\documentclass[12pt]{article}
\usepackage{graphicx}
\usepackage{amsmath}
\usepackage{amsfonts}
\usepackage{enumerate}
\usepackage{latexsym}
\usepackage{epsfig}
\usepackage{float}
\usepackage{color}
\usepackage{amscd}
\usepackage{amssymb}
\usepackage{epstopdf}

\newtheorem{theorem}{Theorem}[section]
\newtheorem{lemma}[theorem]{Lemma}
\newtheorem{coroll}[theorem]{Corollary}

\def\proofbox{\begin{picture}(6.5,6.5)
\put(0,0){\framebox(6.5,6.5){}}\end{picture}}
\newenvironment{proof}{\noindent{\it Proof.\quad}}{\hfill\proofbox}

\begin{document}

\title{Edge-preserving Maps of the Nonseparating Curve Graphs, Curve Graphs and Rectangle Preserving Maps of the Hatcher-Thurston Graphs}
\author{Elmas Irmak}

\maketitle

\renewcommand{\sectionmark}[1]{\markright{\thesection. #1}}

\thispagestyle{empty}
\maketitle
\begin{abstract} Let $R$ be a compact, connected, orientable surface of genus $g$ with $n$ boundary components 
with $g \geq 2$, $n \geq 0$. Let $\mathcal{N}(R)$ be the nonseparating curve graph, $\mathcal{C}(R)$ be the curve graph and $\mathcal{HT}(R)$ be the Hatcher-Thurston graph of $R$. We prove that 
if $\lambda : \mathcal{N}(R) \rightarrow\mathcal{N}(R)$ is an edge-preserving map, then $\lambda$ is induced by a homeomorphism of $R$. We prove that if $\theta : \mathcal{C}(R) \rightarrow \mathcal{C}(R)$ is an edge-preserving map, then $\theta$ is induced by a homeomorphism of $R$. 
We prove that if $R$ is closed and $\tau: \mathcal{HT}(R) \rightarrow\mathcal{HT}(R)$ is a rectangle preserving map, then $\tau$ is induced by a homeomorphism of $R$. We also prove that these homeomorphisms are unique up to isotopy when 
$(g, n) \neq (2, 0)$.\end{abstract}
 
\maketitle

{\small Key words: Mapping class groups, edge-preserving maps, orientable surfaces
	
}
 
\section{Introduction} 
Let $R$ be a compact, connected, orientable surface of genus $g$ with $n$ boundary components. The mapping class group of $R$, $Mod_R$, is defined to be the group of 
isotopy classes of orientation preserving self-homeomorphisms of $R$. The extended mapping class group, $Mod^*_R$, of $R$ is defined to be the group of 
isotopy classes of all self-homeomorphisms of $R$. Abstract simplicial complexes have been studied to get information about the algebraic structure of the mapping class groups. One of these complexes is the complex of curves of $R$. We will define its 1-skeleton called the curve graph $\mathcal{C}(R)$ of $R$ as follows: The vertex set of $\mathcal{C}(R)$  
is the set of isotopy classes of nontrivial simple closed curves on $R$, where nontrivial means 
the curve does not bound a disk and it is not isotopic to a boundary component of $R$. Two vertices in 
$\mathcal{C}(R)$ are connected by an edge if and only if they have pairwise disjoint representatives on $R$. The complex of curves is the flag abstract simplicial complex whose 1-skeleton is $\mathcal{C}(R)$. 
The nonseparating curve graph $\mathcal{N}(R)$ of $R$ is the induced 
subgraph of $\mathcal{C}(R)$ whose vertices are the isotopy classes of nonseparating simple closed curves on $R$. There is a natural action of $Mod^*_R$ on $\mathcal{C}(R)$ by automorphisms. If $[f] \in Mod^*_R$, then $[f]$ induces an automorphism 
$[f]_*: \mathcal{C}(R) \rightarrow \mathcal{C}(R)$ where 
$[f]_*([a]) = [f(a)]$ for every vertex $[a]$ of $\mathcal{C}(R)$. Similarly, each element of $Mod^*_R$ induces an automorphism of $\mathcal{N}(R)$. We recall that a map on a graph is edge-preserving if it sends any two vertices connected by an edge to two vertices connected by an edge. We will first prove that edge-preserving maps of $\mathcal{N}(R)$ are induced by homeomorphisms. 
This means that if $\lambda: \mathcal{N}(R) \rightarrow \mathcal{N}(R)$ is an edge preserving map, then there exists a self-homeomorphism $f$ of $R$ such that $[f]_*(\alpha) = \lambda(\alpha)$ for every vertex $\alpha \in \mathcal{N}(R)$. The main results are the following:

\begin{theorem} \label{A} Let $R$ be a compact, connected, orientable surface of genus $g$ with $n$ boundary components with $g \geq 2$,
$n \geq 0$. If $\lambda :\mathcal{N}(R) \rightarrow \mathcal{N}(R)$ is an edge-preserving map, 
then there exists a homeomorphism $h : R \rightarrow R$ such that $[h]_*(\alpha) = \lambda(\alpha)$
for every vertex $\alpha$ in $\mathcal{N}(R)$ (i.e. $\lambda$ is induced by $h$) and this homeomorphism is unique up to isotopy when $(g, n) \neq (2, 0)$.\end{theorem}
 
\begin{theorem} \label{A2} Let $R$ be a compact, connected, orientable surface of genus $g$ with $n$ boundary components with $g \geq 2$,
$n \geq 0$. If $\theta :\mathcal{C}(R) \rightarrow \mathcal{C}(R)$ is an edge-preserving map, then there exists a homeomorphism $h : R \rightarrow R$ such that $[h]_*(\alpha) = \theta(\alpha)$
for every vertex $\alpha$ in $\mathcal{C}(R)$ (i.e. $\theta$ is induced by $h$) and this homeomorphism is unique up to isotopy when $(g, n) \neq (2, 0)$.\end{theorem}

For a simple closed curve $a$ on $R$, the cut surface $R_a$ (the surface obtained from $R$ by cutting along $a$) is defined as follows: The surface $R_a$ is a compact surface such that there is a homeomorphism $f$ between two of its boundary components, the quotient surface $R_a / (x \sim f(x))$ is homeomorphic to $R$, and the image of these two boundary components 
under the quotient map is $a$ (see section 1.3.1 in \cite{FM}). Cutting $R$ along pairwise disjoint nonseparating simple
closed curves on $R$ can be defined similarly. Let $C= \{c_1, c_2, \ldots, c_g\}$ be a set of pairwise disjoint nonseparating simple
closed curves on $R$ such that $R_C$ (the surface obtained from $R$ by cutting along $C$) is a sphere with
$2g+n$ boundary components. The set $\{ [c_1], [c_2], \cdots,
[c_g] \}$ is called a cut system and denoted by $\langle [c_1], [c_2],
\cdots, [c_g] \rangle$, where $[c_i]$ notation is used for the isotopy class of $c_i$. The Hatcher-Thurston graph, $\mathcal{HT}(R)$, has the cut systems as vertices. This is the 1-skeleton of the Hathcher-Thurston complex which was constructed in \cite{HT} by Hatcher and Thurston in order to find a presentation for the mapping class group. There is a natural action of $Mod^*_R$ on $\mathcal{HT}(R)$ by automorphisms. We will describe how two vertices span an edge in the Hatcher-Thurston graph and prove that rectangle preserving maps of $\mathcal{HT}(R)$ are induced by homeomorphisms  with the following theorem in Section 4. 
 
\begin{theorem} \label{B}  Let $R$ be a closed, connected, orientable surface of genus $g \geq 2$. If $\tau: \mathcal{HT}(R) \rightarrow\mathcal{HT}(R)$ is a rectangle preserving map, then $\tau$ is induced by a homeomorphism of $R$, i.e. there exists a homeomorphism $h : R \rightarrow R$ such that $\tau( \langle [a_1], [a_2], \cdots, [a_g] \rangle) = \langle [h(a_1)], [h(a_2)], \cdots, [h(a_g)] \rangle $ for every vertex $ \langle [a_1], [a_2], \cdots, [a_g] \rangle \in \mathcal{HT}(R)$, and this homeomorphism is unique 
up to isotopy when $g \geq 3$.\end{theorem}

Some results about complex of curves on compact, connected, orientable surfaces and their applications on mapping class groups are as follows: In \cite{Iv1} Ivanov proved that every automorphism of the complex of curves is induced by a homeomorphism of $R$ if the genus is at least two, and he used this result to give a classification of isomorphisms between any two finite index subgroups of the extended mapping class group of $R$. These 
results were extended for surfaces of genus zero 
and one by Korkmaz in \cite{K1} and indepedently by Luo in \cite{L}. 
The author studied superinjective simplicial maps of the complex of curves which are defined as follows:
Let $\alpha$, $\beta$ be two vertices in the complex of curves. The geometric intersection number $i(\alpha, \beta)$ is defined to be the minimum number of points of $a \cap b$ where $a \in \alpha$ and $b \in \beta$. A simplicial map from the complex of curves to itself is called superinjective if it satisfies the following:
if $\alpha, \beta$ are two vertices and $i(\alpha,\beta) \neq 0$, then
$i(\lambda(\alpha),\lambda(\beta)) \neq 0$. A superinjective simplicial map is a simplicial map that maps any two vertices that do not span an edge to vertices that do not span an edge. So, superinjective maps are also edge-preserving maps of the complement graph of $\mathcal{C}(R)$. 
The author proved that the superinjective simplicial maps of the complex of curves on a compact, connected, orientable surface are induced by homeomorphisms if the genus is at least two, and using this result she gave a classification of injective homomorphisms from finite index subgroups to the whole extended 
mapping class group in \cite{Ir1}, \cite{Ir2}, \cite{Ir3}. Behrstock-Margalit and Bell-Margalit proved these results for lower genus cases in \cite{BM} and in \cite{BeM}. Shackleton proved that locally injective simplicial maps of the curve complex are induced by homeomorphisms in \cite{Sh}. Aramayona-Leininger proved the existence of
finite rigid sets (sets such that every locally injective simplicial map from this set to the curve complex is induced by a homeomorphism) in the curve complex in \cite{AL1}. By using this result, they also proved that there is an exhaustion of the curve complex by a sequence of finite rigid sets in \cite{AL2}. In \cite{IlK} Ilbira-Korkmaz proved the existence of finite rigid subcomplexes in the curve complex for nonorientable surfaces when $g+n \neq 4$.
The author proved that there is an exhaustion of the curve complex by a sequence of finite superrigid sets (sets such that every superinjective simplicial map from this set to the curve complex is induced by a homeomorphism) on compact, connected, nonorientable surfaces in \cite{Ir10}. She also proved recently that there is an exhaustion of the curve complex by a sequence of finite rigid sets on compact, connected, nonorientable surfaces in \cite{Ir11}.  
In \cite{H1} Hern\'andez proved that there is a finite set of curves whose union of iterated rigid expansions gives the whole curve complex. In \cite{H2} Hern\'andez proved that if $S_1$ and $S_2$ are orientable surfaces of finite topological type such that $S_1$ has genus at least 3 and the complexity of $S_1$ is an upper bound of the complexity of $S_2$, and $\theta : \mathcal{C}(S_1) \rightarrow \mathcal{C}(S_2)$ is an edge-preserving map, 
then $S_1$ is homeomorphic to $S_2$ and $\theta$ is induced by a homeomorphism. 

Our new results improve/extend the following results on compact, connected, orientable surfaces: In \cite{Ir3} the author proved that superinjective simplicial maps of the nonseparating curve complexes are induced by homeomorphisms when the genus is at least two and the number of boundary 
components is at most $g-1$. Since superinjective simplicial maps of the nonseparating curve complexes are edge-preserving on the nonseparating curve graphs, Theorem 1.1 improves this result when 
$g=2$, $n \geq 0$. Our Theorem 1.2 gives a new proof of Hern\'andez's result about edge preserving maps of $\mathcal{C}(R)$ in our setting and covers $g=2$, $n \geq 0$ case as well. Since superinjective simplicial maps and automorphisms of the complex of curves are both edge-preserving on the curve graphs, Theorem 1.2 improves the results of the author given in \cite{Ir1}, \cite{Ir2}, \cite{Ir3} and also the results of Ivanov, Korkmaz and Luo about the automorphisms of the curve complex given in \cite{Iv1}, \cite{K1}, \cite{L} respectively.
Since automorphisms of the Hatcher-Thurston complex are rectangle preserving on the Hatcher-Thurston graph, Theorem 1.3 improves the result given by Irmak-Korkmaz in \cite{IrK} that automorphisms of the Hatcher-Thurston complex are induced by homeomorphisms when $g \geq 2$ and $n = 0$. In \cite{H3} Hern\'andez proved that an edge-preserving and alternating map 
from the Hatcher-Thurston graph to itself is induced by a homeomorphism of $R$ when $R$ is closed and $g \geq 3$. With our Theorem 1.3 we prove this result when $R$ is closed and $g \geq 2$ since an edge-preserving and alternating map is equivalent to a rectangle preserving map. 
  
Irmak-Paris proved that superinjective simplicial maps of the two-sided curve complex are induced by homeomorphisms on compact, connected, nonorientable surfaces when the genus is at least 5 in \cite{IrP1}. They used this result to classify injective homomorphisms from finite index subgroups to the whole mapping class group on compact, connected, nonorientable surfaces when the genus is at least 5 in \cite{IrP2}. In this paper we use some techniques given by Irmak-Paris in \cite{IrP1}  
and some techniques given by Aramayona-Leininger in \cite{AL2}. 

We note that the author also proved that if $g=0, n \geq 5$ or $g=1, n \geq 3$, and $\lambda : \mathcal{C}(R) \rightarrow\mathcal{C}(R)$ is an edge-preserving map, then $\lambda$ is induced by a homeomorphism of $R$, and this homeomorphism is unique up to isotopy in \cite{Ir4}.

\section{Edge-preserving Maps of the Nonseparating Curve Graphs}

In this section we will always assume that $g \geq 2$, $n \geq 0$ and $\lambda :\mathcal{N}(R) \rightarrow \mathcal{N}(R)$ is an edge-preserving map. Throughout the paper we will use that if $\alpha_1, \alpha_2, \cdots, \alpha_m$ are isotopy classes of essential simple closed curves on $R$, then there exist $a_1 \in \alpha_1, a_2 \in \alpha_2, \cdots, a_m \in \alpha_m$ such that the number of intersection of the curves $a_i$ and $a_j$ equals the geometric intersection number of $\alpha_i$ and $\alpha_j$ for all $i, j= 1, 2, \cdots, m$. This works by giving the surface $R$ a hyperbolic metric (since $R$ has negative Euler characteristic) and choosing the unique closed geodesic in each isotopy class. We will call $a_1, a_2, \cdots, a_m$ as the representatives in minimal position of $\alpha_1, \alpha_2, \cdots, \alpha_m$ respectively.

Let $P$ be a set of pairwise disjoint nontrivial simple closed curves on $R$. The set $P$ is
called a pair of pants decomposition of $R$, if $R_P$ (the surface obtained from $R$ by cutting along $P$) is disjoint union of genus zero surfaces with three boundary components, pairs of pants. A pair of pants of a pants
decomposition is the image of one of these connected components under the quotient map $q:R_P \rightarrow R$. 
We note that if a set $P$ is a pair of pants decomposition consisting of nonseparating simple closed curves on $R$, then $[P] = \{ [x]: x \in P\}$ forms a maximal complete subgraph in the nonseparating curve graph on $R$. The set $[P]$ has exactly $3g-3+n$ elements.
 
\begin{lemma}
\label{inj} If $\mathcal{A}$ is the set of vertices of a complete subgraph of $\mathcal{N}(R)$, then $\lambda$ is injective on $\mathcal{A}$.
\end{lemma}

\begin{proof} The result follows since $\lambda$ is edge-preserving.\end{proof}

\begin{coroll} \label{pd-inj} Let $P$ be a pants decomposition consisting of nonseparating simple closed curves on $R$. 
A set of pairwise disjoint representatives of $\lambda([P])$ is a pants decomposition on $R$.\end{coroll}

\begin{proof} The proof follows from Lemma \ref{inj}.\end{proof}
 
\begin{lemma}
\label{0} Suppose $g \geq 2, n \geq 0$. Let $\alpha, \beta$ be two vertices of $\mathcal{N}(R)$. If $i(\alpha, \beta) = 1$, then $\lambda(\alpha) \neq \lambda(\beta)$.\end{lemma}
 
\begin{figure}[t]
	\begin{center}  \epsfxsize=2.7in \epsfbox{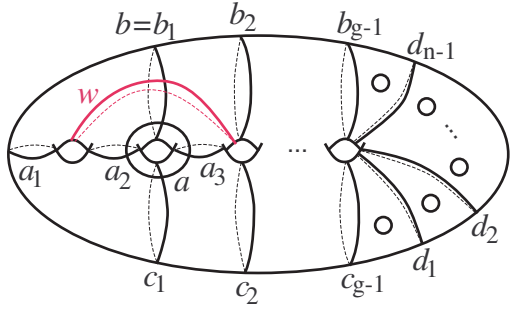}
		
		\caption {Curves $a, b$ intersecting once} \label{fig0-2}
	\end{center}
\end{figure}
 
\begin{proof} {\bf Case (i):} Suppose that $(g, n) \neq (2, 0)$. Let $a$ and $b$ be representatives in minimal position of $\alpha$ and  $\beta$ respectively. Let $b_1=b$ and complete $a, b$ to a curve configuration as shown in Figure \ref{fig0-2}. The set $P = \{a_1, a_2, \cdots, a_{g}, b_1, b_2, \cdots, b_{g-1}, 
c_1, c_2, \cdots, c_{g-1}, d_1, d_2, \cdots, d_{n-1}\}$ is a pants decomposition on $R$. By Corollary \ref{pd-inj}, we know that a set of pairwise disjoint representatives of $\lambda([P])$ is a pants decomposition on $R$. 
Since $[w]$ is connected by an edge to $[x]$ for each curve $x \in P \setminus \{b\}$ and $\lambda$ is edge-preserving, $\lambda([w])$ is connected by an edge to 
$\lambda([x]$ for each $x \in P \setminus \{b\}$. These imply that either (i) $\lambda([w]) = \lambda(\beta)$ or (ii) $i( \lambda([w]), \lambda(\beta)) \neq 0$, and $i( \lambda([w]), \lambda([x])) = 0$ for each $x \in P \setminus \{b\}$. Since $[w]$ is connected by an edge to   
$[a]$, we must have $\lambda([w]) \neq \lambda(\alpha)$ and $i( \lambda([w]), \lambda(\alpha)) = 0$. Now it is easy to see that if $\lambda(\alpha) = \lambda(\beta)$ we get a contradiction with the above information. Hence, $\lambda(\alpha) \neq \lambda(\beta)$.  

{\bf Case (ii):} Suppose that $(g, n) = (2, 0)$.  Let $a, x$ be representatives in minimal position of $\alpha, \beta$ respectively. We complete $a, x$ to a curve configuration $\{a, b, c, x, y, l\}$ as shown in Figure \ref{fig0-1} (i). 
The set $P = \{a, b, c\}$ is a pants decomposition on $R$. By using Corollary \ref{pd-inj}, we see that $\lambda([P])$ forms a maximal complete subgraph in the nonseparating curve graph. Suppose that $\lambda([a]) = \lambda([x])$. We will first prove some implications: 

Claim 1: $\lambda([y]) \neq \lambda([c])$.

Proof of Claim 1: The vertex $[l]$ is connected by an edge to each of
$[x], [b], [y]$. Since $\lambda$ is edge-preserving, $\lambda([l])$ is connected by an edge to each of
$\lambda([x]), \lambda([b]), \lambda([y])$. If $\lambda([y]) = \lambda([c])$, then since 
$\lambda([a]) = \lambda([x])$, we would have that $\lambda([l])$ is connected by an edge to each of
$\lambda([a]), \lambda([b]), \lambda([c])$, but $\{\lambda([a]), \lambda([b]), \lambda([c])\}$ forms a maximal complete subgraph in the nonseparating curve graph, so we get a contradiction. Hence, $\lambda([y]) \neq \lambda([c])$.

Claim 2: $\lambda([y]) \neq \lambda([b])$.
 
Proof of Claim 2: We consider the curve $w$ shown in Figure \ref{fig0-1} (ii). Since $[w]$ is connected by an edge to each of
$[x], [b]$ and $\lambda$ is edge-preserving, we have $\lambda([w])$ is connected by an edge to each of
$\lambda([x]), \lambda([b])$. Since 
$\lambda([a]) = \lambda([x])$, we see that $\lambda([w])$ is connected by an edge to each of
$\lambda([a]), \lambda([b])$, and $\{\lambda([a]), \lambda([b]), \lambda([w])\}$ forms a maximal complete subgraph in the nonseparating curve graph of $R$. Since $[t]$ is connected by an edge to each of
$[a], [y], [w]$ and $\lambda$ is edge-preserving, $\lambda([t])$ is connected by an edge to each of
$\lambda([a]), \lambda([y]), \lambda([w])$. If $\lambda([y]) = \lambda([b])$, then $\lambda([t])$ is connected by an edge to each of $\lambda([a]), \lambda([b]), \lambda([w])$. This gives a contradiction since $\{\lambda([a]), \lambda([b]), \lambda([w])\}$ forms a maximal complete subgraph in the nonseparating curve graph. Hence, $\lambda([y]) \neq \lambda([b])$.

\begin{figure}
	\begin{center}
		\hspace{0cm} \epsfxsize=1.82in \epsfbox{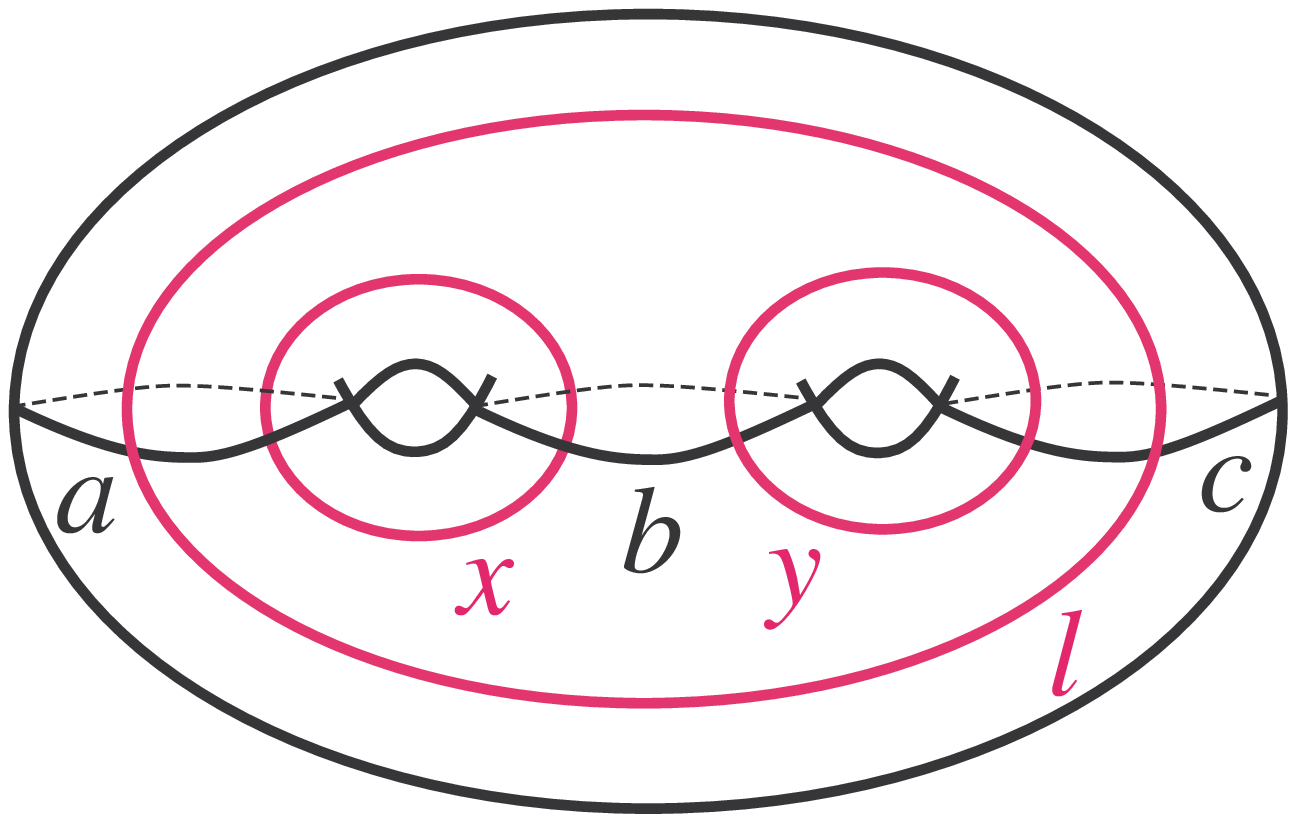} \hspace{0.5cm}  \epsfxsize=1.82in \epsfbox{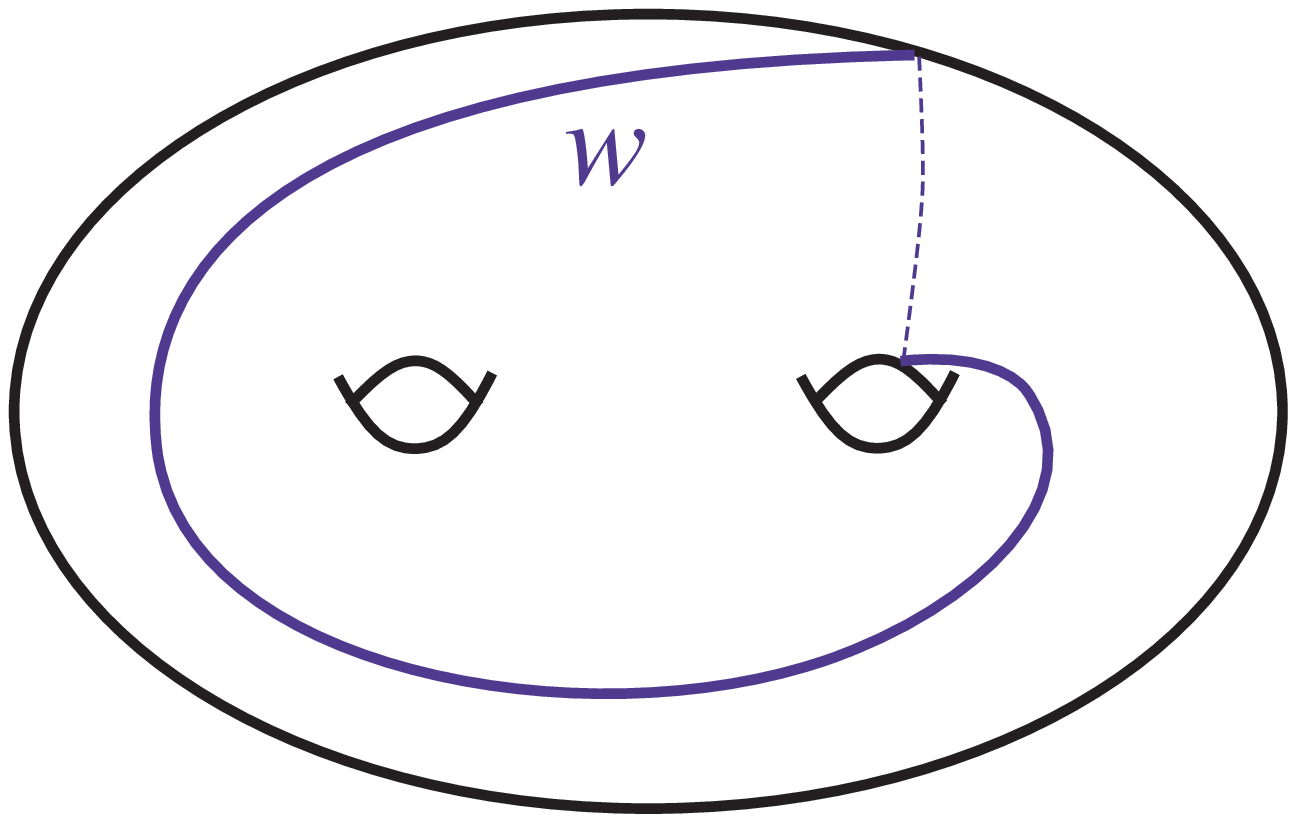} \hspace{0.5cm}  \epsfxsize=1.82in \epsfbox{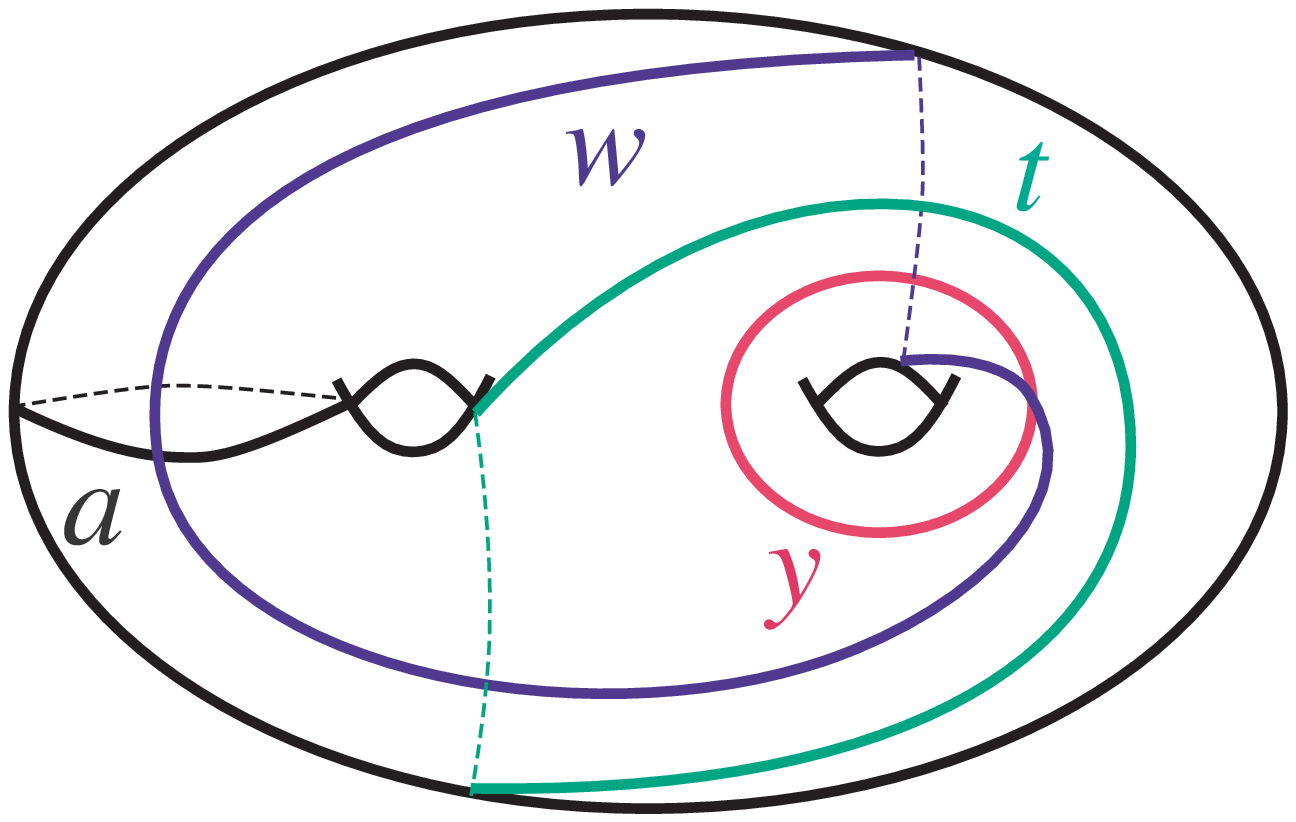}  
		
		\hspace{0cm} (i) \hspace{4.5cm} (ii) \hspace{4.5cm} (iii)
		
		\epsfxsize=1.82in \epsfbox{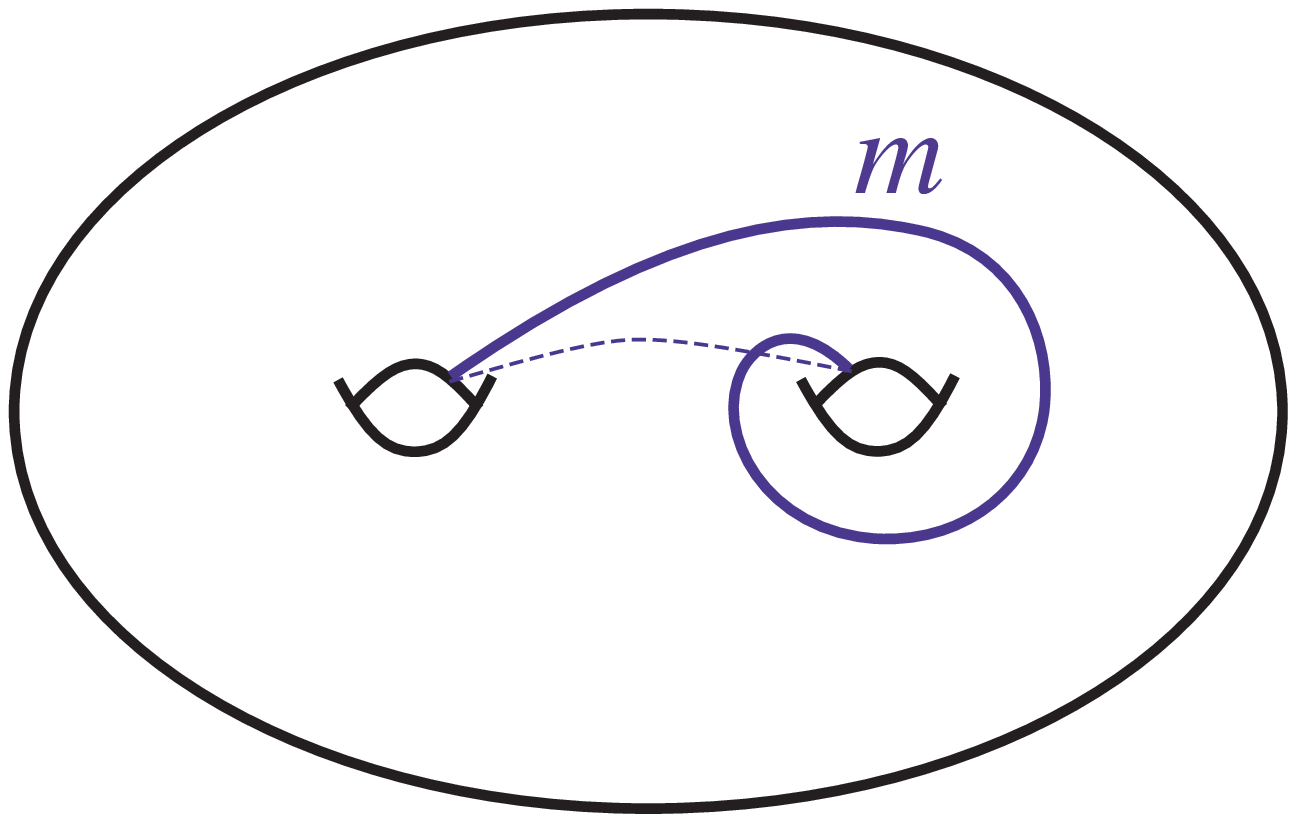}  \hspace{0.5cm} \epsfxsize=1.82in \epsfbox{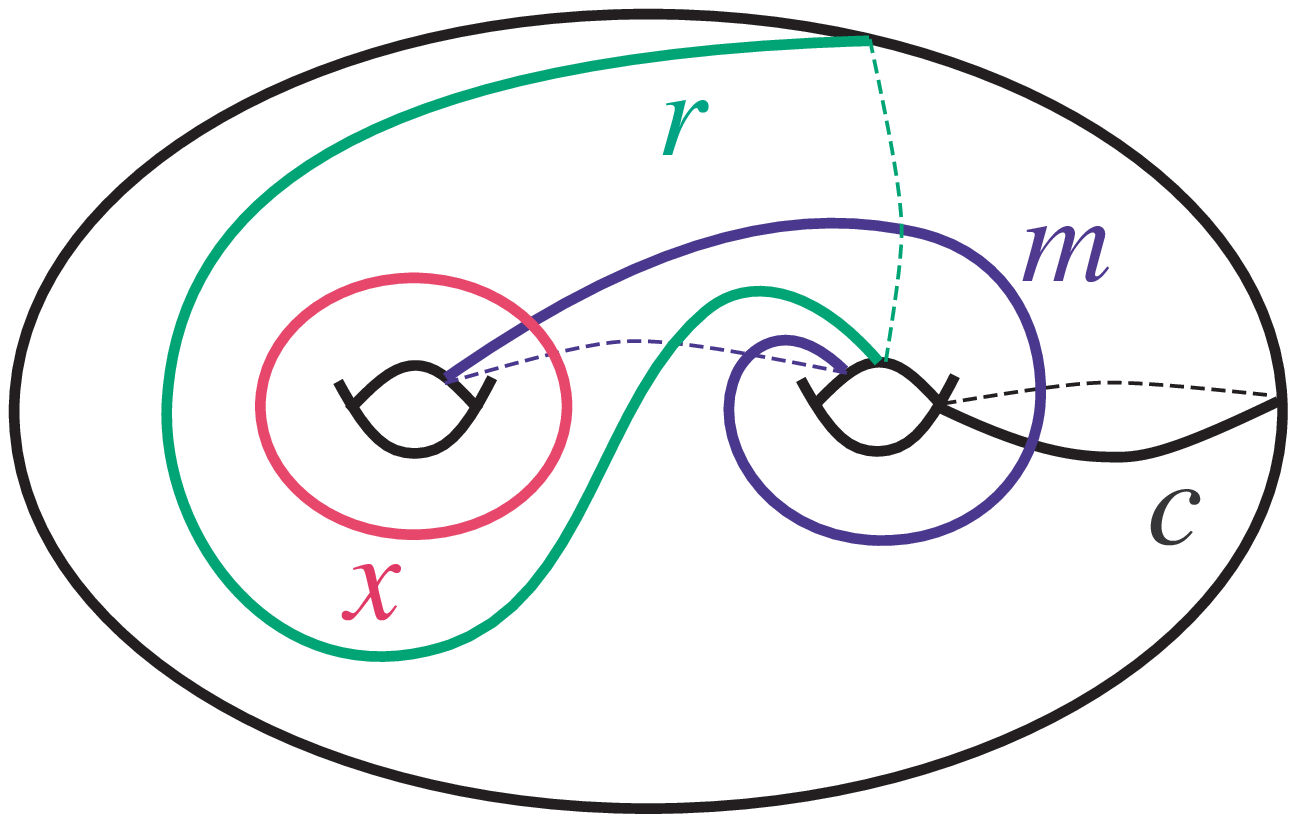}  \hspace{0.5cm}   \epsfxsize=1.82in \epsfbox{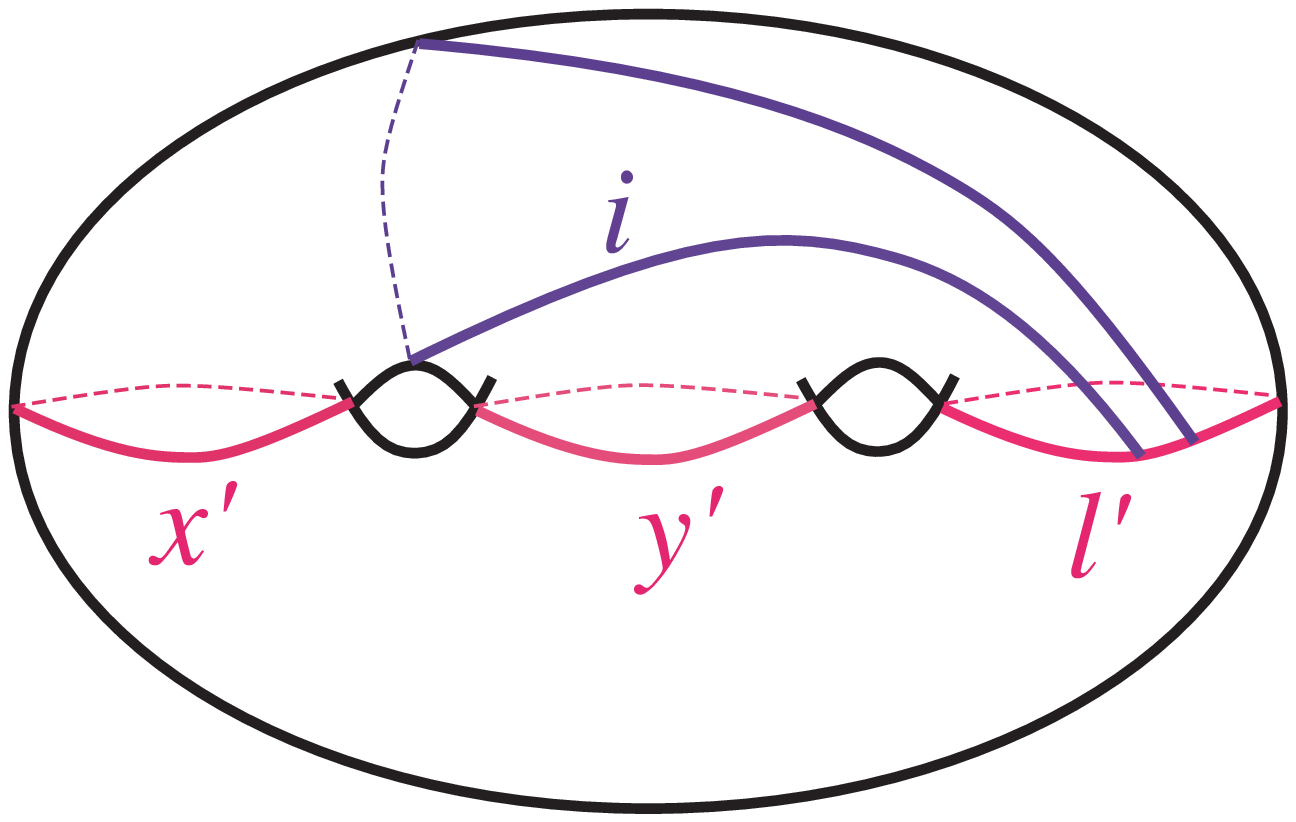}  
		
		\hspace{0cm} (iv) \hspace{4.5cm} (v) \hspace{4.5cm} (vi)
		
		\caption {Curves $a, x$ intersecting once} \label{fig0-1}
	\end{center}
\end{figure}

Claim 3: $i(\lambda([y]), \lambda([b])) \neq 0$ and $i(\lambda([y]), \lambda([c])) \neq 0$. 

Proof of Claim 3: We know that $\{\lambda([a]), \lambda([b]), \lambda([c])\}$ forms a maximal complete subgraph in the nonseparating curve graph and $i(\lambda([y]), \lambda([a])) =0$. Using these information and Claim 1 and Claim 2 we see that either 
$i(\lambda([y]), \lambda([b])) \neq 0$ or $i(\lambda([y]), \lambda([c])) \neq 0$. 

(i) Suppose $i(\lambda([y]), \lambda([c])) \neq 0$ and $i(\lambda([y]), \lambda([b])) = 0$. Since we know that $\lambda([y]) \neq \lambda([b])$ by Claim 2, we can see that $\{\lambda([a]), \lambda([b]), \lambda([y])\}$ forms a maximal complete subgraph in the nonseparating curve graph of $R$. 
Since $[l]$ is connected by an edge to each of
$[x], [b], [y]$ and $\lambda$ is edge-preserving, $\lambda([l])$ is connected by an edge to each of
$\lambda([x]), \lambda([b]), \lambda([y])$. Since $\lambda([x]) = \lambda([a])$, we see that $\lambda([l])$ is connected by an edge to each of
$\lambda([a]), \lambda([b]), \lambda([y])$. That gives a contradiction since $\{\lambda([a]), \lambda([b]), \lambda([y])\}$ forms a maximal complete subgraph in the nonseparating curve graph of $R$. Hence we cannot have $i(\lambda([y]), \lambda([c])) \neq 0$ and $i(\lambda([y]), \lambda([b])) = 0$.

(ii) Suppose $i(\lambda([y]), \lambda([b])) \neq 0$ and $i(\lambda([y]), \lambda([c])) = 0$. Let $a', b', c', l', y'$ be representatives in minimal position of $\lambda([a]), \lambda([b]), \lambda([c]), \lambda([l]), \lambda([y])$ respectively. We know that 
$a', b', c'$ are all nonseparating curves that form a pants decomposition on $R$. Since 
$y'$ intersects $b'$, $y'$ is disjoint from $a'$ and $c'$, and $l'$ is disjoint from each of $y'$ and $b'$, we see that $l'$ is isotopic to either $a'$ or $c'$ (since $y'$ and $b'$ fill $R_{a' \cup c'}$). So, either $\lambda([l]) = \lambda([a])$ or $\lambda([l]) = \lambda([c])$. Since $\lambda([a]) = \lambda([x])$, and $\lambda([l])$ is connected by an edge to $\lambda([x])$, we cannot have $\lambda([l]) = \lambda([a])$. Hence, we must have $\lambda([l]) = \lambda([c])$, but this gives a contradiction as follows: Consider the curves $m, r$ given in Figure \ref{fig0-1} (v). Since $[m]$ is connected by an edge to each of $[a], [l]$ and $\lambda$ is edge-preserving, we see that $\lambda([m])$ is connected by an edge to each of
$\lambda([a]), \lambda([l])$. Since $\lambda([x]) = \lambda([a])$ and $\lambda([l]) = \lambda([c])$, we see $\lambda([m])$ is connected by an edge to each of
$\lambda([x]), \lambda([c])$, and $\{\lambda([m]), \lambda([x]), \lambda([c])\}$ forms a maximal complete subgraph in the nonseparating curve graph of $R$. But 
since $[r]$ is connected by an edge to each of $[m], [x], [c]$ we see that $\lambda([r])$ is connected by an edge to each of
$\lambda([m]), \lambda([x]), \lambda([c])$. This gives a contradiction since $\{\lambda([m]), \lambda([x]), \lambda([c])\}$ forms a maximal complete subgraph in the nonseparating curve graph of $R$. Hence, we cannot have $i(\lambda([y]), \lambda([b])) \neq 0$ and $i(\lambda([y]), \lambda([c])) = 0$.

We conclude that $i(\lambda([y]), \lambda([b])) \neq 0$ and $i(\lambda([y]), \lambda([c])) \neq 0$. 

Claim 4: $i(\lambda([l]), \lambda([c])) \neq 0$, $i(\lambda([l]), \lambda([a])) = 0$, $i(\lambda([l]), \lambda([b])) = 0$.  

Proof of Claim 4: Since $[l]$ is connected by an edge to $[x]$ we know that $\lambda([l])$ is connected by an edge to
$\lambda([x])$. Then since $\lambda([a]) = \lambda([x])$, we have $\lambda([l])$ is connected by an edge to
$\lambda([a])$. So, $i(\lambda([l]), \lambda([a])) = 0$. Since $[l]$ is connected by an edge to $[b]$ we know that $\lambda([l])$ is connected by an edge to
$\lambda([b])$. So, $i(\lambda([l]), \lambda([b])) = 0$. Since $\{\lambda([a]), \lambda([b]), \lambda([c])\}$ forms a maximal complete subgraph in the nonseparating curve graph of $R$ and $i(\lambda([l]), \lambda([a])) = 0$, $i(\lambda([l]), \lambda([b])) = 0$, we see that either $\lambda([l]) = \lambda([c])$ or $i(\lambda([l]), \lambda([c])) \neq 0$. Since $\lambda([l])$ is connected by an edge to
$\lambda([y])$, we have $i(\lambda([l]), \lambda([y])) = 0$. By Claim 3, we know $i(\lambda([y]), \lambda([c])) \neq 0$. Hence, we cannot have $\lambda([l]) = \lambda([c])$. So, 
$i(\lambda([l]), \lambda([c])) \neq 0$. This completes the proof of Claim 4. 

Now to finish the proof for Case (ii) we consider the set $\{\lambda([x]), \lambda([y]), \lambda([l])\}$. Let $x'$ be a representative of $\lambda([x])$ which minimally intersects each of $a', b', c', y', l'$. The curves $x', y', l'$ form a pants decomposition on $R$. It is easy to see that $c'$ is disjoint from $x'$, and $b'$ is disjoint from $l'$. Since $\lambda([b])$ is connected by an edge to $\lambda([a])$ and $\lambda([a]) = \lambda([x])$, we see that $b'$ is disjoint from $x'$.
By Claim 3 and Claim 4, $c'$ intersects each of $l', y'$, and $b'$ intersects $y'$. So, we have (i) $b'$ is disjoint from each of $l', x'$, (ii) $b'$ intersects $y'$, (iii) $c'$ is disjoint from $x'$, and (iv) $c'$ intersects each of $l', y'$. 
All this information will give a contradiction as follows: If $c'$ and $l'$ intersect only once then $b'$ would have to be a separating curve since $b'$ is disjoint from $c' \cup l' \cup x'$ and $x'$ is disjoint from $c' \cup l'$. That gives a contradiction. Suppose $c'$ and $l'$ intersect more than once. If there exists an arc $i$ of $c'$ starting and ending on $l'$ in a pair of pants bounded by $x', y', l'$ (see Figure \ref{fig0-1} (vi)) then we would get a contradiction as $b'$ is disjoint from $c' \cup l' \cup x'$ and $b'$ intersects $y'$. So, all the arcs of $c'$ in any of the pair of pants bounded by $x', y', l'$ either 
must connect $y'$ to $l'$ or must start and end on $y'$ in that pair of pants. All this information would imply that either $b'$ is a separating curve or $b'$ cannot exists as a nonseparating curve which is disjoint from $c' \cup l' \cup x'$ while intersecting $y'$, as there would not be any nontrivial simple closed curve disjoint from $c' \cup l' \cup x'$ in $R$. 
So, the intersection information we have between the curves $a', b', c', x', y', l'$ gives us a contradiction when we assume that 
$\lambda([a]) = \lambda([x])$. So, $\lambda([a]) \neq \lambda([x])$. Hence, $\lambda(\alpha) \neq \lambda(\beta)$.\end{proof}\\

\begin{figure}[t]
	\begin{center}
		\hspace{0.cm}  \epsfxsize=2.7in \epsfbox{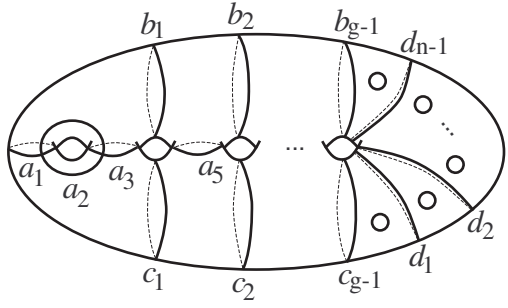} 
		\caption {Curves $a_1, a_2, a_3$ forming a chain} \label{figure-0}
	\end{center}
\end{figure}

A set of simple closed nonseparating curves $\{a_1, a_2, \cdots, a_k\}$ is called a chain if $i([a_i]), [a_{i+1}])=1$ and $i([a_i]), [a_{j}])=0$ for $|i-j|>1$.    

\begin{lemma}
\label{1} Suppose $g \geq 2, n \geq 0$. Let $\alpha_1, \alpha_2, \alpha_3$ be three vertices in 
$\mathcal{N}(R)$ such that $i(\alpha_1, \alpha_2) = 1$,  $i(\alpha_2, \alpha_3) = 1$,  $i(\alpha_1, \alpha_3) = 0$. 
Then $i(\lambda(\alpha_2), \lambda(\alpha_1)) \neq 0$ or $i(\lambda(\alpha_2), \lambda(\alpha_3)) \neq 0$.\end{lemma}

\begin{proof} Let $a_1, a_2, a_3$ be representatives in minimal position of $\alpha_1, \alpha_2, \alpha_3$ respectively. We see that $\{a_1, a_2, a_3\}$ is a chain. We complete $a_1, a_2, a_3$ to a curve configuration as shown in Figure \ref{figure-0}. The set $P = \{a_1, a_3, \cdots, a_{2g-1},$ $b_1, b_2, \cdots, b_{g-1}, 
c_1, c_2, \cdots, c_{g-1}, d_1, d_2, \cdots, d_{n-1}\}$ is a pants decomposition on $R$. By using that $\lambda([P])$ forms a maximal complete subgraph in the nonseparating curve graph of $R$ we can see the following:
Since $a_2$ is disjoint from and nonisotopic to each curve in $P \setminus \{a_1, a_3\}$, $\lambda$ is edge-preserving and $\lambda(\alpha_1) \neq \lambda(\alpha_2)$, $\lambda(\alpha_2) \neq \lambda(\alpha_3)$ by Lemma \ref{0}, we see that either $i(\lambda(\alpha_2), \lambda(\alpha_1)) \neq 0$ or $i(\lambda(\alpha_2), \lambda(\alpha_3)) \neq 0$.\end{proof}\\
  
\begin{figure}
\begin{center}
\hspace{0.cm} \epsfxsize=2.7in \epsfbox{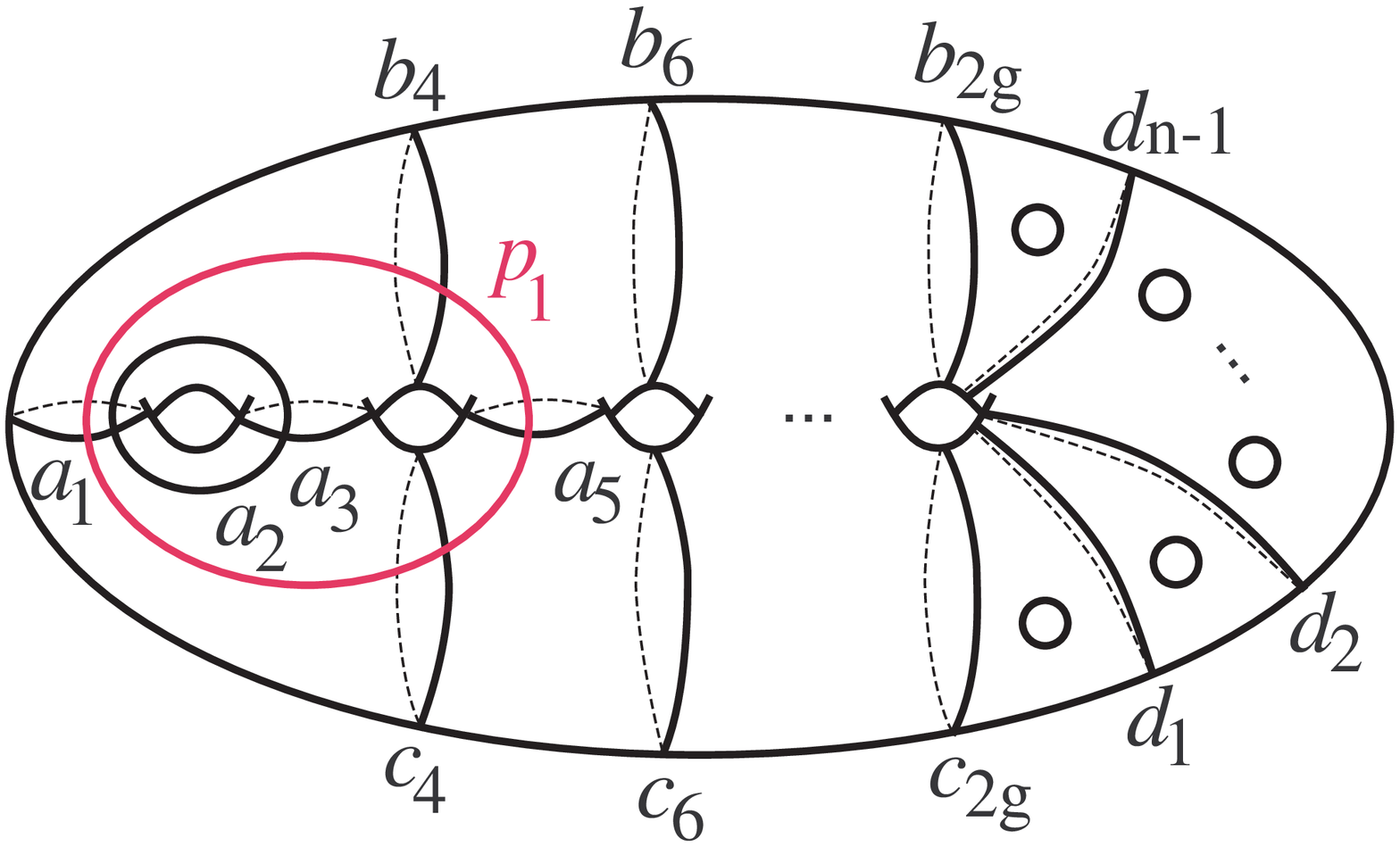}  \hspace{1.2cm}  \epsfxsize=2.2in \epsfbox{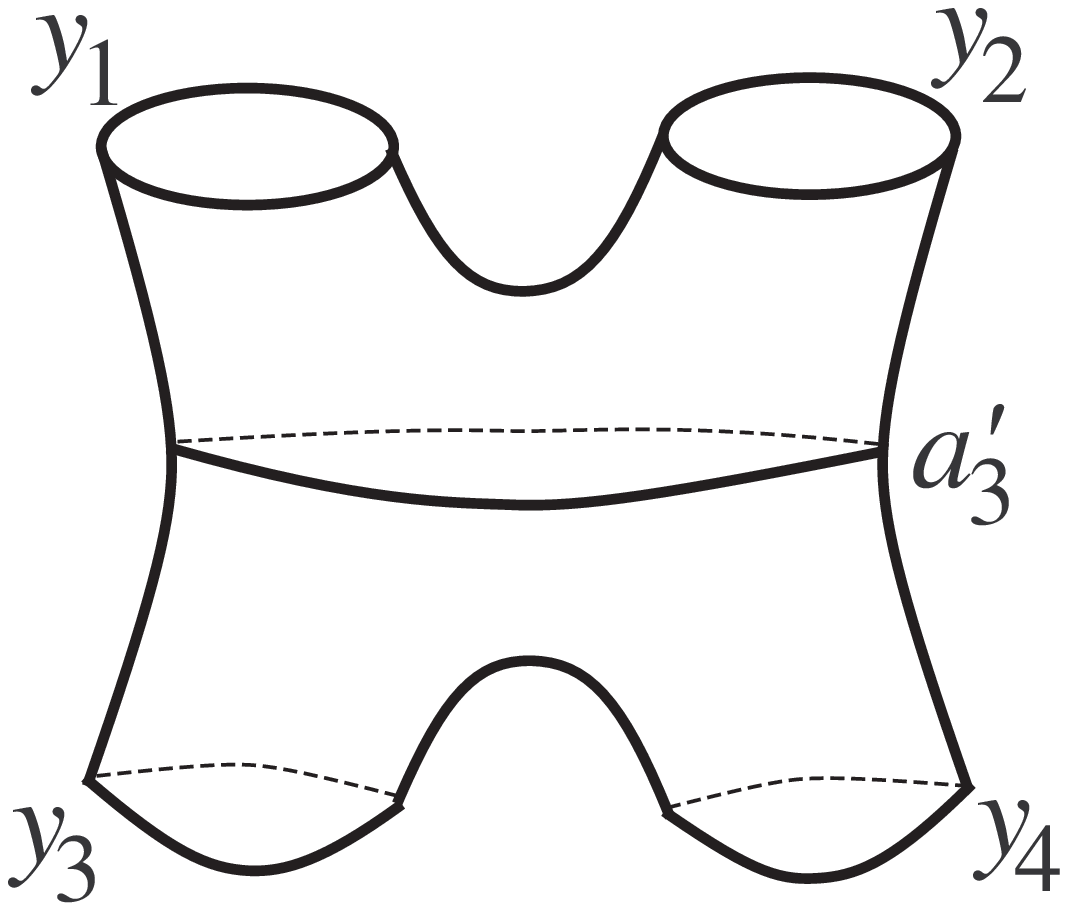} 

\hspace{-0.2cm} (i) \hspace{6.1cm} (ii)
 
\caption {Curves $a_1, a_2$ intersecting once} \label{figure-0-2}
\end{center}
\end{figure}
 
We will say that two simple closed curves have essential (geometric) intersection if their isotopy classes have nonzero geometric intersection.   
 
\begin{lemma}
\label{2} Suppose $g \geq 2, n \geq 0$. Let $\alpha_1, \alpha_2$ be two vertices of 
$\mathcal{N}(R)$ such that $i(\alpha_1, \alpha_2) = 1$. Then $i(\lambda(\alpha_1), \lambda(\alpha_2)) \neq 0$.\end{lemma}
 
\begin{proof} Let $a_1, a_2$ be representatives in minimal position of $\alpha_1, \alpha_2$ respectively. We complete $a_1, a_2$ to a curve configuration as shown in Figure \ref{figure-0-2} (i). The set $P = \{a_1, a_3, \cdots,$ $ a_{2g-1}, 
b_4, b_6, \cdots, b_{2g}, 
c_4, c_6, \cdots, c_{2g}, d_1, d_2, \cdots, d_{n-1}\}$ is a pants decomposition on $R$. Let $P'$ be a set of pairwise disjoint representatives of $\lambda([P])$. The set $P'$ is a pants decomposition on $R$. 
We know that $i(\lambda(\alpha_2), \lambda(\alpha_1)) \neq 0$ or $i(\lambda(\alpha_2), \lambda(\alpha_3)) \neq 0$ by Lemma \ref{1}. 
Suppose $i(\lambda(\alpha_1), \lambda(\alpha_2)) = 0$. Then $\lambda(\alpha_1)$ and 
$\lambda(\alpha_2)$ must have nonisotopic, pairwise disjoint representatives since $\lambda(\alpha_2) \neq \lambda(\alpha_1)$ by Lemma \ref{0} and 
we also know that $i(\lambda(\alpha_2), \lambda(\alpha_3)) \neq 0$ by Lemma \ref{1}.
Let $a_3'$ be a representative of $\lambda(\alpha_3)$ that is in $P'$. Let $a_2'$ be a representative of $\lambda(\alpha_2)$ that has minimal intersection with $P'$. 
The curve $a_3'$ belongs to two pair of pants $A, B$ in $P'$. Let $y_1, y_2, y_3, y_4$ be the other boundary components of these two pair of pants as shown in Figure \ref{figure-0-2} (ii). We note that some of these boundary components may be equal to each other or may be nonessential.   
Since $a_2$ is disjoint from and nonisotopic to all the curves in $P \setminus \{a_1, a_3\}$, and $\lambda$ is edge-preserving, we have $i(\lambda(\alpha_2), \lambda([x])) = 0$ for all $x \in P \setminus \{a_1, a_3\}$ and $\lambda(\alpha_2) \neq \lambda([x])$
for all $x \in P \setminus \{a_1, a_3\}$. Since we also know that $i(\lambda(\alpha_2), \lambda(\alpha_3)) \neq 0$, we see that
$a'_2$ is in $A \cup B$, $a_2'$ is disjoint from $y_1, y_2, y_3, y_4$ and $a_2'$ intersects $a'_3$ essentially (that means $i([a'_2], [a'_3]) \neq 0$). 
We will get a contradiction as follows: For every pair of curves $x, y$ in $\{P \setminus \{a_3\}\}$ we can find a curve $z$ that intersects $x$ and $y$ once essentially and disjoint from $a_2 \cup a_3$ (for example the curve $p_1$ for the pair $a_1, b_4$).  
So, for every pair $x, y$ in $\{y_1, y_2, y_3, y_4\}$ we can find a curve $z$ that intersects both of them once essentially and disjoint from $a_2 \cup a_3$. 
By Lemma \ref{1} we must have $i(\lambda([z]), \lambda([x])) \neq 0$ or $i(\lambda([z]),  \lambda([y])) \neq 0$. Let $z'$ be a representative of $\lambda([z])$ that has minimal intersection with $P' \cup \{a'_2\}$. Then, $z'$ must intersect some $y_i$ essentially but this is not possible since $i(\lambda([z]), \lambda([a_2])) = 0$ and $i(\lambda([z]),  \lambda([a_3])) = 0$, and so $z'$ is disjoint from $a'_2 \cup a'_3$. This gives a contradiction. Hence, $i(\lambda(\alpha_1), \lambda(\alpha_2)) \neq 0$.\end{proof}\\

\begin{figure} \begin{center}
\hspace{-0.4cm} \epsfxsize=2.99in \epsfbox{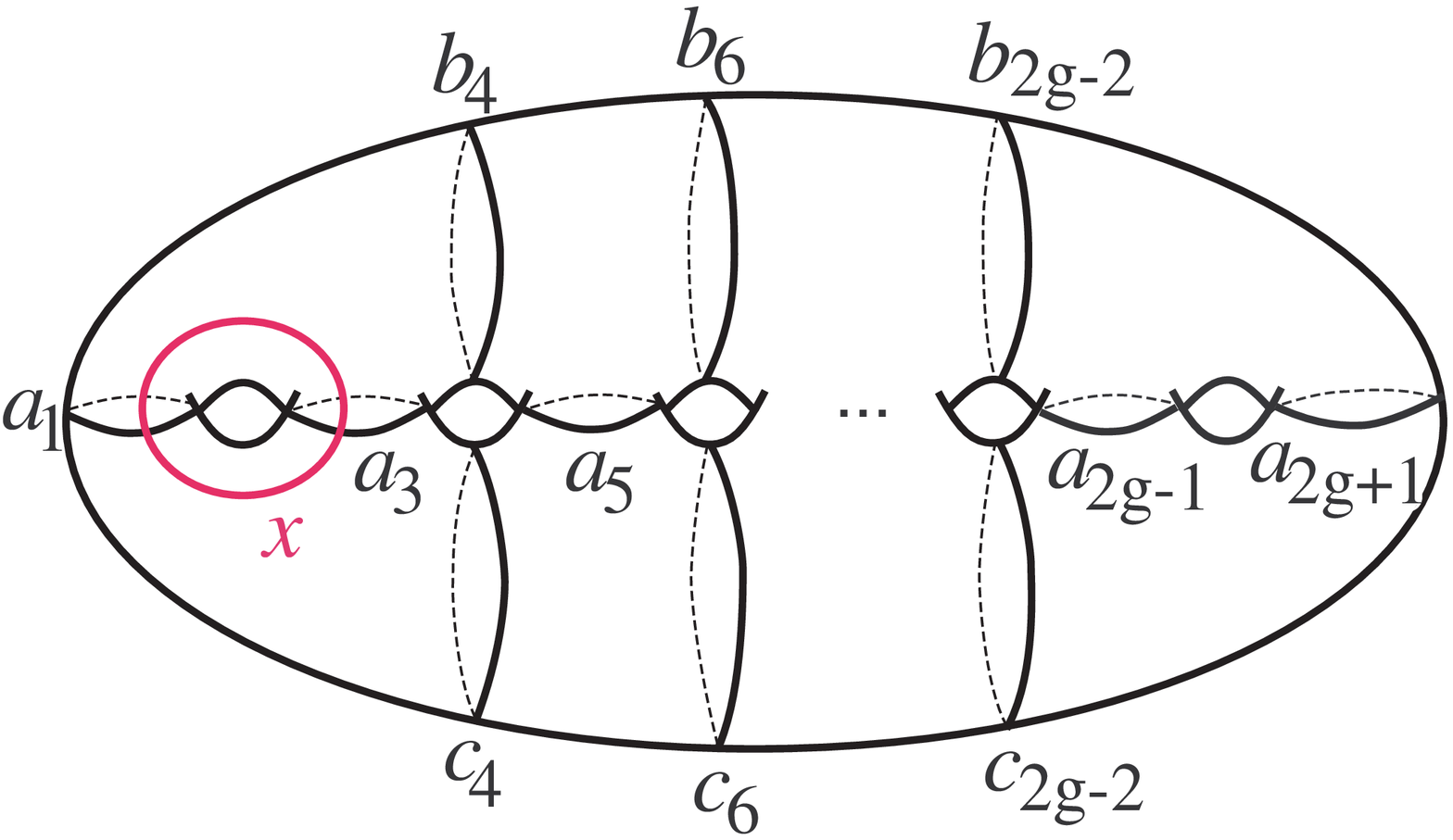} \hspace{-1.2cm} \epsfxsize=2.99in \epsfbox{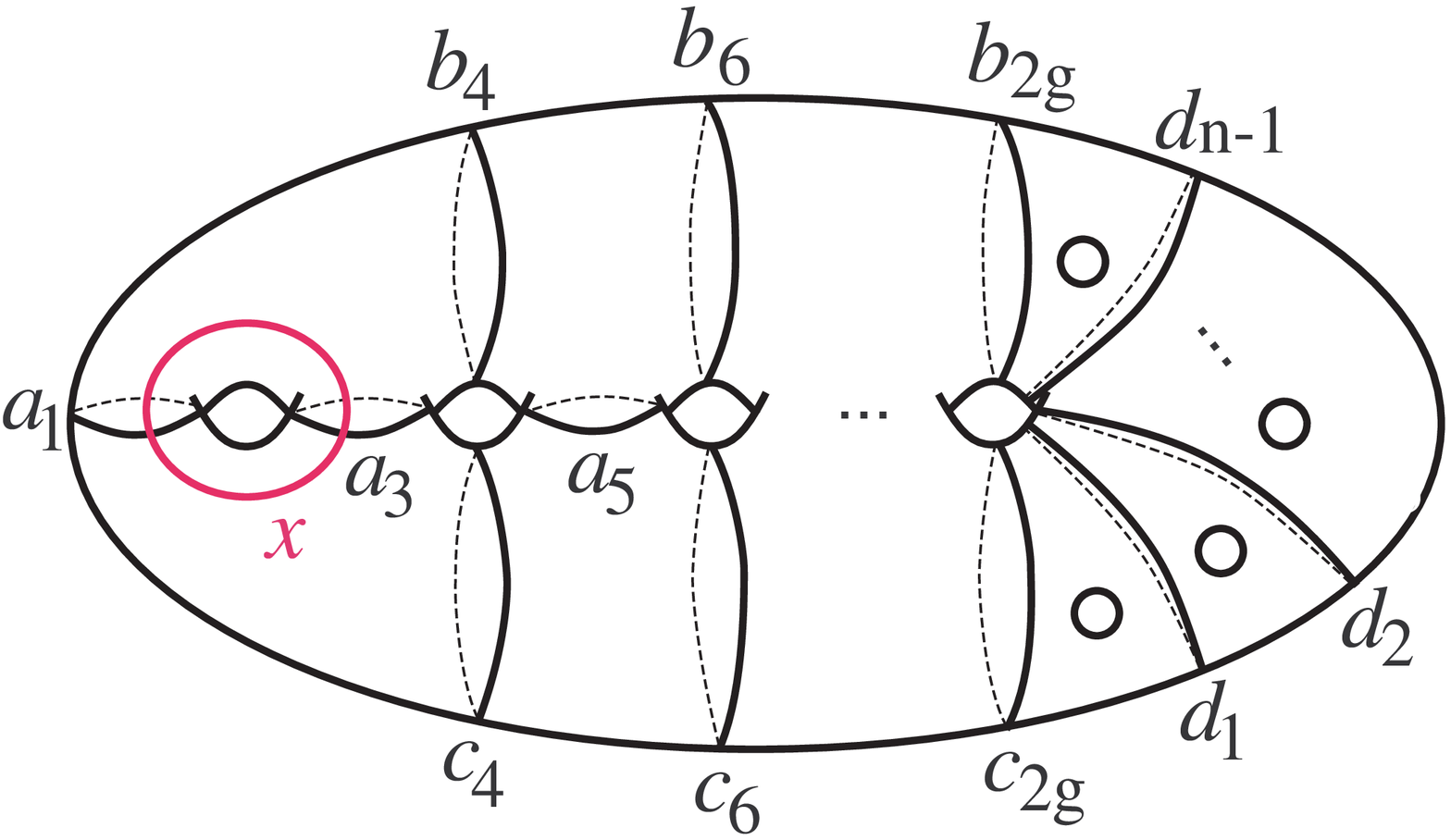}
 
\hspace{-0.9cm} (i) \hspace{6.3cm} (ii)
 
\hspace{-0.4cm} \epsfxsize=2.99in \epsfbox{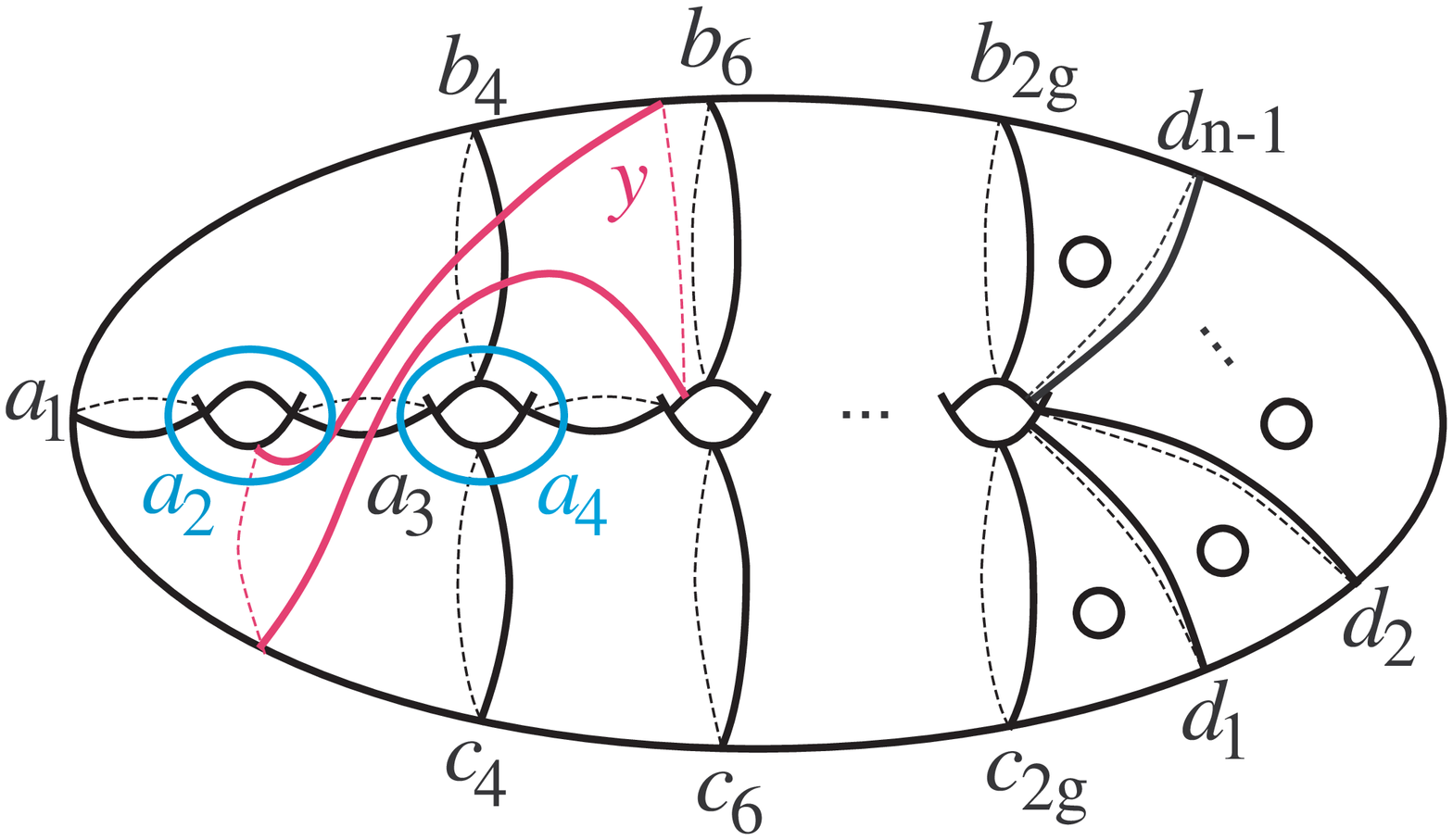} \hspace{-1.2cm} \epsfxsize=2.99in \epsfbox{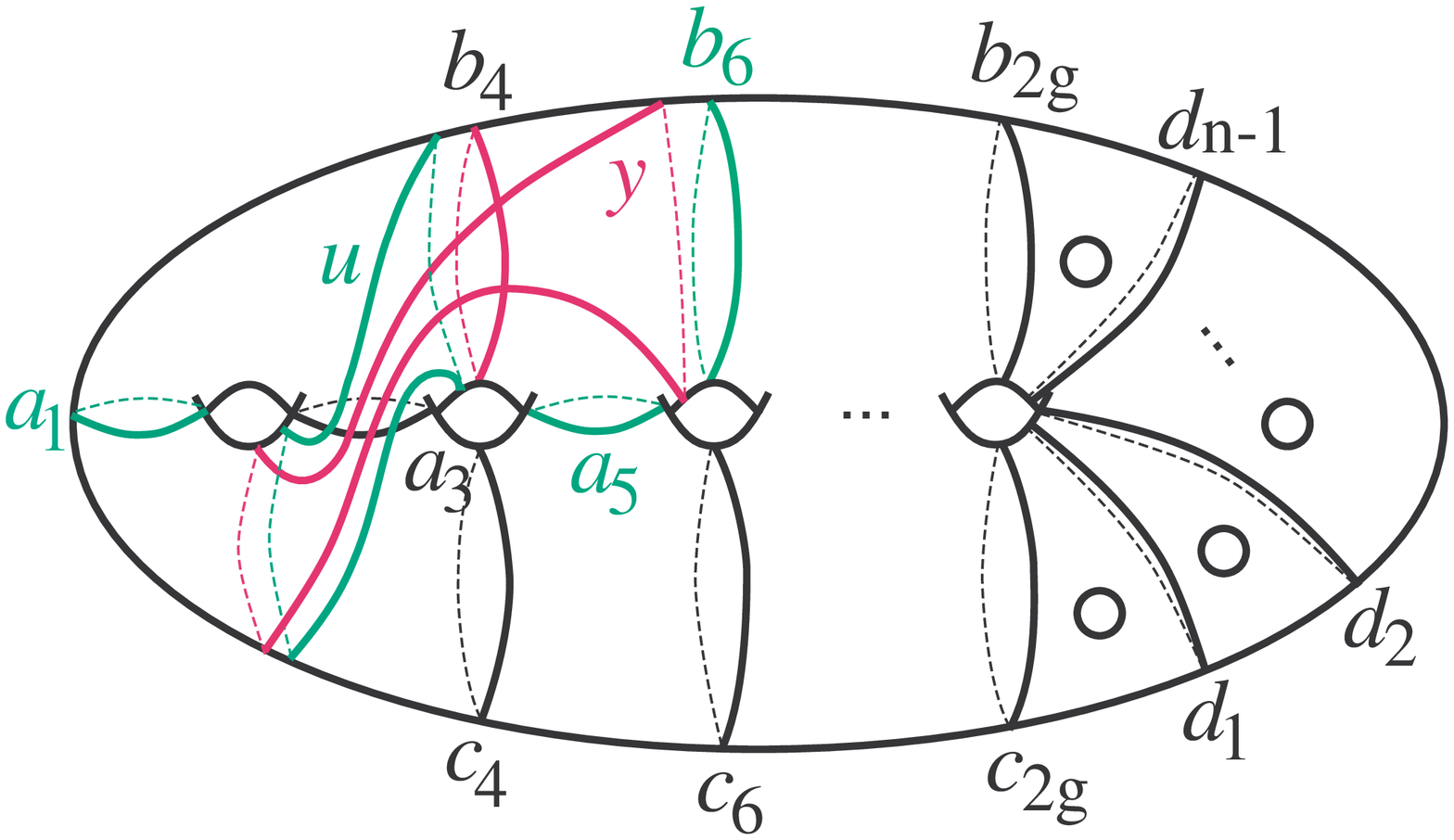}
  
\hspace{-1cm} (iii) \hspace{6.2cm} (iv)
  
\hspace{-0.4cm} \epsfxsize=2.99in \epsfbox{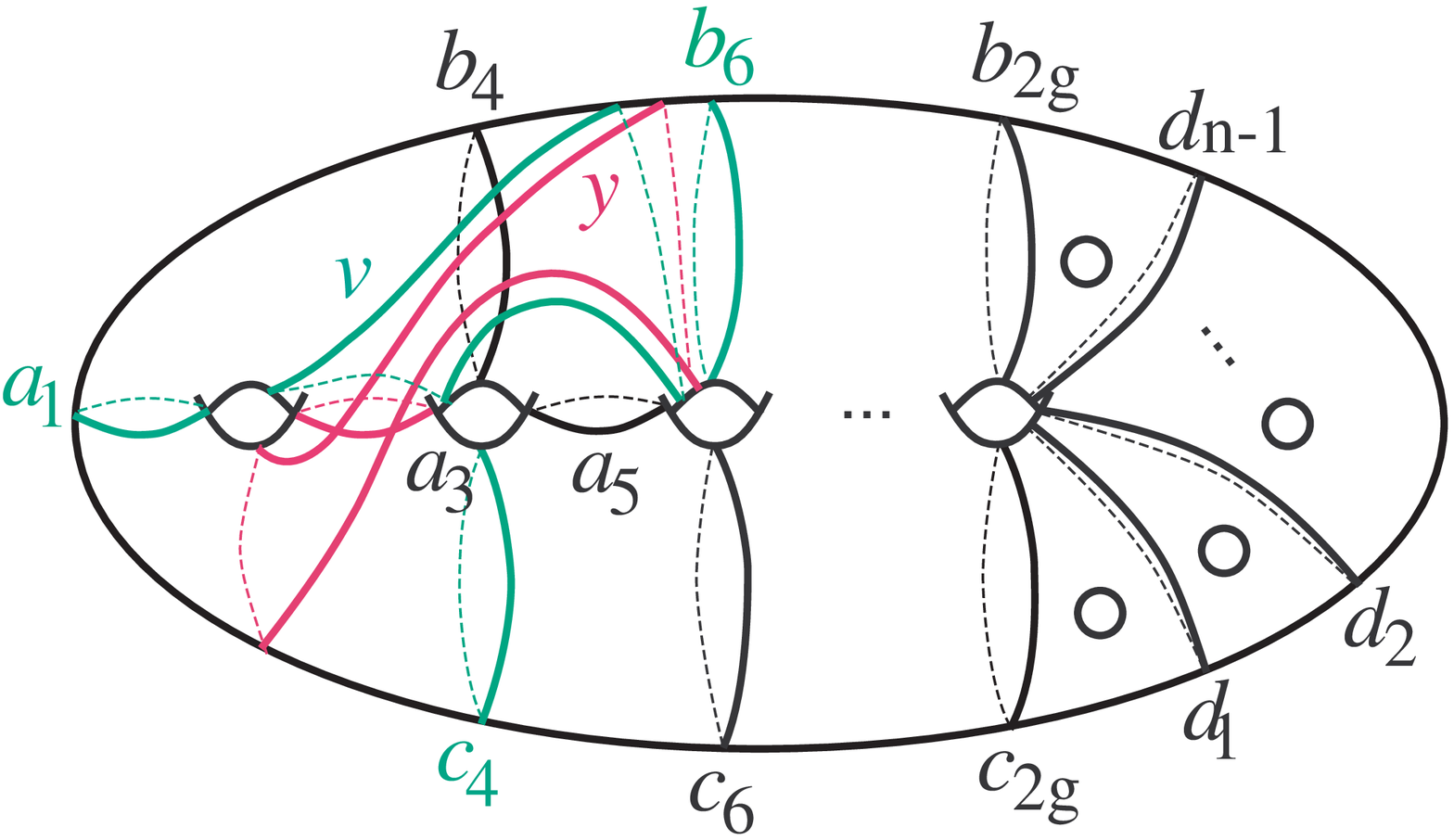} \hspace{-1.2cm} \epsfxsize=2.99in \epsfbox{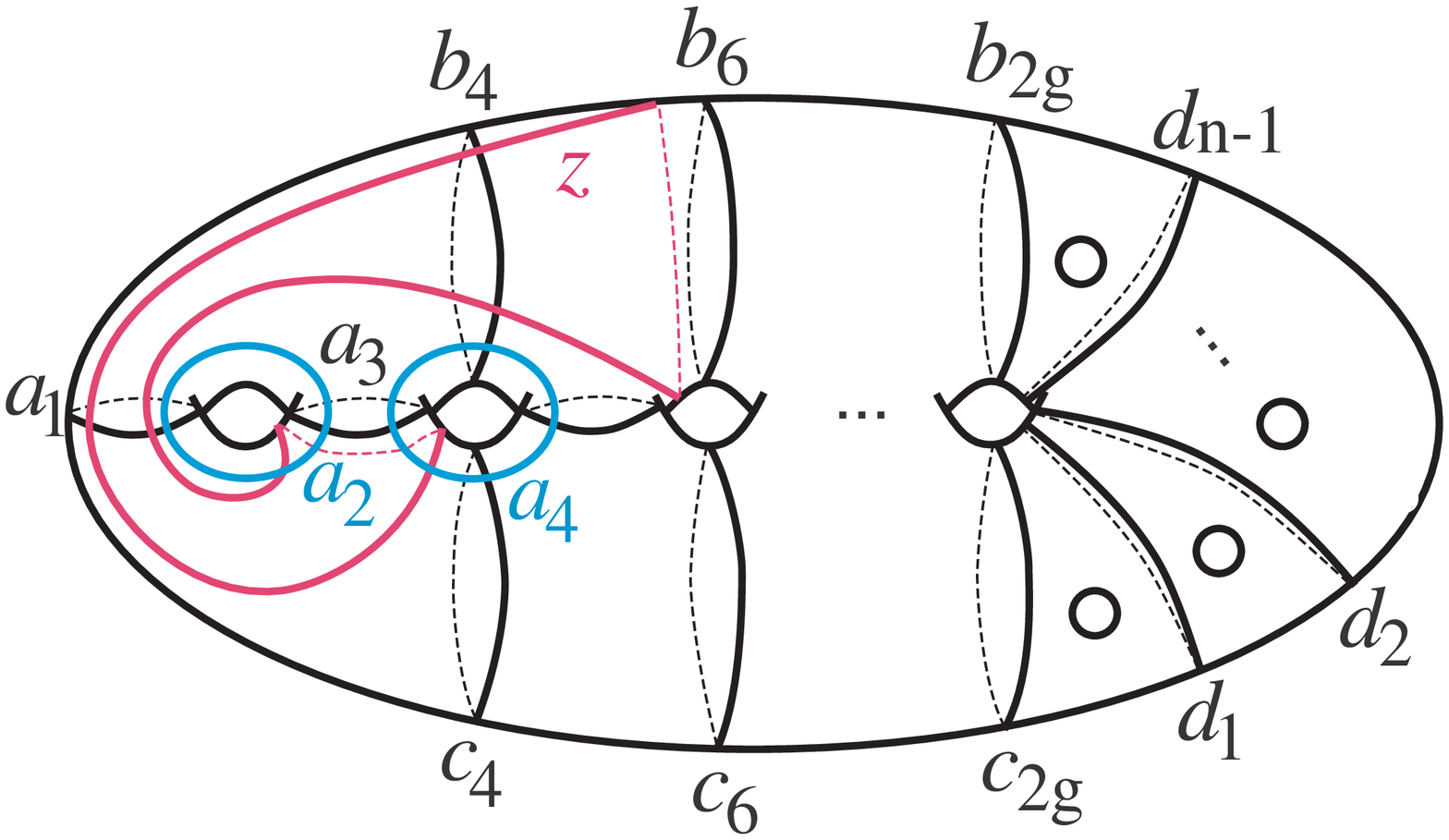}
  
\hspace{-1cm} (v) \hspace{6.3cm} (vi)
  
\hspace{-0.4cm} \epsfxsize=2.99in \epsfbox{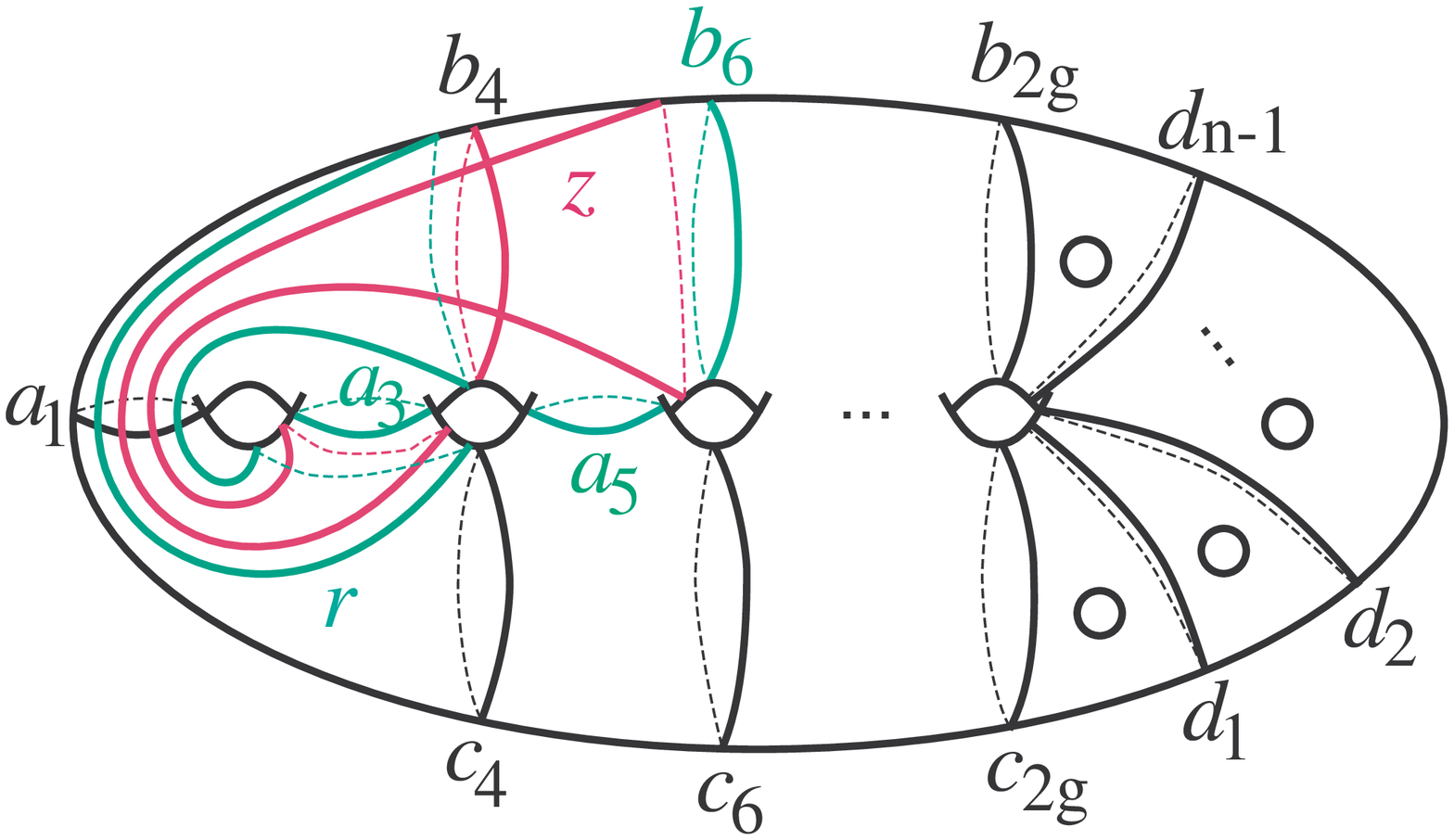}  \hspace{-1.2cm} \epsfxsize=2.99in \epsfbox{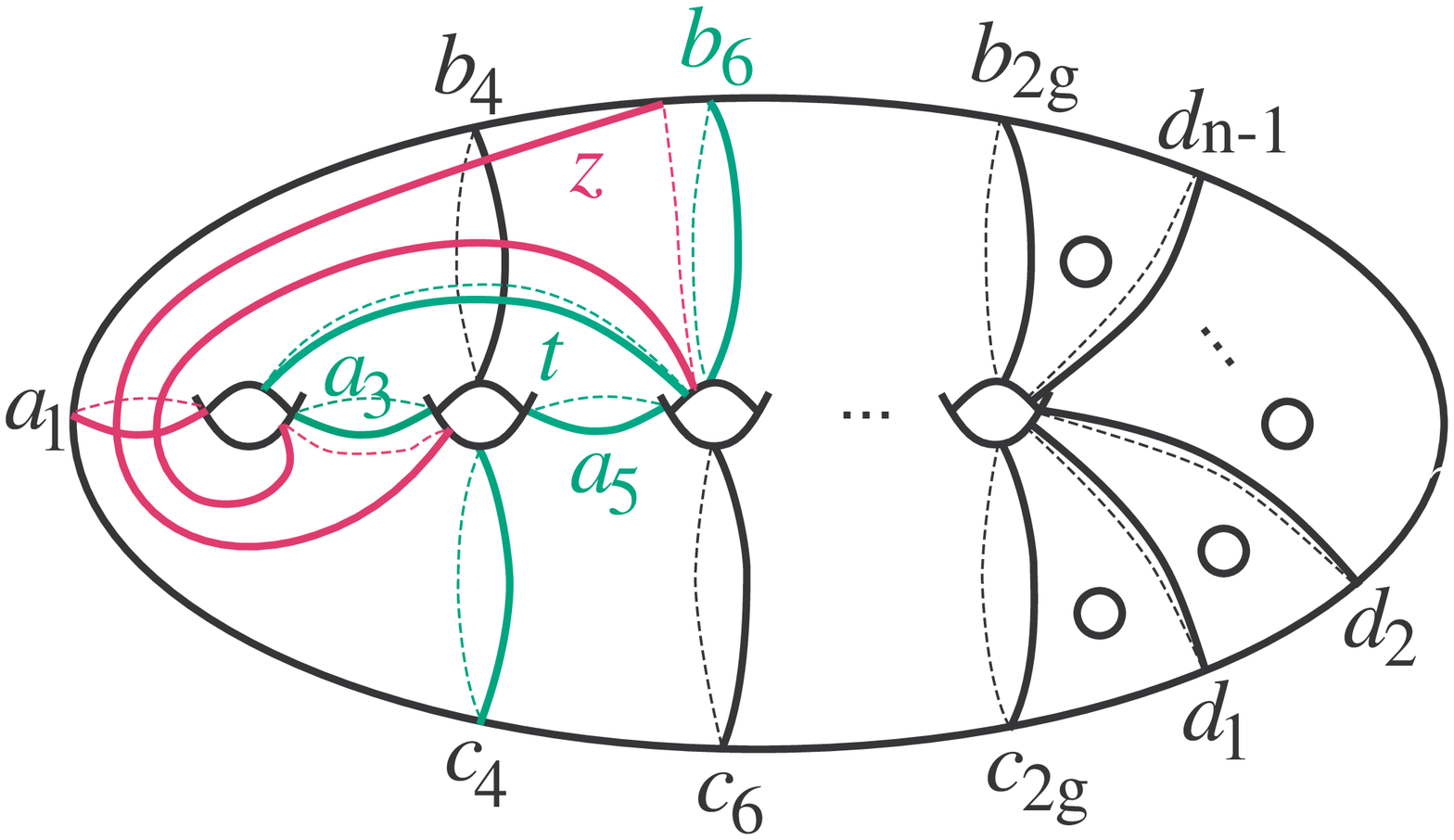}
  
\hspace{-0.8cm} (vii) \hspace{6.1cm} (viii)
 
\caption{Control of adjacency relation} \label{fig1}
\end{center}
\end{figure}

Let $a$ and $b$ be two distinct elements in a pair of pants decomposition $P$ on $R$. Then $a$ is called {\it adjacent} to $b$ with respect to $P$ if and only if there exists a pair of pants in $P$ which has $a$ and $b$ on its boundary.
 
\begin{lemma}
\label{adj} Suppose $g \geq 2$ and $n \geq 0$. Let $P = \{a_1, a_3, \cdots, a_{2g+1}, b_4, b_6, \cdots, b_{2g-2}, c_4,$ 
$c_6, \cdots, c_{2g-2}\}$ when $R$ is closed, and 
$P = \{a_1, a_3, \cdots, a_{2g-1}, b_4, b_6, \cdots, b_{2g}, c_4, c_6,$ 
$\cdots, c_{2g},$ $d_1, d_2,$ $ \cdots, d_{n-1}\}$ when $R$ has boundary
where the curves are as shown in Figure \ref{fig1} (i)-(ii). Let $P'$ be a pair of pants
decomposition of $R$ such that $\lambda([P]) = [P']$. Let $a'_1, a'_3, b'_4$ be the representatives of $\lambda([a_1]), \lambda([a_3]),
\lambda([b_4])$ in $P'$ respectively. Then any two of $a'_1, a'_3, b'_4$ are adjacent to each other with respect to $P'$.\end{lemma}
  
\begin{proof} We will first give the proof when $g \geq 3, n > 0$. Any two of $a_1, a_3, b_4$ are adjacent to each other with respect to $P$.
Let $x$ be the curve shown in Figure \ref{fig1} (ii). 
Assume that $a'_1$ and $a'_3$ are not adjacent with respect to $P'$. Since 
$i([x], [a_1]) = 1$ and $i([x], [a_3]) =1$, we have $i(\lambda([x]), \lambda([a_1])) \neq 0$ and
$i(\lambda([x]), \lambda([a_3])) \neq 0$ by Lemma \ref{2}. Since $i([x], [e]) = 0$ for all $e$ in $P \setminus \{a_1, a_3\}$,
we have $i(\lambda([x]), \lambda([e]))=0$ for all $e$ in $P \setminus \{a_1, a_3\}$. But this is not possible because
$\lambda([x])$ has to intersect geometrically essentially with some isotopy class other than $\lambda([a_1])$ and
$\lambda([a_3])$ in the image pair of pants decomposition to be able to make essential intersections with $\lambda([a_1])$ and
$\lambda([a_3])$. So, $a_1'$ is adjacent to $a_3'$ with respect to $P'$.

To see that $a_3'$ is adjacent to $b_4'$ with respect to $P'$ we consider the curves $y$ and $a_4$ shown in Figure \ref{fig1} (iii). Let 
$Q_1 = (P \setminus \{a_3\}) \cup \{u\}$ where the curve $u$ is as shown in Figure \ref{fig1} (iv). We see that $Q_1$ is a pants decomposition on $R$. Since $i([y], [e]) = 0$ for all $e$ in $Q_1 \setminus \{b_4\}$, we have 
$i(\lambda([y]), \lambda([e]))=0$ for all $e$ in $Q_1 \setminus \{b_4\}$. The map $\lambda$ is edge-preserving, so 
$\lambda([y]) \neq \lambda([e])$ for all $e$ in $Q_1 \setminus \{b_4\}$. Since $i([a_4], [b_4]) = 1$ and 
$i([a_4], [y]) = 0$, we have  $i(\lambda([a_4]), \lambda([b_4])) \neq 0$ by Lemma \ref{2} and 
$i(\lambda([a_4]), \lambda([y])) = 0$. So, $\lambda([y]) \neq \lambda([b_4])$. These imply that  
$i(\lambda([y]), \lambda([b_4])) \neq 0$. Let 
$Q_2 = (P \setminus \{b_4\}) \cup \{v\}$ where the curve $v$ is as shown in Figure \ref{fig1} (v). We see that $Q_2$ is a pants 
decomposition on $R$. 
Since $i([y], [e]) = 0$ for all $e$ in $Q_2 \setminus \{a_3\}$,
we have $i(\lambda([y]), \lambda([e]))=0$ for all $e$ in $Q_2 \setminus \{a_3\}$. The map $\lambda$ is edge-preserving, so 
$\lambda([y]) \neq \lambda([e])$ for all $e$ in $Q_2 \setminus \{a_3\}$. Since $i([a_4], [a_3]) = 1$ and 
$i([a_4], [y]) = 0$, we have  $i(\lambda([a_4]), \lambda([a_3])) \neq 0$ by Lemma \ref{2} and 
$i(\lambda([a_4]), \lambda([y])) = 0$. So, $\lambda([y]) \neq \lambda([a_3])$. These imply that  
$i(\lambda([y]), \lambda([a_3])) \neq 0$. Since $i(\lambda([y]), \lambda([b_4])) \neq 0$, 
$i(\lambda([y]), \lambda([a_3])) \neq 0$ and $i(\lambda([y]), \lambda([e]))=0$ for all $e$ in $P \setminus \{a_3, b_4\}$, we see that $a_3'$ is adjacent to $b_4'$ with respect to $P'$. 
 
\begin{figure} \begin{center}
		\hspace{-1.5cm} \epsfxsize=1.7in \epsfbox{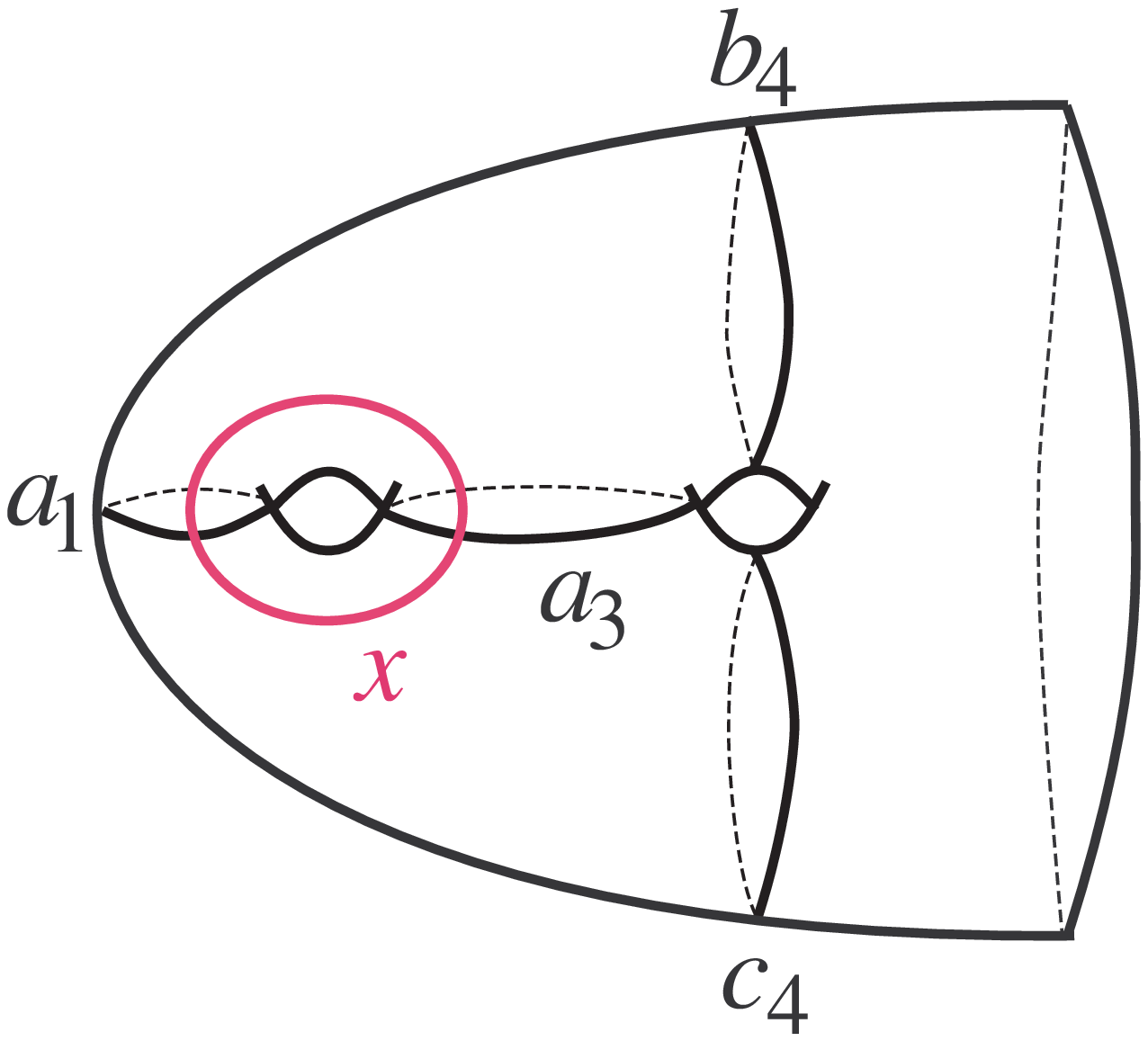} \hspace{0.4cm} \epsfxsize=1.7in \epsfbox{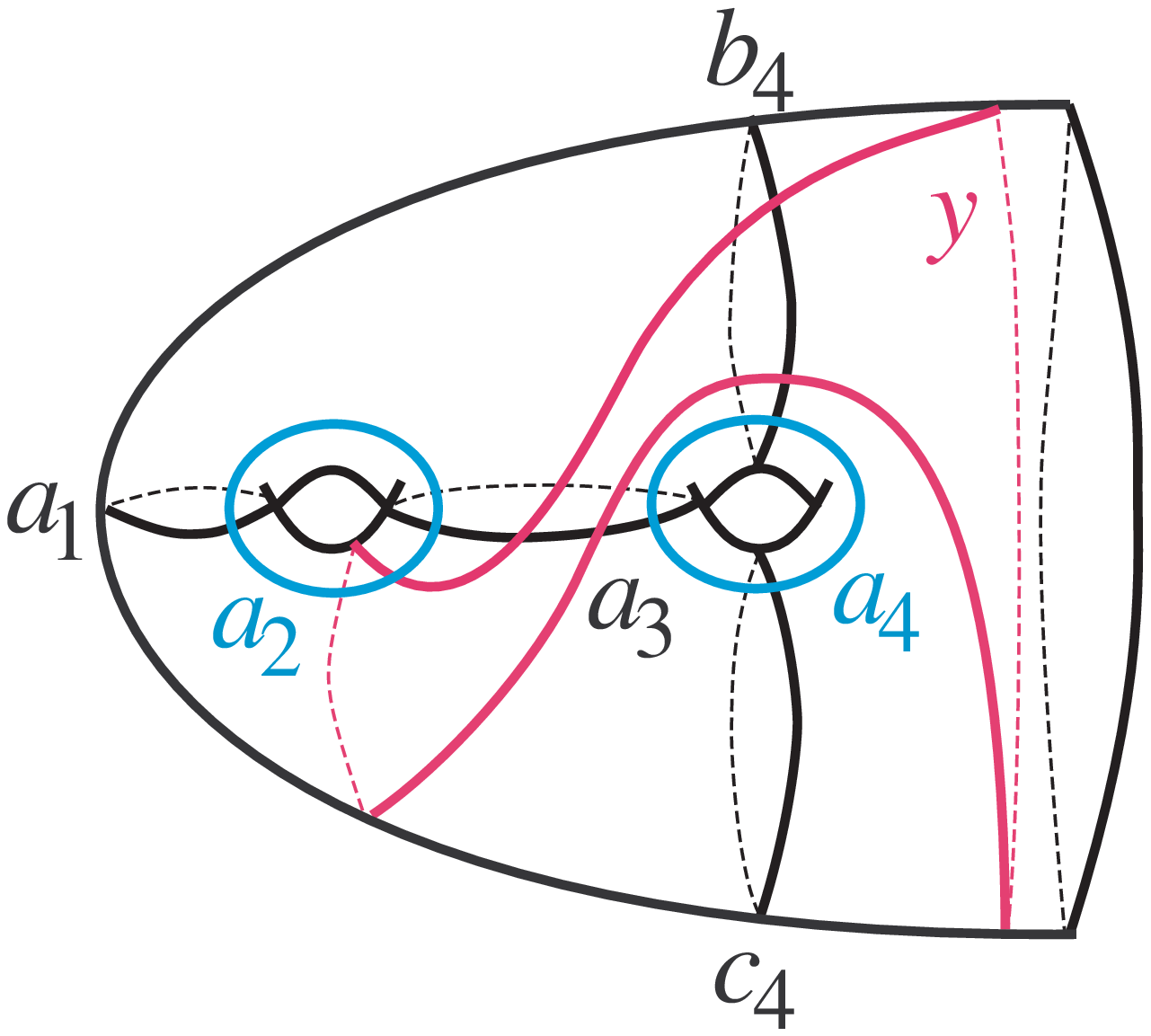} \hspace{0.4cm} \epsfxsize=1.7in \epsfbox{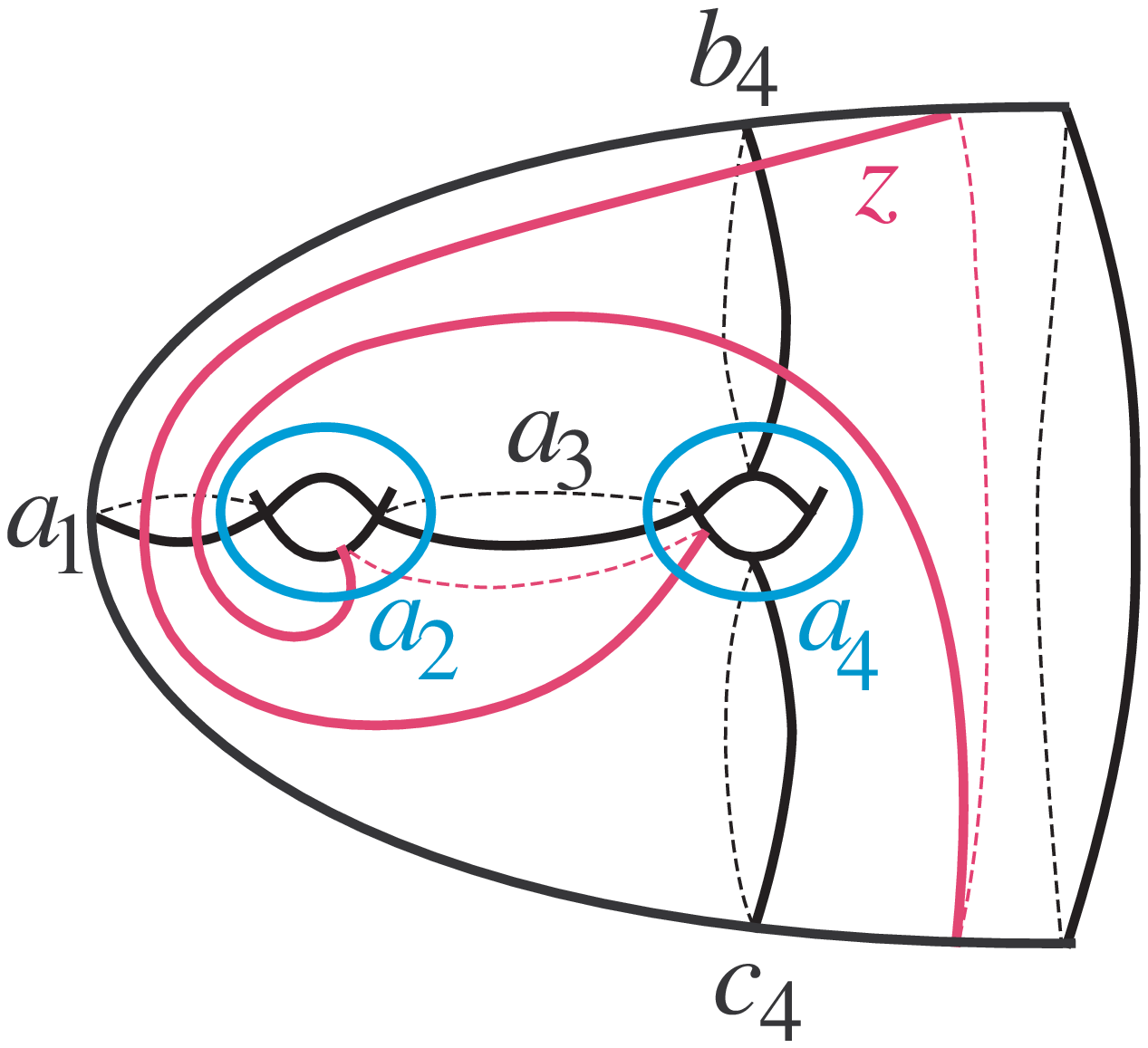}
		
		\hspace{-1cm} (i) \hspace{3.7cm} (ii) \hspace{3.7cm} (iii)
		
		\hspace{-1.5cm} \epsfxsize=1.7in \epsfbox{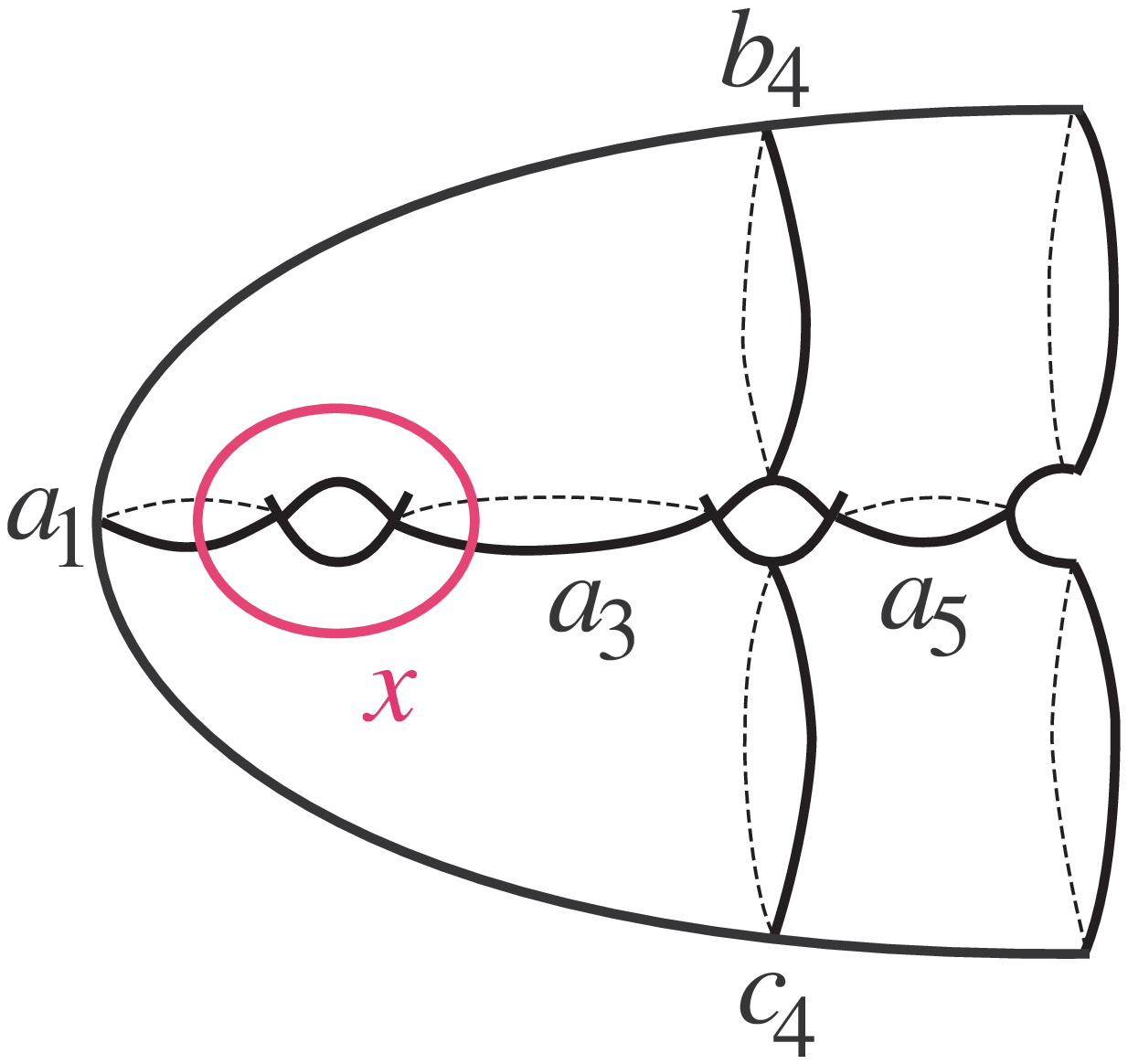} \hspace{0.4cm} \epsfxsize=1.7in \epsfbox{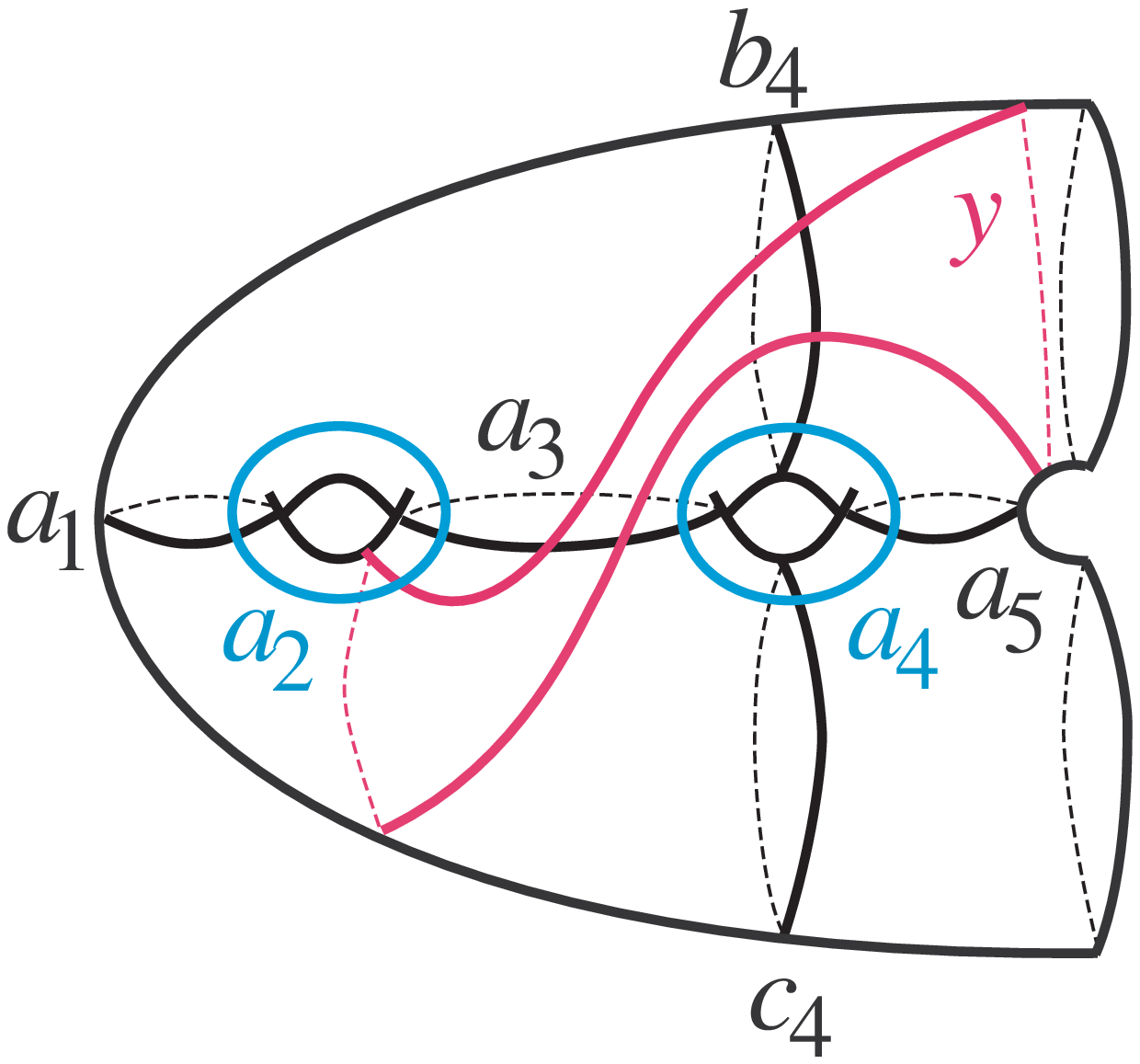} \hspace{0.4cm} \epsfxsize=1.7in \epsfbox{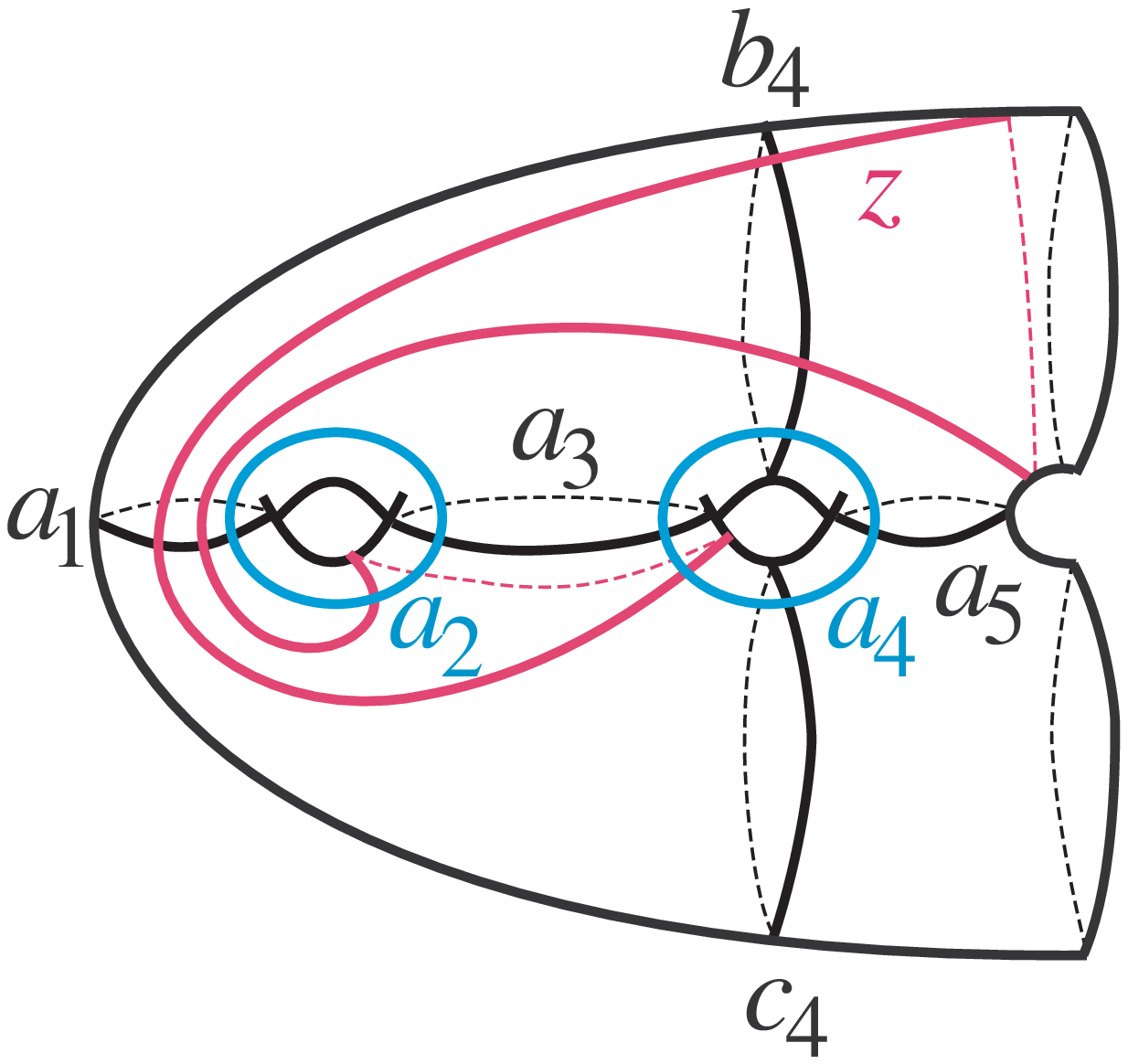}
		
		\hspace{-1cm} (iv) \hspace{3.7cm} (v) \hspace{3.7cm} (vi)
		
		\caption{Control of adjacency relation} \label{fig2}
\end{center} \end{figure}
 
Similarly, we can see that $a_1'$ is adjacent to $b_4'$ with respect to $P'$ as follows: Consider the curve $z$ shown in Figure \ref{fig1} (vi). Let 
$Q_3 = (P \setminus \{a_1\}) \cup \{r\}$ where the curve $r$ is as shown in Figure \ref{fig1} (vii). The set $Q_3$ is a pants decomposition on $R$. Since $i([z], [e]) = 0$ for all $e$ in $Q_3 \setminus \{b_4\}$,
we have $i(\lambda([z]), \lambda([e]))=0$ for all $e$ in $Q_3 \setminus \{b_4\}$. Since $\lambda$ is edge-preserving 
$\lambda([z]) \neq \lambda([e])$ for all $e$ in $Q_3 \setminus \{b_4\}$. Since $i([a_4], [b_4]) = 1$ and 
$i([a_4], [z]) = 0$, we have  $i(\lambda([a_4]), \lambda([b_4])) \neq 0$ by Lemma \ref{2} and 
$i(\lambda([a_4]), \lambda([z])) = 0$. So, $\lambda([z]) \neq \lambda([b_4])$. These imply that  
$i(\lambda([z]), \lambda([b_4])) \neq 0$. Let 
$Q_4 = (P \setminus \{b_4\}) \cup \{t\}$ where the curve $t$ is as shown in Figure \ref{fig1} (viii). The set $Q_4$ is a pants 
decomposition on $R$. Since $i([z], [e]) = 0$ for all $e$ in $Q_4 \setminus \{a_1\}$,
we have $i(\lambda([z]), \lambda([e]))=0$ for all $e$ in $Q_4 \setminus \{a_1\}$. Since $\lambda$ is edge-preserving 
$\lambda([z]) \neq \lambda([e])$ for all $e$ in $Q_4 \setminus \{a_1\}$. Since $i([a_4], [z]) = 1$ and 
$i([a_4], [a_1]) = 0$, we have  $i(\lambda([a_4]), \lambda([z])) \neq 0$ by Lemma \ref{2} and 
$i(\lambda([a_4]), \lambda([a_1])) = 0$. So, $\lambda([z]) \neq \lambda([a_1])$. These imply that  
$i(\lambda([z]), \lambda([a_1])) \neq 0$. Since $i(\lambda([z]), \lambda([b_4])) \neq 0$, 
$i(\lambda([z]), \lambda([a_1])) \neq 0$ and $i(\lambda([z]), \lambda([e]))=0$ for all $e$ in $P \setminus \{a_1, b_4\}$, 
we see that $a_1'$ is adjacent to $b_4'$ with respect to $P'$. Hence, any two of $a'_1, a'_3, b'_4$ are adjacent to each other with respect to $P'$. 

The proof follows similarly when $g \geq 3, n = 0$.
The statement of the lemma is obvious when $g=2, n=0$. For the remaining cases for $g=2$, the proof is similar (see Figure
\ref{fig2}).\end{proof} 

\begin{lemma}
\label{embedded} Suppose $g \geq 2$ and $n \geq 0$. If $\alpha, \beta, \gamma$ are distinct vertices in $\mathcal{N}(R)$ having pairwise disjoint representatives 
which bound a pair of pants on $R$, then $\lambda(\alpha), \lambda(\beta), \lambda(\gamma)$ are distinct vertices in $\mathcal{N}(R)$ having pairwise disjoint representatives which bound a pair of pants on $R$.\end{lemma}

\begin{figure}
\begin{center}
\hspace{1cm} \epsfxsize=3.2in \epsfbox{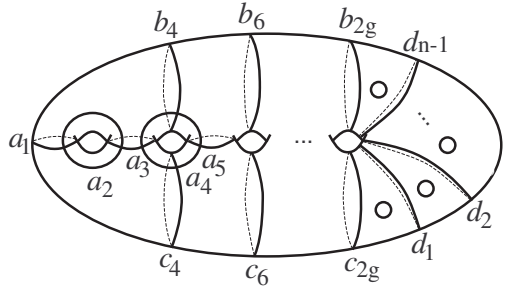} 

\caption {Curves $a_1, a_3, b_4$ bounding a pair of pants} \label{fig3-a}
\end{center}
\end{figure}
 
\begin{proof} Let $a, b, c$ be pairwise disjoint representatives of $\alpha, \beta, \gamma$ respectively. If $R$ is a closed surface
of genus two then the statement is obvious. In the other cases we complete $\{a, b, c\}$ to a pair of pants decomposition $P$ as in Lemma \ref{adj} 
where we let $a=a_1, b=a_3, c=b_4$. Let $a_2, a_4$ be as shown in  Figure \ref{fig3-a}. Let $P'$ be a pair of pants decomposition of $R$ such that $\lambda([P]) = [P']$. Let $a_1', a_3', b_4', c_4'$ be the representatives of $\lambda([a_1]), \lambda([a_3]), \lambda([b_4]), \lambda([c_4])$ in $P'$ respectively. Let $a'_2, a'_4$ be disjoint representatives of $\lambda([a_2]), \lambda([a_4])$ which intersects elements of $P'$ minimally. Using Lemma \ref{adj}, we see that since any two curves in $\{a_1, a_3, b_4\}$ are adjacent with respect to $P$, any two curves in $\{a'_1, a'_3, b'_4\}$ are adjacent with respect to $P'$. Similarly, since any two curves in $\{a_1, a_3, c_4\}$ are adjacent with respect to $P$, any two curves in $\{a'_1, a'_3, c'_4\}$ are adjacent with respect to $P'$. Then there is a pair of pants $Q_1$ in $P'$ having $a_1'$ and $a_3'$
on its boundary. Let $x$ be the other boundary component of $Q_1$. If $x = b_4'$, then we are done. Suppose $x \neq b_4'$. Since $a_1'$ is adjacent to $b_4'$ with respect to $P'$ there is another pair of pants $Q_2$ in $P'$ having $a_1'$ and $b_4'$
on its boundary. Let $y$ be the other boundary component of $Q_2$. If $y= a_3'$ we are done, if that is not the case then since $a_3'$ is adjacent to $b_4'$ with respect to $P'$ there is a pair of pants $Q_3$ in $P'$ having $a_3'$ and $b_4'$
on its boundary. Since $i([a_1], [a_2])=1$, $i([a_2], [a_3])=1$ and $a_2$ is disjoint from all the other curves in $P$, we know $a_2'$ intersects $a_1'$ and $a_3'$ nontrivially, and $a_2'$ is disjoint from all the other curves in $P'$. 
Then there exists an essential arc $w$ of $a_2'$ in $Q_2$ which starts and ends on $a_1'$ and does not intersect $b_4'$ and the other boundary component of $Q_2$. But this gives a contradiction since $a_4$ is disjoint from $a_1$ and $a_2$ 
and $i([a_4], [b_4]) =1$, so $a_4'$ is disjoint from $a'_1$ and $a'_2$ and $i([a'_4], [b'_4]) \neq 0$. Hence, there is a pair of pants which has $a'_1, a'_3, b'_4$ on its boundary.\end{proof}\\
 
The author gave a characterization of geometric intersection one property in the nonseparating curve complex on $R$ with the following lemma in $\cite{Ir3}$.

\begin{figure}
	\begin{center}
		\hspace{0.1cm} \epsfxsize=2.5in \epsfbox{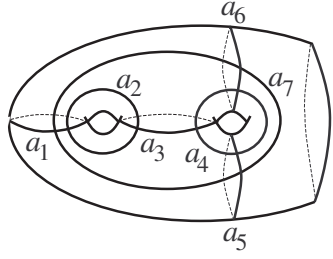} \caption{Curves $a_1, a_2$ intersecting once} \label{fig-int-one}
	\end{center}
\end{figure}

\begin{lemma} \label{Irmak-lemma} (Irmak \cite{Ir3}) Suppose $g \geq 2$ and $n \geq 0$.
Let $\alpha_{1}$ and $\alpha_{2}$ be two vertices in
$\mathcal{N}(R)$. Then $i( \alpha_{1}, \alpha_{2})=1$ if and only
if there exist isotopy classes $\alpha_{3}, \alpha_{4},
\alpha_{5}, \alpha_{6}, \alpha_{7}$ such that

\indent (i) $i(\alpha_{i}, \alpha_{j})=0$ if and only if the
the circles $a_i$ and $a_j$ on Figure \ref{fig-int-one} are disjoint.

\indent (ii) $\alpha_{1}, \alpha_{3}, \alpha_{5}, \alpha_{6}$ have
pairwise disjoint representatives $a_1, a_3, a_5, a_6$
respectively such that $a_5 \cup a_6$ divides $R$ into two pieces,
one of these is a torus with two holes, $T$, containing some
representatives of the isotopy classes $\alpha_1, \alpha_2$ and
$a_1, a_3, a_5$ bound a pair of pants on $T$ and $a_1, a_3, a_6$
bound a pair of pants on $T$.\end{lemma}

\begin{lemma} \label{int-one} Suppose $g \geq 2$ and $n \geq 0$.  
Let $\alpha_1, \alpha_2$ be two vertices in $\mathcal{N}(R)$. If $i(\alpha_1, \alpha_2) = 1$, then $i(\lambda(\alpha_1), \lambda(\alpha_2) )= 1$.\end{lemma}

\begin{proof} Let $a_1, a_2$ be disjoint representatives of $\alpha_1, \alpha_2$ respectively. We complete $\{a_1, a_2\}$ to a curve configuration 
$\{a_1, a_2, \cdots, a_7\}$ as shown in Figure \ref{fig-int-one}. Let $\alpha_i=[a_i]$ for $i = 3, 4, 5, 6, 7$. We have $i(\alpha_{i}, \alpha_{j})=0$ if
and only if the circles $a_i, a_j$ on Figure \ref{fig-int-one} are disjoint. By Lemma \ref{2} we know that if $i(\alpha, \beta) = 1$ then $i(\lambda(\alpha), \lambda(\beta)) \neq 0$. Using this and knowing that $\lambda$ is edge-preserving, 
we see that $i(\lambda(\alpha_{i}), \lambda(\alpha_{j}))=0$ if
and only if the circles $a_i, a_j$ on Figure \ref{fig-int-one} are disjoint. Using Lemma \ref{embedded}, that $\lambda$ is edge-preserving, and that it sends curves with geometric 
intersection one to curves that have nontrivial intersection, we can see that there are pairwise disjoint representatives $a_1', a_3', a_5', a_6'$ of
$\lambda(\alpha_{1}), \lambda(\alpha_{3}), \lambda(\alpha_{5}), \lambda(\alpha_{6})$ respectively, such that $a_5' \cup a_6'$
divides $R$ into two pieces, one of these is a torus with two holes, $T$, containing some representatives of the isotopy classes
$\lambda(\alpha_1), \lambda(\alpha_2)$ and $a_1', a_3', a_5'$ bound a pair of pants on $T$ and $a_1', a_3', a_6'$ bound a pair
of pants on $T$. Then, by Lemma \ref{Irmak-lemma}, we get $i(\lambda(\alpha_1), \lambda(\alpha_2))=1$.\end{proof}\\

Let $\alpha$, $\beta$ be two distinct vertices in
$\mathcal{N}(R)$. We call $(\alpha, \beta)$ to be a
\textit{peripheral pair} in $\mathcal{N}(R)$ if they have disjoint
representatives $x, y$ respectively such that $x, y$ and a
boundary component of $R$ bound a pair of pants on $R$.

\begin{lemma}
\label{peripheral} Suppose $g \geq 2$ and $n \geq 1$. Let $(\alpha, \beta)$ be a peripheral pair in
$\mathcal{N}(R)$. Then $(\lambda(\alpha), \lambda(\beta))$ is a
peripheral pair in $\mathcal{N}(R)$.
\end{lemma}

\begin{proof} Let $x, y$ be disjoint representatives of $\alpha, \beta$
respectively such that $x, y$ and a boundary component of $R$
bound a pair of pants, say $Q$, on $R$. If $g \geq 2, n=1$, then we complete $x, y$ to a pair of pants decomposition $P$ as in Lemma \ref{embedded} 
where we let $x=b_{2g}, y=c_{2g}$. We note that all the pairs of pants of $P$ except for $Q$ have three essential boundary components. Let $P'$ be a pair of pants decomposition of $R$ such that $\lambda([P]) = [P']$. Let $b_{2g}', c_{2g}'$ be the representatives of $\lambda([b_{2g}]), \lambda([c_{2g}])$ in $P'$ respectively.
By Lemma \ref{embedded} we know that if three curves in $P$ bound an embedded pair of pants, then the corresponding curves in $P'$ bound an embedded pair of pants on $R$. So, if a pair of pants in $P$ has three essential curves on its boundary, then the corresponding pair of pants in $P'$ has three essential curves on its boundary. 
Because how $P$ was constructed, by gluing all such pairs of pants of $P'$ along their common boundary components we get a genus $g-1$ surface with two boundary components $b_{2g}'$ and $c_{2g}'$. Hence, $b_{2g}', c_{2g}'$ and the boundary component of $R$ have to bound a pair of pants on $R$ so that gluing all the pair of pants in $P'$ gives us genus $g$ surface with one boundary component, $R$.  

\begin{figure} \begin{center}
\hspace{-1.5cm} \epsfxsize=1.8in \epsfbox{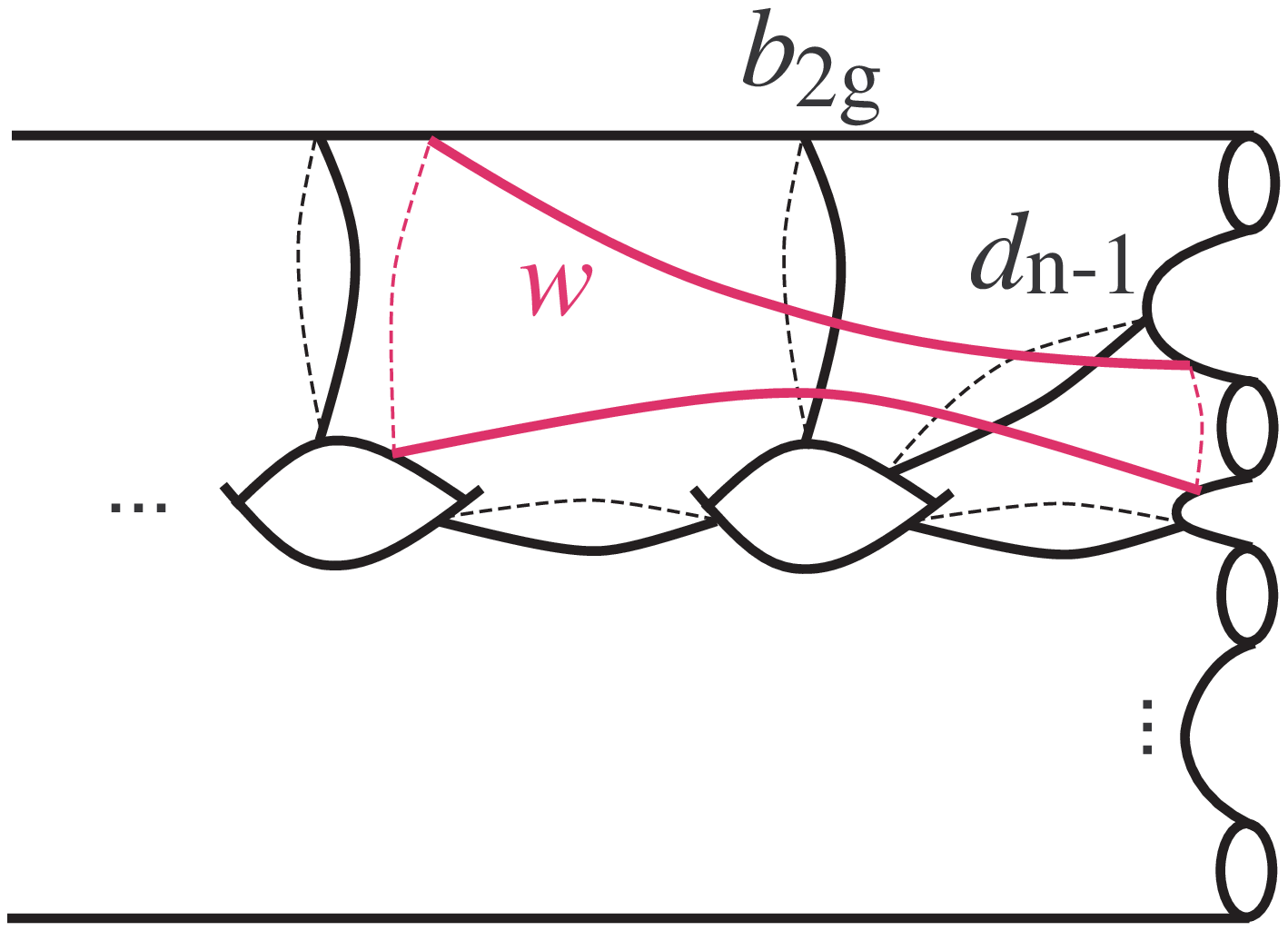} \hspace{0.4cm} \epsfxsize=1.8in \epsfbox{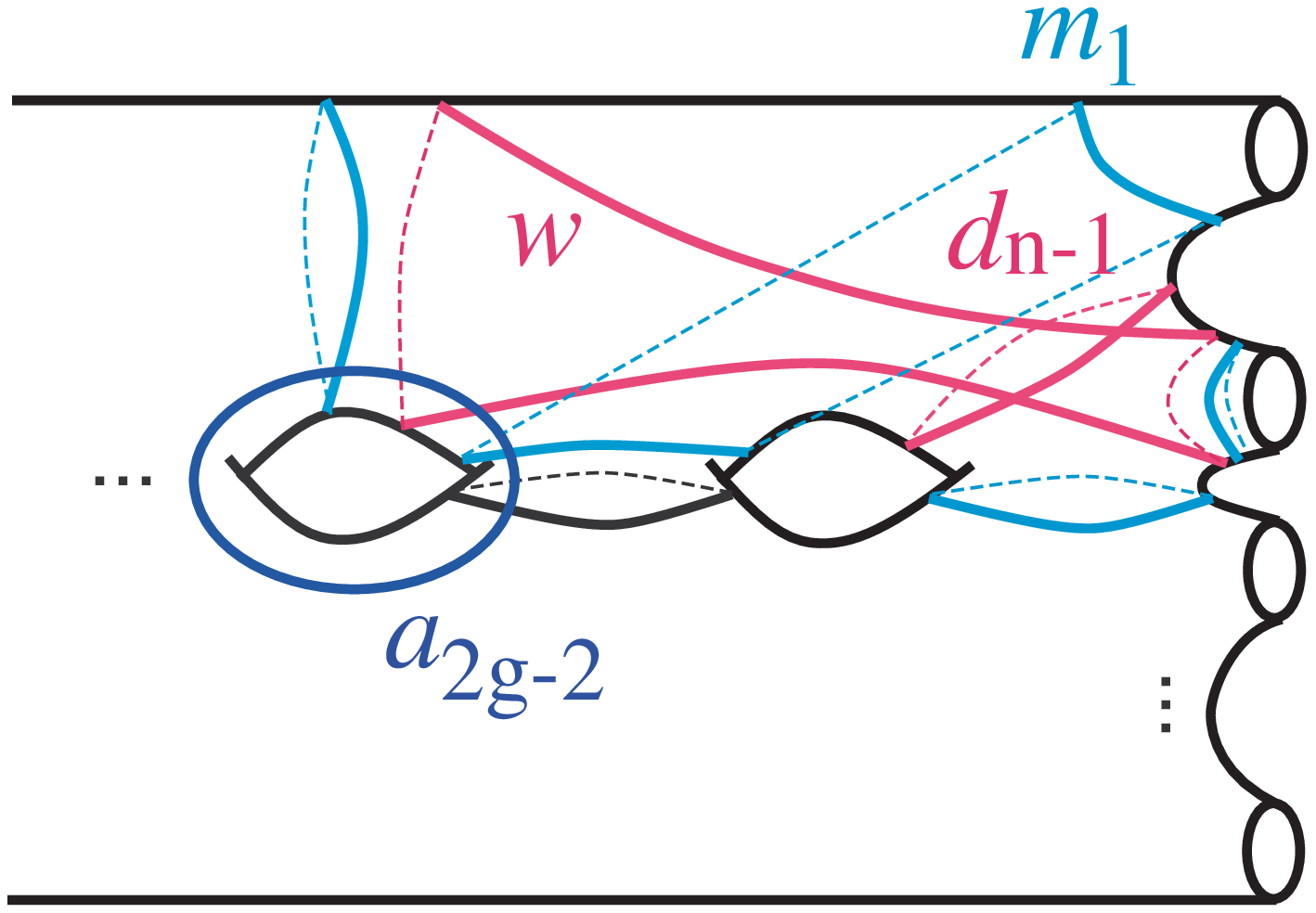}  \hspace{0.4cm} \epsfxsize=1.8in \epsfbox{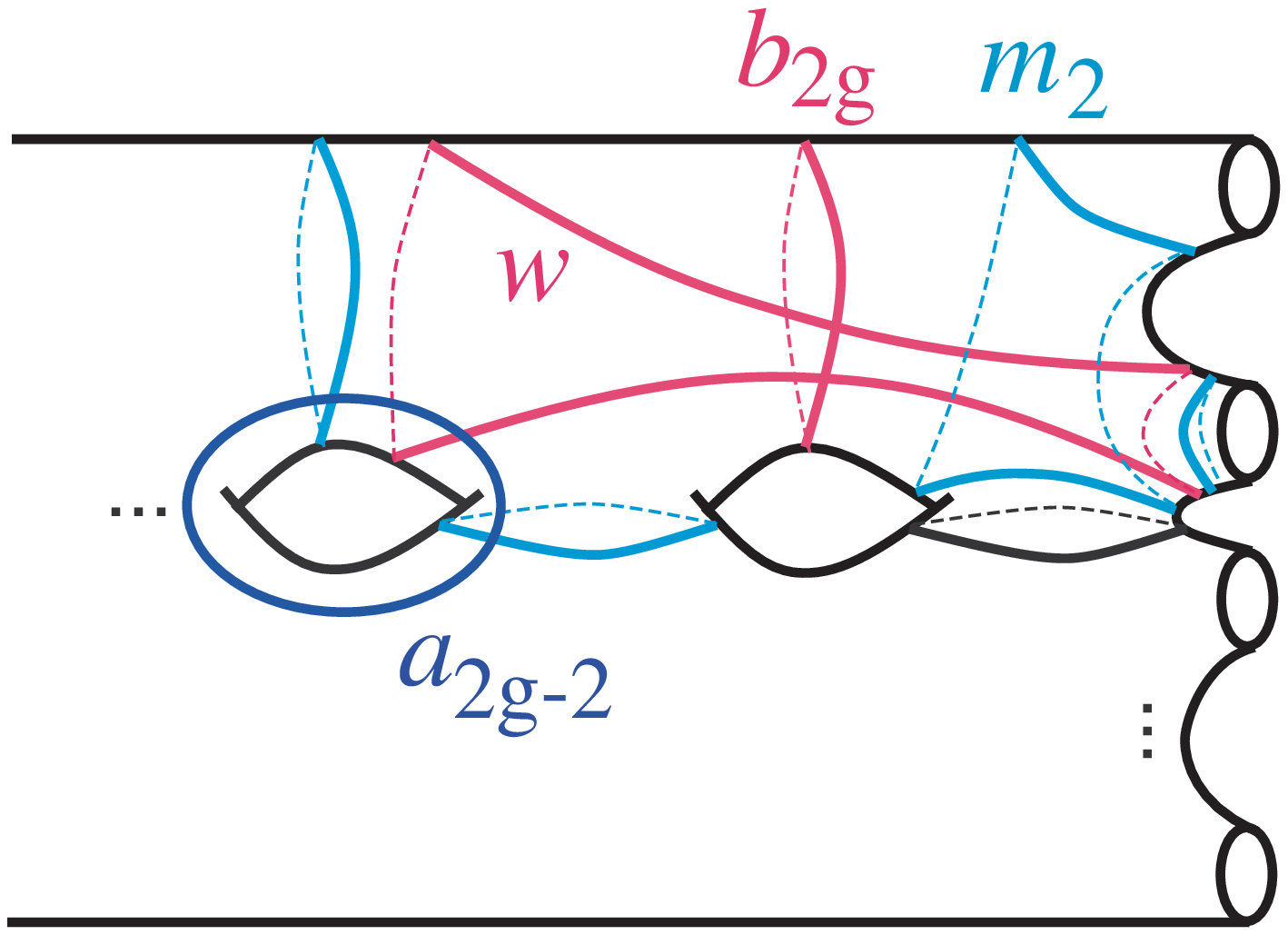}  

\hspace{-1.5cm} (i) \hspace{3.9cm} (ii) \hspace{3.9cm} (iii)
 
\caption{Peripheral pairs} \label{fig35}
\end{center} \end{figure} 

Suppose that $g \geq 2, n \geq 2$. We complete $x, y$ to a pair of pants decomposition $P$ as in Lemma \ref{embedded} 
where we let $x=b_{2g}, y=d_{n-1}$. Let $P'$ be a pair of pants decomposition of $R$ such that $\lambda([P]) = [P']$. Let $b_{2g}', d_{n-1}'$ be the representatives of $\lambda([b_{2g}]), \lambda([d_{n-1}])$ in $P'$ respectively. To see that $b_{2g}'$ is adjacent to $d_{n-1}'$ with respect to $P'$ we consider the curve $w$ given in Figure \ref{fig35} (i). Let 
$Q_1 = (P \setminus \{b_{2g}\}) \cup \{m_1\}$ where the curve $m_1$ is as shown in Figure \ref{fig35} (ii). The set $Q_1$ is a pants decomposition on $R$. Since $i([w], [e]) = 0$ for all $e$ in $Q_1 \setminus \{d_{n-1}\}$,
we have $i(\lambda([w]), \lambda([e]))=0$ for all $e$ in $Q_1 \setminus \{d_{n-1}\}$. Since $\lambda$ is edge-preserving 
$\lambda([w]) \neq \lambda([e])$ for all $e$ in $Q_1 \setminus \{d_{n-1}\}$. Since $i([a_{2g-2}], [w]) = 1$ and 
$i([a_{2g-2}], [d_{n-1}]) = 0$, we have  $i(\lambda([a_{2g-2}]), \lambda([w])) \neq 0$ and 
$i(\lambda([a_{2g-2}]), \lambda([d_{n-1}])) = 0$. So, $\lambda([w]) \neq \lambda([d_{n-1}]))$. These imply that  
$i(\lambda([w]), \lambda([d_{n-1}])) \neq 0$. Let $Q_2 = (P \setminus \{d_{n-1}\}) \cup \{m_2\}$ where the curve $m_2$ is as shown in Figure \ref{fig35} (iii). The set $Q_2$ is a pants 
decomposition on $R$. 
Since $i([w], [e]) = 0$ for all $e$ in $Q_2 \setminus \{b_{2g}\}$,
we have $i(\lambda([w]), \lambda([e]))=0$ for all $e$ in $Q_2 \setminus \{b_{2g}\}$. Since $\lambda$ is edge-preserving 
$\lambda([w]) \neq \lambda([e])$ for all $e$ in $Q_2 \setminus \{b_{2g}\}$. Since $i([a_{2g-2}], [w]) = 1$ and 
$i([a_{2g-2}], [b_{2g}]) = 0$, we have $i(\lambda([a_{2g-2}]), \lambda([w])) \neq 0$ and 
$i(\lambda([a_{2g-2}]), \lambda([b_{2g}])) = 0$. So, $\lambda([w]) \neq \lambda([b_{2g}])$. These imply that  
$i(\lambda([w]), \lambda([b_{2g}])) \neq 0$. Since $i(\lambda([w]), \lambda([b_{2g}])) \neq 0$, 
$i(\lambda([w]), \lambda([d_{n-1}])) \neq 0$ and $i(\lambda([w]), \lambda([e]))=0$ for all $e$ in $P \setminus \{b_{2g}, d_{n-1}\}$, we see that $b_{2g}'$ is adjacent to $d_{n-1}'$ with respect to $P'$. 

\begin{figure}[htb]
\begin{center}
\hspace{1.3cm} \epsfxsize=3.0in \epsfbox{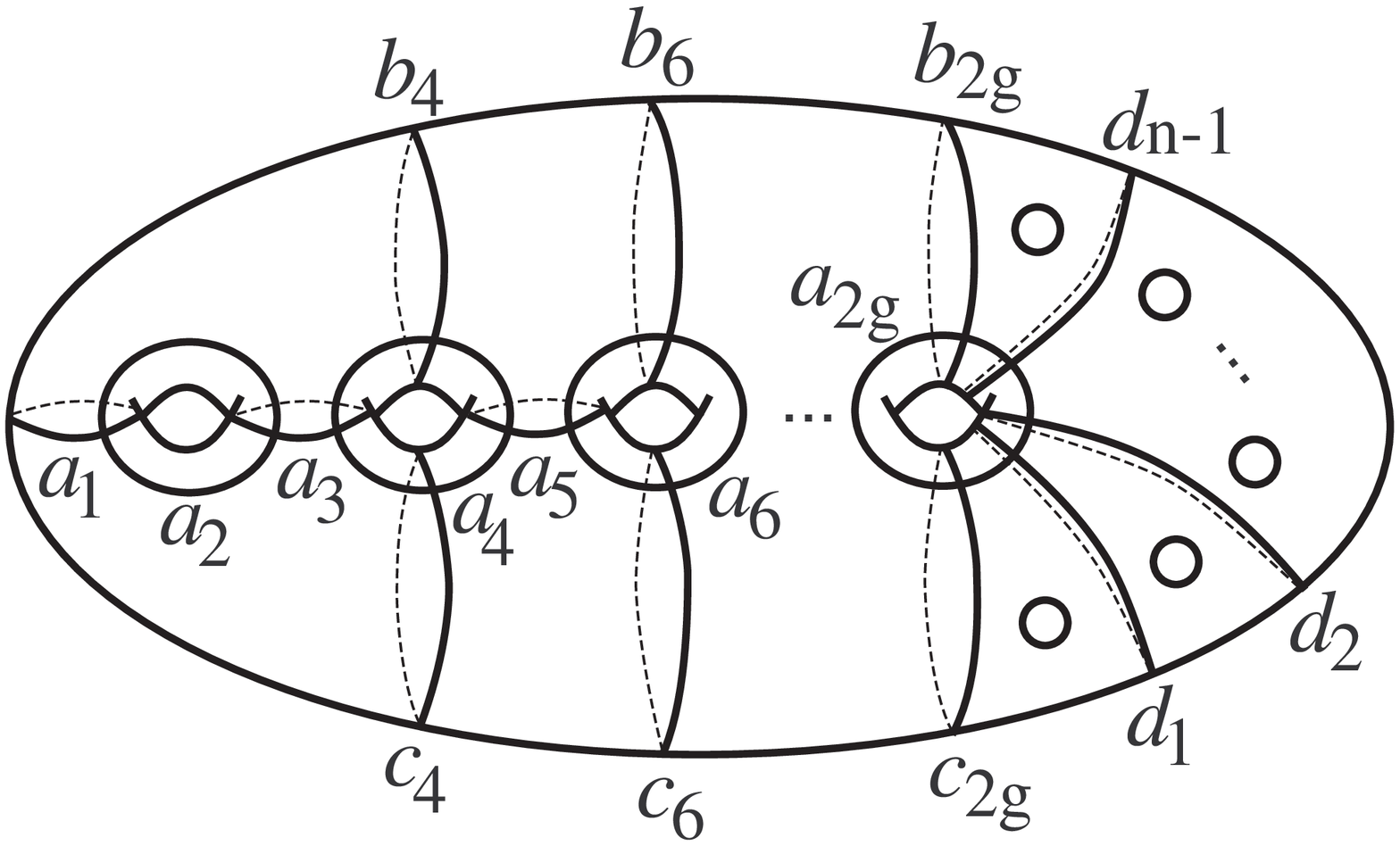} \epsfxsize=1.8in \hspace{-1.4cm} \epsfxsize=3.0in \epsfbox{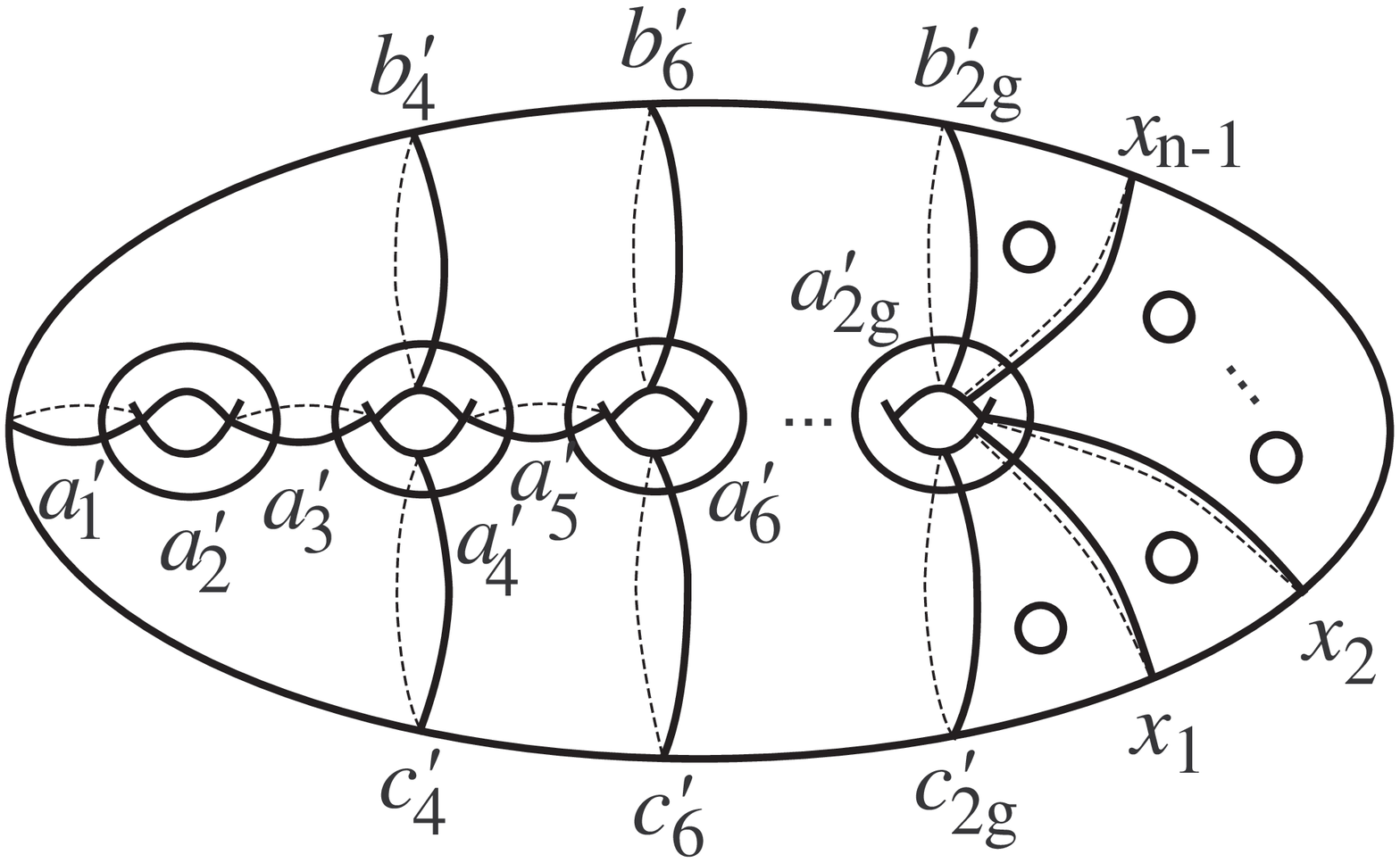} \epsfxsize=1.8in

\hspace{-0.1cm}   (i) \hspace{5.6cm} (ii)
\caption {Peripheral pairs}
\label{fig4-c}
\end{center}
\end{figure}
 
Now consider the curve configuration $P \cup \{a_2, a_4, \cdots, a_{2g}\}$ shown in Figure \ref{fig4-c} (i). Let $a'_2, a'_4, \cdots, a'_{2g}$ be pairwise disjoint representatives of $\lambda([a_{2}]), \lambda([a_{4}]), \cdots,  \lambda([a_{2g}])$ respectively such that $a'_2, a'_4, \cdots, a'_{2g}$ have minimal intersection with the elements of $P'$. By Lemma \ref{int-one} geometric intersection one is preserved. So, a regular neighborhood of union of
all the elements in $\{a'_1, a'_2, \cdots, a_{2g}', b'_4, b'_6, \cdots, b_{2g}, c'_4, c'_6, \cdots, c_{2g}, d'_1, d'_2, \cdots, d'_{n-1}\}$ is an orientable surface of genus $g$ with several boundary components. By using Lemma \ref{embedded} we can see that if three curves in $P$ bound a pair of pants then the corresponding curves bound a pair of pants. This shows that the curves are as shown in Figure \ref{fig4-c} (ii) where 
$\{x_1, x_2, \cdots, x_{n-1}\} = \{d'_1, d'_2, \cdots, d'_{n-1}\}$. Since we know that $b_{2g}'$ is adjacent to $d_{n-1}'$ with respect to $P'$, we see that 
$x_{n-1}= d_{n-1}'$ and $\{x_1, x_2, \cdots, x_{n-2}\} = \{d'_1, d'_2, \cdots, d'_{n-2}\}$.\end{proof}\\

\begin{figure}[htb]
\begin{center}
\hspace{1.3cm} \epsfxsize=3.0in \epsfbox{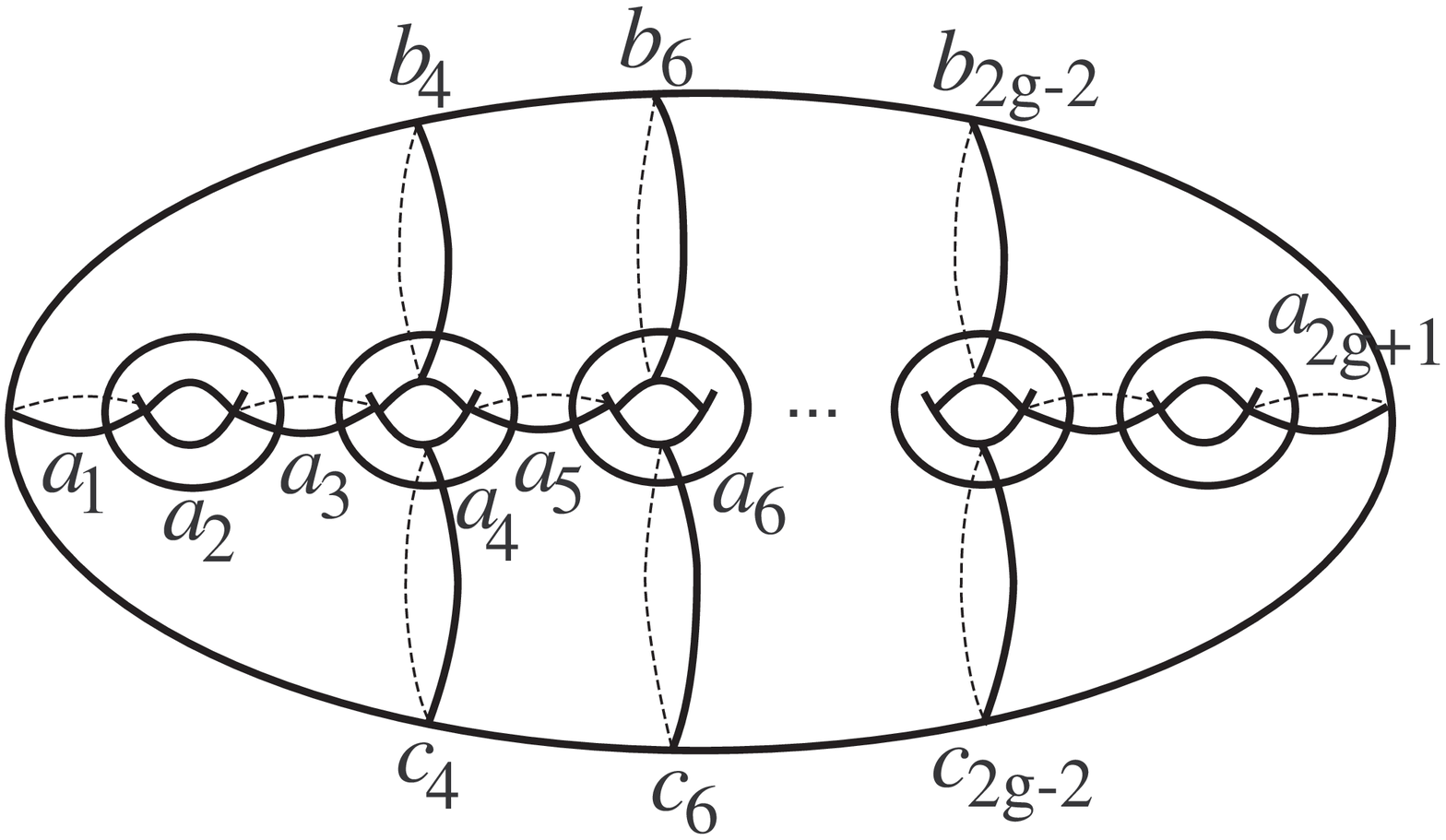}  \hspace{-1.4cm} \epsfxsize=3.0in \epsfbox{fig1-aa.eps} 

\hspace{-0.1cm}   (i) \hspace{5.6cm} (ii)

\hspace{1.5cm} \epsfxsize=3.0in \epsfbox{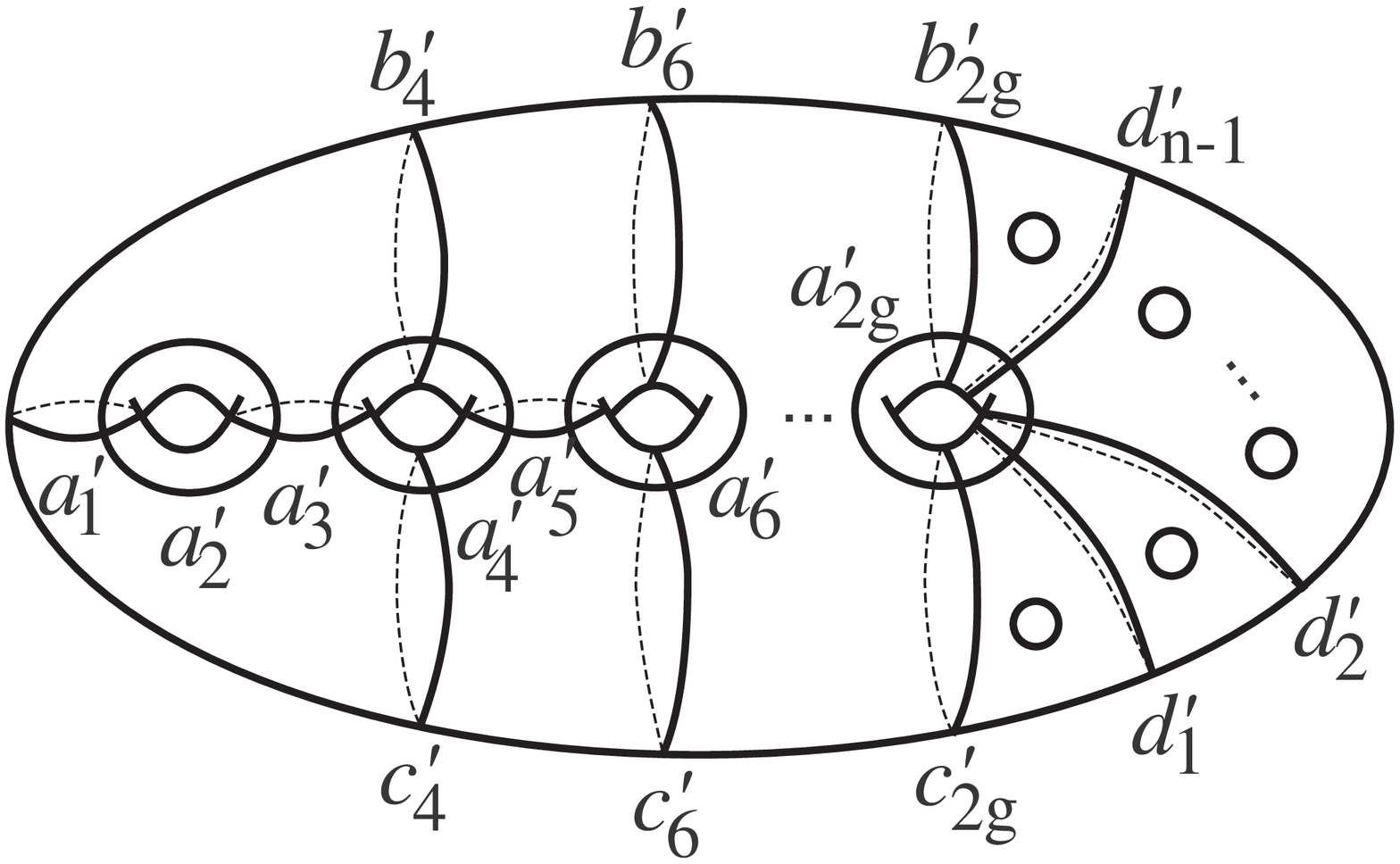} 
 
\hspace{-0.1cm}  (iii)

\caption {Curves in $\mathcal{C}_1$}
\label{fig10}
\end{center}
\end{figure}
 
If $f: R \rightarrow R$ is a homeomorphism, then we will use the same notation for $f$ and $[f]_*$.
Let $\mathcal{C}_1 = \{a_1, a_2, \cdots, a_{2g+1}, b_4, b_6, \cdots, b_{2g-2}, c_4, c_6,$ $ \cdots,$ $ c_{2g-2}\}$ when $R$ is closed, and
$\mathcal{C}_1 = \{a_1, a_2, \cdots, a_{2g}, b_4, b_6, \cdots, b_{2g}, c_4, c_6, \cdots,$ $ c_{2g}, d_1, d_2,$ $ \cdots, d_{n-1}\}$ when $R$ has boundary 
where the curves are as shown in Figure \ref{fig10} (i), (ii).  

\begin{lemma} \label{curves} Suppose $g \geq 2$ and $n \geq 0$. There exists a 
homeomorphism $h: R \rightarrow R$ such that $h([x]) = \lambda([x])$ $\forall \ x \in \mathcal{C}_1$.\end{lemma}

\begin{proof} We will give the proof when $n \geq 1$. The proof for the closed case is similar. We will consider all the curves in  $\mathcal{C}_1$ as 
shown in Figure \ref{fig10} (ii). Let the curves $a'_1, a'_2, \cdots,
a_{2g}', b'_4, b'_6, \cdots, b_{2g}', c'_4, c'_6,$ $ \cdots, c_{2g}',
d'_1, d'_2, \cdots, d_{n-1}'$ be representatives in minimal position of the elements $\lambda([a_1]), \lambda([a_2]), \cdots,  \lambda([a_{2g}]),$ $ \lambda([b_4]), \lambda([b_6]), \cdots, \lambda([b_{2g}]), \lambda([c_4]), \lambda([c_6]),$ $ \cdots, \lambda([c_{2g}]),$ $ \lambda([d_1]), \lambda([d_2]), \cdots, \lambda([d_{n-1}])$ respectively.
By Lemma \ref{int-one} geometric intersection one is preserved. So, a regular neighborhood of union of all the elements in
$\mathcal{C}_1'= \{a'_1, a'_2, \cdots,$ $ a_{2g}', b'_4, b'_6, \cdots,$ $ b_{2g}', c'_4, c'_6, \cdots, c_{2g}', d'_1, d'_2, \cdots,$
$d'_{n-1}\}$ is an orientable surface of genus $g$ with several boundary components. By Lemma \ref{embedded}, if three nonseparating curves in $\mathcal{C}_1$ 
bound a pair of pants on $R$, then the corresponding curves in $\mathcal{C}_1'$ bound a pair of pants on $R$. By Lemma \ref{peripheral} if two curves in $\mathcal{C}_1$ are peripheral pairs, then the corresponding curves
in $\mathcal{C}_1'$ are peripheral pairs. All these imply that the curves in $\mathcal{C}_1'$ are as shown in 
Figure \ref{fig10} (iii). Hence, there exists a homeomorphism $h: R \rightarrow R$ such that $h([x]) = \lambda([x])$ $\forall \ x \in \mathcal{C}_1$.\end{proof}\\

\begin{figure} \begin{center}
\hspace{-0.2cm} \epsfxsize=2.99in \epsfbox{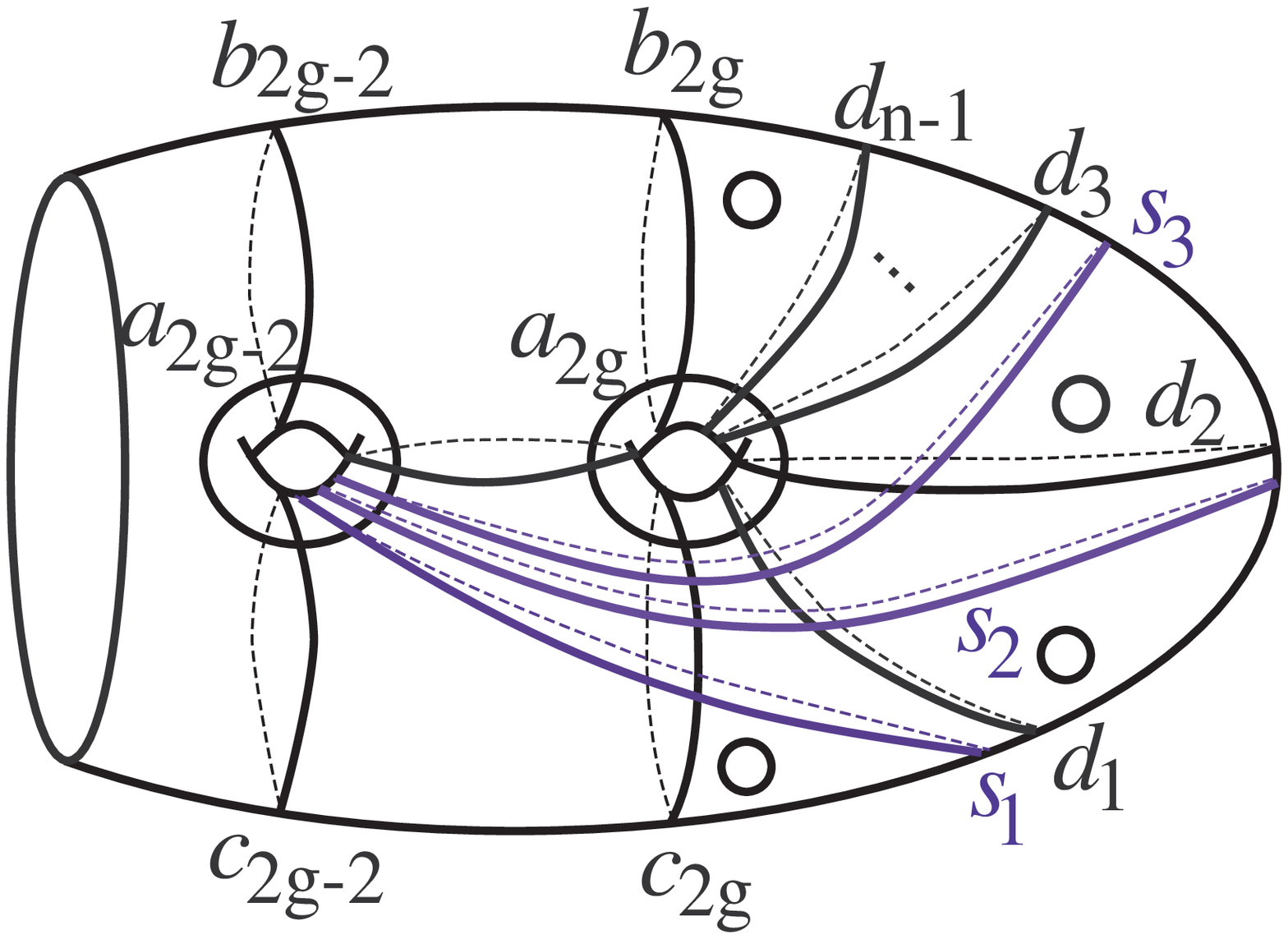} \hspace{-1.2cm} \epsfxsize=2.99in \epsfbox{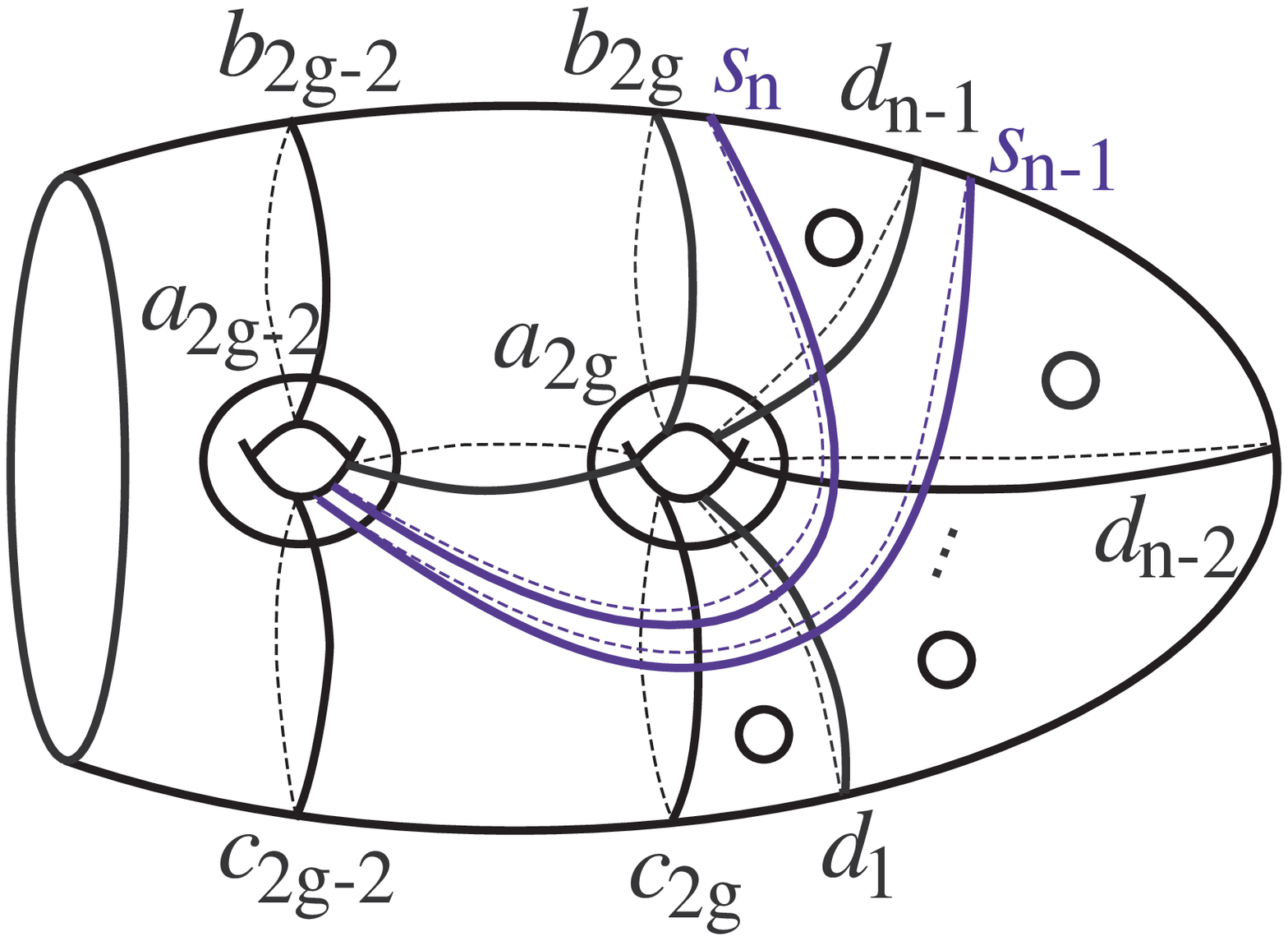}
 
\hspace{-0.6cm} (i) \hspace{6.3cm} (ii)
 
\hspace{-0.2cm} \epsfxsize=2.99in \epsfbox{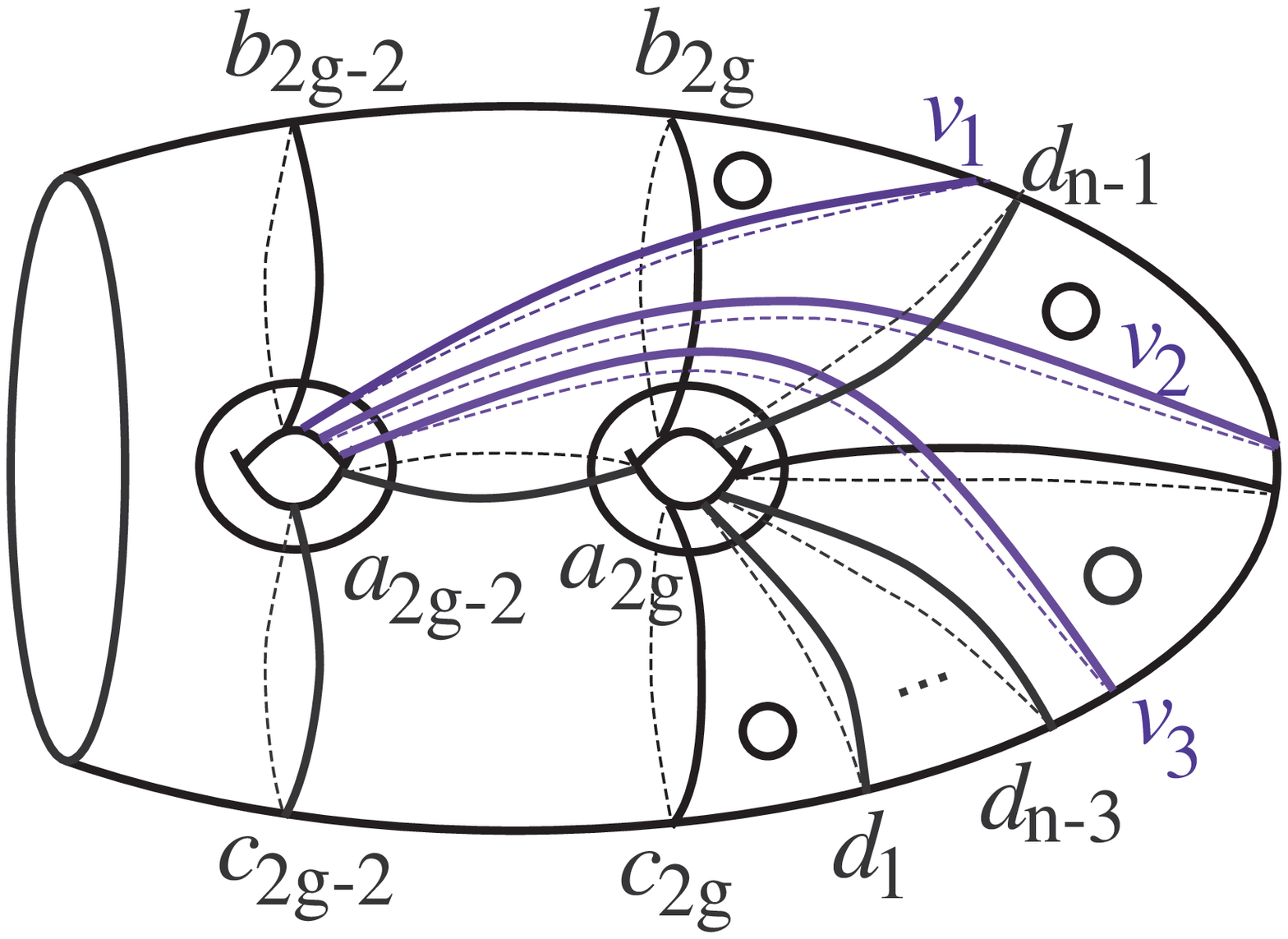} \hspace{-1.2cm} \epsfxsize=2.99in \epsfbox{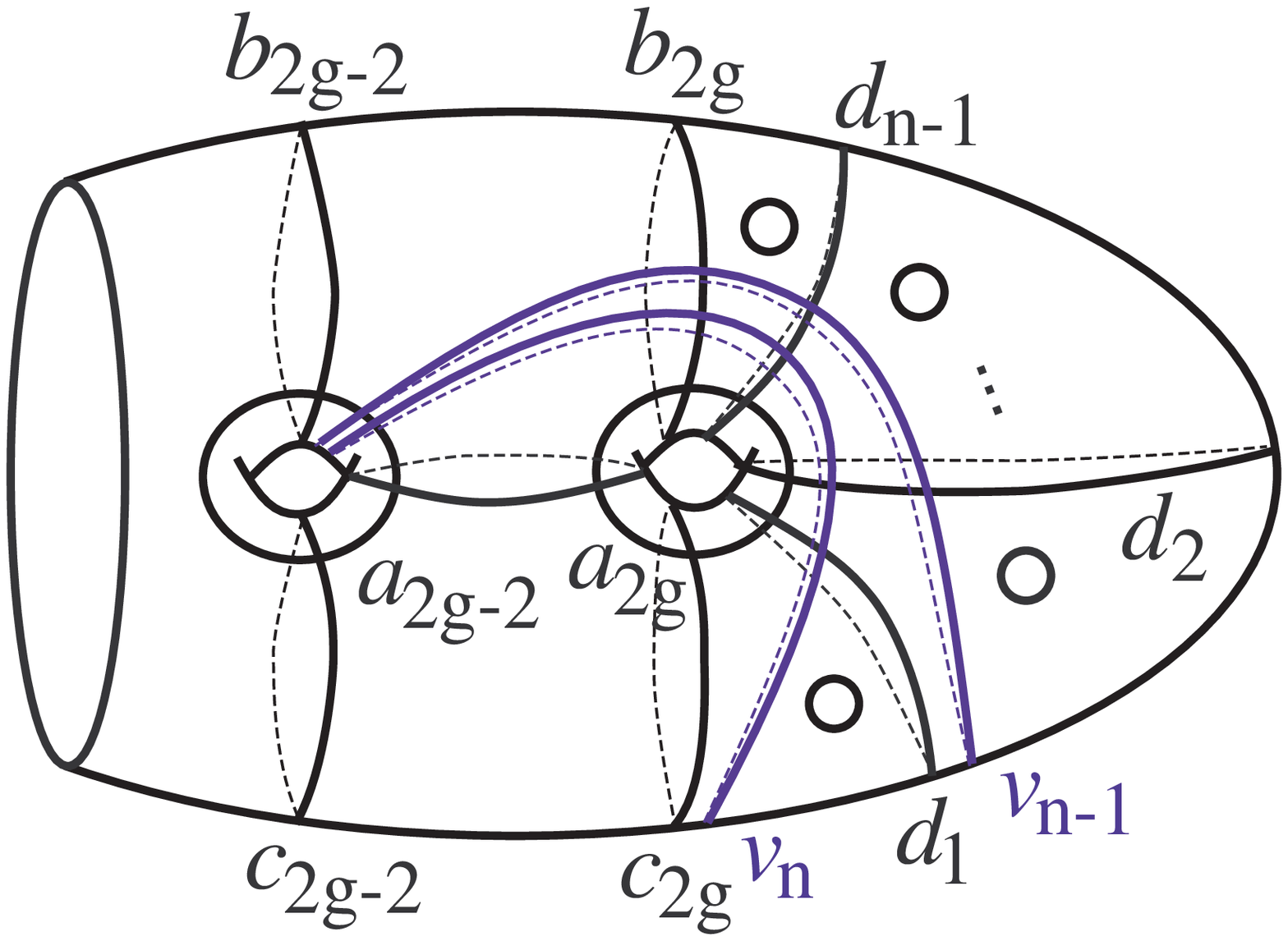}
  
\hspace{-0.6cm} (iii) \hspace{6.1cm} (iv)
  
\hspace{-0.2cm} \epsfxsize=2.99in \epsfbox{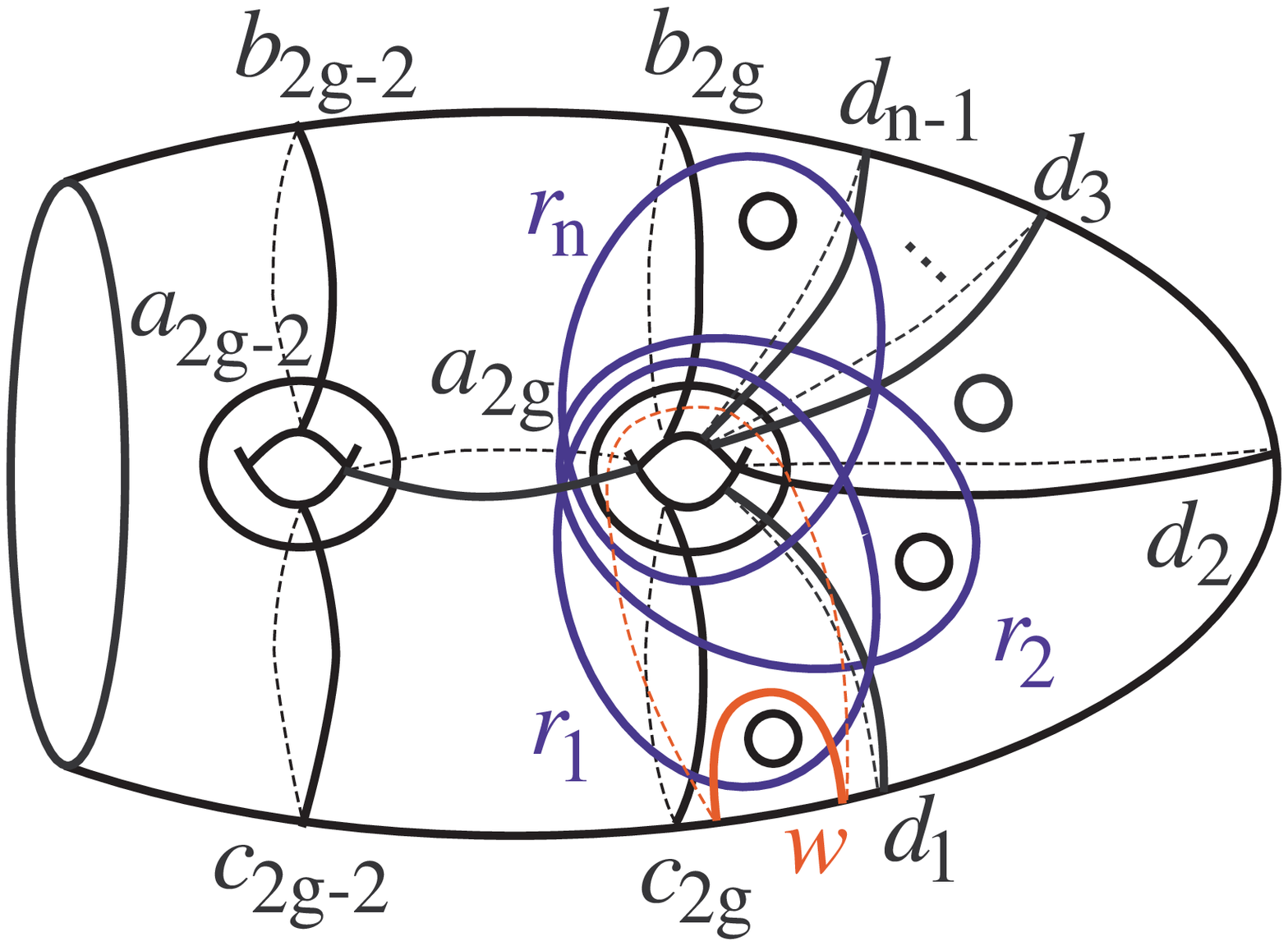} \hspace{-1.2cm} \epsfxsize=2.99in \epsfbox{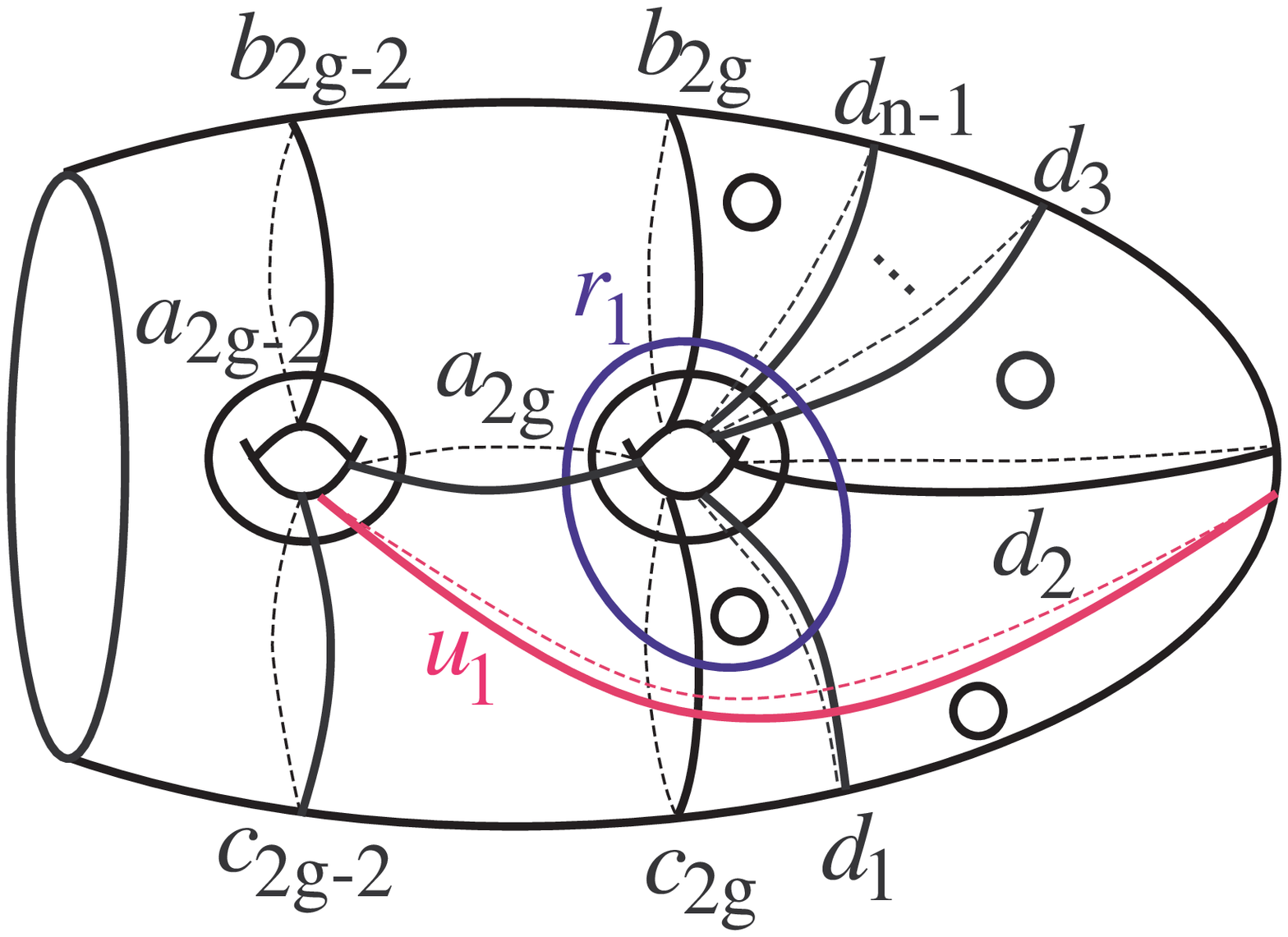}
  
\hspace{-0.6cm} (v) \hspace{6.1cm} (vi)
 
\hspace{-0.2cm} \epsfxsize=2.99in \epsfbox{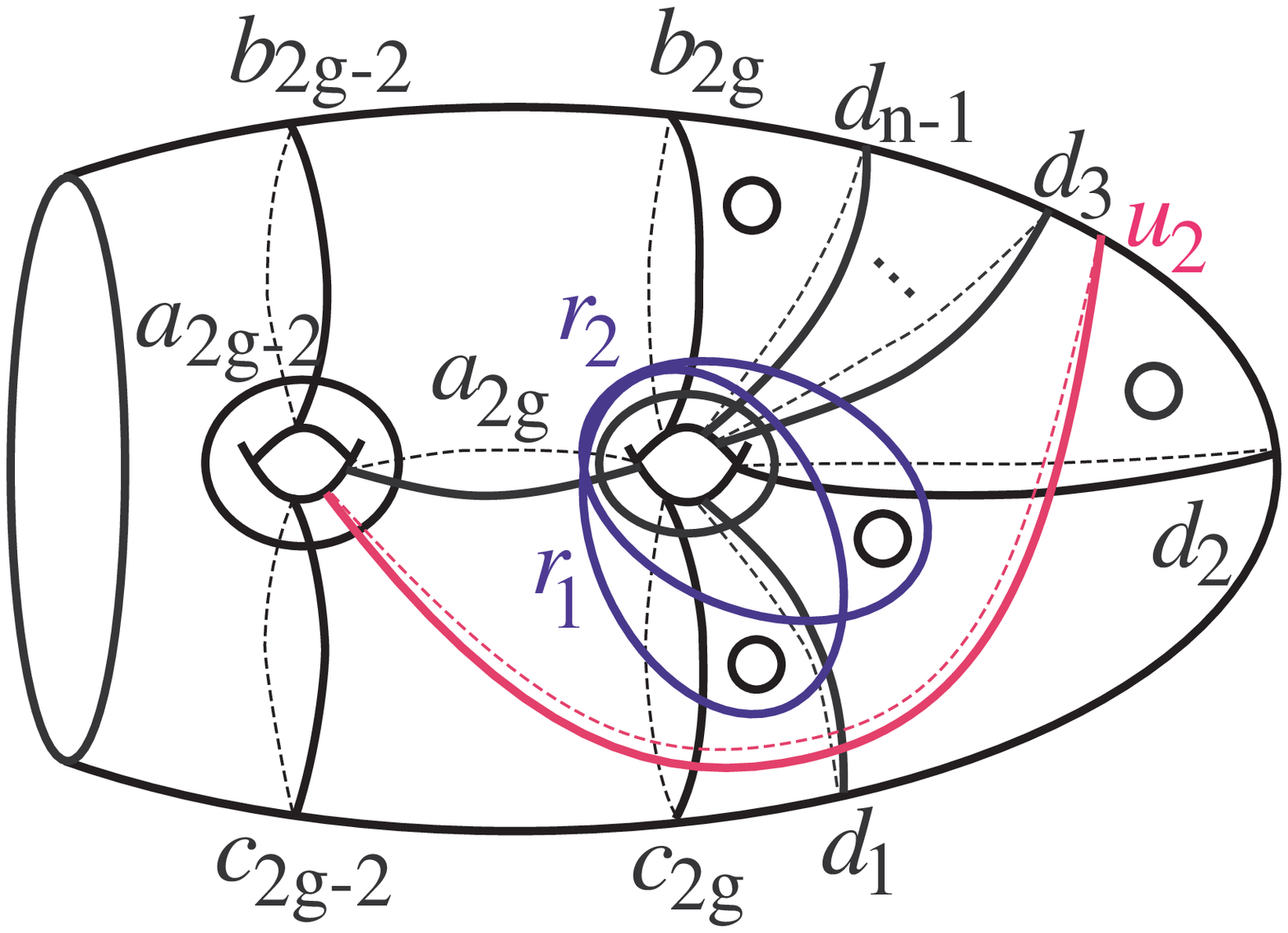} \hspace{-1.2cm} \epsfxsize=2.99in \epsfbox{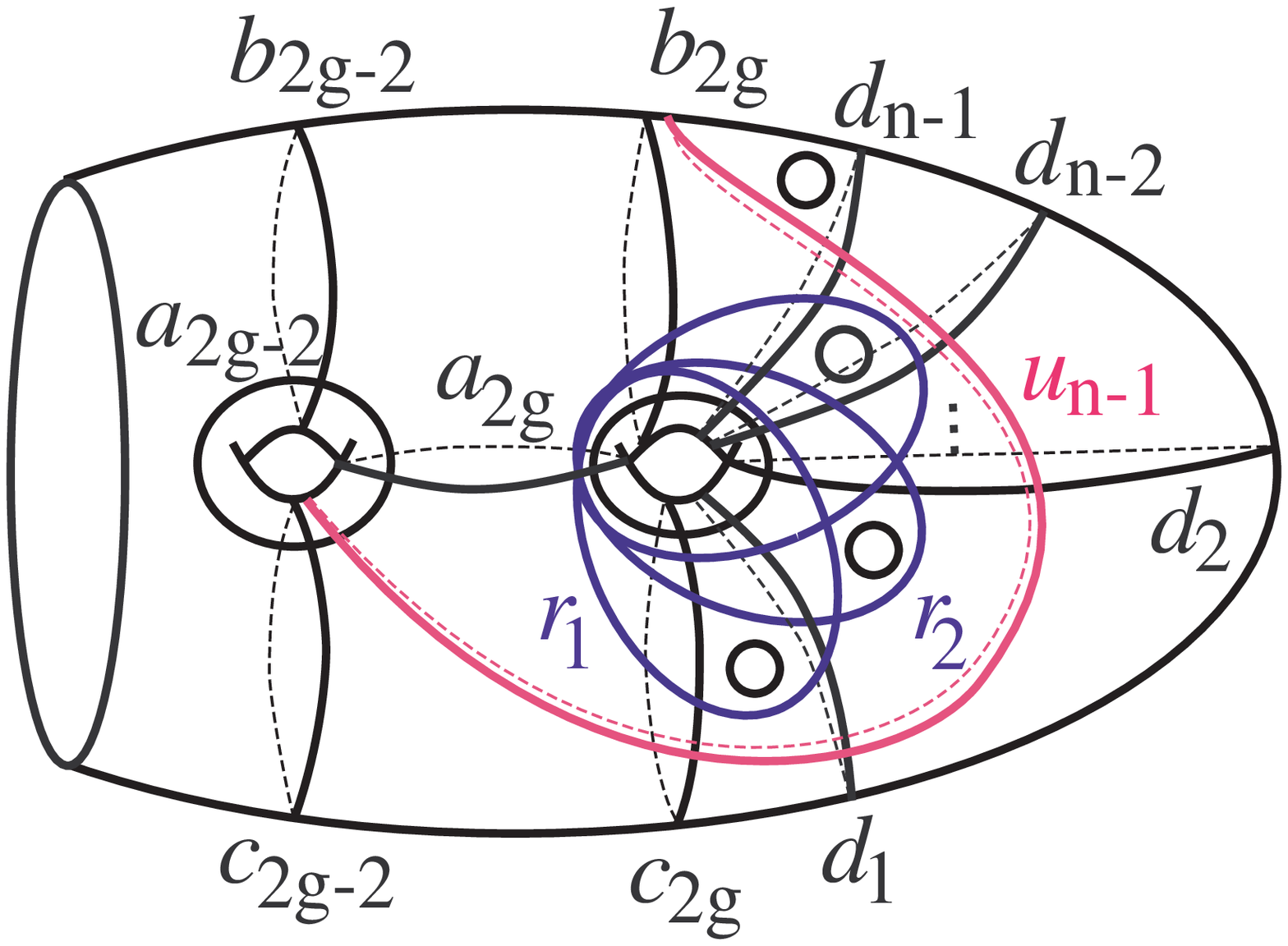} 

\hspace{-0.3cm} (vii) \hspace{5.8cm} (viii)
\caption{Curves in $\mathcal{C}_2$} \label{fig4}
\end{center}
\end{figure}

When $n \geq 1$, let $\mathcal{C}_2 = \{r_1, r_2, \cdots, r_{n}, s_1, s_2, \cdots, s_{n}, u_1, u_2, \cdots, u_{n-1}, v_1, v_2, \cdots, v_{n}, w\}$ 
where the curves are as shown in Figure \ref{fig4}.

\begin{lemma} \label{curves-II} Suppose $g \geq 2$ and $n \geq 1$. There exists a 
homeomorphism $h: R \rightarrow R$ such that $h([x]) = \lambda([x])$ $\forall \ x \in \mathcal{C}_1 \cup \mathcal{C}_2$.\end{lemma}

\begin{proof} We will consider all the curves in  Figure \ref{fig4}. By Lemma \ref{curves} there exists a homeomorphism $h: R \rightarrow R$ such tha $h([x]) = \lambda([x])$ $\forall \ x \in \mathcal{C}_1$. 

The curve $s_1$ is the unique nontrivial simple closed curve up to isotopy that is disjoint from each curve in 
$\mathcal{C}_1 \setminus \{a_{2g-2}, c_{2g}\}$ that intersects $a_{2g-2}$ once and it is not isotopic to $c_{2g-2}$. Since $h([x]) = \lambda([x])$ for all these curves, $\lambda$ is edge-preserving and it preserves intersection one, we have $h([s_1]) = \lambda([s_1])$. The curve $s_2$ is the unique nontrivial simple closed curve up to 
isotopy that is disjoint from each curve in $(\mathcal{C}_1 \cup \{s_1\}) \setminus \{a_{2g-2}, c_{2g}, d_{1}\}$ that intersects $a_{2g-2}$ once and it is not isotopic to $s_{1}$. Since $h([x]) = \lambda([x])$ for all these curves and these properties are preserved by $\lambda$, we have $h([s_2]) = \lambda([s_2])$. Similarly, we get  $h([s_i]) = \lambda([s_i])$ $ \forall i= 3, 4, \cdots, n$.

The curve $v_1$ is the unique nontrivial simple closed curve up to isotopy that is disjoint from each curve in 
$\mathcal{C}_1 \setminus \{a_{2g-2}, b_{2g}\}$ that intersects $a_{2g-2}$ once and it is not isotopic to $b_{2g-2}$. Since $h([x]) = \lambda([x])$ for all these curves and these properties are preserved by $\lambda$, we have $h([v_1]) = \lambda([v_1])$. The curve $v_2$ is the unique nontrivial simple closed curve  up to 
isotopy that is disjoint from each curve in $(\mathcal{C}_1 \cup \{v_1\}) \setminus \{a_{2g-2}, b_{2g}, d_{n-1}\}$ that intersects $a_{2g-2}$ once, and it is not isotopic to $v_{1}$. Since $h([x]) = \lambda([x])$ for all these curves and these properties are preserved by $\lambda$, we have $h([v_2]) = \lambda([v_2])$. Similarly, we get $h([v_i]) = \lambda([v_i])$ $ \forall i= 3, 4, \cdots, n$.
 
Consider the curve $w$ as shown in the Figure \ref{fig4} (v). There exists a
homeomorphism $\phi : R \rightarrow R$ of order two such that the map $\phi_{*}$ induced by $\phi$ on $\mathcal{N}(R)$ sends the isotopy class of each curve in $\mathcal{C}_1 \cup \{s_1, s_2, \cdots, s_{n}, v_1, v_2, \cdots, v_{n}\}$ to itself and switches 
$[r_1]$ and $[w]$. We can see that $\lambda([r_1]) \neq \lambda([w])$ as follows: Since $h([x]) = \lambda([x])$ when $x = v_1$ and $x = b_{2g}$, 
$\lambda([v_1]), \lambda([b_{2g}])$ intersect nontrivially. Consider the curve $p_{n-1}$ given in Figure \ref{fig5} (vii). Since there is a homeomorphism sending the pair 
$(p_{n-1}, r_1)$ to the pair $(v_1, b_{2g})$, we can choose similar curve configurations to see that  $\lambda([p_{n-1}]), \lambda([r_{1}])$ intersect nontrivially. 
Since there is an edge between $[p_{n-1}]$ and $[w]$, there is an edge between $\lambda([p_{n-1}]) $ and $\lambda([w])$, so 
$\lambda([p_{n-1}]), \lambda([w])$ have geometric intersection zero. Hence, $\lambda([r_1]) \neq \lambda([w])$. 
There are only two nontrivial simple closed curves, namely $r_1$ and $w$, up to isotopy that are disjoint from each of $c_{2g-2}, a_{2g-2}, v_{n-1}$, bounds a pair of pants with $a_{2g}$ and a boundary component of $R$
and intersects each of $a_{2g-1}, b_{2g}, c_{2g}, d_1, d_2, \cdots, d_{n-1}$ once. Since we know that $h([x]) = \lambda([x])$ for all these curves, $\lambda$ preserves these properties by Lemma \ref{embedded} and Lemma \ref{int-one}, and $\lambda([r_1]) \neq \lambda([w])$, by replacing $h$ with $h \circ \phi$ if necessary, we can assume that we have $h([r_1]) = \lambda([r_1])$ and 
$h([w]) = \lambda([w])$. We note that to get the proof of the lemma, it is enough to prove the result for
this $h$. The curve $r_2$ is the unique nontrivial simple closed curve  up to isotopy that is disjoint from each of $a_{2g-2}, b_{2g-2}, s_{1}, a_{2g}, v_{n-2}, w$, bound a pair of pants with $a_{2g}$ and a boundary component of $N$,  
and intersects each of $a_{2g-1}, b_{2g}, c_{2g}, d_1, d_2, \cdots, d_{n-1}$ once. 
Since we know that $h([x]) = \lambda([x])$ for all these curves and $\lambda$ preserves these properties, we see that $h([r_2]) = \lambda([r_2])$. Similarly, we get  $h([r_i]) = \lambda([r_i])$ $ \forall i= 3, 4, \cdots, n$.
The curve $u_1$ is the unique nontrivial simple closed curve  up to isotopy that is disjoint from each curve in $(\mathcal{C}_1 \cup \{r_1\}) \setminus \{a_{2g-2}, c_{2g}, d_1\}$ that intersects $a_{2g-2}$ once and it is not isotopic to $c_{2g-2}$. 
Since $h([x]) = \lambda([x])$ for all these curves and these properties are preserved by $\lambda$, we see that $h([u_1]) = \lambda([u_1])$. The curve $u_2$ is the unique nontrivial simple closed curve  up to isotopy that is disjoint from each of $(\mathcal{C}_1 \cup \{r_1, r_2\}) \setminus \{a_{2g-2}, c_{2g}, d_1, d_2\}$ that intersects $a_{2g-2}$ once and it is not isotopic to $c_{2g-2}$. 
Since $h([x]) = \lambda([x])$ for all these curves and these properties are preserved by $\lambda$, we have $h([u_2]) = \lambda([u_2])$. Similarly, we get  $h([u_i]) = \lambda([u_i])$ $ \forall i= 3, 4, \cdots, n-1$.
Hence, there exists a homeomorphism $h: R \rightarrow R$ such that $h([x]) = \lambda([x])$ $\forall \ x \in  \mathcal{C}_1 \cup \mathcal{C}_2$.\end{proof}\\

\begin{figure} \begin{center} 
\hspace{-0.4cm} \epsfxsize=2.9in \epsfbox{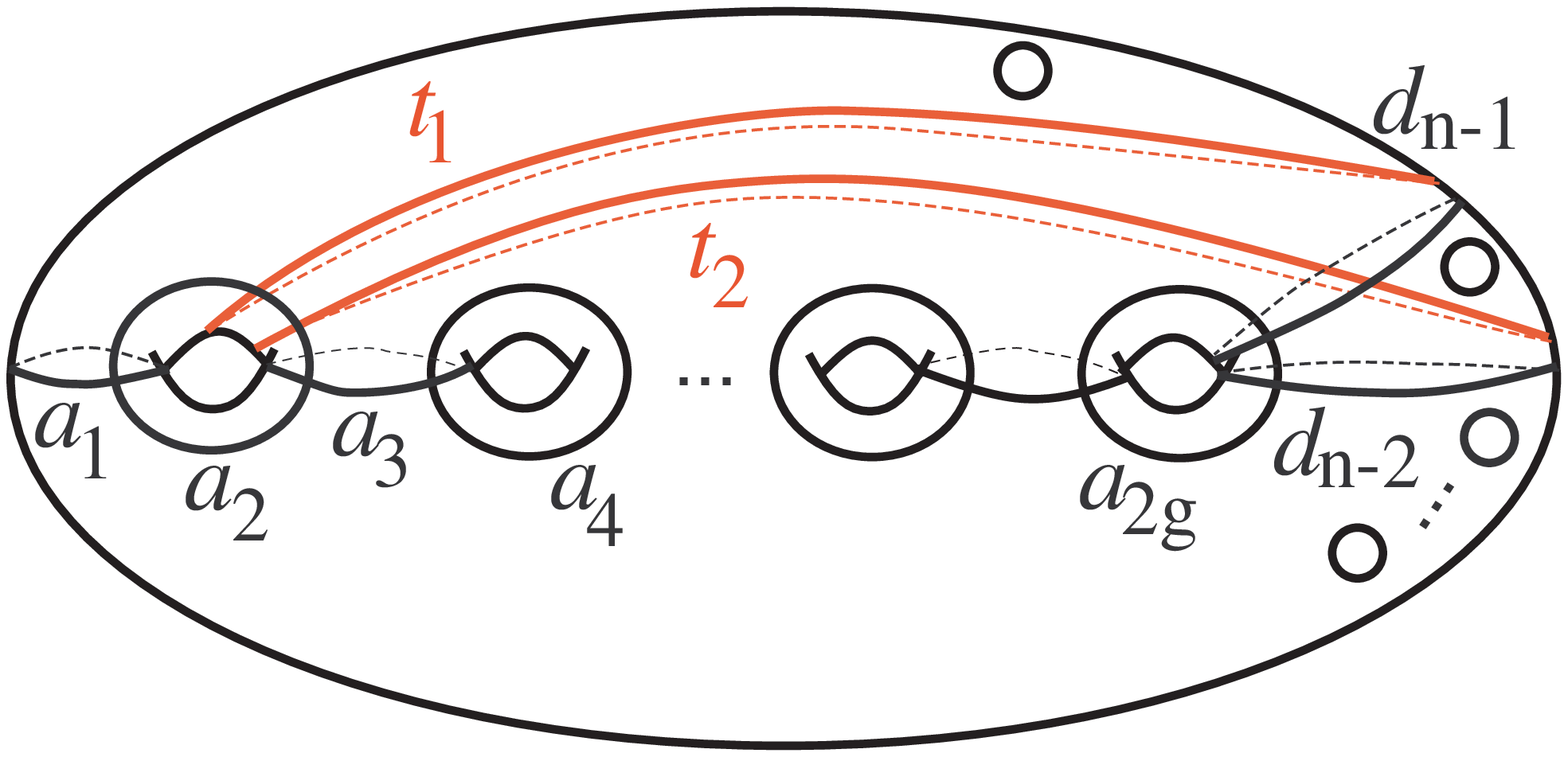} \hspace{-0.9cm} \epsfxsize=2.69in \epsfbox{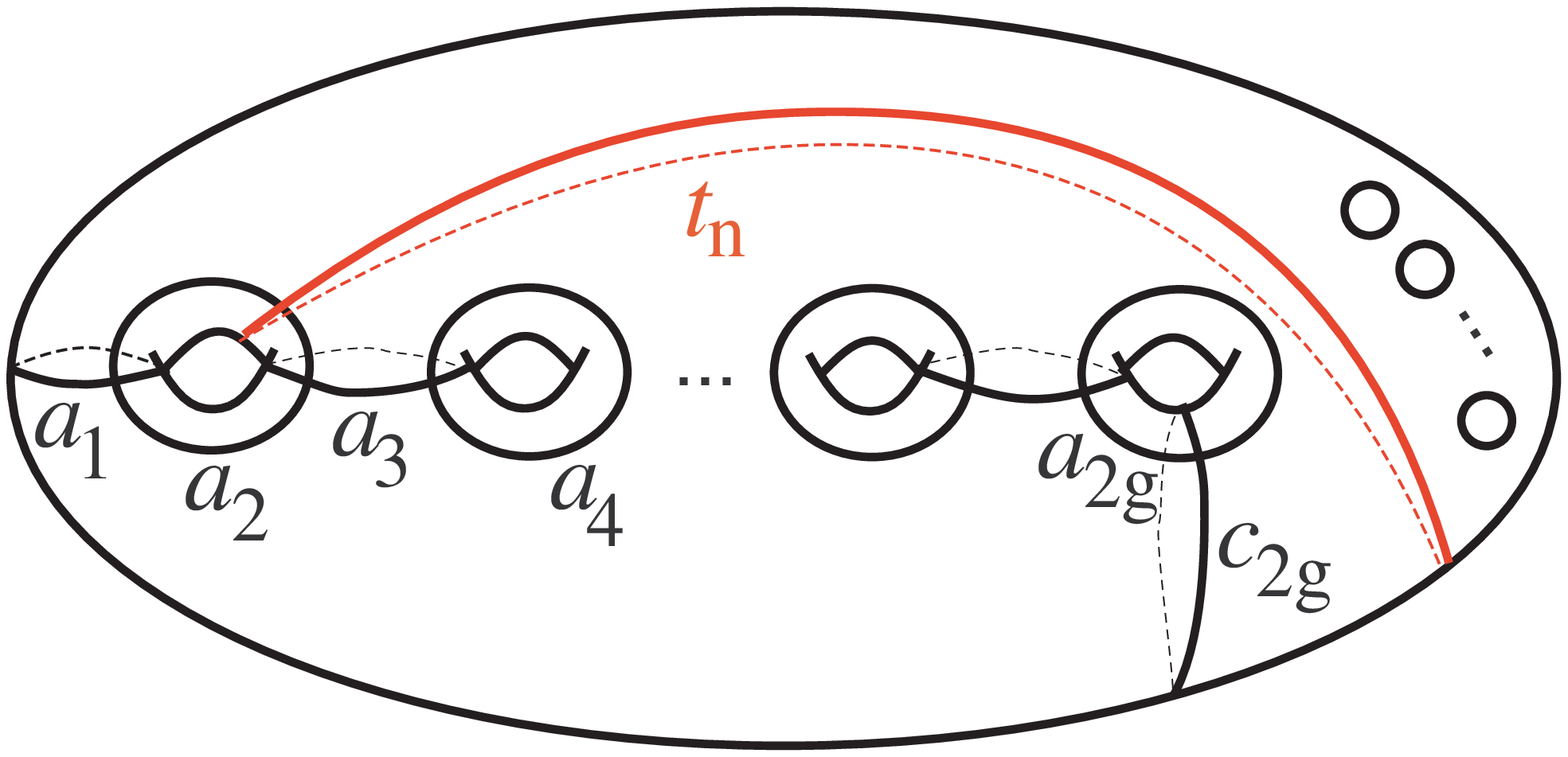} 

\hspace{-0.5cm} (i) \hspace{6.2cm} (ii)

\hspace{0.8cm} \epsfxsize=3.2in \epsfbox{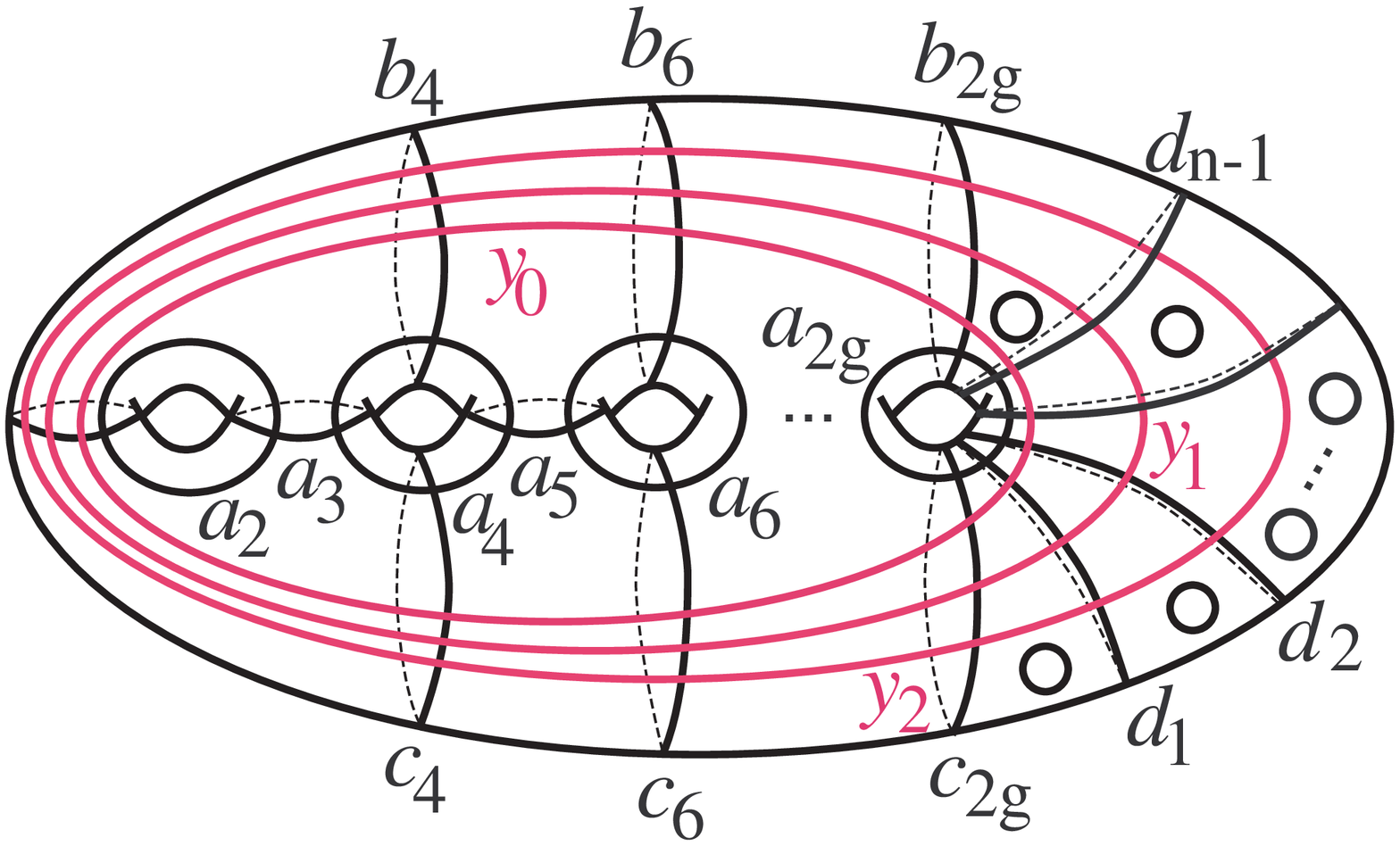} \hspace{-1.7cm} \epsfxsize=3.2in \epsfbox{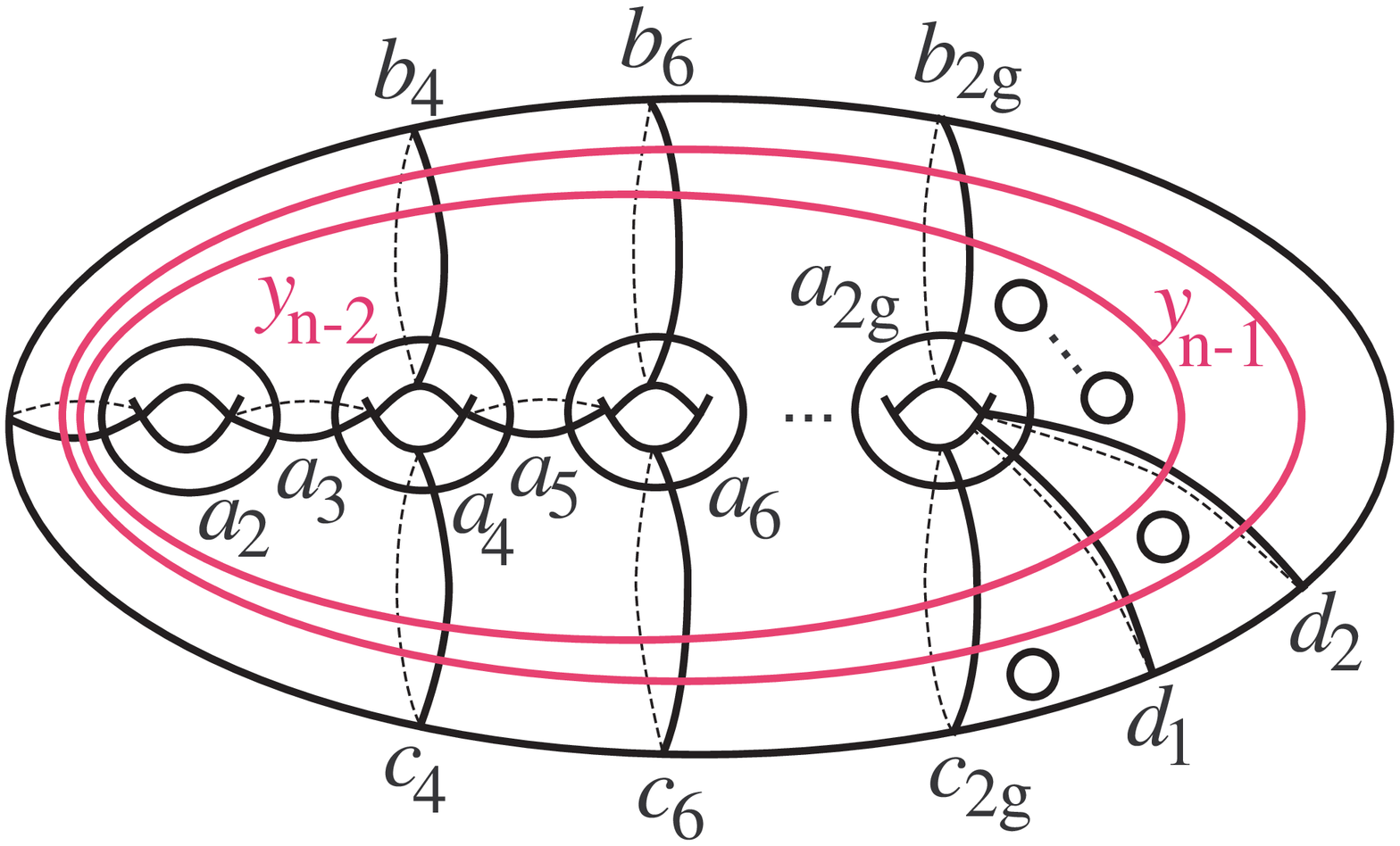} 
 
\hspace{-0.5cm} (iii) \hspace{6.1cm} (iv)

 \hspace{-0.5cm} \epsfxsize=3.2in \epsfbox{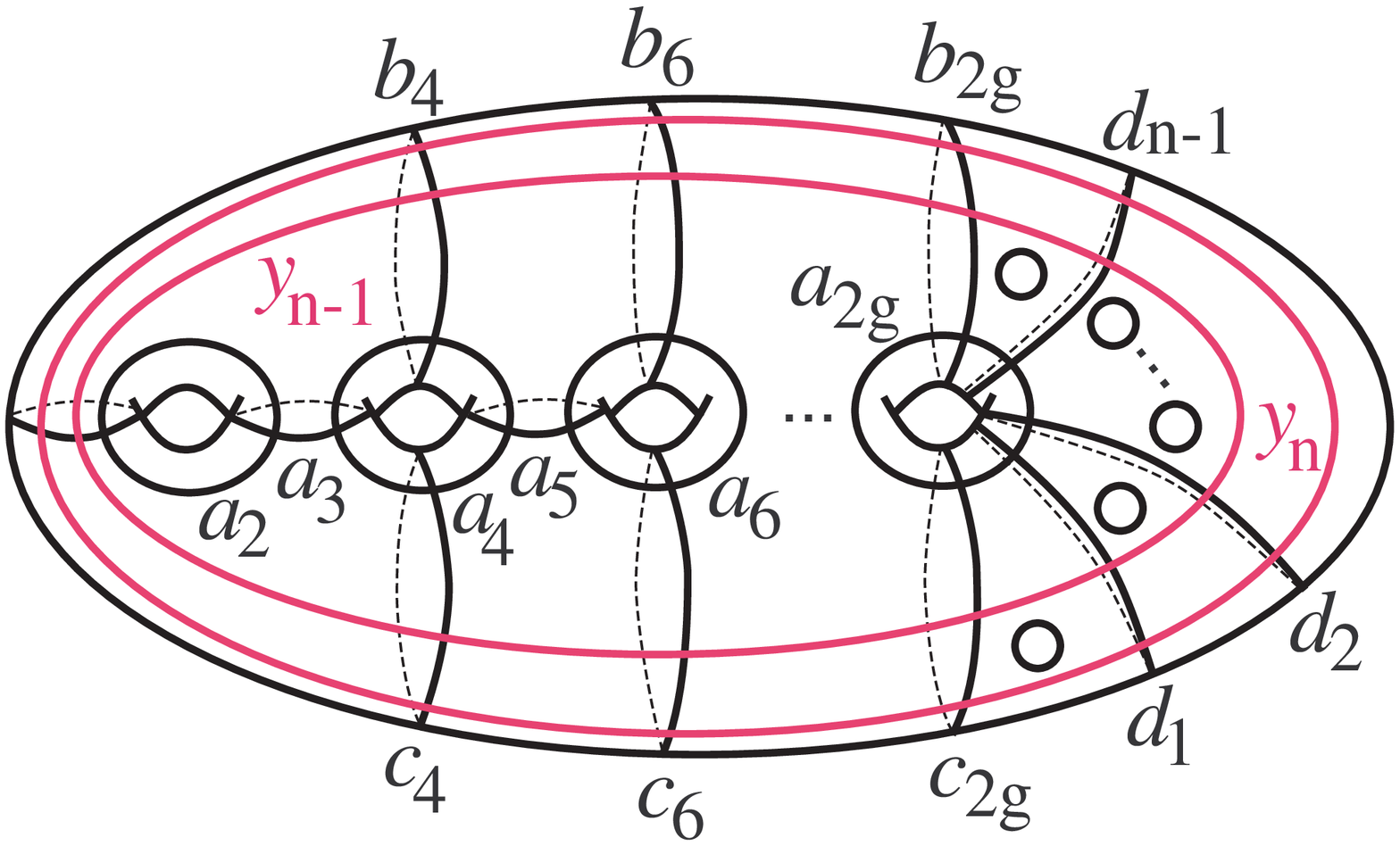} \hspace{-1.7cm}  \epsfxsize=2.69in \epsfbox{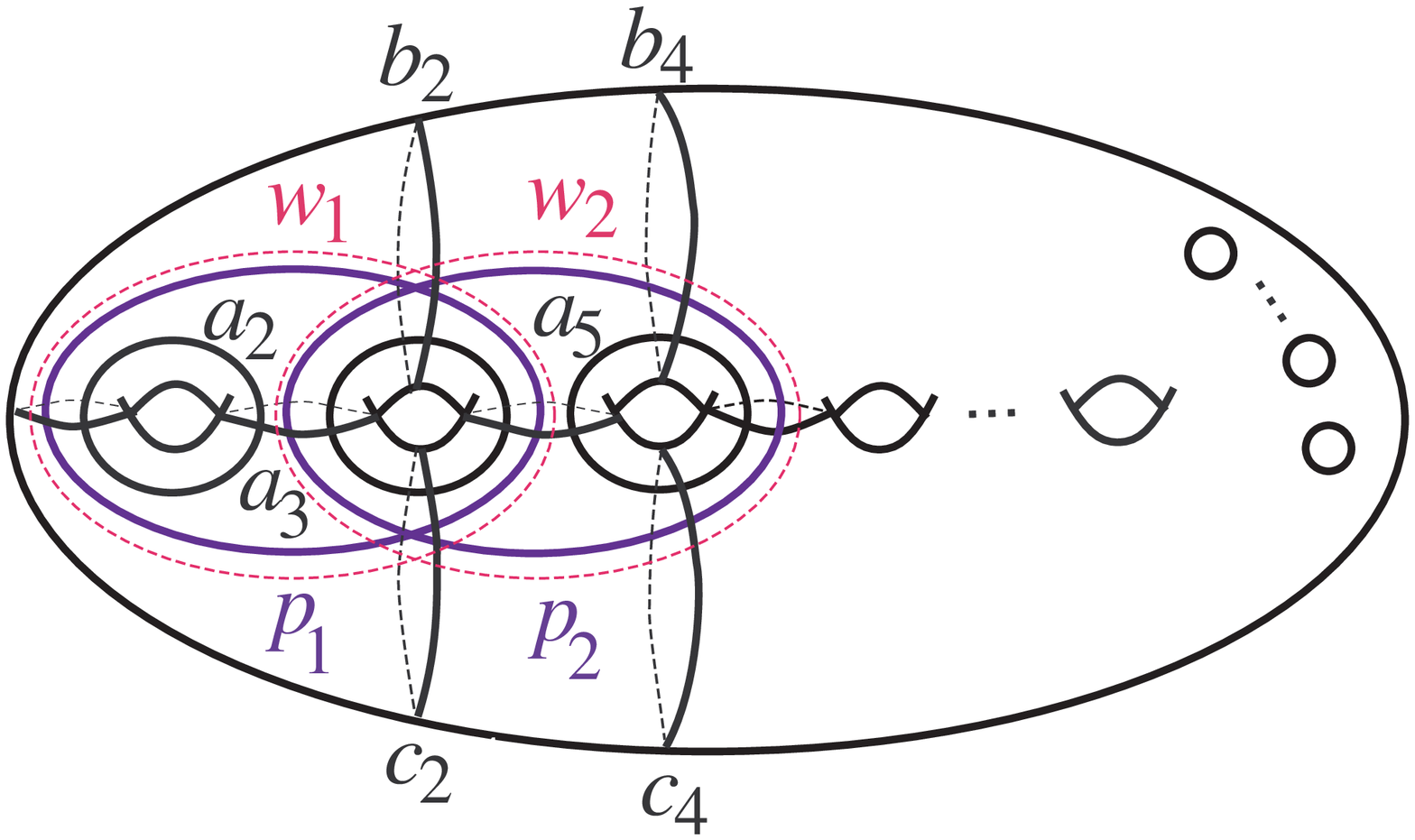} 

\hspace{-0.5cm} (v) \hspace{6.1cm} (vi)

 \hspace{-0.4cm}  \epsfxsize=2.69in \epsfbox{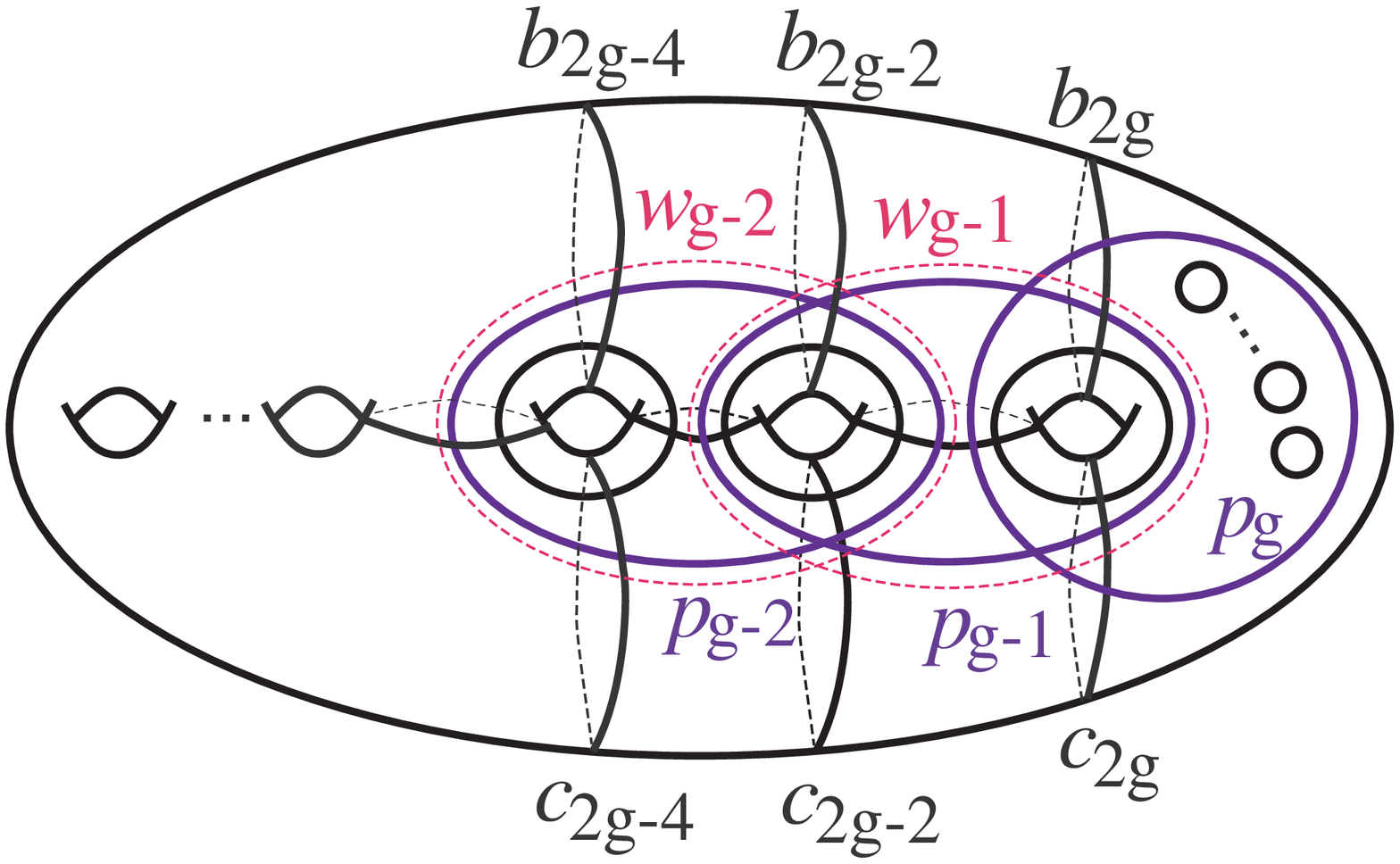}  \hspace{-0.3cm}  \epsfxsize=2.69in \epsfbox{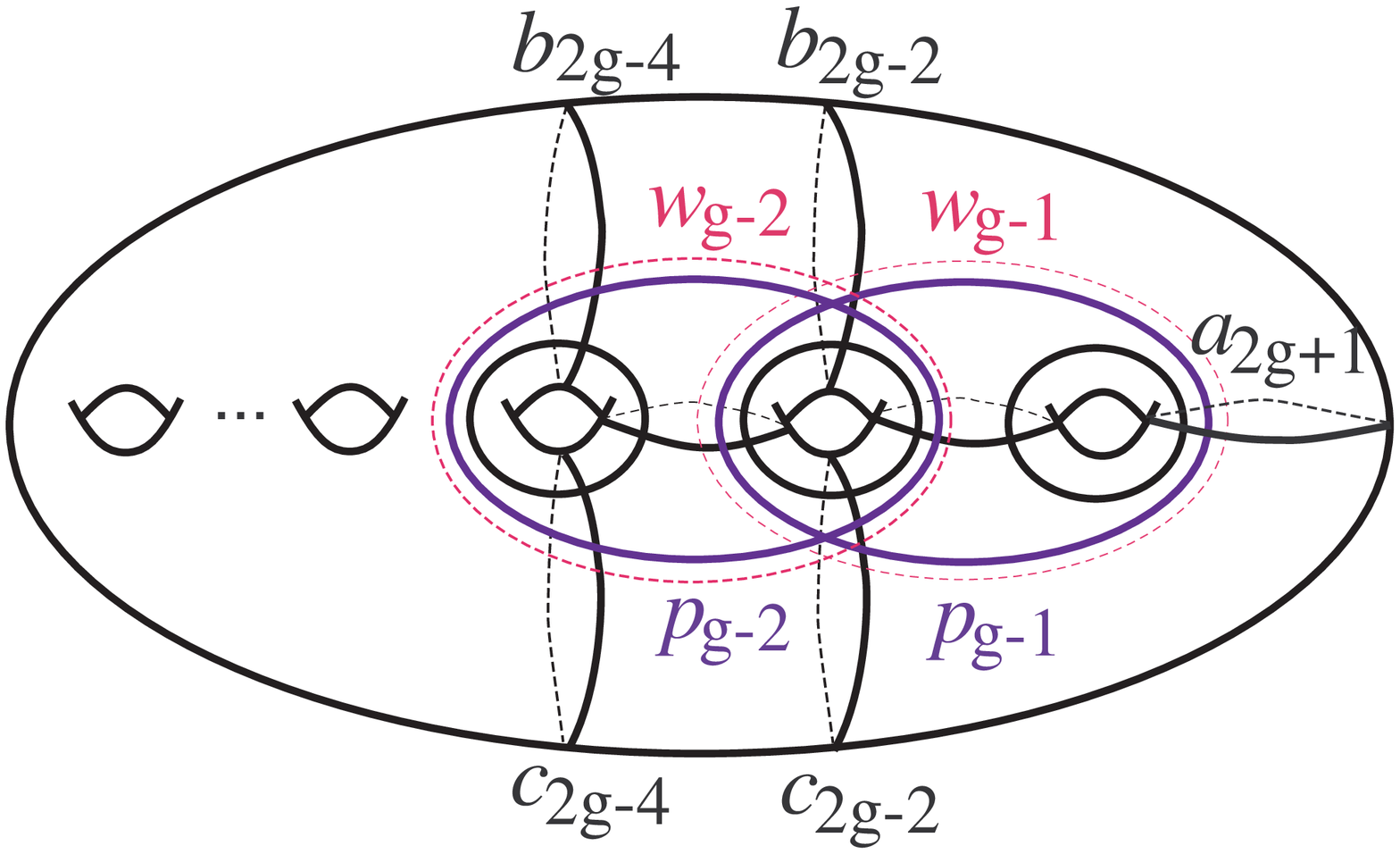} 
 
\hspace{-0.36cm} (vii) \hspace{5.8cm} (viii)
\caption{Curves in $\mathcal{C}_3$} \label{fig5}
\end{center}
\end{figure}

\begin{figure}[htb]
\begin{center}
\hspace{0cm} \epsfxsize=2.5in \epsfbox{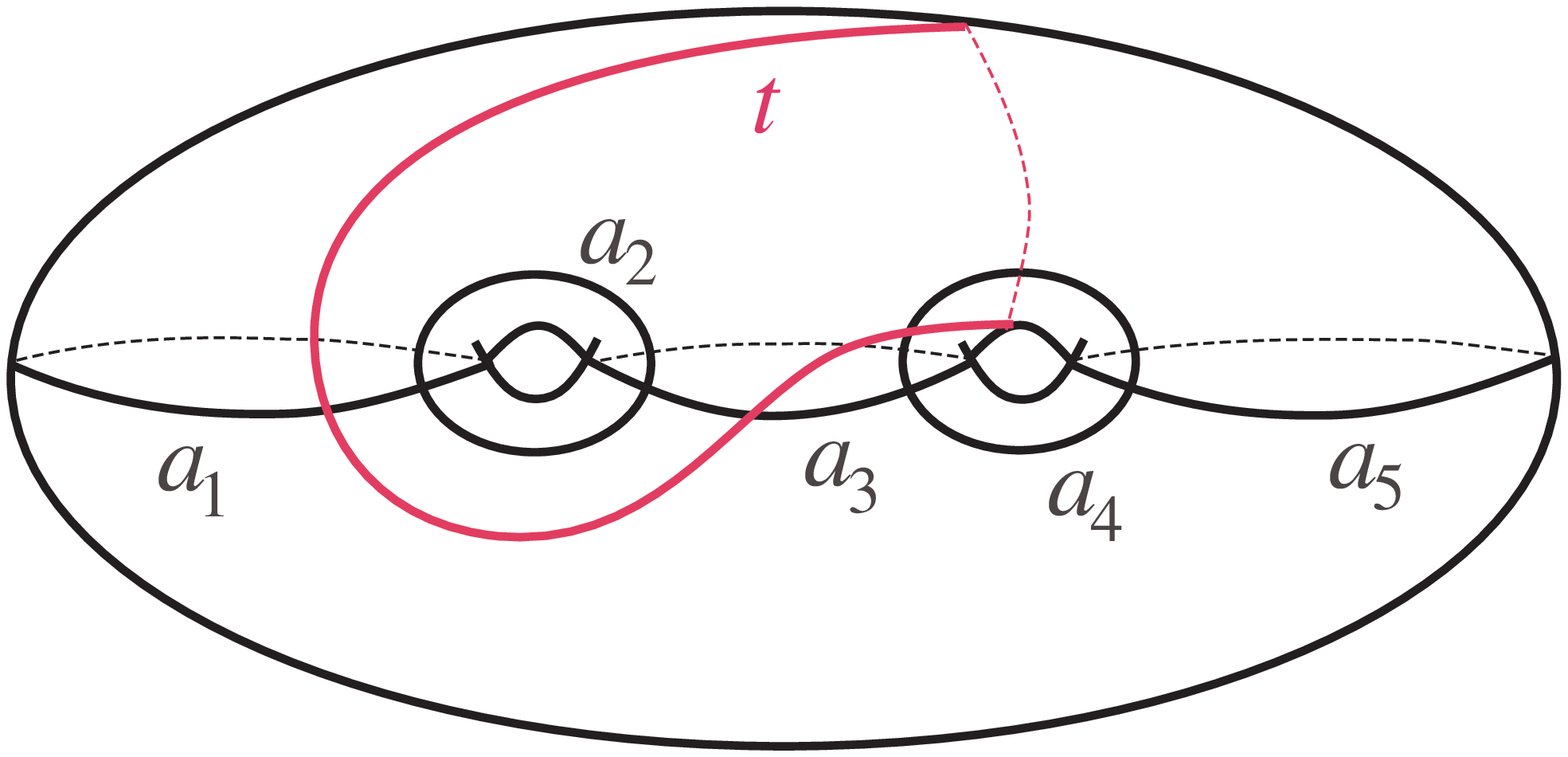} \hspace{0.2cm} \epsfxsize=2.5in \epsfbox{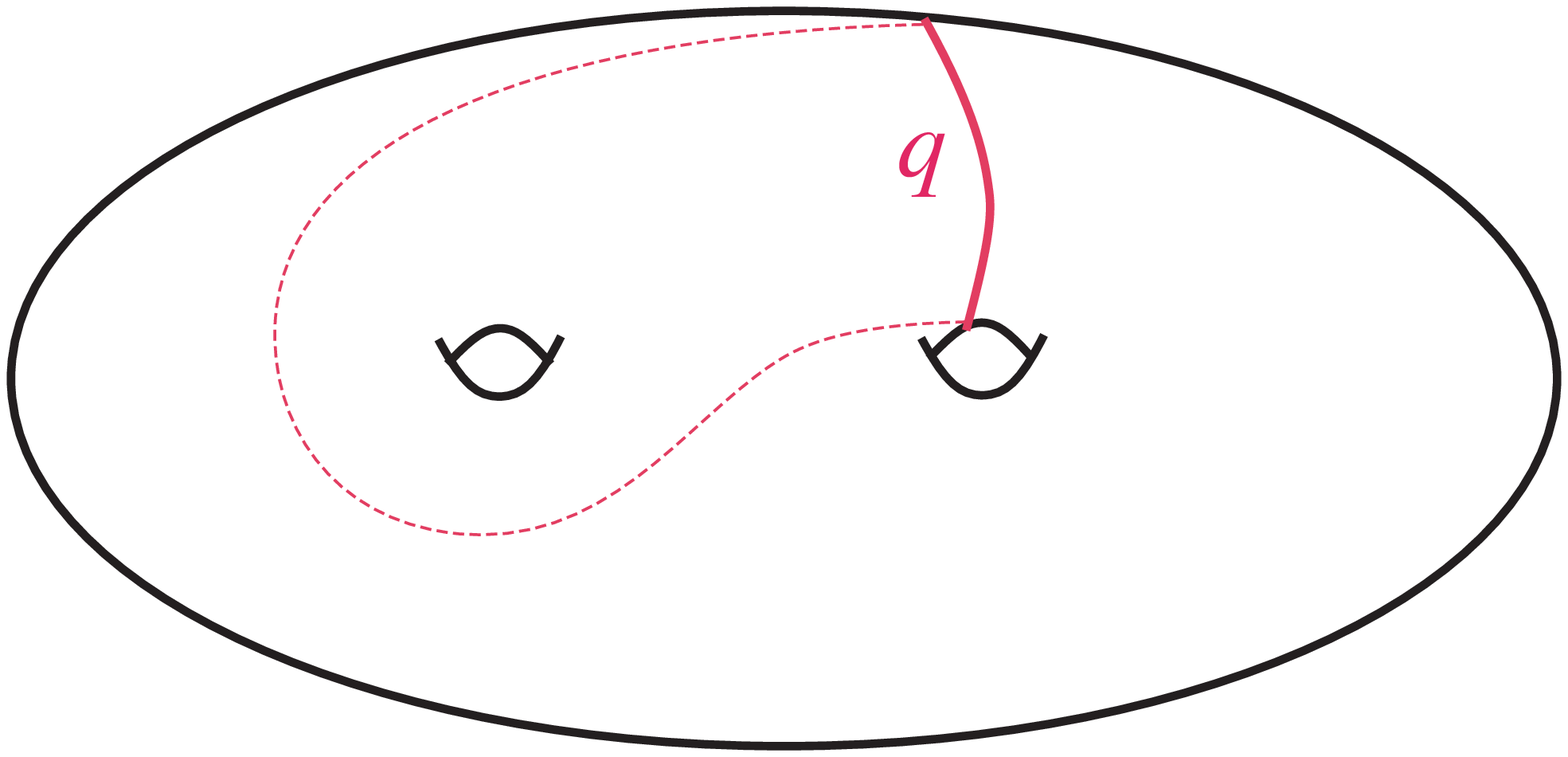} 

\hspace{0cm} (i) \hspace{5.7cm} (ii)
   
\caption{Curves in $\mathcal{C}_3$} \label{fig5-b}
\end{center}
\end{figure}
  
Let $\mathcal{C}_3 = \{p_1, p_2, \cdots, p_{g}, t_1, t_2, \cdots, t_{n}, y_0, y_1, \cdots, y_{n}, w_1, w_2, \cdots, w_{g-1}\}$ when $n \geq 1$, 
$\mathcal{C}_3 = \{p_1, p_2, \cdots, p_{g-1}, w_1, w_2, \cdots, w_{g-1}\}$ when $R$ is closed and $(g, n) \neq (2, 0)$, and let  
$\mathcal{C}_3 = \{t\}$ when $(g, n) = (2, 0)$ where the curves are as shown in Figure \ref{fig5} and Figure \ref{fig5-b}. We 
let $\mathcal{C}_2 = \emptyset$ when $R$ is closed.

\begin{lemma} \label{curves-III} Suppose $g \geq 2$ and $n \geq 0$. There exists a 
homeomorphism $h: R \rightarrow R$ such that $h([x]) = \lambda([x])$ $\forall \ x \in \mathcal{C}_1 \cup \mathcal{C}_2 \cup \mathcal{C}_3$.\end{lemma}

\begin{proof} We will first give the proof when $R$ has nonempty boundary. We will consider all the curves in 
Figure \ref{fig5}. By Lemma \ref{curves-II}  there exists a 
homeomorphism $h: R \rightarrow R$ such that $h([x]) = \lambda([x])$ $\forall \ x \in \mathcal{C}_1 \cup \mathcal{C}_2$.  

The curve $t_1$ is the unique nontrivial simple closed curve up to isotopy that is disjoint from each of $\mathcal{C}_1 \setminus \{a_{2}, b_4, b_6, \cdots, b_{2g}\}$, intersects $a_{2g}$ once and nonisotopic to $a_1$. Since we
know that $h([x]) = \lambda([x])$ for all these curves and $\lambda$ preserves these properties we see that $h([t_1]) = \lambda([t_1])$. The curve $t_2$ is the unique nontrivial simple closed curve  up to isotopy that is disjoint from each of $(\mathcal{C}_1 \cup \{t_1\}) \setminus \{a_{2}, b_4, b_6, \cdots, b_{2g}, d_{n-1}\}$, intersects $a_{2g}$ once and nonisotopic to $t_1$. Since we
know that $h([x]) = \lambda([x])$ for all these curves and $\lambda$ preserves these properties we see that $h([t_2]) = \lambda([t_2])$. Similarly, we get  $h([t_i]) = \lambda([t_i])$ $ \forall i= 3, 4, \cdots, n$.

The curve $y_{n}$ is the unique nontrivial simple closed curve  up to isotopy that is disjoint from each of $a_{2}, a_{3}, \cdots, a_{2g}, r_{1}, r_2, \cdots,
r_n$, intersects each of $a_ 1, b_4, b_6, \cdots, b_{2g}, c_{4}, c_6, \cdots, c_{2g},$ $ d_1, d_2, \cdots,$ $ d_{n-1}$ once. 
Since $h([x]) = \lambda([x])$ for all these curves and these properties are preserved by $\lambda$ we have $h([y_n]) = \lambda([y_n])$. The curve $y_{n-1}$ is the unique nontrivial simple closed curve  up to isotopy that is disjoint from each of $a_{2}, a_{3}, \cdots, a_{2g}, r_2,  r_3, \cdots,
r_n, y_n$, intersects each of $a_1, b_{4}, b_{6}, \cdots, b_{2g}, c_{4}, c_6, \cdots, c_{2g}, d_1, d_2,$ $ \cdots,$ $ d_{n-1}$ once, and nonisotopic to $y_n$. 
Since $h([x]) = \lambda([x])$ for all these curves and these properties are preserved by $\lambda$ we have $h([y_{n-1}]) = \lambda([y_{n-1}])$. Similarly, we get  $h([y_i]) = \lambda([y_i])$ $ \forall i= 0, 1, 2, \cdots, n-2$. 

The curve $p_{g}$ is the unique 
nontrivial simple closed curve up to isotopy that is disjoint from each of $b_{2g-2}, a_{2g-2}, c_{2g-2}, y_{n}, a_{2g}, r_1, r_2, \cdots, r_n$, 
intersects each of $a_{2g-1}, b_{2g}, c_{2g}, d_1, d_2,$ $ \cdots, d_{n-1}$ once and nonisotopic to $a_{2g}$. Since $h([x]) = \lambda([x])$ for all these curves and these properties are preserved by $\lambda$, we have $h([p_{g}]) = \lambda([p_{g}])$. When $g=2$, we have $p_{g-1}= y_0$, so $h([p_{g-1}]) = \lambda([p_{g-1}])$. 
When 
$g \neq 2$, the curve $p_{g-1}$ is the unique nontrivial simple closed curve up to isotopy that is disjoint from each of $b_{2g-4}, a_{2g-4}, 
c_{2g-4}, y_{0}, t_n, a_{2g-2}, a_{2g-1}, a_{2g}, w$, and intersects each of $a_{2g-3}, b_{2g-2}, b_{2g}, c_{2g-2}, c_{2g}, d_1, d_2,$ $ \cdots, d_{n-1}$ once. Since we
know that $h([x]) = \lambda([x])$ for all these curves and $\lambda$ preserves these properties we see that $h([p_{g-1}]) = \lambda([p_{g-1}])$. Now, substituting $p_{g-1}$ for $w$ and considering the above intersection information, 
it is easy to see that $h([w_{g-1}]) = \lambda([w_{g-1}])$.
The curve $p_{g-2}$ is the unique nontrivial simple closed curve up to isotopy that is disjoint from each of $b_{2g-6}, a_{2g-6}, c_{2g-6},$ $ b_{2g}, a_{2g}, c_{2g}, a_{2g-4}, a_{2g-3}, a_{2g-2}, w_{g-1}$, and intersects each of $a_{2g-5}, b_{2g-4}, b_{2g-2},$ $ a_{2g-1}, c_{2g-2},$ $ c_{2g-4}$ once. Since we
know that $h([x]) = \lambda([x])$ for all these curves and $\lambda$ preserves these properties we see that
$h([p_{g-2}]) = \lambda([p_{g-2}])$. Similarly, we get $h([p_i]) = \lambda([p_i])$ $ \forall i= 1, 2, \cdots, g-3$, and $h([w_i]) = \lambda([w_i])$ $ \forall i= 1, 2, \cdots, g-2$. 
Hence, there exists a homeomorphism $h: R \rightarrow R$ such that $h([x]) = \lambda([x])$ $\forall \ x \in \mathcal{C}_1 \cup \mathcal{C}_2 \cup \mathcal{C}_3$.

Suppose that $R$ is closed and $(g, n) \neq (2, 0)$. Then $\mathcal{C}_2$ is empty and $\mathcal{C}_3 =  \{p_1, p_2, \cdots,$ $ p_{g-1}, w_1, w_2, \cdots, w_{g-1}\}$. By Lemma \ref{curves} we know that there exists a homeomorphism $h$ of $R$ such that $h([x]) = \lambda([x])$ $\forall \ x \in \mathcal{C}_1$. There exists a homeomorphism $\phi : R \rightarrow R$ of order two such that the map $\phi_{*}$ induced by $\phi$ on $\mathcal{N}(R)$ sends the isotopy class of each curve in $\mathcal{C}_1$ to itself and switches 
$[p_{g-1}]$ and $[w_{g-1}]$.
There are only two nontrivial simple closed curves, namely $p_{g-1}$ and $w_{g-1}$, up to isotopy that are disjoint from each of $b_{2g-4}, a_{2g-4}, c_{2g-4}, a_{2g-2}, a_{2g-1}, a_{2g}$, bounds a pair of pants with $a_{2g}$ and $a_{2g-2}$, and intersects each of $a_{2g-3}, b_{2g-2}, b_{2g}, c_{2g}$ once. Since $h([x]) = \lambda([x])$ for all these curves and $\lambda$ preserves these properties by Lemma \ref{embedded} and Lemma \ref{int-one}, 
by replacing $h$ with $h \circ \phi$ if necessary, we can assume that we have $h([p_{g-1}]) = \lambda([p_{g-1}])$ and $h([w_{g-1}]) = \lambda([w_{g-1}])$. To get the proof of the lemma, it is enough to prove the result for
this $h$. Now we continue as follows: The curve $p_{g-2}$ is the unique nontrivial simple closed curve up to isotopy that is disjoint from each of $b_{2g-6}, a_{2g-6}, c_{2g-6}, b_{2g}, a_{2g},$ $ c_{2g}, a_{2g-4}, a_{2g-3}, a_{2g-2}, w_{g-1}$, 
bounds a pair of pants with $a_{2g-2}$ and $a_{2g-4}$, and intersects each of $a_{2g-5}, b_{2g-4}, b_{2g-2},$ $ a_{2g-1}, c_{2g-2}, c_{2g-4}$ once. Since we
know that $h([x]) = \lambda([x])$ for all these curves and $\lambda$ preserves these properties we see that
$h([p_{g-2}]) = \lambda([p_{g-2}])$. Now it is easy to see that $h([w_{g-2}]) = \lambda([w_{g-2}])$. Similarly, we get $h([p_i]) = \lambda([p_i])$ $ \forall i= 1, 2, \cdots, g-3$, and $h([w_i]) = \lambda([w_i])$ $ \forall i= 1, 2, \cdots, g-2$. Hence, there exists a homeomorphism $h: R \rightarrow R$ such that $h([x]) = \lambda([x])$ $\forall \ x \in \mathcal{C}_1 \cup \mathcal{C}_3$.
 
Now suppose that $R$ is closed and $(g, n) = (2, 0)$. Then $\mathcal{C}_2$ is empty and $\mathcal{C}_3 = \{t\}$. 
Let $\sigma$ be the homeomorphism of $R$ of order two defined by the symmetry with respect to the plane of the paper, see Figure \ref{fig5-b}.  
We have $\sigma(x) = x$ for all $x \in \{ a_1, a_2, \cdots, a_{5}\}$ and $\sigma(t) = q$. 
The curves $t$ and $q$ are the only nontrivial simple closed curves up to isotopy that intersect each of $a_1, a_3, a_4$ once and bound a pair of pants with $a_2, a_5$. Since $h([x]) = \lambda([x])$ for all these curves and these properties are preserved by $\lambda$, by composing $h$ with $\sigma$ if necessary we can assume that 
we have $h([t]) = \lambda([t])$. So, we have $\lambda([x])= h([x])$ $\forall \ x \in \{a_1, a_2, a_3, a_4, a_5, t\}$. Hence, there exists a homeomorphism $h: R \rightarrow R$ such that $h([x]) = \lambda([x])$ $\forall \ x \in \mathcal{C}_1 \cup \mathcal{C}_3$ in this case as well.\end{proof}\\
 
\begin{figure}[htb]
\begin{center}
\hspace{1.5cm} \epsfxsize=3in \epsfbox{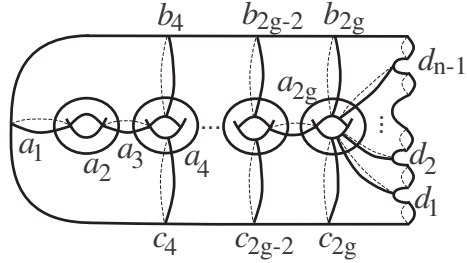}
\caption{The homeomorphism $\sigma$} \label{fig1-e}
\end{center}
\end{figure} 

\begin{lemma} \label{abcd} Suppose that $g \geq 2$ and $n \geq 0$ and $(g, n) \neq (2, 0)$. 
The set $\mathcal{C}_1 \cup \{p_1\}$ has trivial pointwise stabilizer.
\end{lemma}
 
\begin{proof} Let $\sigma$ be the homeomorphism of $R$ of order two defined by the symmetry with respect to the plane of the paper, see Figure \ref{fig1-e}.
We have $\sigma(x) = x$ for all $x \in \{ a_1, a_2, \cdots, a_{2g}, b_4, b_6, \cdots,$ $b_{2g}, c_4, c_6, \cdots, c_{2g}, d_1, d_2, \cdots, d_{n-1} \}$.
Let $h$ be a homeomorphism of $R$ such that $[h(x)] = [x]$ for every $x \in \mathcal{C}_1 \cup \{p_1\}$.
By Assertion 1 given in the proof of Lemma 3.6 in \cite{IrP2} (see also Proposition 2.8 in \cite{FM}), $h$ is isotopic to a homeomorphism which fixes every curve $x \in \{a_1, a_2, \cdots, a_{2g}, b_4, b_6, \cdots, b_{2g}, c_4,$ $ c_6, \cdots, c_{2g}, d_1, d_2, \cdots, d_{n-1} \}$, so we can assume that $h(x) = x$ for all $x \in \{a_1, a_2, \cdots, a_{2g},$ $ b_4, b_6, \cdots, b_{2g}, c_4,$ $ c_6, \cdots, c_{2g}, d_1, d_2, \cdots, d_{n-1} \}$. 
Let $K = a_1 \cup \cdots \cup a_{2g} \cup b_1 \cup \cdots \cup b_{2g} \cup c_1 \cup \cdots \cup c_{2g} \cup d_1 \cup \cdots \cup d_{n-1}$. We have $h(K) = K$.
Cutting $R$ along $K$ we get $2g-2$ disks and $n$ annuli. 
The homeomorphism $h$ sends each of these pieces onto itself.
It is easy to see that either the restriction of $h$ to each piece preserves the orientation for each piece or reverses the orientation for each piece.
Suppose that the restriction of $h$ to each piece preserves the orientation for each piece.
Then $h$ also preserves the orientation of each $x \in \{a_1, \cdots, a_{2g}, b_4, b_6, \cdots, b_{2g}, c_4, c_6, \cdots, c_{2g}, d_1, \cdots, d_{n-1}\}$, hence we can assume that the restriction of $h$ to $K = a_1 \cup \cdots \cup a_{2g} \cup b_1 \cup \cdots \cup b_{2g} \cup c_1 \cup \cdots \cup c_{2g} \cup d_1 \cup \cdots \cup d_{n-1}$ is the identity.
Then the restriction of $h$ to each disk is isotopic to the identity with an isotopy that fixes the boundary pointwise, and the restriction of $h$ to each annuli is isotopic to the identity with an isotopy that pointwise fixes the boundary of the annuli which is not a boundary component of $R$. This shows that $h$ is isotopic to the identity on $R$. If the restriction of $h$ to each piece reverses the orientation of each piece, then $h \sigma$ preserves the orientation of each piece, so $h \sigma$ is isotopic to the identity on $R$. 
Hence, either $h$ is isotopic to the identity or $\sigma$ on $R$. We have $[h(p_1)] = [p_1]$, but $[\sigma (p_1)] \neq [p_1]$, so $h$ cannot be isotopic to $\sigma$. Hence, $h$ is isotopic to identity on $R$.\end{proof}\\
 
When $(g, n) = (2, 0)$, we have $\mathcal{C}_1 \cup \{t\} = \{a_1, \cdots, a_5, t\}$ where the curves are as shown in  
Figure \ref{fig5-b} (i).

\begin{lemma} \label{abcde} Suppose that $(g, n) = (2, 0)$. 
The pointwise stabilizer of $\mathcal{C}_1 \cup \{t\}$ is the center of $Mod_R^*$.
\end{lemma}
 
\begin{proof} Let $\sigma$ be the homeomorphism of $R$ of order two defined by the symmetry with respect to the plane of the paper, see Figure \ref{fig5-b} (i). We see that $\sigma(x) = x$ for all $x \in \{ a_1, a_2, \cdots, a_5\}$. 
Let $h$ be a homeomorphism of $R$ such that $[h(x)] = [x]$ for every $x \in \mathcal{C}_1 \cup \{t\}$.
By Assertion 1 given in the proof of Lemma 3.6 in \cite{IrP2} (see also Proposition 2.8 in \cite{FM}), $h$ is isotopic to a homeomorphism which fixes every curve $x \in \{a_1, a_2, \cdots, a_5\}$, so we can assume that $h(x) = x$ for all $x \in \{a_1, a_2, \cdots, a_5\}$. Let $K = a_1 \cup \cdots \cup a_5$. We have $h(K) = K$. Cutting $R$ along $K$ we get $2$ disks $D_1, D_2$. Suppose $h$ sends each of these pieces onto itself. Then $h$ is either orientation preserving on both $D_1$ and $D_2$ or orientation reversing on both $D_1$ and $D_2$. If $h$ is orientation reversing on both 
of them then $h \circ \sigma$ is orientation preserving on both of them. 
Then $h \circ \sigma$ also preserves the orientation of each $x \in \{a_1, \cdots, a_5\}$, hence we may assume that the restriction of 
$h \circ \sigma$ to $K$ is the identity. 
Then the restriction of $h \circ \sigma$ to each disk is isotopic to the identity with an isotopy that fixes the boundary pointwise. This shows that $h \circ \sigma$ is isotopic to the identity on $R$. This would imply that $h$ is isotopic to $\sigma$, but that gives a contradiction since  $h(t)$ is not isotopic to $\sigma(t)=q$. So, $h$ is orientation preserving on both $D_1$ and $D_2$, and hence $h$ is isotopic to $identity$.
Now suppose that $h(D_1) = D_2$ and $h(D_2)=D_1$. Let $i$ be the hyperelliptic involution which generates the center of 
$Mod_R^* \cong \mathbb{Z}_2$ and swicthes $D_1$ and $D_2$. Then, $h \circ i$ sends each of $D_1, D_2$ onto itself, and we can see as before that $h \circ i$ is isotopic to $\sigma$ or $identity$. But $(h \circ i)(t)$ is isotopic to $t$ but not isotopic to $\sigma(t)=q$.
So,  $h \circ i$ is isotopic to $identity$. Hence $h$ is isotopic to $i$ and the pointwise stabilizer of $\mathcal{C}_1 \cup \{t\}$ is the center of $Mod_R^*$.\end{proof}\\
  
\begin{figure}[htb]
\begin{center}
\hspace{1.7cm} \epsfxsize=3.3in \epsfbox{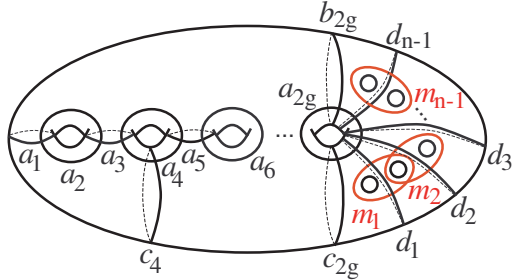}
\caption {Curves for the generating set}
\label{fig6}
\end{center}
\end{figure}

Consider the curves given in Figure \ref{fig6}. Let $t_x$ be the Dehn twist about $x$. Let $\sigma_i$ be the half twist along 
$m_i$. The mapping class group, $Mod_R$, of $R$ is the group of 
isotopy classes of all orientation preserving self-homeomorphisms of $R$.  
If $g \geq 2$ then $Mod_R$ can be generated by $G = \{t_x: x \in \{a_1, a_2, \cdots, a_{2g}, b_{2g}, c_4, c_{2g}, d_1, d_2, \cdots, d_{n-1}\} \} \cup \{\sigma_1,$ $\sigma_2, \cdots, \sigma_{n-1}\}$, see Corollary 4.15 in \cite{FM}.  
Let $h: R \rightarrow R$ be a homeomorphism which satisfies the statement of Lemma \ref{curves-III}. We know $h([x]) = \lambda([x])$ $\forall \ x \in \mathcal{C}_1 \cup \mathcal{C}_2 \cup \mathcal{C}_3$. 
We will follow the techniques given by Irmak-Paris \cite{IrP2} to obtain the homeomorphism we want.

\begin{figure} \begin{center}  
\hspace{-0.7cm} \epsfxsize=2.7in \epsfbox{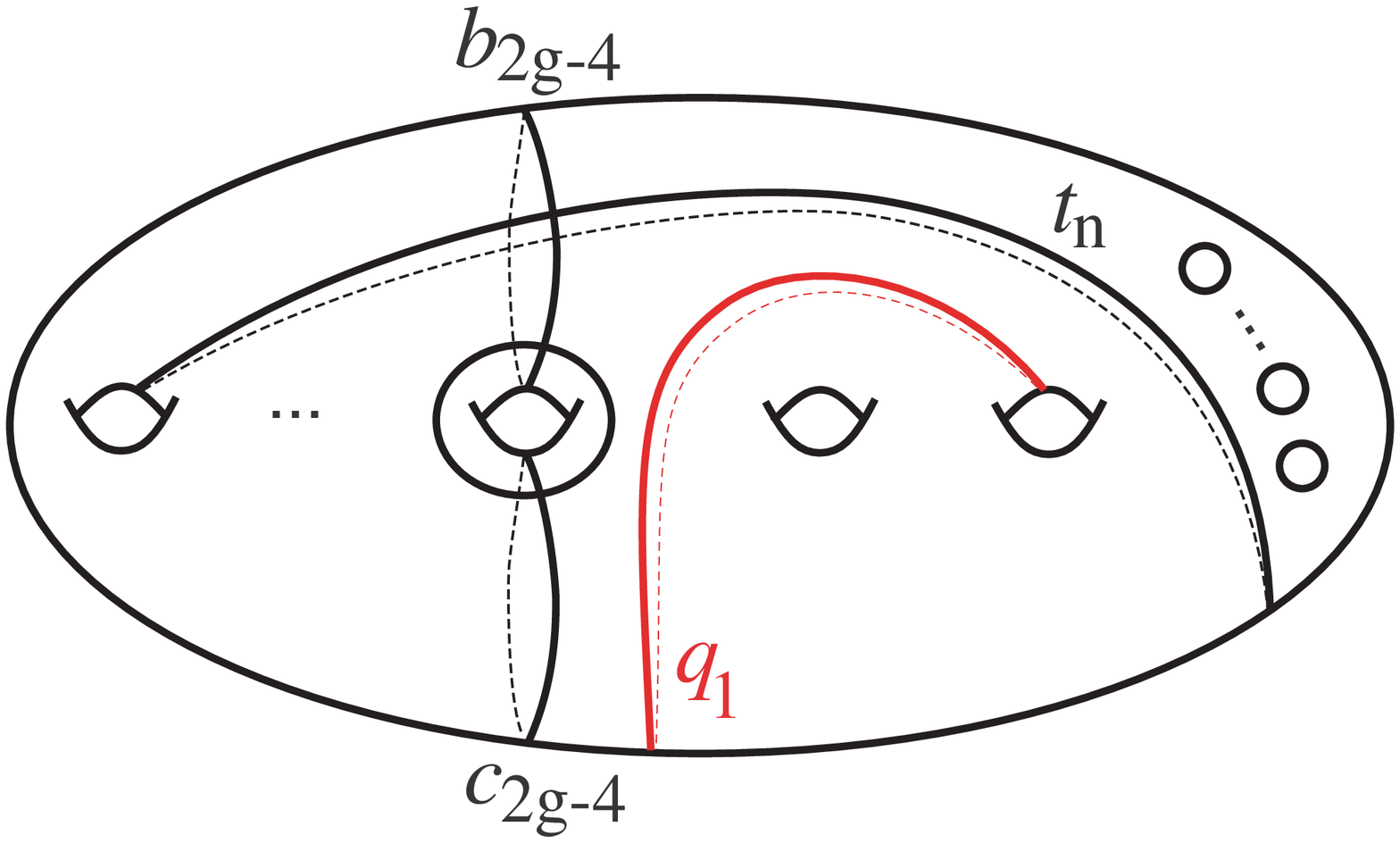} \hspace{0.2cm} \epsfxsize=2.7in \epsfbox{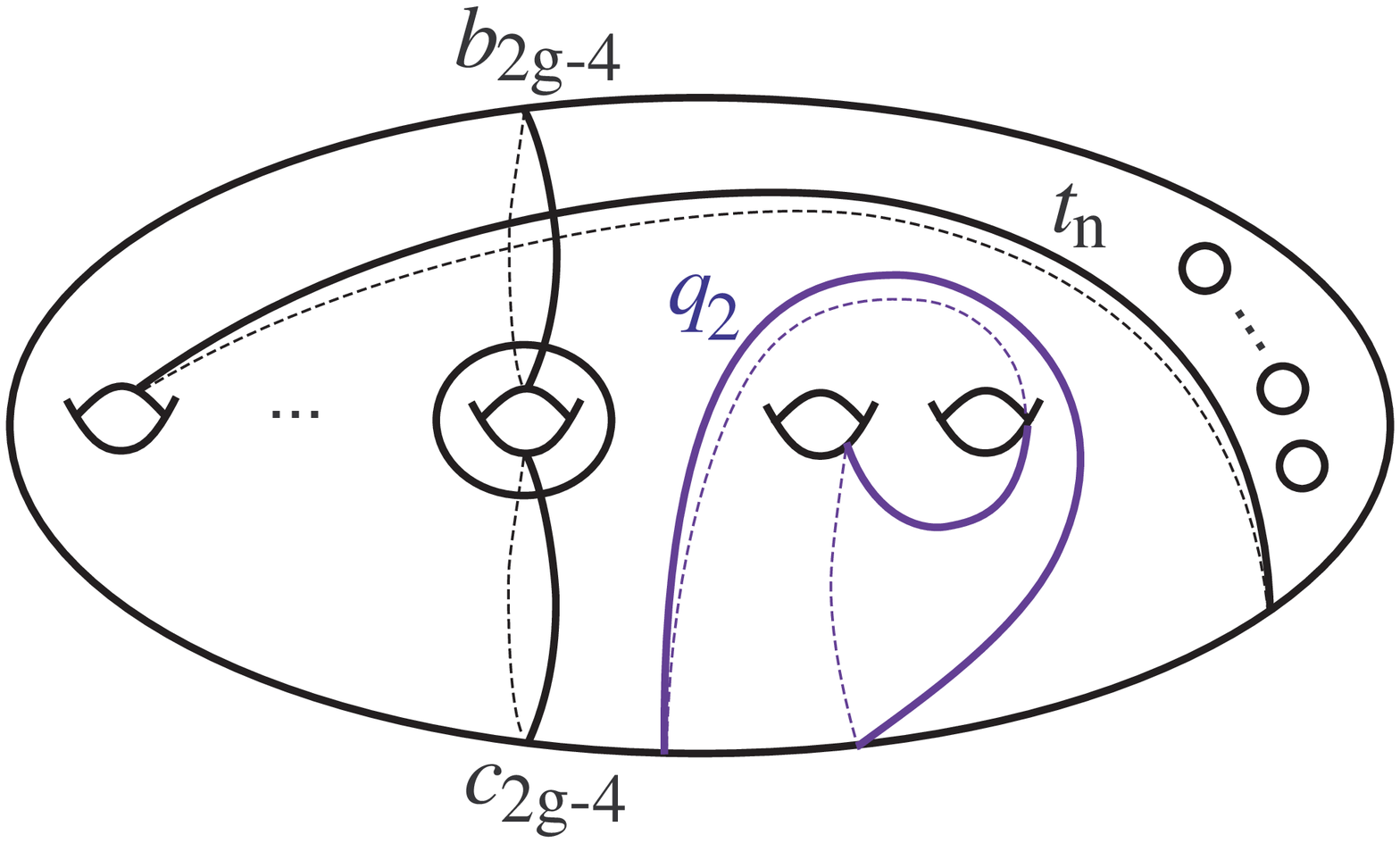}

\hspace{-1cm} (i) \hspace{6.5cm} (ii)

\hspace{-0.7cm}  \epsfxsize=2.7in \epsfbox{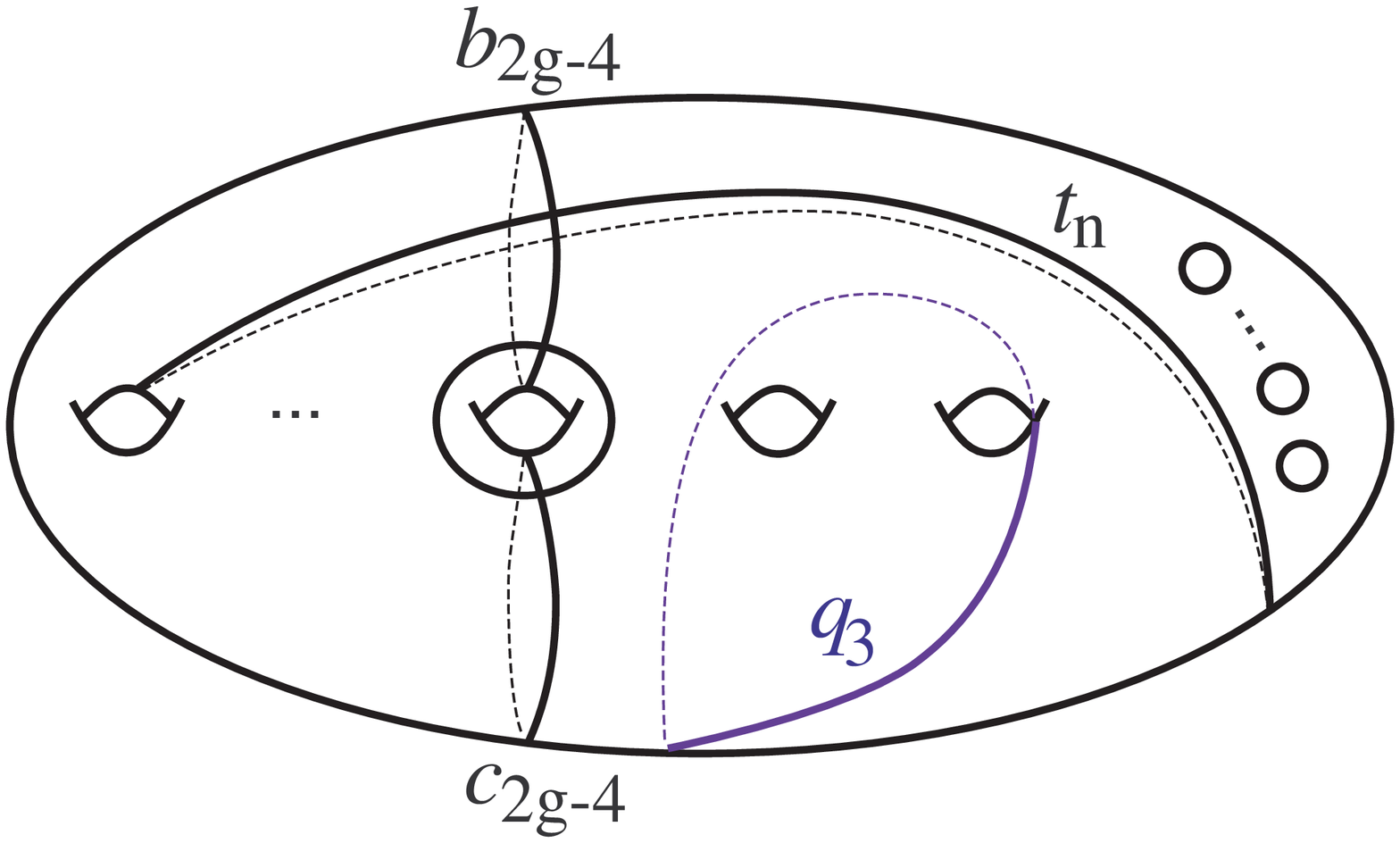} \hspace{0.2cm}  \epsfxsize=2.7in \epsfbox{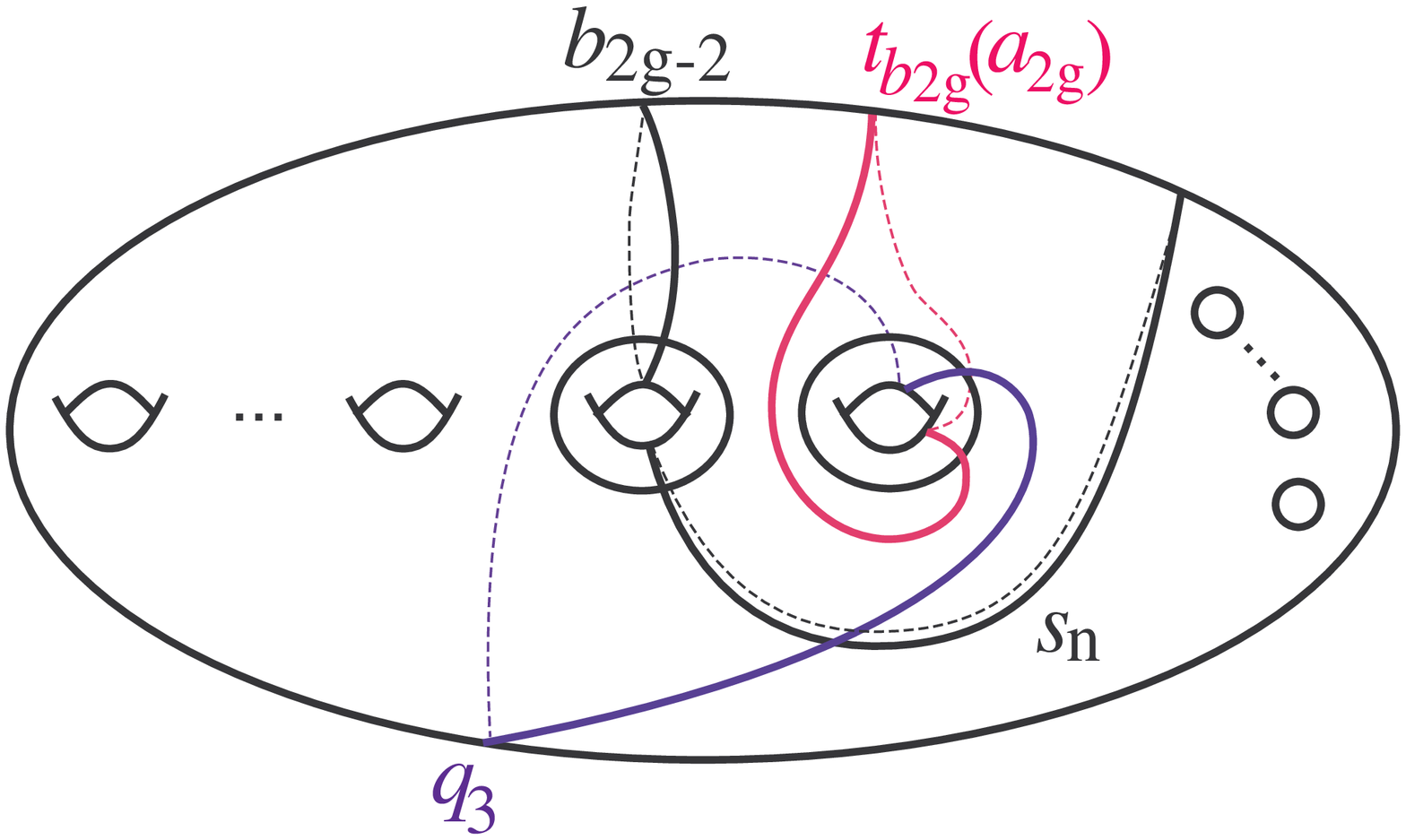} 

\hspace{-1cm} (iii) \hspace{6.4cm} (iv)
  
\hspace{-0.4cm} \epsfxsize=2.7in \epsfbox{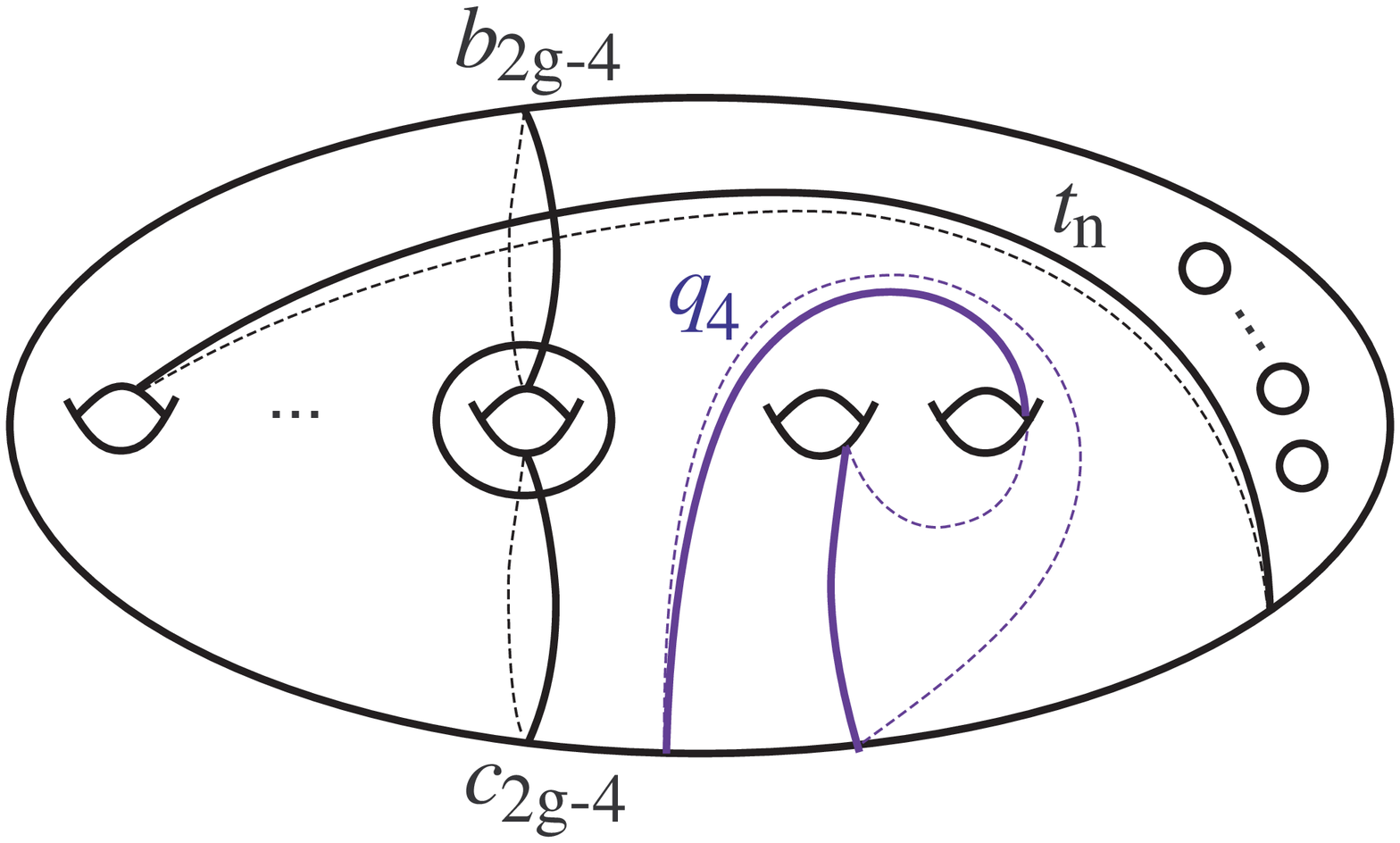} \hspace{0.3cm} \epsfxsize=2.7in \epsfbox{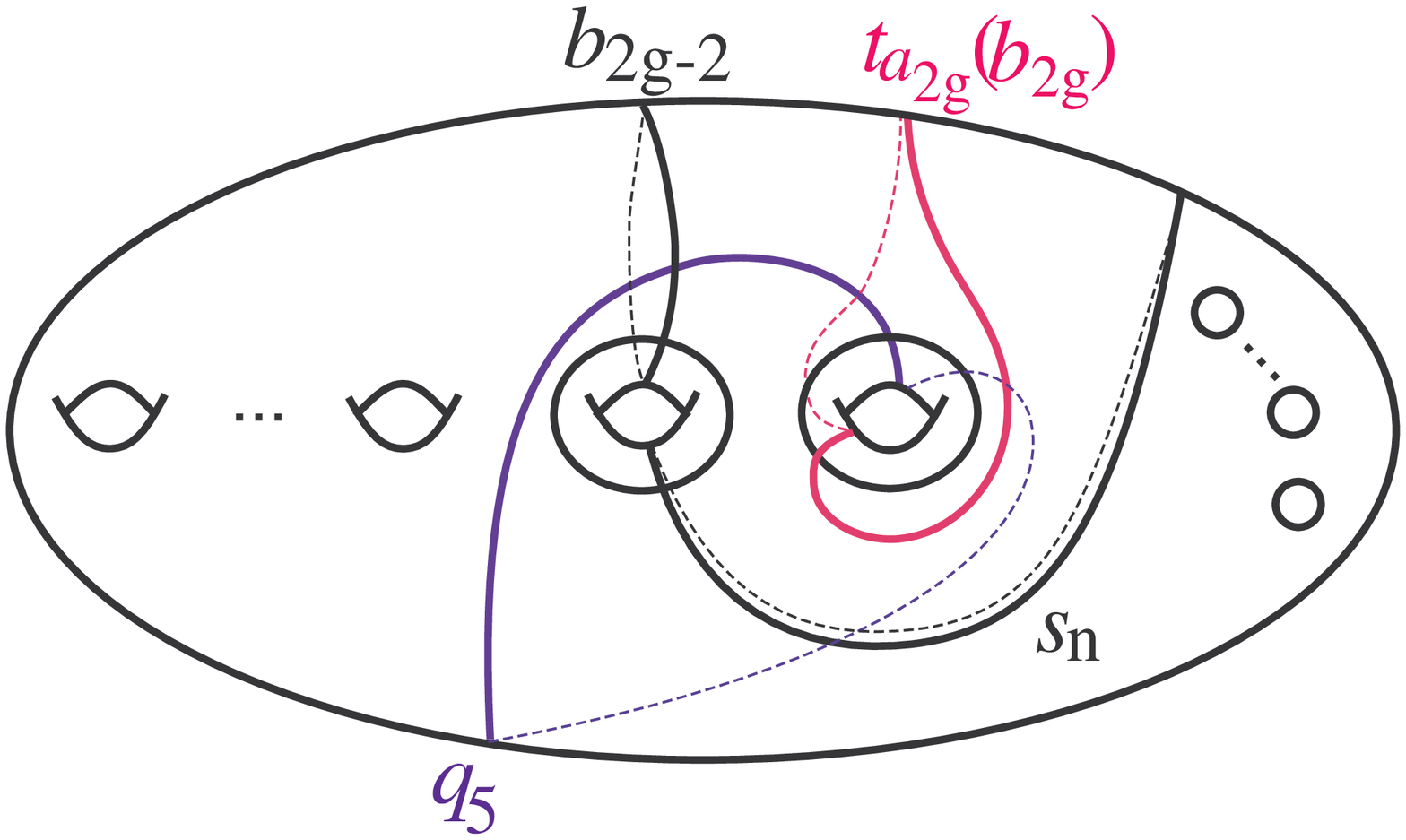} 

\hspace{-1cm} (v) \hspace{6.2cm} (vi)

 \hspace{-0.4cm} \epsfxsize=2.69in \epsfbox{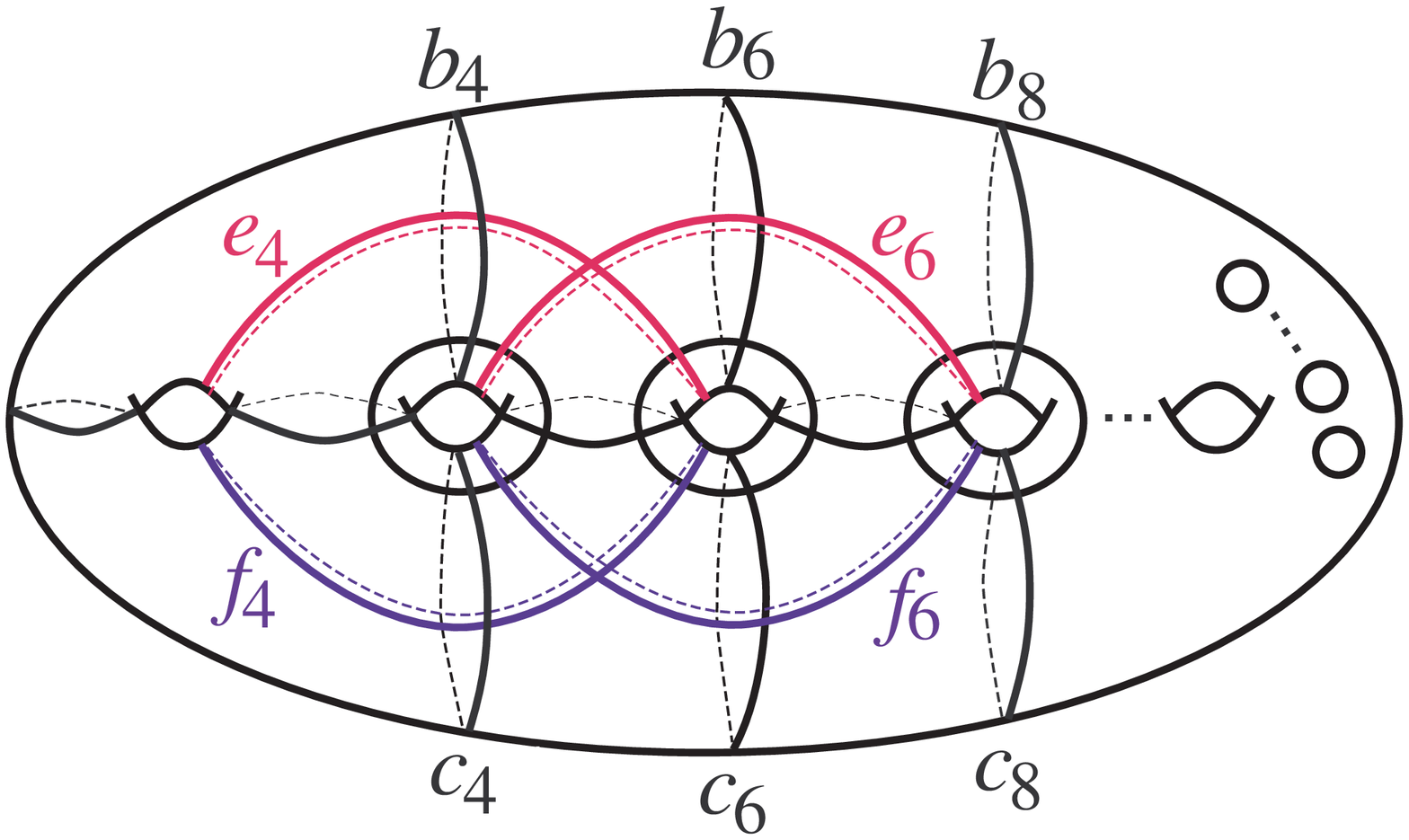} \hspace{0.3cm} \epsfxsize=2.69in \epsfbox{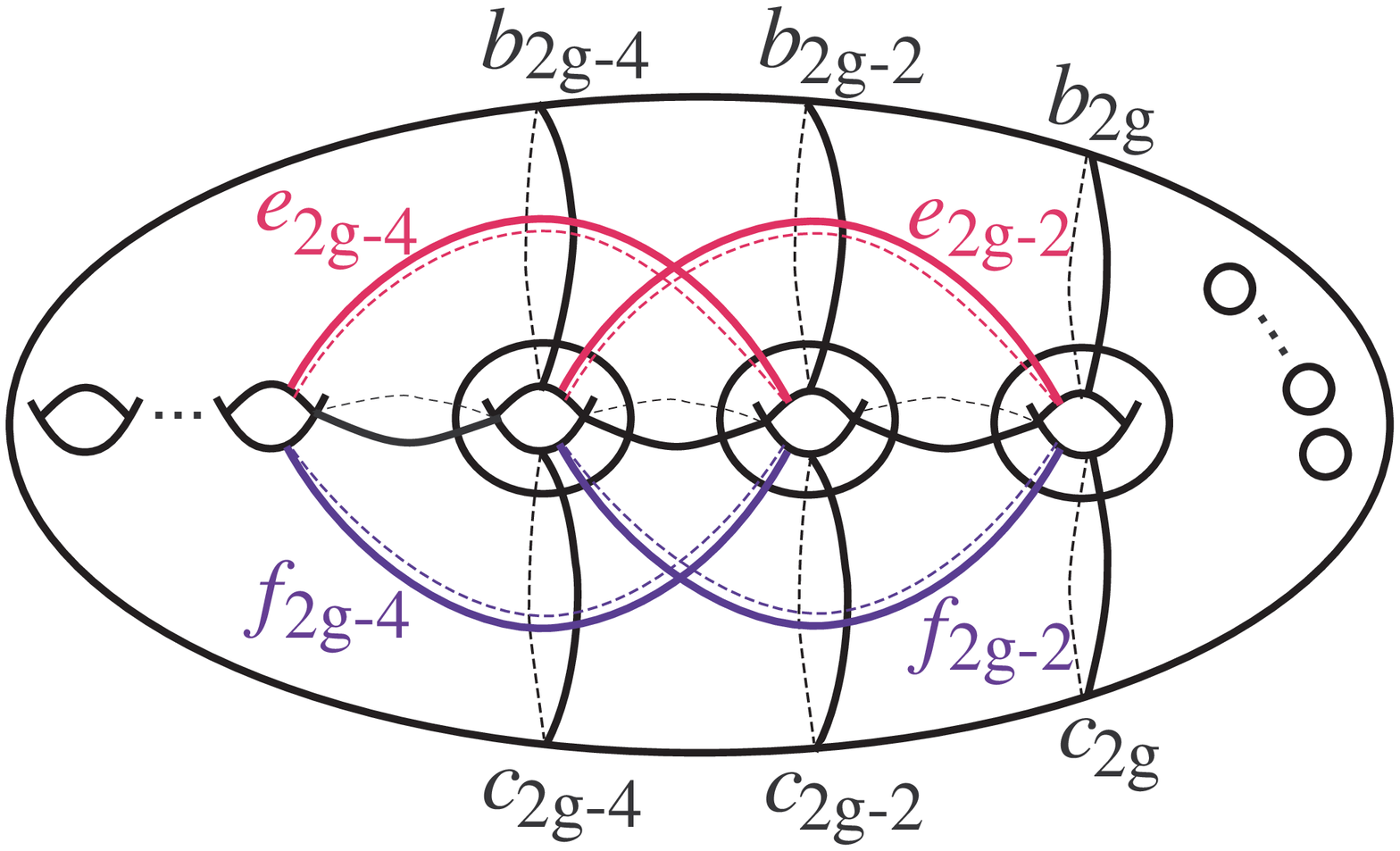} 

\hspace{-0.5cm} (vii) \hspace{6.3cm} (viii)

\caption{Curve configurations for twists} \label{fig7}
\end{center}
\end{figure} 

\begin{lemma} \label{prop1e} Suppose $g \geq 2$, $n \geq 0$. $\forall \ f \in G$, 
$\exists$ a set $L_f \subset \mathcal{N}(R)$
such that $\lambda([x])= h([x])$ $\ \forall \ x \in L_f \cup f(L_f)$. If $(g, n) \neq (2, 0)$, then $L_f$ can be chosen to have trivial  pointwise stabilizer. If $(g, n) = (2, 0)$, then $L_f$ can be chosen so that the pointwise stabilizer of $L_f$ is the center of $Mod_R^*$.\end{lemma}

\begin{proof} {\bf Case (i):} Suppose that $g \geq 3, n \geq 1$. We have $h([x]) = \lambda([x])$ $\forall \ x \in  \mathcal{C}_1 \cup \mathcal{C}_2 \cup \mathcal{C}_3$ by Lemma \ref{curves-III}. Let $f \in G$.  
For $f=t_{b_{2g}}$, let $L_f = \mathcal{C}_1 \cup \{p_1\}$. The set $\mathcal{C}_1 \cup \{p_1\}$ has trivial pointwise stabilizer by Lemma \ref{abcd}. We know $\lambda([x])= h([x])$
$\forall \ x \in \mathcal{C}_1 \cup \{p_1\}$. We just need to check the equation for $t_{b_{2g}}(a_{2g})$ since the other 
curves in $\mathcal{C}_1 \cup \{p_1\}$ are fixed by $t_{b_{2g}}$. 
Consider the curves given in Figure \ref{fig7} (i)-(vi). 
The curve $q_1$ is the unique nontrivial simple closed curve up to isotopy that is disjoint from all the curves in
$\{c_{2g-4}, a_{2g-4}, b_{2g-4}, t_n,  c_{2g-2}, a_{2g-2}, a_{2g-1}, c_{2g}\}$, intersects $a_{2g}$ once and nonisotopic to $c_{2g}$. Since we know that $h([x]) = \lambda([x])$ for all
these curves and these properties are preserved by $\lambda$, we have $h([q_1]) = \lambda([q_1])$.  
The curve $q_2$ is the unique nontrivial simple closed curve up to isotopy that is disjoint from all the curves in
$\{c_{2g-4}, a_{2g-4}, b_{2g-4}, t_n,  q_1, c_{2g-2}, a_{2g-1}, p_{n-1}\}$, intersects $a_{2g}$ once and nonisotopic to $a_{2g-1}$. Since we know that $h([x]) = \lambda([x])$ for all
these curves and these properties are preserved by $\lambda$, we have $h([q_2]) = \lambda([q_2])$. 
The curve $q_3$ is the unique nontrivial simple closed curve up to isotopy that is disjoint from all the curves in
$\{c_{2g-4}, a_{2g-4}, b_{2g-4}, t_n,  a_{2g-2}, a_{2g-1}, q_{2}, p_{n-2}\}$ and intersects $c_{2g-2}$ once. Since we know that $h([x]) = \lambda([x])$ for all
these curves and these properties are preserved by $\lambda$, we have $h([q_3]) = \lambda([q_3])$.  
The curve $t_{b_{2g}}(a_{2g})$ is the unique nontrivial simple closed curve up to isotopy that is disjoint from all the curves in
$\{b_{2g-2}, a_{2g-2}, s_n, q_3\}$ and intersects each of $a_{2g}$ and $b_{2g}$ once.
Since $h([x]) = \lambda([x])$ for all these curves and these properties are preserved by $\lambda$, we have
$h([t_{b_{2g}}(a_{2g}) ]) = \lambda([t_{b_{2g}}(a_{2g})])$. Hence we get $\lambda([x])= h([x])$ $\ \forall \ x \in L_f \cup f(L_f)$ when $f=t_{b_{2g}}$. 

The curve $q_4$ is the unique nontrivial simple closed curve up to isotopy that is disjoint from all the curves in
$\{c_{2g-4}, a_{2g-4}, b_{2g-4}, t_n,  q_1, c_{2g-2}, a_{2g-1}, w_{n-1}\}$, intersects $a_{2g}$ once and nonisotopic to $a_{2g-1}$. Since we know that $h([x]) = \lambda([x])$ for all
these curves and these properties are preserved by $\lambda$, we have $h([q_4]) = \lambda([q_4])$. 
The curve $q_5$ is the unique nontrivial simple closed curve up to isotopy that is disjoint from all the curves in
$\{c_{2g-4}, a_{2g-4}, b_{2g-4}, t_n,  a_{2g-2}, a_{2g-1}, q_{4},$ $ w_{n-2}\}$ and intersects $c_{2g-2}$ once. Since we know that $h([x]) = \lambda([x])$ for all
these curves and these properties are preserved by $\lambda$, we have $h([q_5]) = \lambda([q_5])$.  
The curve $t_{a_{2g}}(b_{2g})$ is the unique nontrivial simple closed curve up to isotopy that is disjoint from all the curves in
$\{b_{2g-2}, a_{2g-2}, s_n, q_5\}$ and intersects each of $a_{2g}$ and $b_{2g}$ once.
Since $h([x]) = \lambda([x])$ for all these curves and these properties are preserved by $\lambda$, we have
$h([t_{a_{2g}}(b_{2g}) ]) = \lambda([t_{a_{2g}}(b_{2g})])$. 
Now we can control other twists as follows: The curve $t_{d_{n-1}}(a_{2g})$ is the unique nontrivial simple closed curve up to isotopy that is disjoint from each of 
$a_{2g-2}, v_1, s_{n-1}, t_{b_{2g}}(a_{2g})$ and intersects each of $a_{2g}$ and $d_{n-1}$ once.
Since $h([x]) = \lambda([x])$ for all these curves and these properties are preserved by $\lambda$, we have
$h([t_{d_{n-1}}(a_{2g}) ]) = \lambda([t_{d_{n-1}}(a_{2g})])$. The curve $t_{d_{n-2}}(a_{2g})$ is the unique nontrivial simple closed curve up to isotopy that is disjoint from each of 
$a_{2g-2}, v_2, s_{n-2}, t_{d_{n-1}}(a_{2g})$ and intersects each of $a_{2g}$ and $d_{n-2}$ once.
Since $h([x]) = \lambda([x])$ for all these curves and these properties are preserved by $\lambda$, we have
$h([t_{d_{n-2}}(a_{2g}) ]) = \lambda([t_{d_{n-2}}(a_{2g})])$. Similarly, we have $h([t_{d_{i}}(a_{2g}) ]) = \lambda([t_{d_{i}}(a_{2g})])$ for all $i= 1, 2, \cdots, n-3$. 
The curve $t_{a_{2g}}(d_{n-1})$ is the unique nontrivial simple closed curve up to isotopy that is disjoint from 
$a_{2g-2}, v_1, s_{n-1}, t_{a_{2g}}(b_{2g})$ and intersects each of $a_{2g}$ and $d_{n-1}$ once.
Since $h([x]) = \lambda([x])$ for all these curves and these properties are preserved by $\lambda$, we have
$h([t_{a_{2g}}(d_{n-1}) ]) = \lambda([t_{a_{2g}}(d_{n-1})])$. The curve $t_{a_{2g}}(d_{n-2})$ is the unique nontrivial simple closed curve up to isotopy that is disjoint from each of 
$a_{2g-2}, v_2, s_{n-2}, t_{a_{2g}}(d_{n-1})$ and intersects each of $a_{2g}$ and $d_{n-2}$ once.
Since $h([x]) = \lambda([x])$ for all these curves and these properties are preserved by $\lambda$, we have
$h([t_{a_{2g}}(d_{n-2}) ]) = \lambda([t_{a_{2g}}(a_{n-2})])$. Similarly, we have $h([t_{a_{2g}}(d_{i}) ]) = \lambda([t_{a_{2g}}(d_{i})])$ for all $i= 1, 2, \cdots, n-3$. 
The curve $t_{a_{2g-1}}(a_{2g})$ is the unique nontrivial simple closed curve up to isotopy that is disjoint from each of 
$v_n, s_n, t_{b_{2g}}(a_{2g})$ and intersects each of $a_{2g-1}$ and $a_{2g}$ once.
Since $h([x]) = \lambda([x])$ for all these curves and these properties are preserved by $\lambda$, we have
$h([t_{a_{2g-1}}(a_{2g}) ]) = \lambda([t_{a_{2g-1}}(a_{2g})])$. The curve $t_{a_{2g}}(a_{2g-1})$ is the unique nontrivial simple closed curve up to isotopy that is disjoint from each of 
$v_n, s_n, t_{a_{2g}}(b_{2g})$ and intersects each of $a_{2g-1}$ and $a_{2g}$ once.
Since $h([x]) = \lambda([x])$ for all these curves and these properties are preserved by $\lambda$, we have
$h([t_{a_{2g}}(a_{2g-1}) ]) = \lambda([t_{a_{2g}}(a_{2g-1})])$.   
Hence, we have $h ([t_x(y)]) = \lambda ([t_x(y)])$ for all $x,y \in \{a_{2g-1}, a_{2g}, b_{2g}, d_1, \dots, d_{n-1}\}$. 

\begin{figure} \begin{center}   
\hspace{-0.4cm} \epsfxsize=2.69in \epsfbox{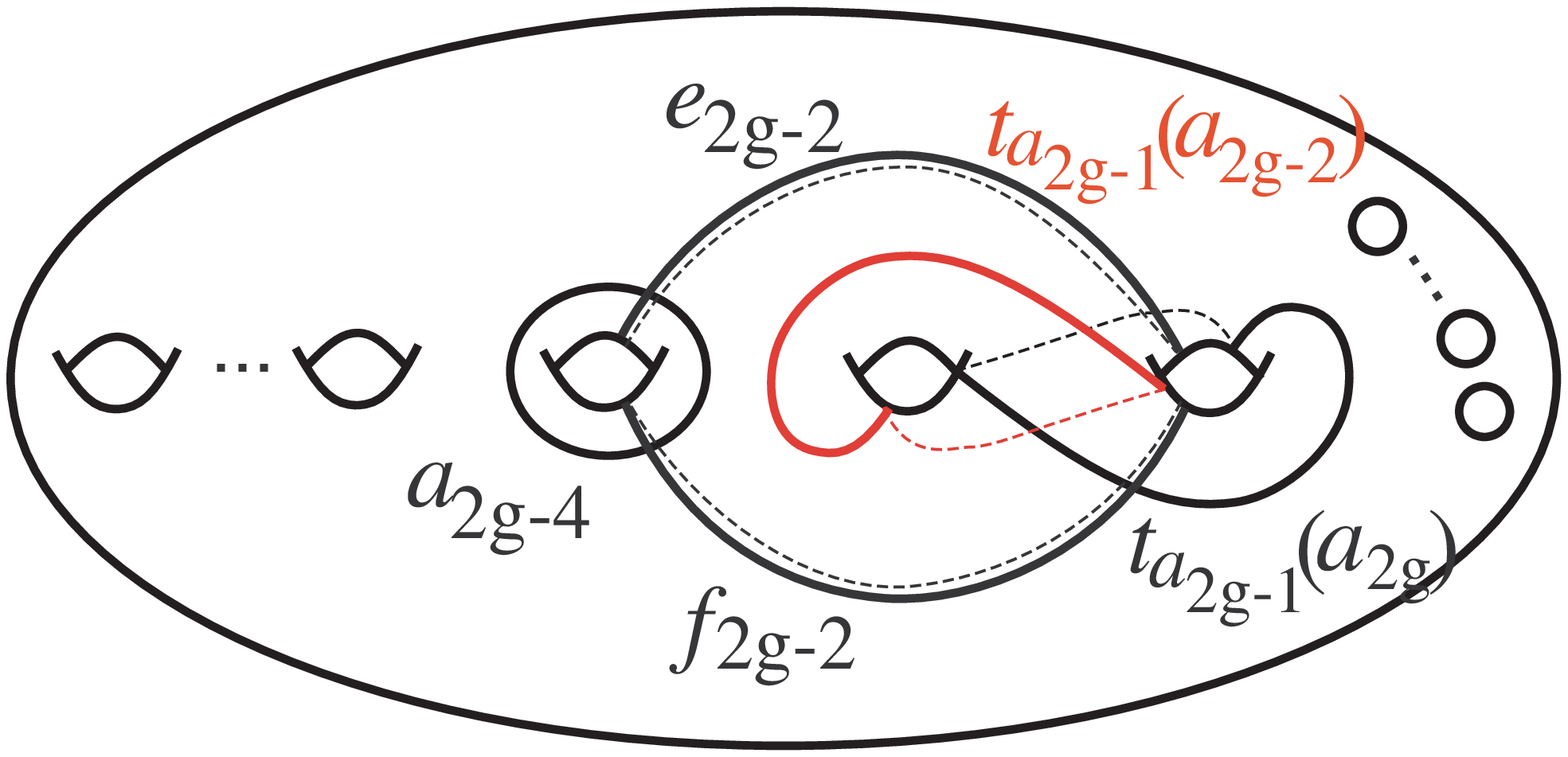} \hspace{0.3cm} \epsfxsize=2.69in \epsfbox{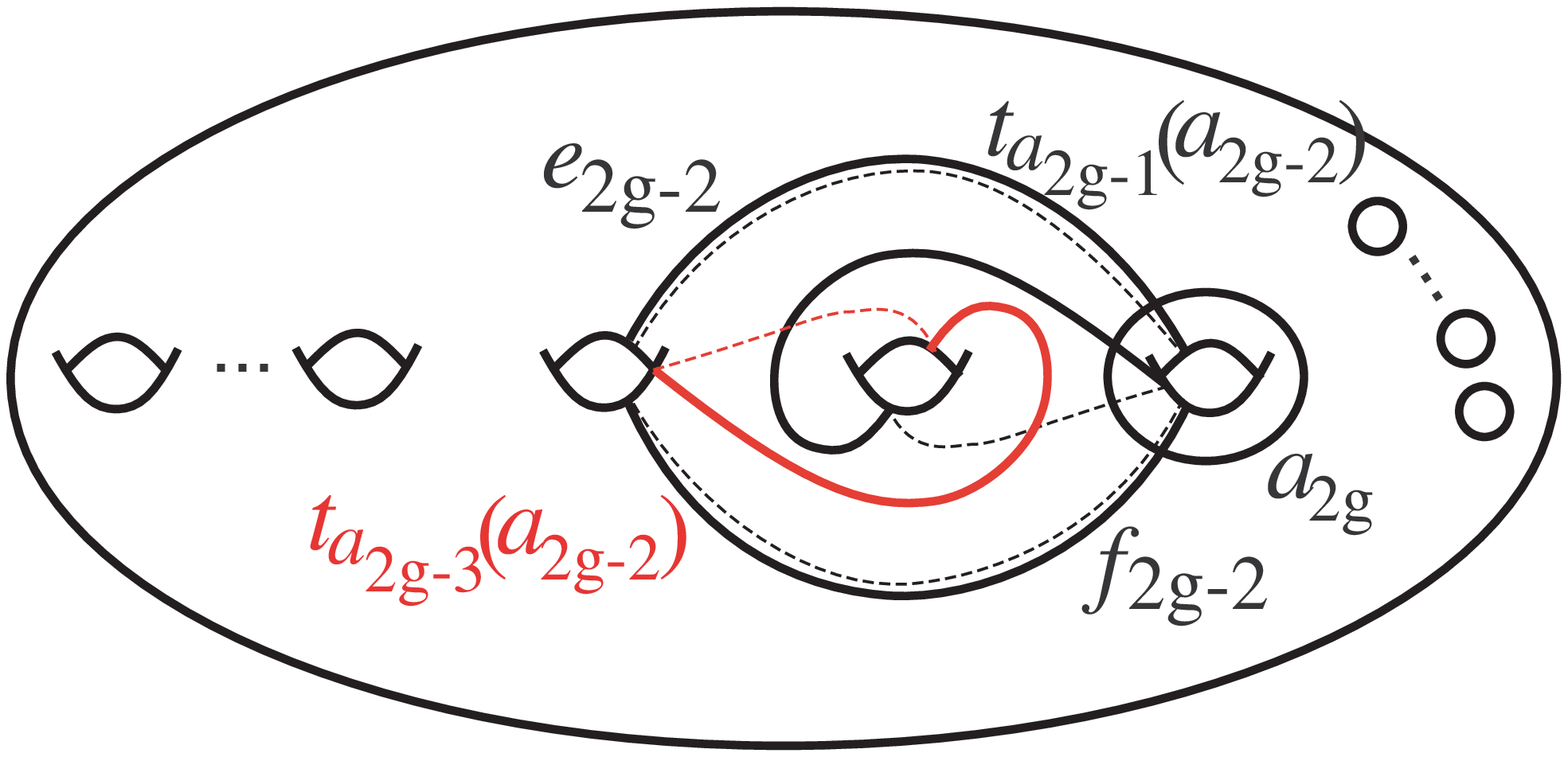} 

\hspace{-0.5cm} (i) \hspace{6.3cm} (ii)

\hspace{-0.4cm} \epsfxsize=2.69in \epsfbox{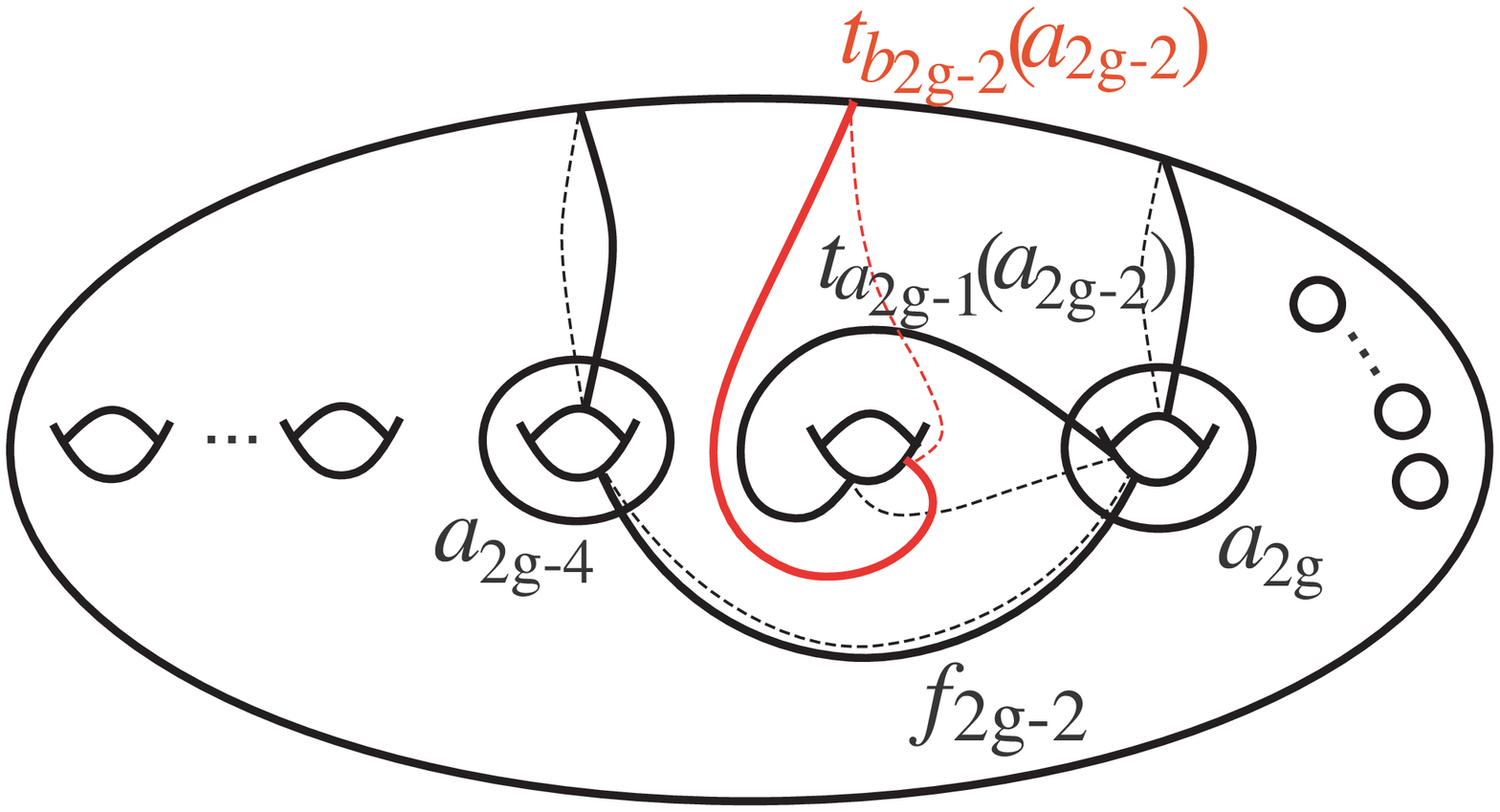} \hspace{0.3cm}   \epsfxsize=2.69in \epsfbox{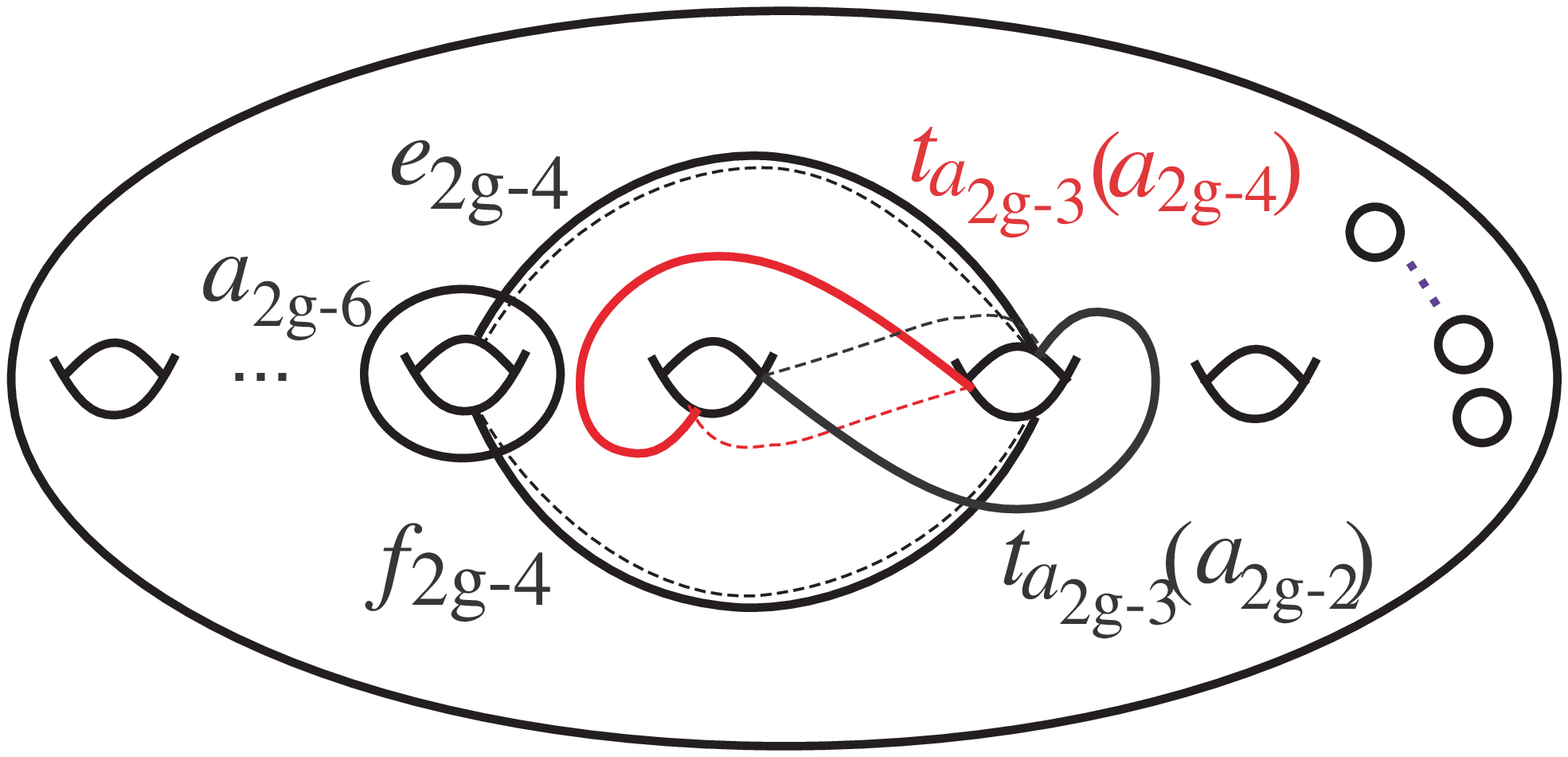}
  
\hspace{-0.5cm} (iii) \hspace{6.1cm} (iv)
 
\hspace{-0.4cm} \epsfxsize=2.7in \epsfbox{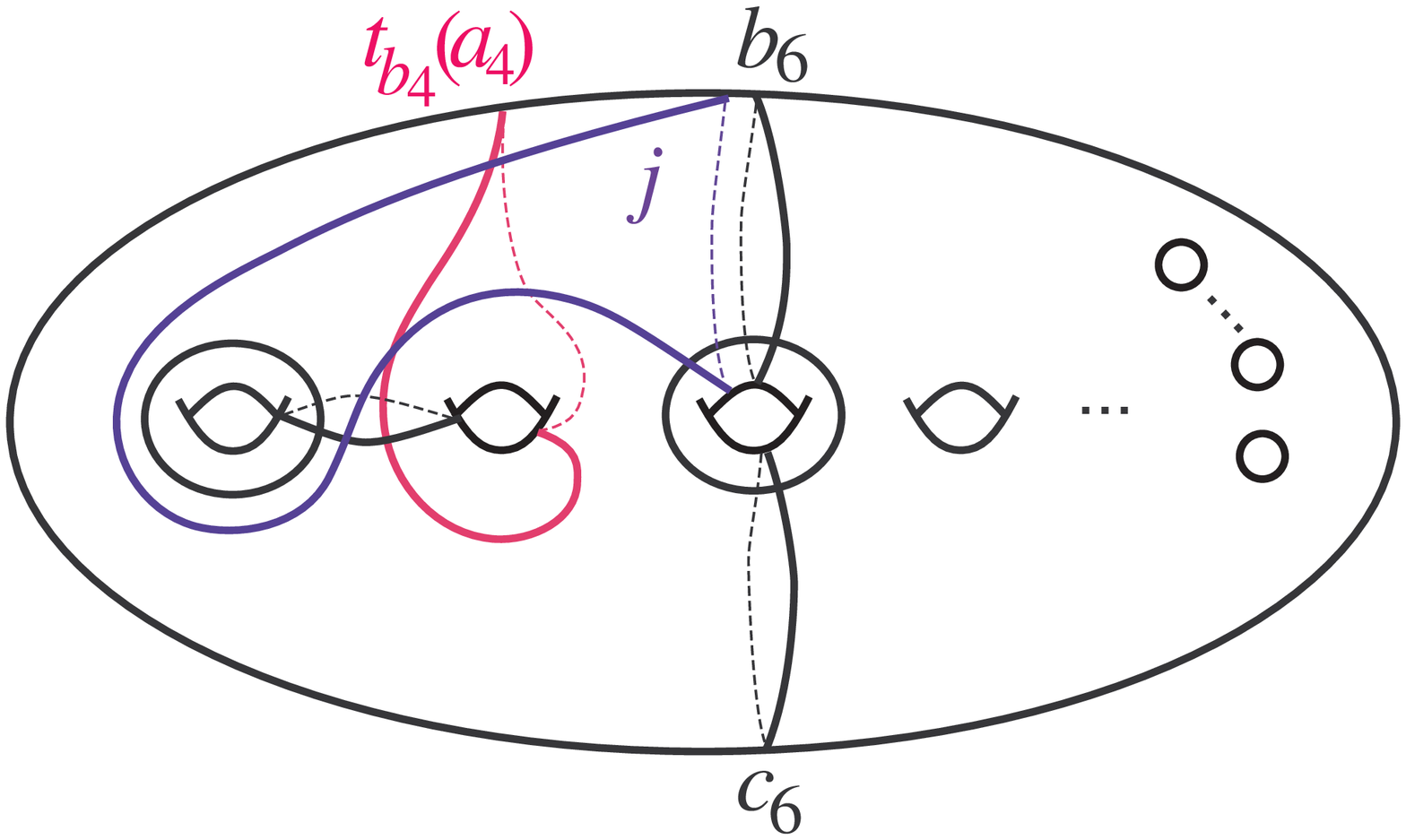} \hspace{0.2cm} \epsfxsize=2.7in \epsfbox{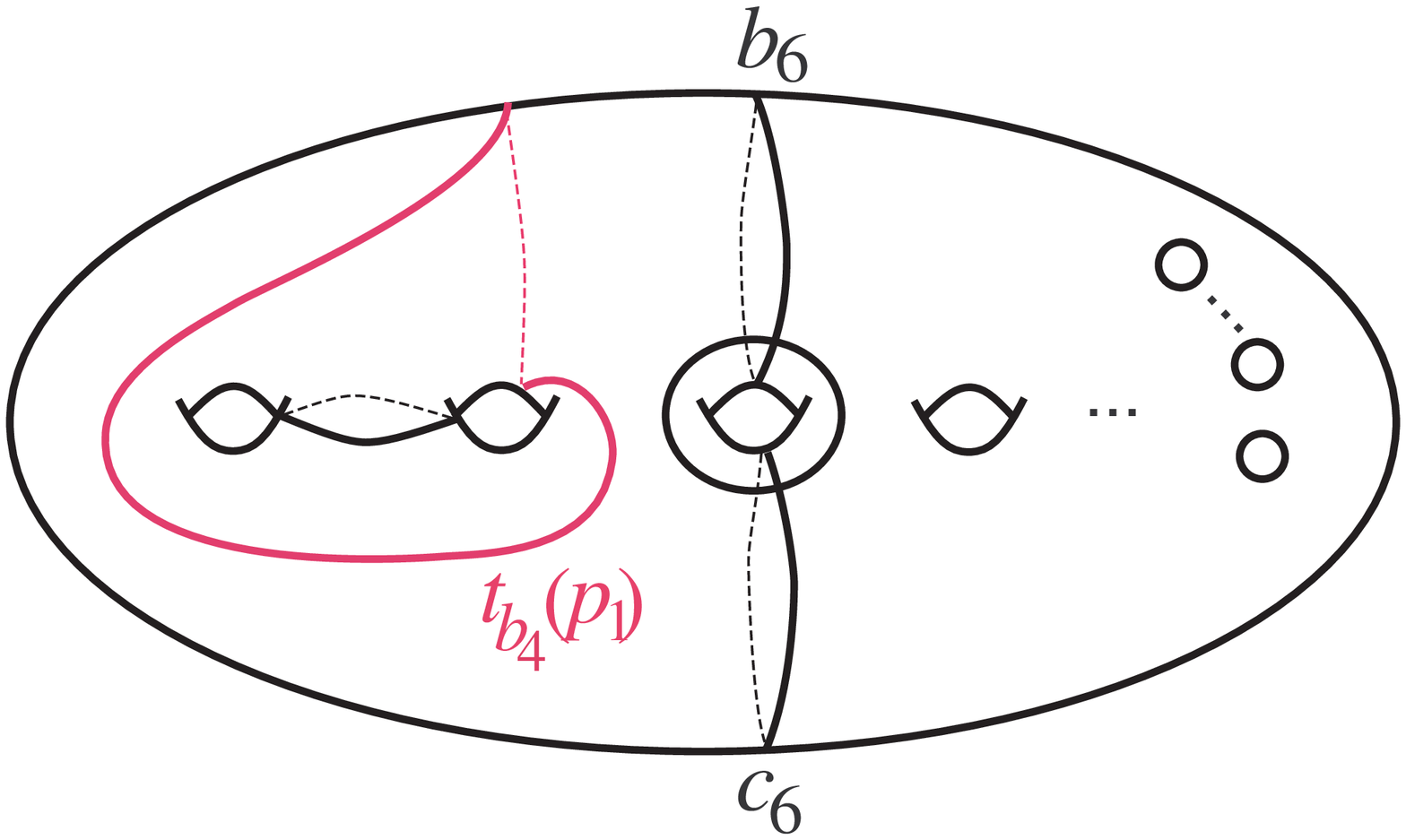} 

\hspace{-0.5cm} (v) \hspace{6.3cm} (vi)

\hspace{-0.9cm} \epsfxsize=2.83in \epsfbox{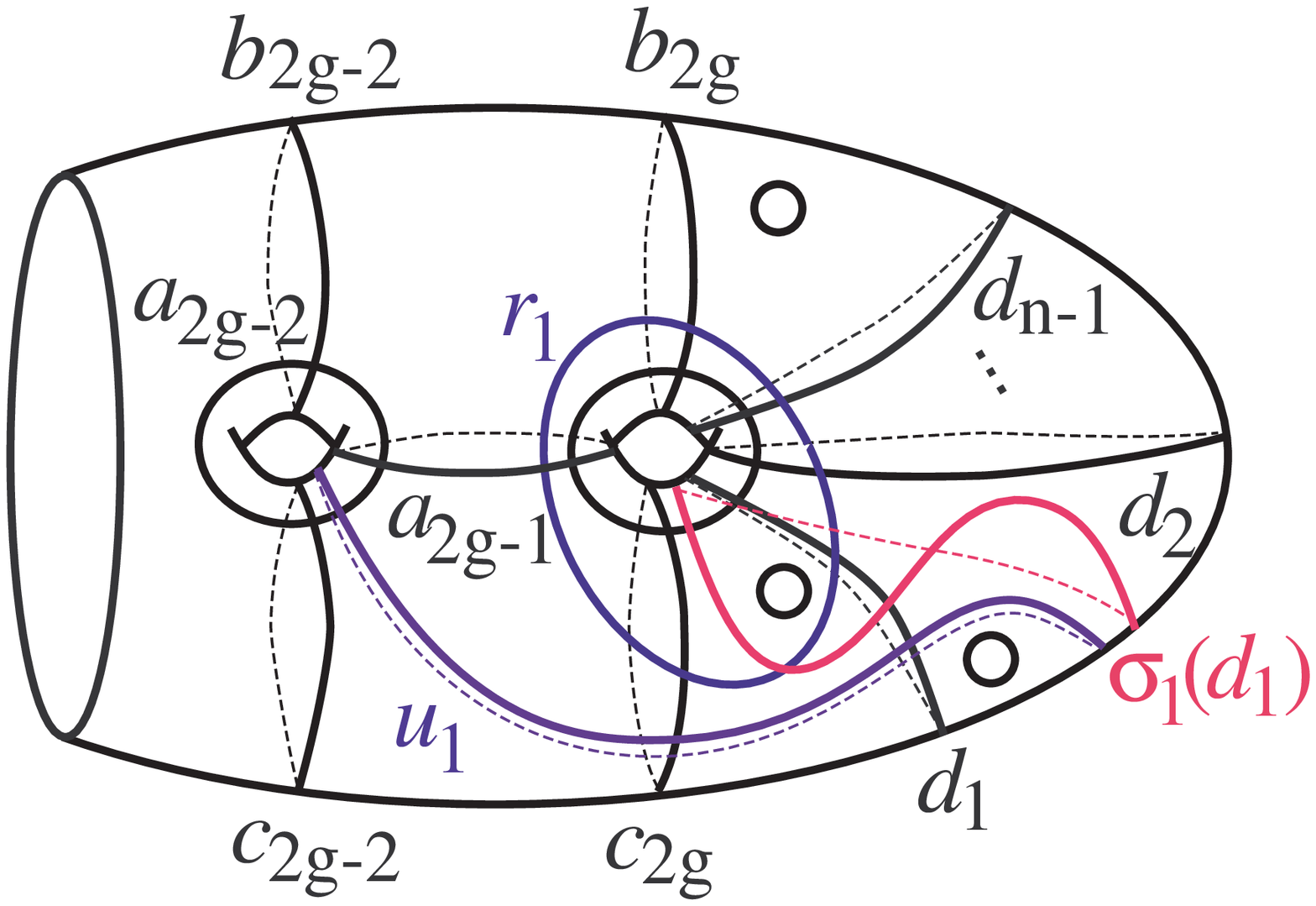} \hspace{-0.3cm}  \epsfxsize=2.83in \epsfbox{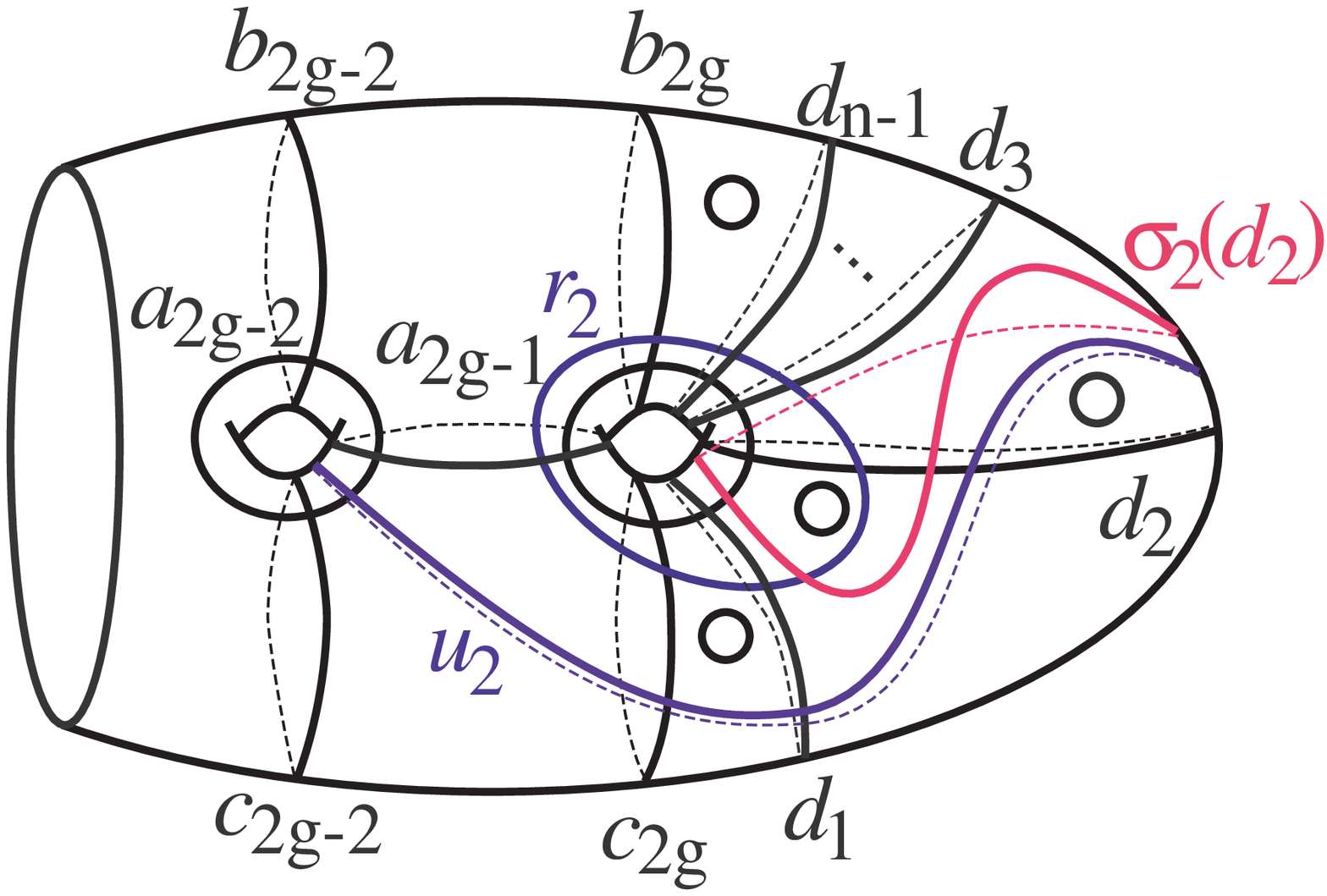}

\hspace{-0.5cm} (vii) \hspace{6cm} (viii)
\caption{Curve configurations for twists and half twists} \label{fig7b}
\end{center}
\end{figure} 

Before we control other twists we will consider the curves $e_i, f_i$ for $i = 4, 6, \cdots, 2g-2$ shown in Figure \ref{fig7} (vii), (viii). The curve 
$e_i$ is the unique nontrivial simple closed curve up to isotopy that is disjoint from each curve in $\mathcal{C}_1 \setminus \{a_{i-2}, b_i, a_{i+2}\}$ that 
intersects each of $a_{i-2}$ and $a_{i+2}$ once.
Since $h([x]) = \lambda([x])$ for all these curves and these properties are preserved by $\lambda$, we have
$h([e_i]) = \lambda([e_i])$ for $i = 4, 6, \cdots, 2g-2$.  The curve 
$f_i$ is the unique nontrivial simple closed curve up to isotopy that is disjoint from each curve in $\mathcal{C}_1 \setminus \{a_{i-2}, c_i, a_{i+2}\}$ that 
intersects each of $a_{i-2}$ and $a_{i+2}$ once.
Since $h([x]) = \lambda([x])$ for all these curves and these properties are preserved by $\lambda$, we have
$h([f_i]) = \lambda([f_i])$ for $i = 4, 6, \cdots, 2g-2$. Consider the curves given in Figure \ref{fig7b} (i), (ii).
The curve $t_{a_{2g-1}}(a_{2g-2})$ is the unique nontrivial simple closed curve up to isotopy that is disjoint from 
$e_{2g-2}, f_{2g-2}, a_{2g-4}, t_{a_{2g-1}}(a_{2g})$ and intersects each of $a_{2g-1}$ and $a_{2g-2}$ once.
Since $h([x]) = \lambda([x])$ for all these curves and these properties are preserved by $\lambda$, we have
$h([t_{a_{2g-1}}(a_{2g-2}) ]) = \lambda([t_{a_{2g-1}}(a_{2g-2})])$. The curve $t_{a_{2g-3}}(a_{2g-2})$ is the unique nontrivial simple closed curve up to isotopy that is disjoint from 
$e_{2g-2}, f_{2g-2}, a_{2g}, t_{a_{2g-1}}(a_{2g-2})$ and intersects each of $a_{2g-3}$ and $a_{2g-2}$ once.
Since $h([x]) = \lambda([x])$ for all these curves and these properties are preserved by $\lambda$, we have
$h([t_{a_{2g-3}}(a_{2g-2}) ]) = \lambda([t_{a_{2g-3}}(a_{2g-2})])$.
Similarly, we get $h ([t_x(y)]) = \lambda ([t_x(y)])$ for all $x,y \in \{ a_1, a_2, \dots, a_{2g}, b_4, b_6, \dots, b_{2g}, c_4,$ $ c_6, \dots, c_{2g}, 
d_1, \dots, d_{n-1}\}$, (see Figure \ref{fig7b} (iii), (iv) for some similar configurations). 

Consider the curves given in Figure \ref{fig7b} (v), (vi). The curve $j$ is the unique nontrivial simple closed curve up to isotopy that is disjoint from 
$c_4, a_2, a_4,$ $ a_5, b_6, w_1$ and intersects each of $a_1, a_3, a_6$ once and bounds a pair of pants together with $a_2, b_6$.
Since $h([x]) = \lambda([x])$ for all these curves and these properties are preserved by $\lambda$, we have
$h([j]) = \lambda([j])$. The curve $t_{b_{4}}(p_{1})$ is the unique nontrivial simple closed curve up to isotopy that is disjoint from 
$a_6, b_6, c_6, j, a_2, a_3, t_{b_{4}}(a_{4})$ and intersects each of $a_1, a_5, b_4, p_1$ once.
Since $h([x]) = \lambda([x])$ for all these curves and these properties are preserved by $\lambda$, we have
$h([t_{b_{4}}(a_{4})]) = \lambda([t_{b_{4}}(a_{4})])$. Similarly, we get $h ([t_x(y)]) = \lambda ([t_x(y)])$ 
for all $x,y \in \{ a_1, a_5, c_4, p_1 \}$. Hence, we get $h ([t_x(y)]) = \lambda ([t_x(y)])$ for all $x,y \in \{ a_1, a_2, \dots, a_{2g}, b_4, b_6,$ $ \dots, b_{2g}, c_4,$ $ c_6, \dots, c_{2g}, 
d_1, \dots, d_{n-1}, p_1\}$. 
 
For $f=t_x$, where $x \in \{a_1, a_2, \dots, a_{2g}, c_4, c_{2g}, d_1, \dots, d_{n-1} \}$, we let $L_f = \mathcal{C}_1 \cup \{p_1\}$.  
By the above, we know that $\lambda ([x]) = h([x])$ for all $x \in L_f \cup f(L_f)$.

For $f = \sigma_i$, where $i \in \{1, \dots, n-1\}$, we let $L_f = \mathcal{C}_1 \cup \{p_1\}$.
We know that $\lambda ([x]) = h ([x])$ for all $x \in L_f$, and we have $\sigma_i (x) = x$ for all $x \in \mathcal{C}_1 \setminus \{ d_i \}$. 
So, we just need to check $h ([ \sigma_i (d_i) ]) = \lambda ([ \sigma_i (d_i) ])$ for each $i$. For $i=1$, we use the curve $u_1$ shown in Figure \ref{fig7b} (vii).
The curve $u_1 \in \mathcal{C}_2$, so we know $h([u_1]) = \lambda ([u_1])$.
The curve $\sigma_1 (d_1)$, shown in Figure \ref{fig7b} (vii), is the unique nontrivial simple closed curve up to isotopy disjoint from $a_{2g-1}, b_{2g}, c_{2g}, d_2, d_3, \cdots, d_{n-1}, u_1$ which intersects $a_{2g}$ once and nonisotopic to $d_2$.
Since we know that $h([x]) = \lambda([x])$ for all these curves and $\lambda$ preserves these properties, we see that 
$h([\sigma_1 (d_1)]) = \lambda ([\sigma_1 (d_1)])$. Similarly we get $h ([ \sigma_i (d_i) ]) = \lambda ([ \sigma_i (d_i) ])$ for each $i$, see Figure \ref{fig7b} (viii). This finishes the proof for $g \geq 3, n \geq 1$. The proof for $g \geq 3, n=0$ is similar.

\begin{figure}
\begin{center} 
\hspace{-0.4cm} \epsfxsize=2.7in \epsfbox{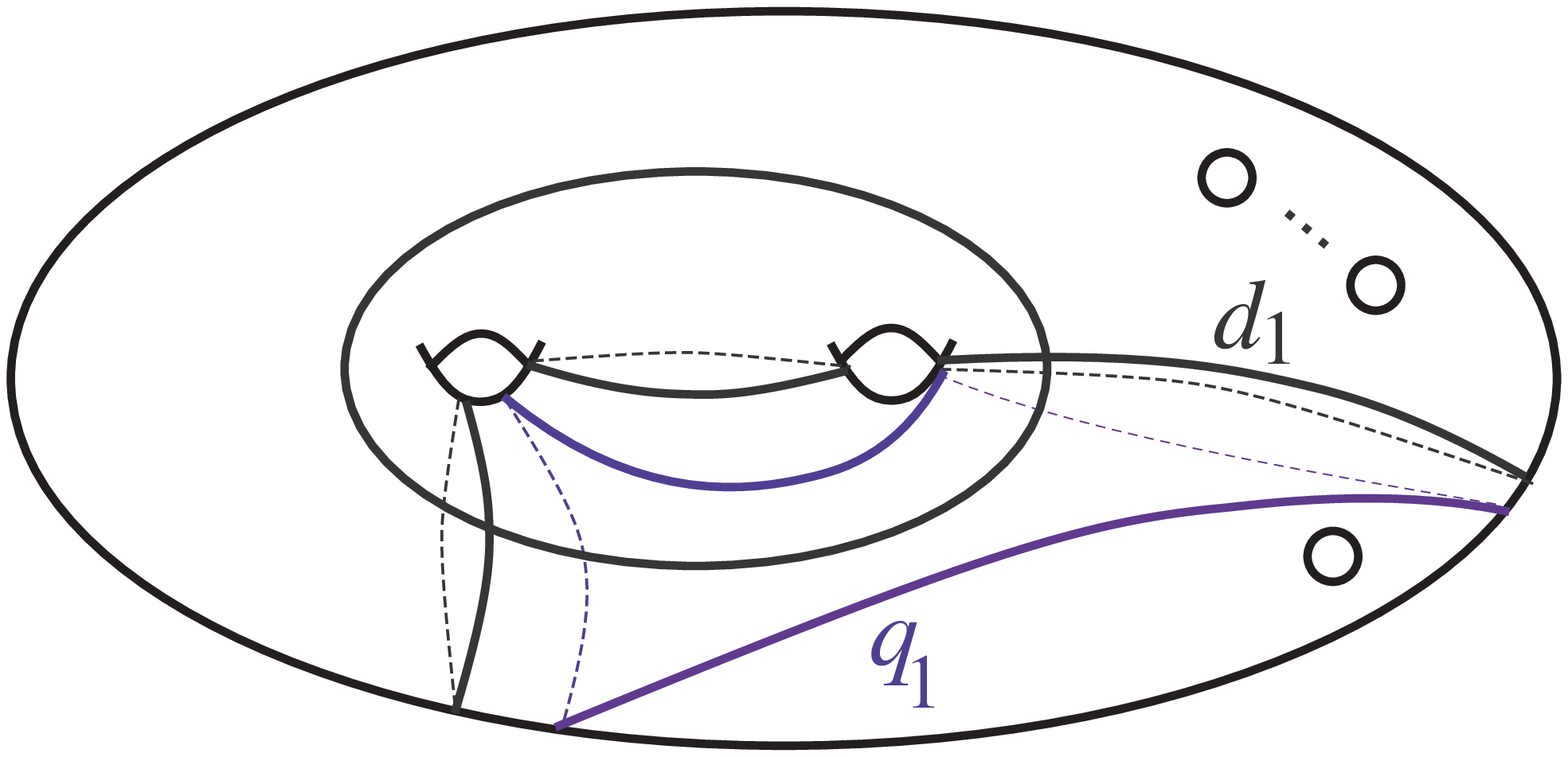} \hspace{0.2cm} \epsfxsize=2.7in \epsfbox{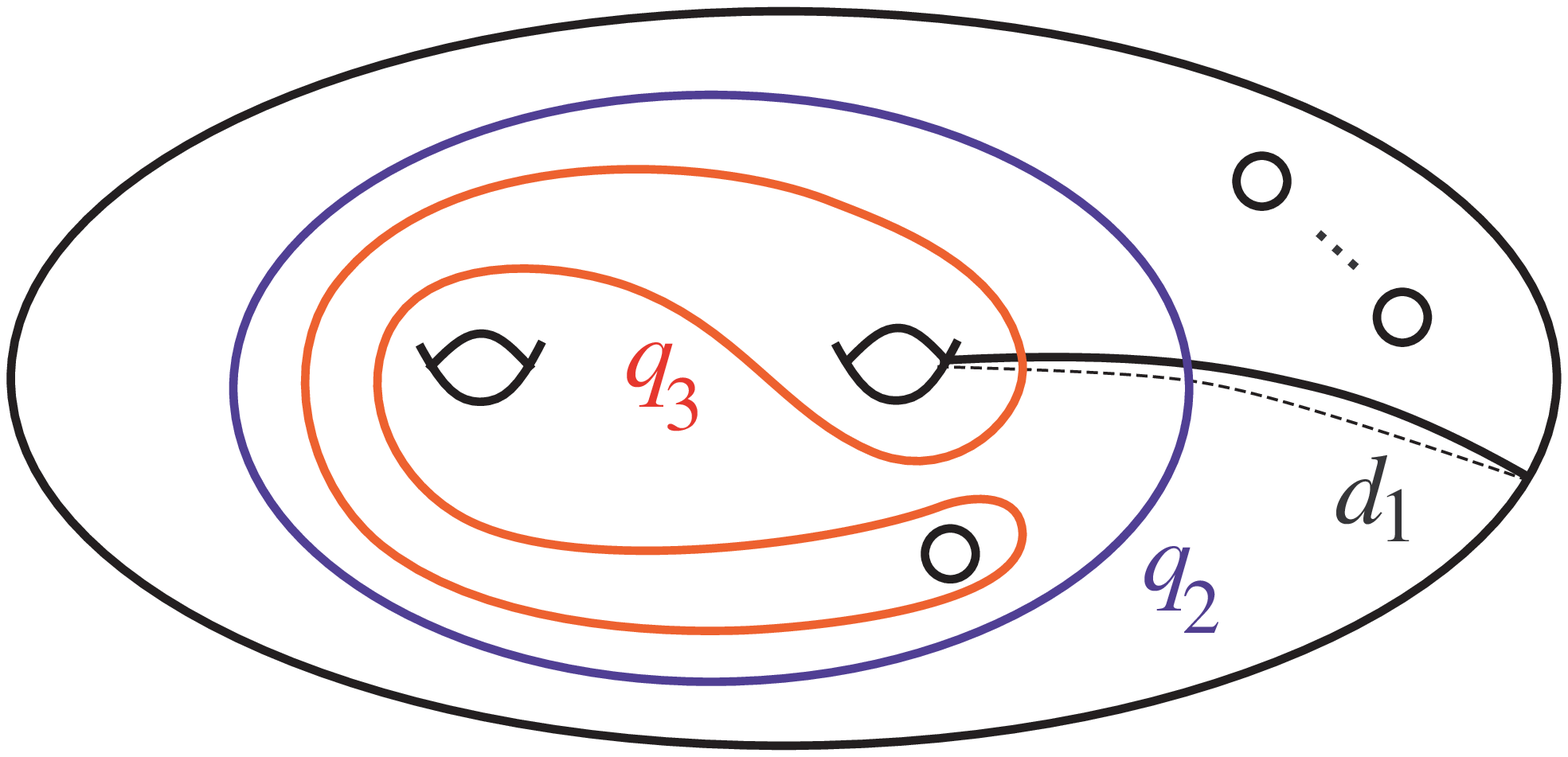} 

\hspace{-1cm} (i) \hspace{6.5cm} (ii)

\hspace{-0.4cm} \epsfxsize=2.7in \epsfbox{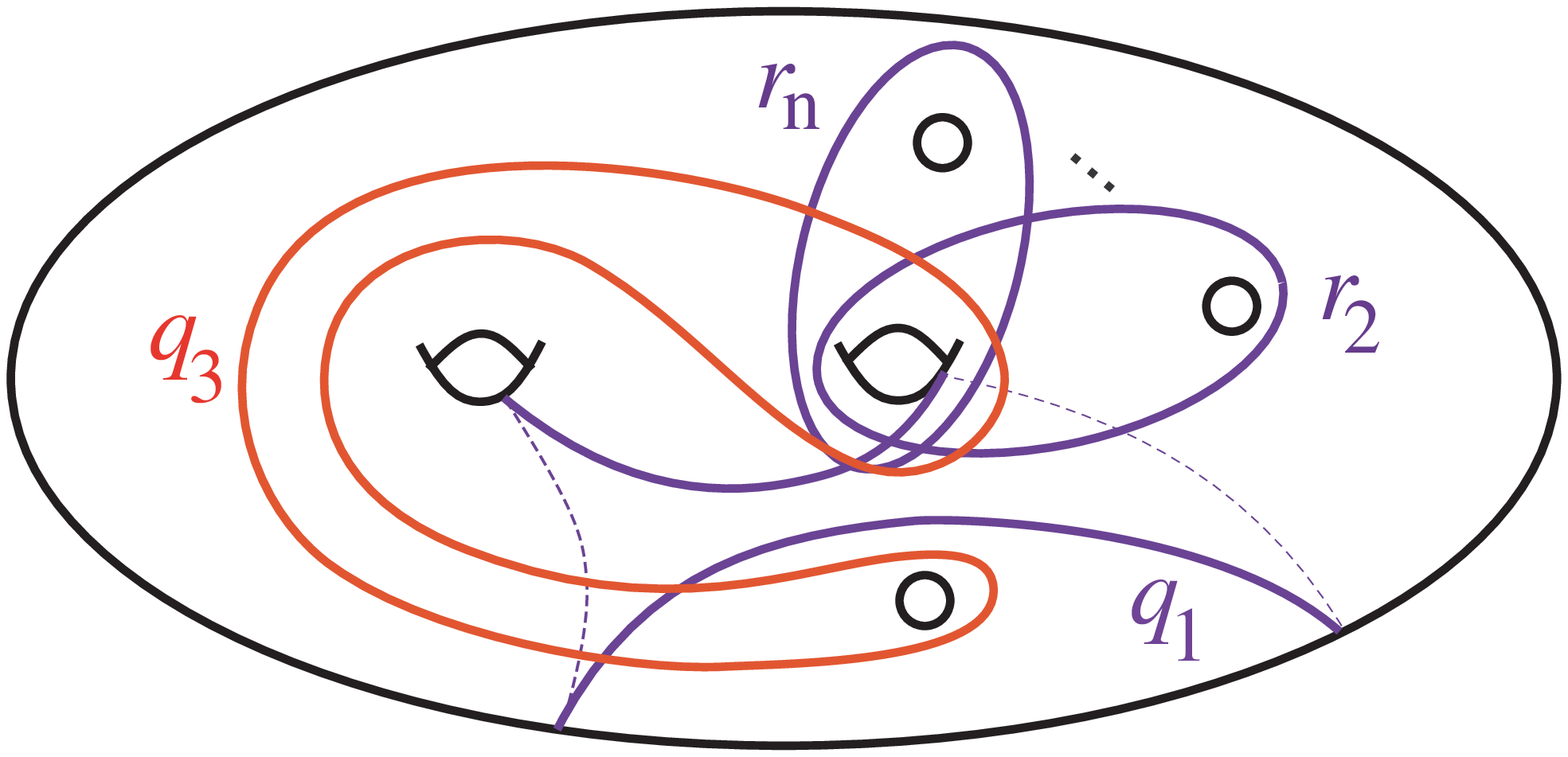} \hspace{0.2cm} \epsfxsize=2.7in \epsfbox{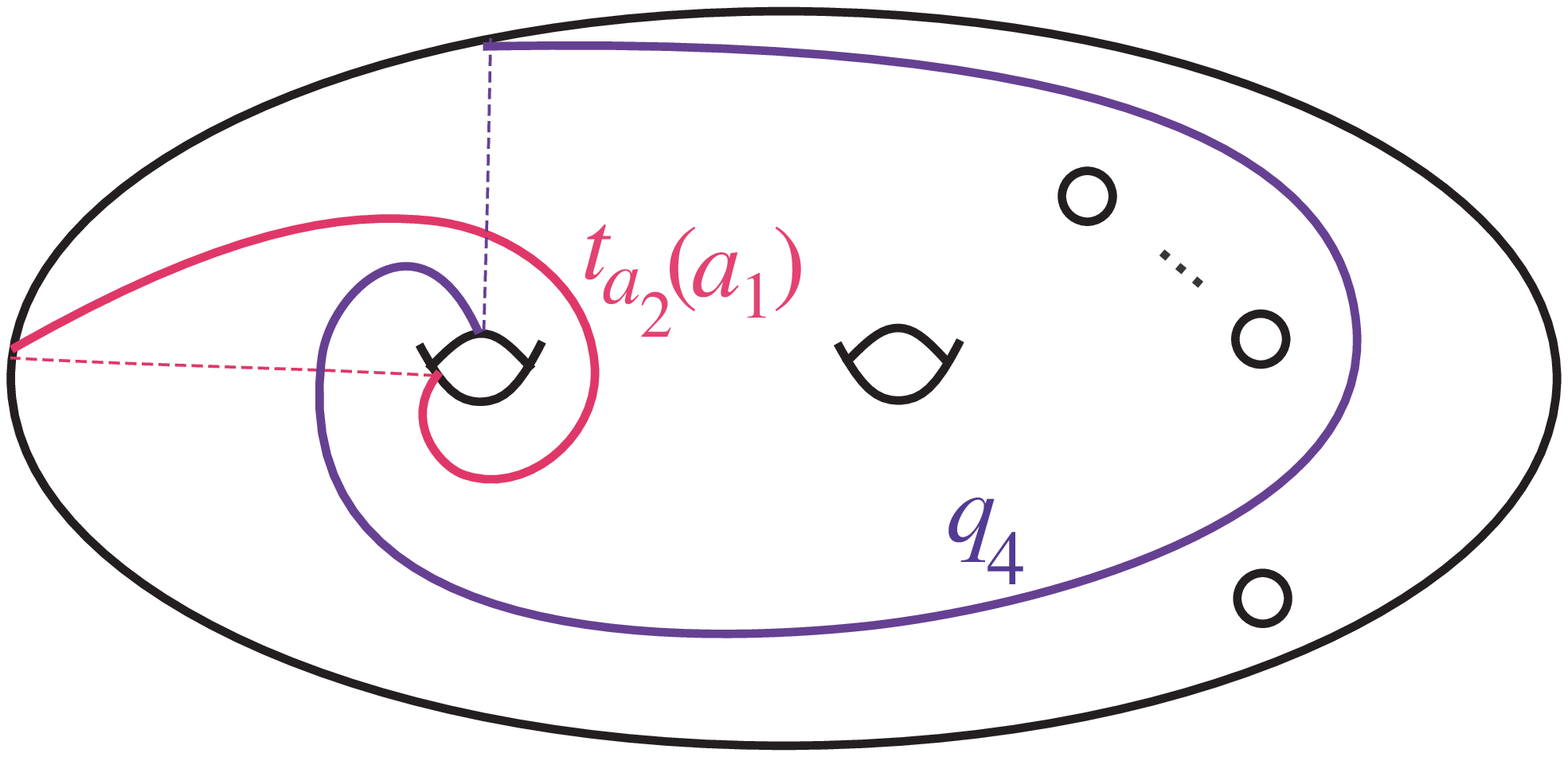} 

\hspace{-0.9cm} (iii) \hspace{6.3cm} (iv)

\hspace{-0.4cm} \epsfxsize=2.7in \epsfbox{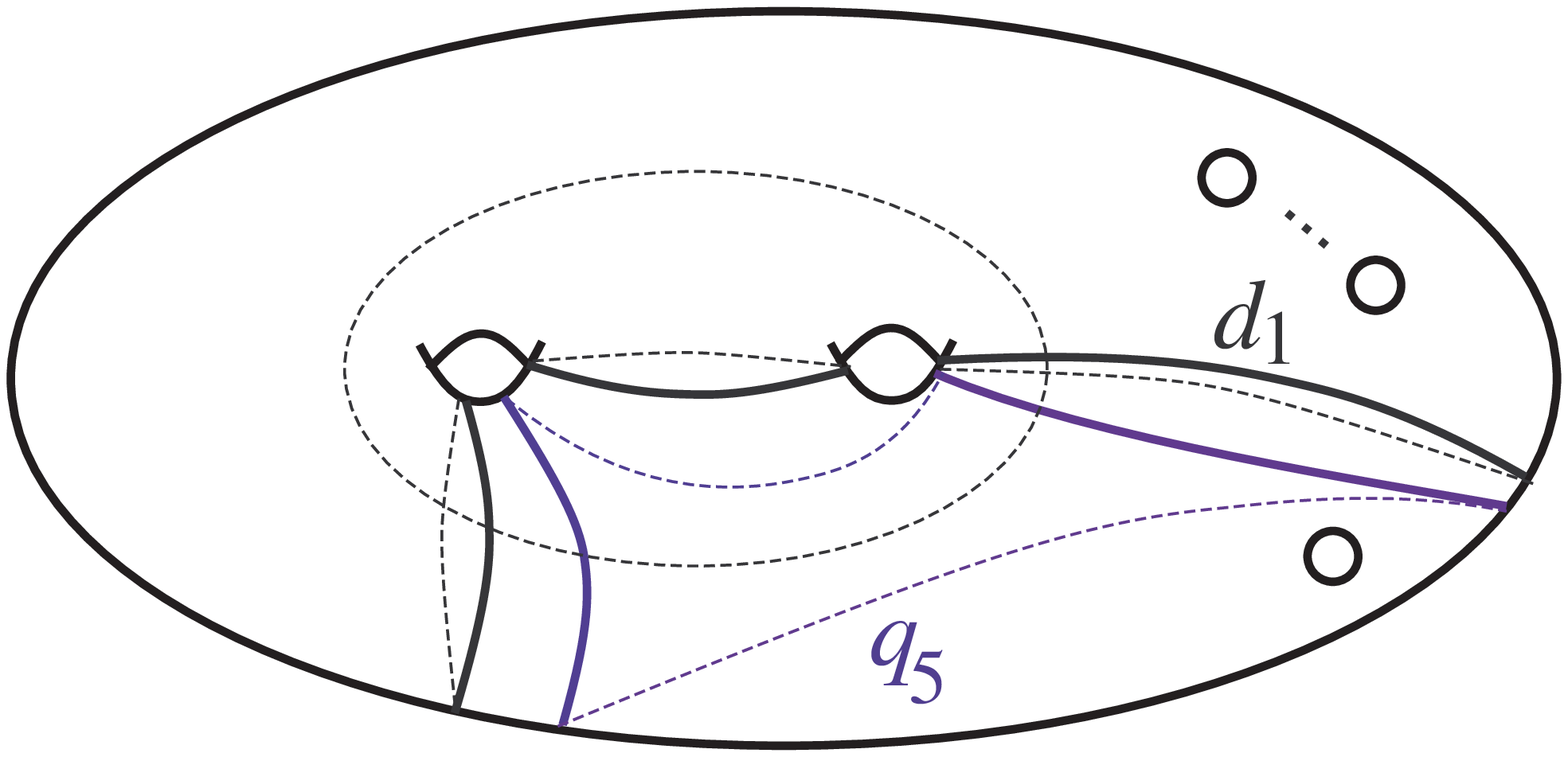} \hspace{0.2cm} \epsfxsize=2.7in \epsfbox{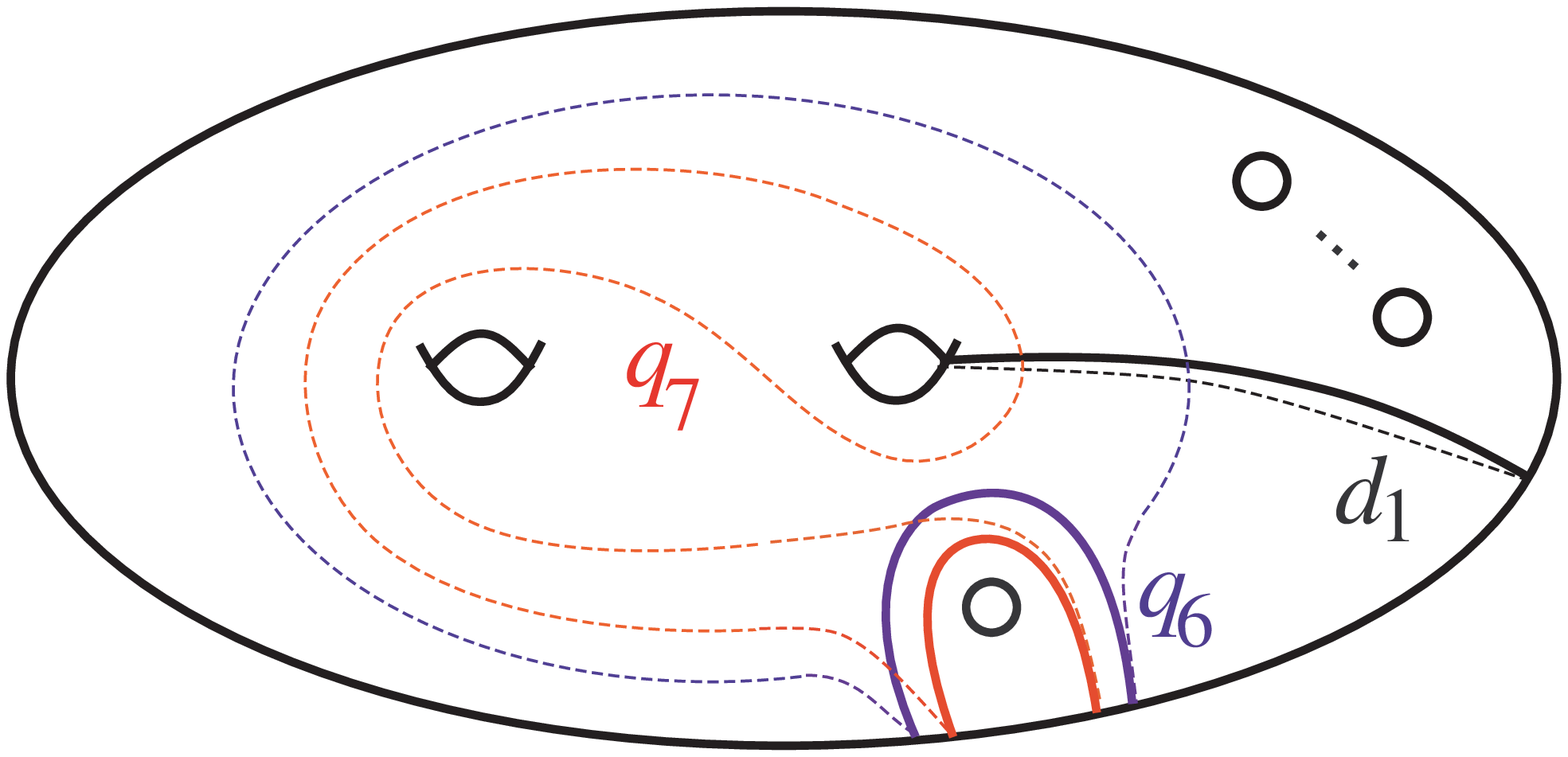} 

\hspace{-1cm} (v) \hspace{6.5cm} (vi)

\hspace{-0.4cm} \epsfxsize=2.7in \epsfbox{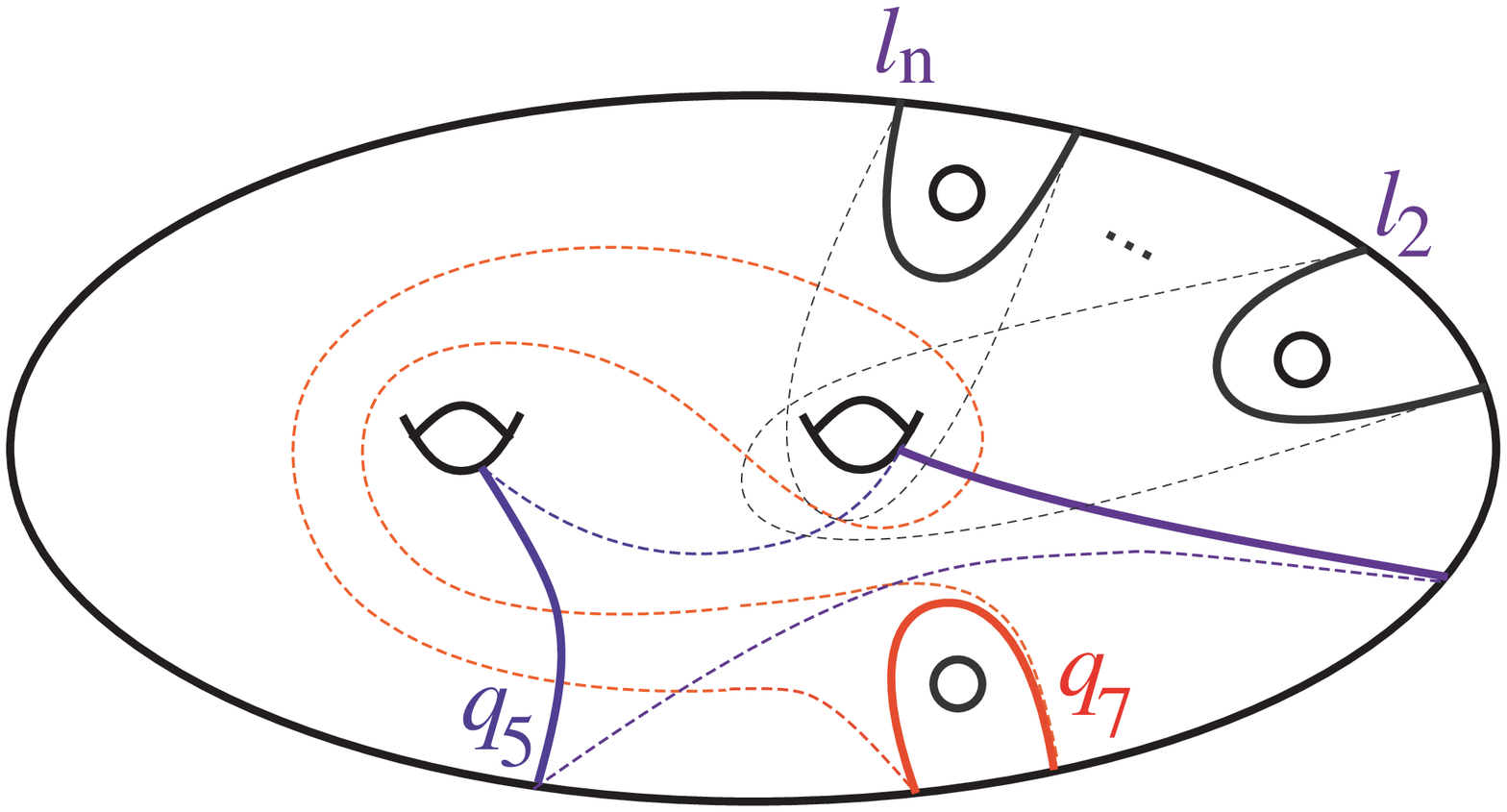} \hspace{0.2cm} \epsfxsize=2.7in \epsfbox{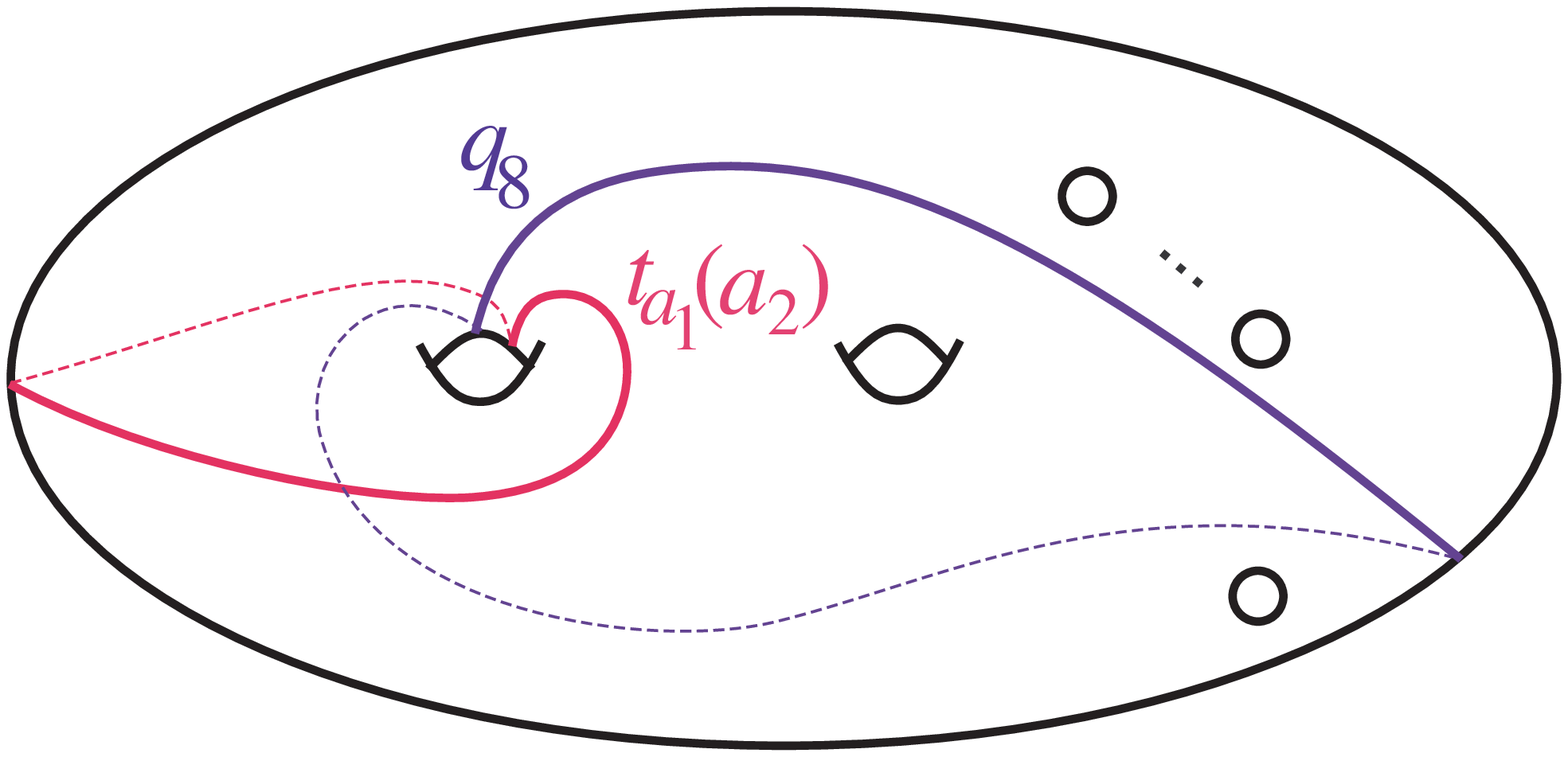} 

\hspace{-0.9cm} (vii) \hspace{6.3cm} (viii)

\caption{Curve configurations for twists} \label{fig9}
\end{center}
\end{figure}

\begin{figure} \begin{center}    
\hspace{1.5cm} \epsfxsize=2.7in \epsfbox{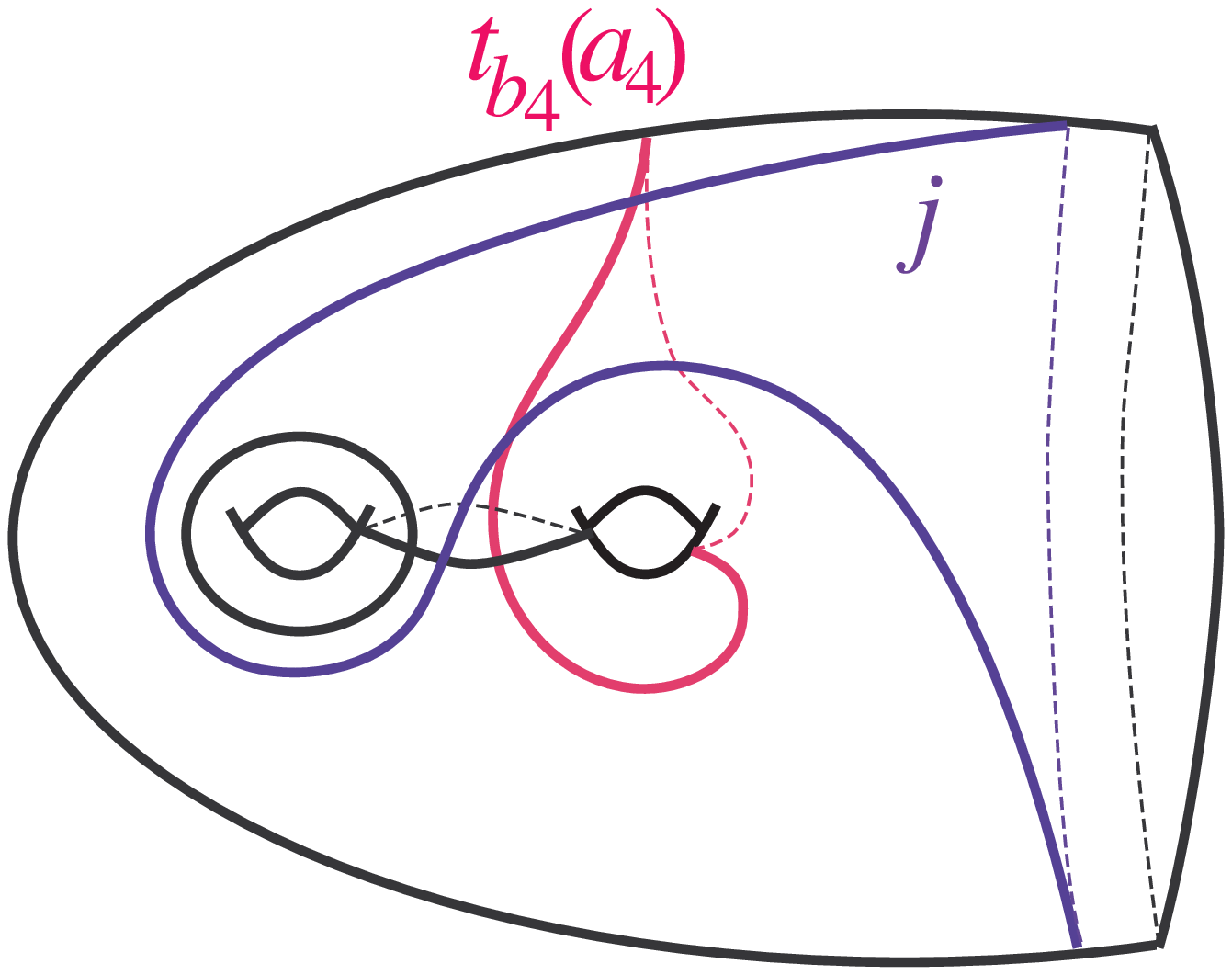} \hspace{-1cm} \epsfxsize=2.7in \epsfbox{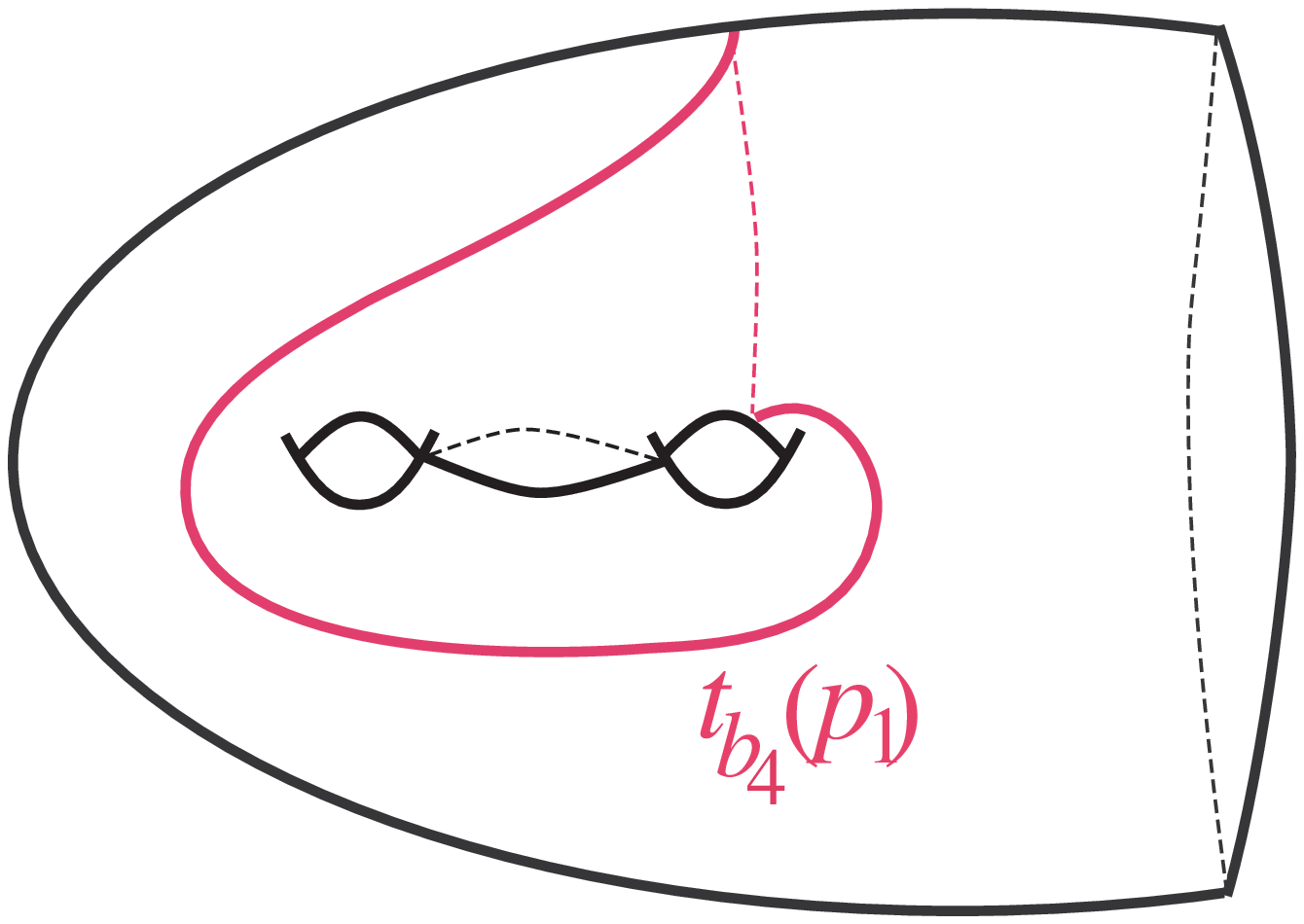} 

\hspace{-0.1cm} (i) \hspace{6cm} (ii)
 
\caption{Curve configurations for twists} \label{fig9b}
\end{center}
\end{figure} 

{\bf Case (ii):} Suppose $g=2$, $n \geq 1$. We have $h([x]) = \lambda([x])$ $\forall \ x \in  \mathcal{C}_1 \cup \mathcal{C}_2 \cup \mathcal{C}_3$ by Lemma \ref{curves-III}. Let $f \in G$.  
For $f=t_x$ when $x \in \{a_1, a_2, a_3, a_4, b_4, c_4, d_1, d_2, \cdots,$ $d_{n-1}\}$, let $L_f = \mathcal{C}_1 \cup \{p_1\}$. The 
set $\mathcal{C}_1 \cup \{p_1\}$ has trivial pointwise stabilizer by Lemma \ref{abcd}. 
We will first give the proof for $f=t_{a_{2}}$. We know $\lambda([x])= h([x])$
$\forall \ x \in \mathcal{C}_1 \cup \{p_1\}$. We will check the equation for $t_{a_2}(a_1)$. 
Consider the curves given in Figure \ref{fig9}. 
The curve $q_1$ is the unique nontrivial simple closed curve up to isotopy that is disjoint from all the curves in
$\{a_1, a_3, d_1, \cdots, d_{n-1}, b_4, p_1\}$, intersects each of $a_2, a_4$ once.  
Since we know that $h([x]) = \lambda([x])$ for all
these curves and $\lambda$ preserves these properties, we see that $h([q_1]) = \lambda([q_1])$. 
We also have $h([q_2]) = \lambda([q_2])$ (the proof is similar to the proof we gave for $h([y_1]) = \lambda([y_1])$ in Lemma \ref{curves-III}). The curve $q_3$ is the unique nontrivial simple closed curve up to isotopy that is disjoint from all the curves in
$\{s_1, \cdots, s_n, a_2, a_4, q_2\}$, intersects each $d_i$ once, and nonisotopic to $a_4$. Since $h([x]) = \lambda([x])$ for all
these curves and $\lambda$ preserves these properties, we have $h([q_3]) = \lambda([q_3])$. 
The curve $q_4$ is the unique nontrivial simple closed curve up to isotopy that is disjoint from all the curves in  
$\{a_3, a_4, r_2, r_3, \cdots, r_n, q_1, q_3\}$ and intersects each of $a_1$ and $a_2$ once. Since $h([x]) = \lambda([x])$ for all
these curves and $\lambda$ preserves these properties, we have $h([q_4]) = \lambda([q_4])$. 
The curve $t_{a_2}(a_1)$ is the unique nontrivial simple closed curve up to isotopy that is disjoint from all the curves in $\{a_4, b_4, c_4, q_4\}$ 
and intersects $a_1, a_2$ once.
Since $h([x]) = \lambda([x])$ for all these curves and $\lambda$ preserves these properties, we have
$h([t_{a_2}(a_1) ]) = \lambda([t_{a_2}(a_1)])$. Consider the curves $q_5, q_6, q_7, q_8$ given in Figure \ref{fig9}. 
By considering the model in Figure \ref{fig1-e} and taking the reflections of the curves $q_1, q_2, q_3, q_4$ through the plane of the paper we get the curves, 
$q_5, q_6, q_7, q_8$. We can make similar arguments (using the curve $w_1$ in the role of $p_1$) to see that $h([x]) = \lambda([x])$ for all the curves $q_5, q_6, q_7, q_8$, and hence 
$h([t_{a_1}(a_2) ]) = \lambda([t_{a_1}(a_2)])$ as in the previous case. Now using $h([t_{a_2}(a_1) ]) = \lambda([t_{a_2}(a_1)])$, $h([t_{a_1}(a_2) ]) = \lambda([t_{a_1}(a_2)])$ and similar curve configurations as in case (i)  
we obtain $h ([t_x(y)]) = \lambda ([t_x(y)])$ for all $x,y \in \{ a_1, a_2, a_3, a_4, b_4, c_4, d_1, \dots, d_{n-1}, p_1 \}$ (see Figure \ref{fig9b}). Hence, we have
$\lambda ([x]) = h ([x])$ for all $x \in L_f \cup f(L_f)$ for every twist $f$ in $G$.
For $f = \sigma_i$, where $i \in \{1, \dots, n-1\}$, we let $L_f = \mathcal{C}_1 \cup \{p_1\}$. The proof is similar to case (i). 
  
\begin{figure}
\begin{center} 
\hspace{-0.4cm} \epsfxsize=2.7in \epsfbox{figy-10-new.eps} \hspace{0.2cm} \epsfxsize=2.7in \epsfbox{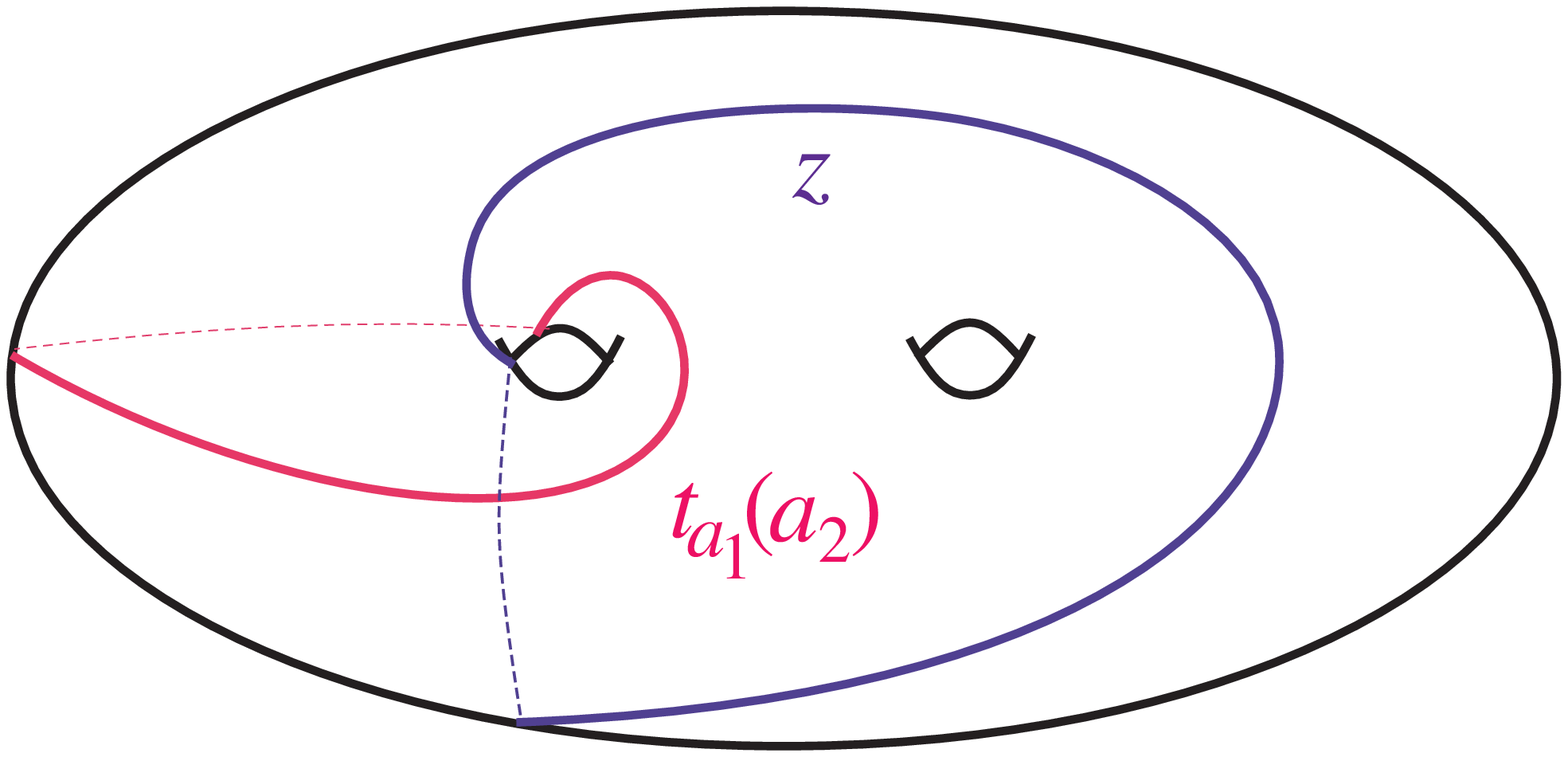} 

\hspace{-0.4cm} (i) \hspace{6.5cm} (ii)

\hspace{-0.4cm} \epsfxsize=2.7in \epsfbox{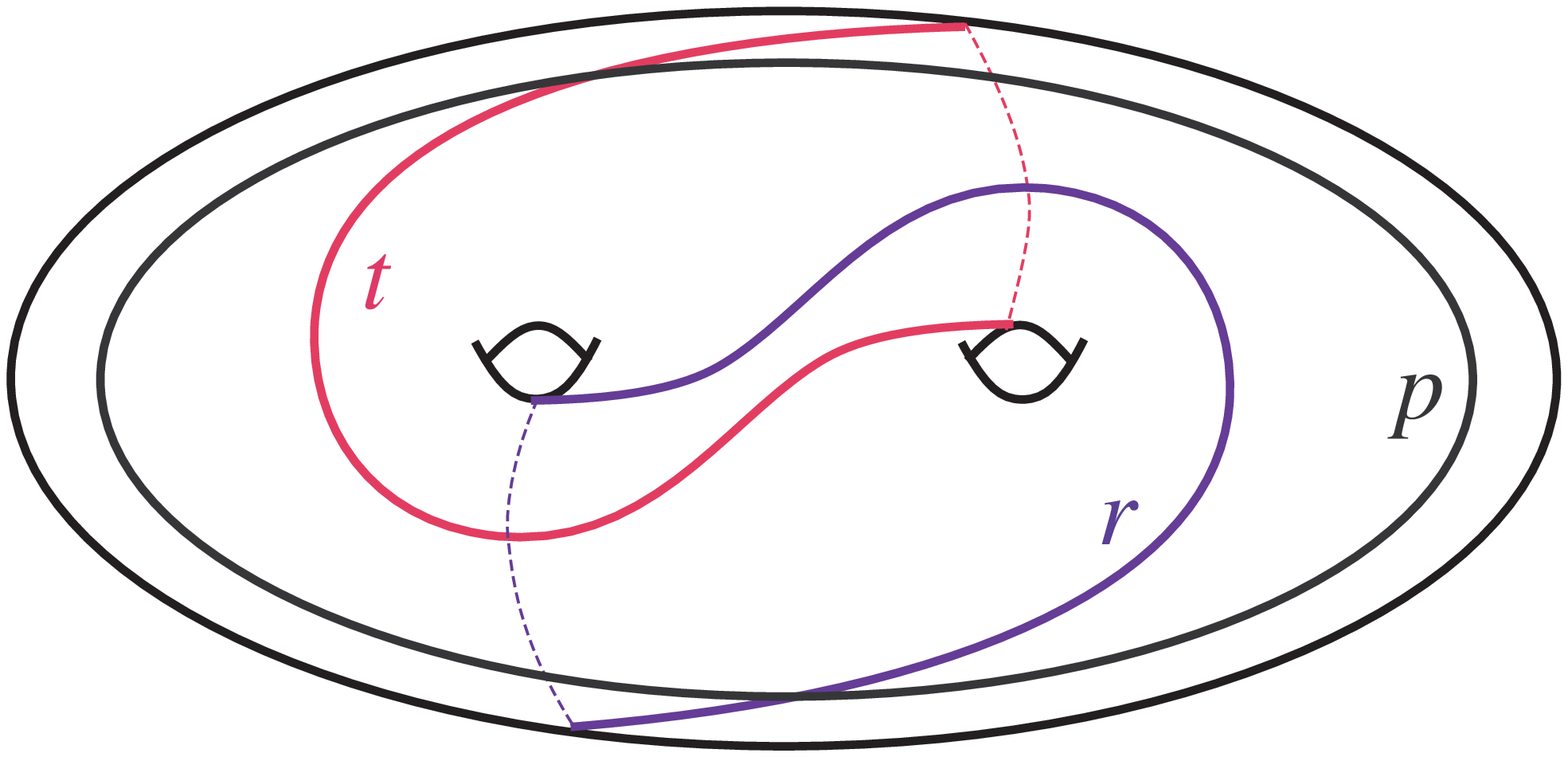} \hspace{0.2cm} \epsfxsize=2.7in \epsfbox{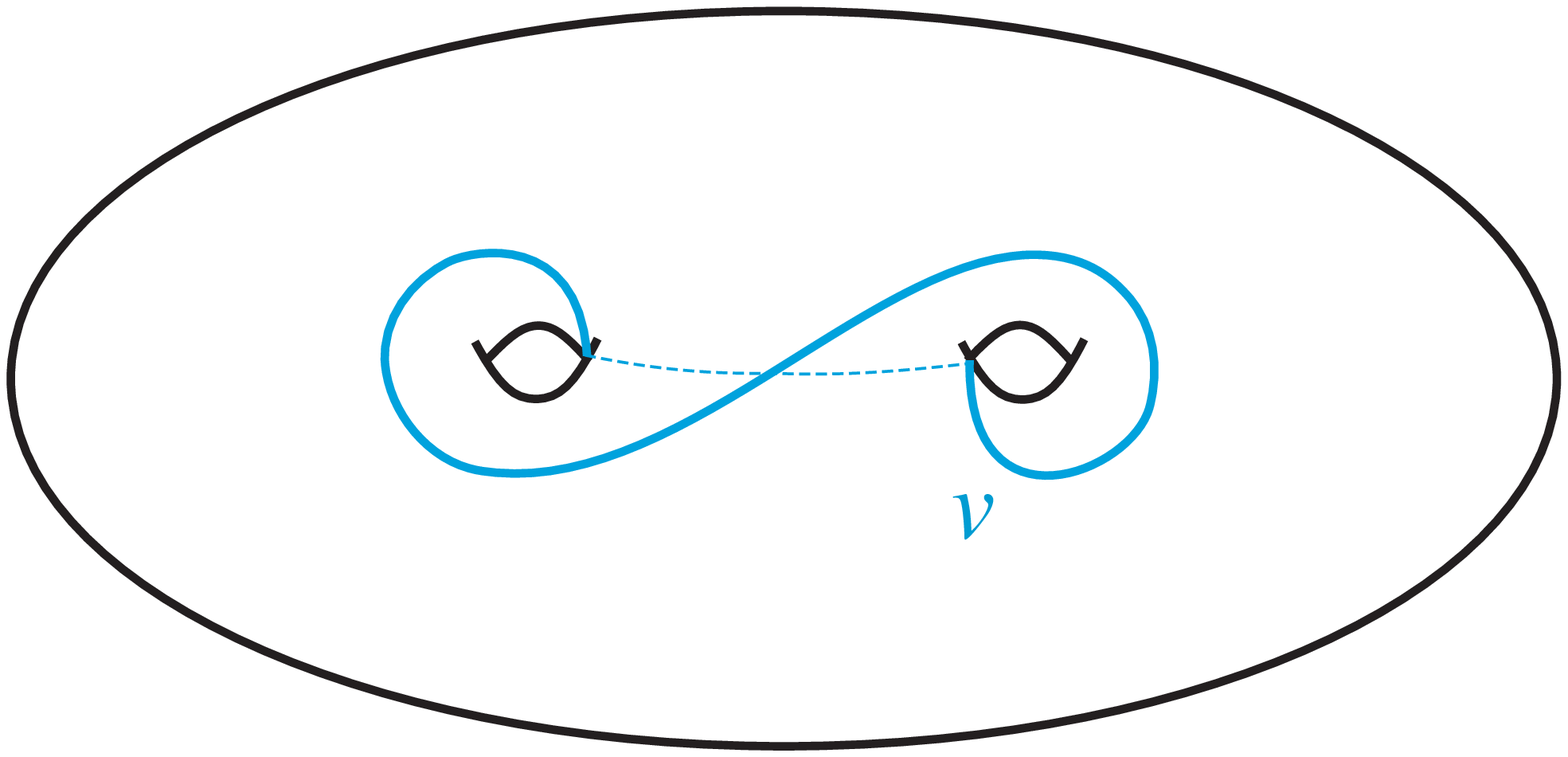} 

\hspace{-0.4cm} (iii) \hspace{6.5cm} (iv)

\hspace{-0.4cm} \epsfxsize=2.7in \epsfbox{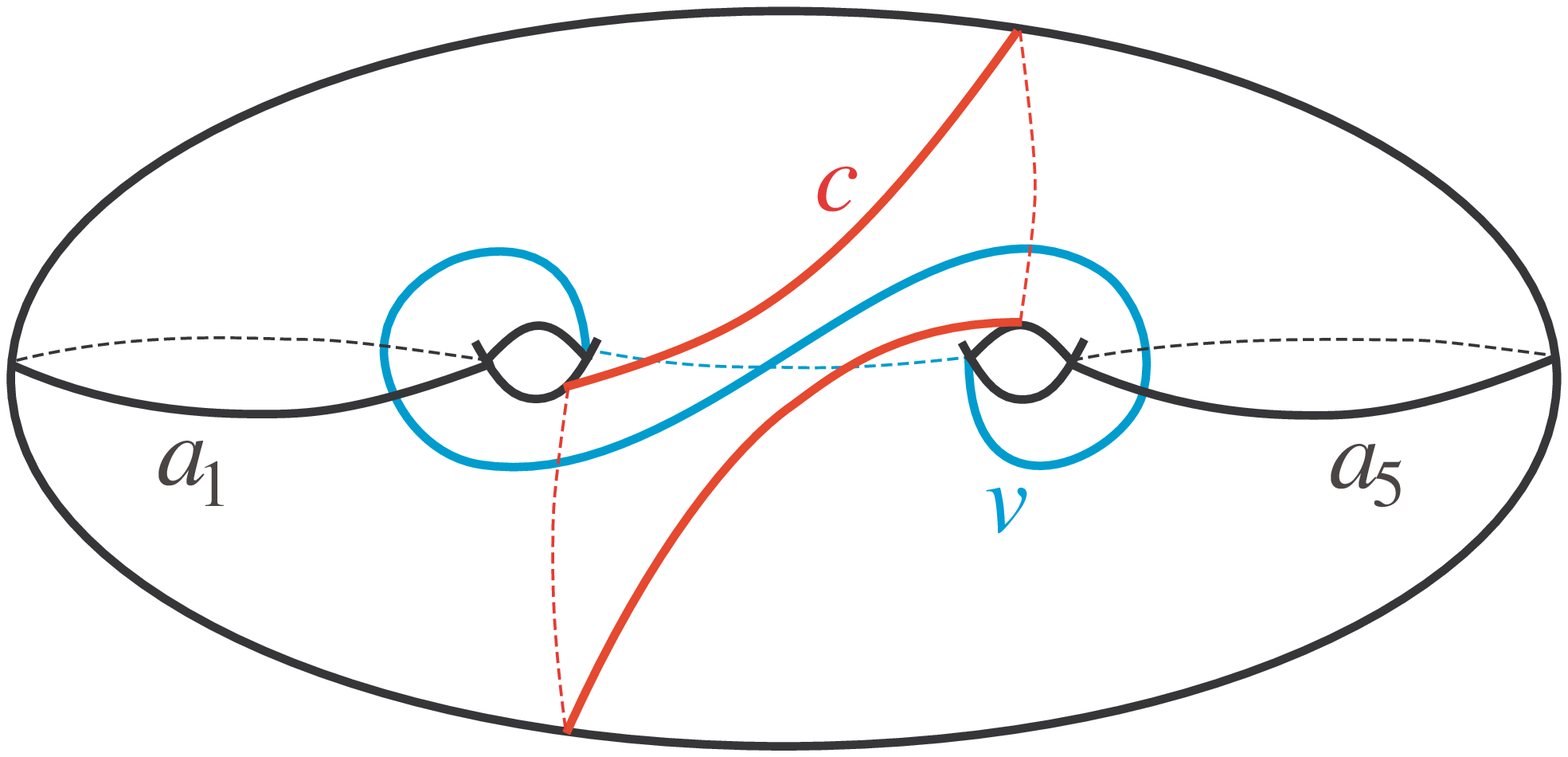} \hspace{0.2cm} \epsfxsize=2.7in \epsfbox{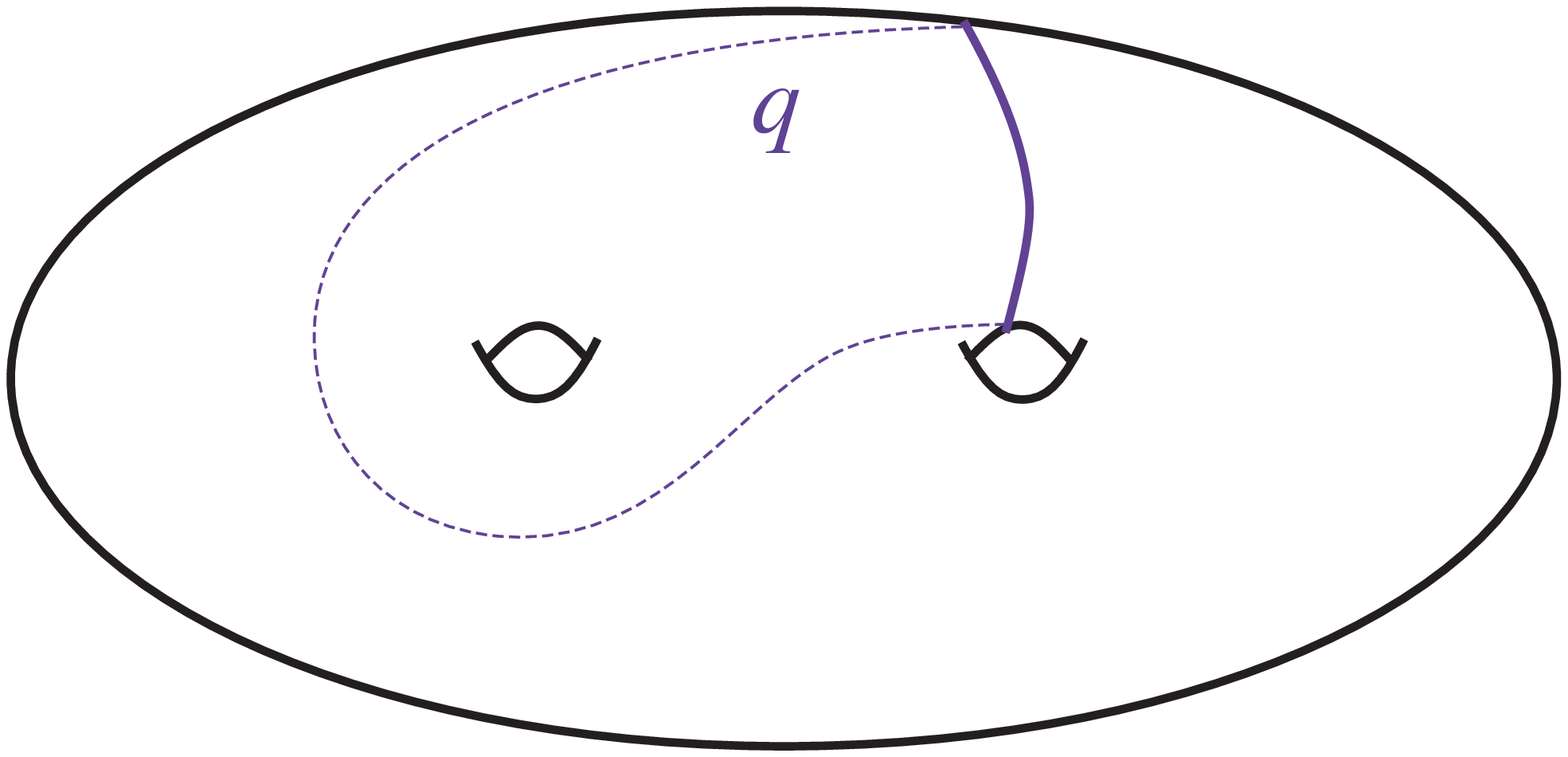} 

\hspace{-0.4cm} (v) \hspace{6.5cm} (vi)

\hspace{-0.4cm} \epsfxsize=2.7in \epsfbox{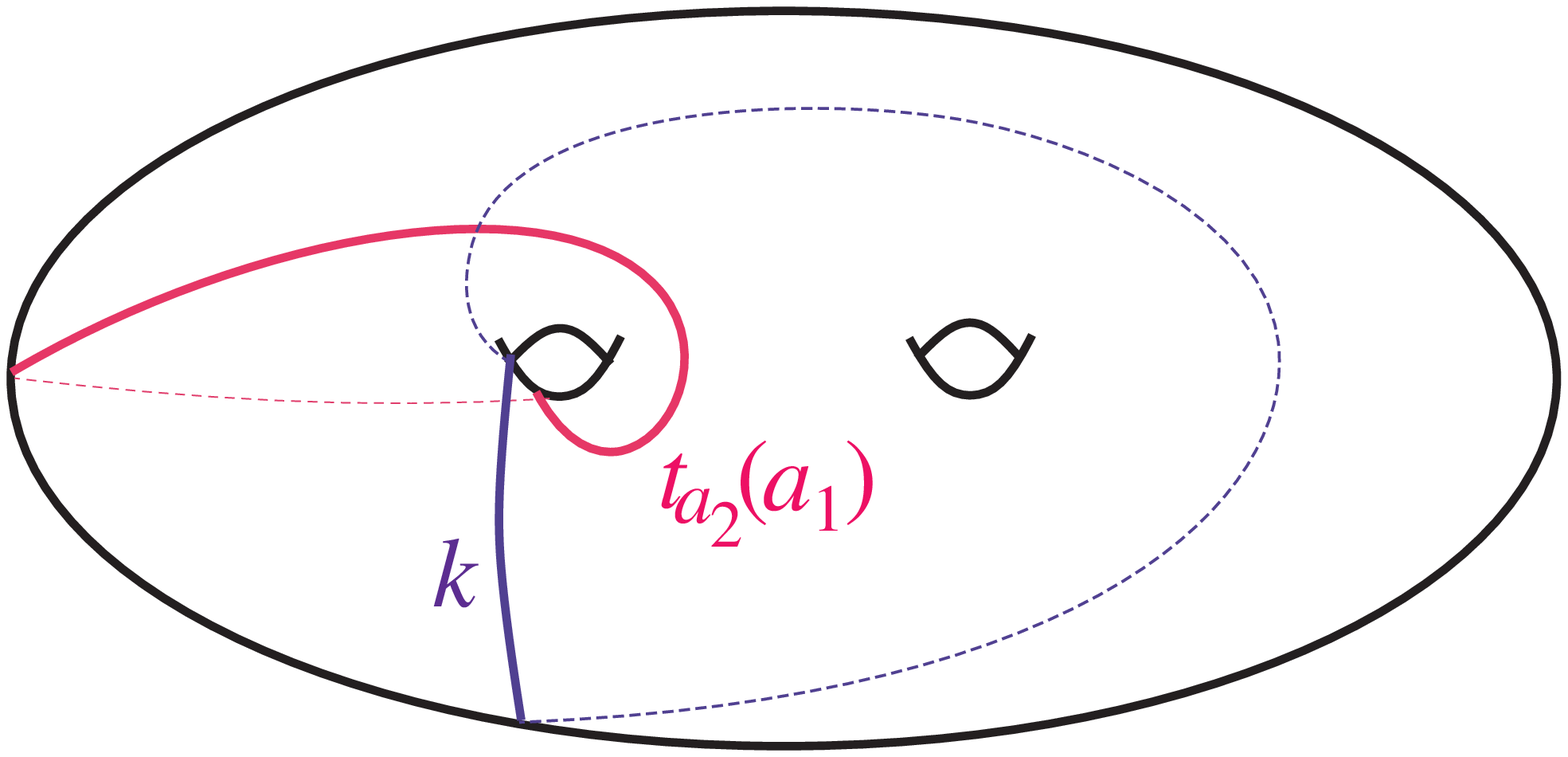} \hspace{0.2cm} \epsfxsize=2.7in \epsfbox{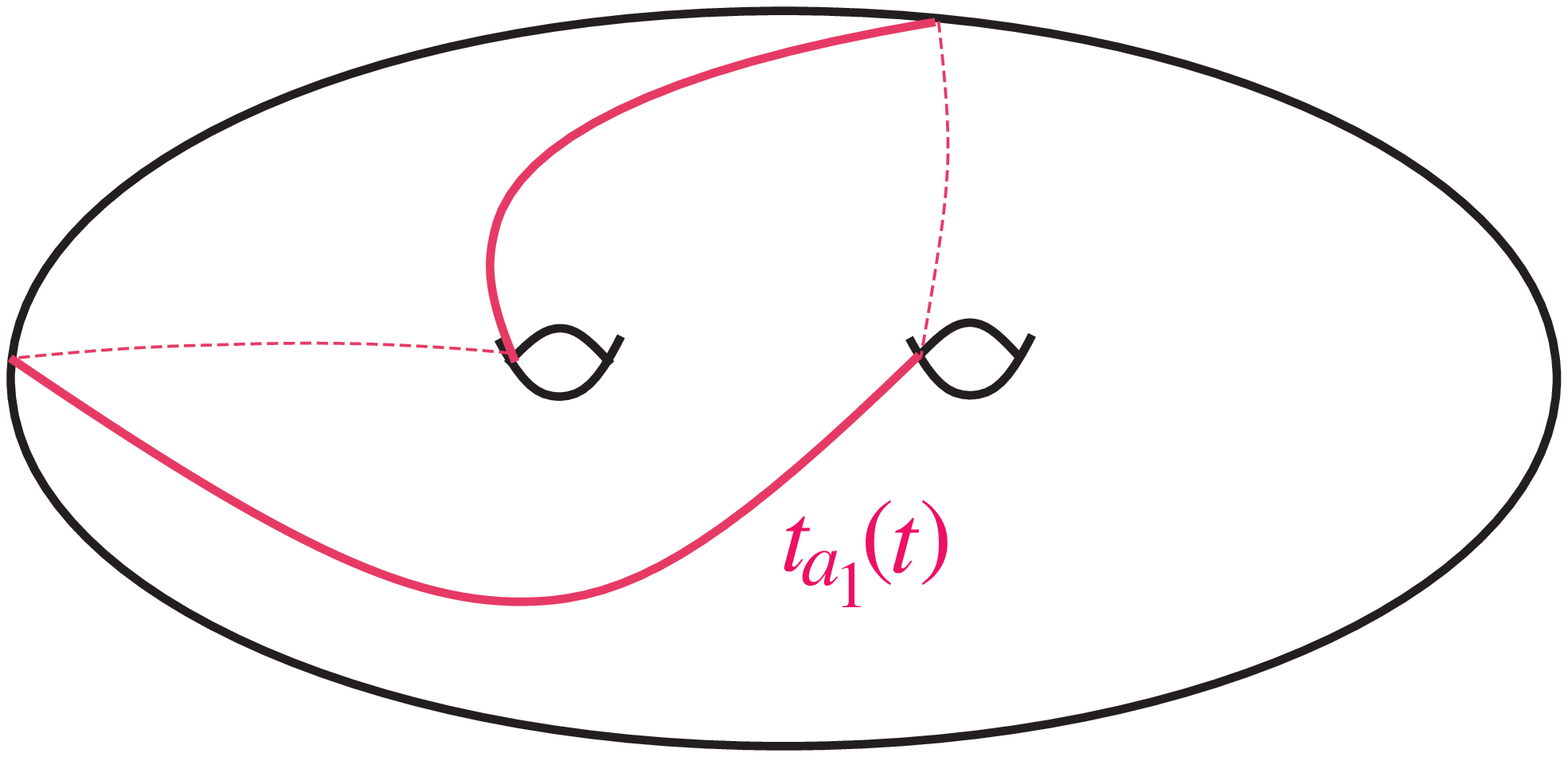} 

\hspace{-0.3cm} (vii) \hspace{6.5cm} (viii)
\caption{Curve configurations for twists} \label{fig88}
\end{center}
\end{figure}

{\bf Case (iii):} Suppose $g=2$, $n=0$. Consider the curves given in Figure \ref{fig88} (i), (ii). 
For $f=t_x$ when $x \in \{a_1, a_2, \cdots, a_{5} \}$, let $L_f = \{a_1, a_2, a_3, a_4, a_5, t\}$. By Lemma \ref{abcde}, the 
pointwise stabilizer of $L_f$ is $Z(Mod^*_R) \cong \mathbb{Z}_2$.
We know $\lambda([x])= h([x])$ $\forall \ x \in \{a_1, a_2, a_3, a_4, a_5, t\}$. 
Let $f=t_{a_1}$. We will check the equation for $t_{a_{1}}(a_{2})$. The curve $z$ is the unique nontrivial simple closed curve up to isotopy that is disjoint from all the curves in
$\{t, a_3, a_4\}$. Since $h([x]) = \lambda([x])$ for all
these curves and $\lambda$ preserves this property, we have $h([z]) = \lambda([z])$. 
The curve $t_{a_{1}}(a_{2})$ is the unique nontrivial simple closed curve up to isotopy that is disjoint from 
$a_4, a_5, z$ and intersects $a_1, a_2$ once. Since $h([x]) = \lambda([x])$ for all
these curves and $\lambda$ preserves these properties, we have $h([t_{a_{1}}(a_{2})]) = \lambda([t_{a_{1}}(a_{2})])$.

Let $p, r, t, v, c$ be as shown in Figure \ref{fig88} (iii)-(v). The curve $r$ is the unique nontrivial simple closed curve up to isotopy that is disjoint from all the curves in
$\{t, a_1, a_4\}$. Since $h([x]) = \lambda([x])$ for all
these curves and $\lambda$ preserves this property, we have $h([r]) = \lambda([r])$. The curve $p$ is the unique nontrivial simple closed curve up to isotopy that is disjoint from all the curves in
$\{a_2, a_3, a_4\}$. Since $h([x]) = \lambda([x])$ for all
these curves and $\lambda$ preserves this property, we have $h([p]) = \lambda([p])$. The curve $v$ is the unique nontrivial simple closed curve up to isotopy that is disjoint from all the curves in
$\{p, r, t\}$. Since $h([x]) = \lambda([x])$ for all
these curves and $\lambda$ preserves this property, we have $h([v]) = \lambda([v])$. The curve $c$ is the unique nontrivial simple closed curve up to isotopy that is disjoint from all the curves in
$\{v, a_1, a_5\}$. Since $h([x]) = \lambda([x])$ for all
these curves and $\lambda$ preserves this property, we have $h([c]) = \lambda([c])$. Since $h([a_3]) = \lambda([a_3]), h([c]) = \lambda([c])$ and $i(h([c]), h([a_3]))=2$, we have $i(\lambda([c]), \lambda([a_3]))=2$. 
There exists a homeomorphism sending the pair 
$(a_3, c)$ to the pair $(t, q)$ where $q$ is the curve shown in Figure \ref{fig88} (vi). Hence, we have $i(\lambda([t]), \lambda([q]))=2$. 

The curves $t, q$ are the only nontrivial simple closed curves up to isotopy that are disjoint from the curves in
$\{a_2, a_5\}$ and intersect each of $a_1, a_3, a_4$ once. Since these properties are preserved by $\lambda$, $h([t]) = \lambda([t])$ and $i(\lambda([t]), \lambda([q]))=2$, and so $\lambda([t]) \neq \lambda([q])$, we see that 
$h([q]) = \lambda([q])$. Consider the curves shown in Figure \ref{fig88} (vii). The curve $k$ is the unique nontrivial simple closed curve up to isotopy that is disjoint from all the curves in
$\{q, a_3, a_4\}$. Since $h([x]) = \lambda([x])$ for all
these curves and $\lambda$ preserves this property, we have $h([k]) = \lambda([k])$. 
The curve $t_{a_{2}}(a_{1})$ is the unique nontrivial simple closed curve up to isotopy that is disjoint from 
$a_4, a_5, k$ and intersects $a_1, a_2$ once. Since $h([x]) = \lambda([x])$ for all
these curves and $\lambda$ preserves these properties, we have $h([t_{a_{2}}(a_{1})]) = \lambda([t_{a_{2}}(a_{1})])$.
Now with similar arguments given above we can see that $h ([t_x(y)]) = \lambda ([t_x(y)])$ for all $x, y \in \{ a_1, a_2, \dots, a_{5}\}$.

The curve $t_{a_{1}}(t)$ is the unique nontrivial simple closed curve up to isotopy that is disjoint from each of 
$t_{a_{1}}(a_2), r, a_5$, see Figure \ref{fig88} (viii). Since $h([x]) = \lambda([x])$ for all
these curves and $\lambda$ preserves this property, we have $h([t_{a_{1}}(t)]) = \lambda([t_{a_{1}}(t)])$.
%The curve $t_{t}(a_{1})$ is the unique nontrivial simple closed curve up to isotopy that is disjoint from 
%$t_{a_{2}}(a_1), r, a_5$. Since $h([x]) = \lambda([x])$ for all
%these curves and $\lambda$ preserves this property, we have $h([t_{t} (a_1)]) = \lambda([t_{t}(a_1)])$.
The curve $t_{a_{3}}(t) = t_{a_{1}}(a_2)$, so $h([t_{a_3} (t)]) = \lambda([t_{a_3}(t)])$.
%The curve $t_{t}(a_{3})$ is the unique nontrivial simple closed curve up to isotopy that is disjoint from 
%$t_{a_{2}}(a_3), z, a_5$. Since $h([x]) = \lambda([x])$ for all
%these curves and $\lambda$ preserves this property, we have $h([t_{t} (a_3)]) = \lambda([t_{t}(a_3)])$.
%The curve $t_{t}(a_{4})$ is the unique nontrivial simple closed curve up to isotopy that is disjoint from 
%$t_{a_{5}}(a_4), a_2, a_3$. Since $h([x]) = \lambda([x])$ for all
%these curves and $\lambda$ preserves this property, we have $h([t_{t} (a_4)]) = \lambda([t_{t}(a_4)])$.
The curve $t_{a_{4}}(t)$ is the unique nontrivial simple closed curve up to isotopy that is disjoint from each of 
$t_{a_{4}}(a_5), r, a_2$. Since $h([x]) = \lambda([x])$ for all
these curves and $\lambda$ preserves this property, we have $h([t_{a_{4}}(t)]) = \lambda([t_{a_{4}}(t)])$. We get 
$h([t_x(y)]) = \lambda ([t_x(y)])$ for all $x \in \{ a_1, a_2, \dots, a_{5}\}$ and for all $y \in L_f$. 
Hence, we have $\lambda ([x]) = h ([x])$ for all $x \in L_f \cup f(L_f)$.\end{proof}
   
\begin{theorem} \label{A} Let $g \geq 2$ and $n \geq 0$. There exists a homeomorphism $h : R \rightarrow R$ such that $[h]_*(\alpha) = \lambda(\alpha)$
for every vertex $\alpha$ in $\mathcal{N}(R)$ and this homeomorphism is unique up to isotopy when $(g, n) \neq (2, 0)$.\end{theorem}

\begin{proof} Let $f \in G$. There exists $L_f \subset \mathcal{N}(R)$ which satisfies the statement of Lemma \ref{prop1e}. Let $\mathcal{X}= \mathcal{C}_1 \cup \mathcal{C}_2 \cup \mathcal{C}_3 \cup \big( \bigcup _{f \in G} (L_f \cup f(L_f) \big ))$. For each vertex $x$ in the nonseperating curve graph, there 
exists $r \in Mod_R$ and a vertex $y$ in the set $\mathcal{X}$ such that $r(y)=x$. By following the construction given in \cite{IrP1}, we let
$\mathcal{X}_1 = \mathcal{X}$ and
$\mathcal{X}_k = \mathcal{X}_{k-1} \cup (\bigcup _{f \in G} (f(\mathcal{X}_{k-1}) \cup f^{-1}(\mathcal{X}_{k-1})))$ when $k \geq 2$. We observe that 
$\bigcup _{k=1} ^{\infty} \mathcal{X}_k$ is the set of all vertices in  $\mathcal{N}(R)$. We will prove that $h([x]) = \lambda([x])$ for all $x \in \mathcal{X}_{k}$ for each $k \geq 1$. We will give the proof by induction on $k$. By using Lemma \ref{curves-III} and Lemma \ref{prop1e}, we see that $h([x]) = \lambda([x])$ for each
$x \in \mathcal{X}_1$. Assume that $h([x]) = \lambda([x])$ for all $x \in \mathcal{X}_{k-1}$ for some $k \geq 2$. Let $f \in G$. There exists a homeomorphism $h_f$ of $R$ such that $h_f([x]) = \lambda([x])$ 
for all $x \in f(\mathcal{X}_{k-1})$. We have $f(L_f) \subset \mathcal{X}_{k-1} \cap f(\mathcal{X}_{k-1})$. This implies that when $(g, n) \neq (2, 0)$ we have $h_f = h$ since $f(L_f)$ has trivial  pointwise stabilizer, and when $(g, n) = (2, 0)$ we have $h_f = h$ or 
$h_f = h \circ i$ where $i$ is the generator for the center of $Mod_R^*$ since the pointwise stabilizer of $f(L_f)$ is the center of $Mod_R^*$. Then we have   
$h_f([x]) = h([x])$ for each vertex $[x] \in \mathcal{N}(R)$ since  
$i([x])=[x]$ for each vertex $[x] \in \mathcal{N}(R)$. Hence, we have  
$h_f([x]) = h([x]) = \lambda([x])$ for each $x \in \mathcal{X}_{k-1} \cup f(\mathcal{X}_{k-1})$. Similarly, there exists a homeomorphism 
$h'_f$ of $R$ such that $h'_f([x]) = \lambda([x])$ for all $x \in f^{-1}(\mathcal{X}_{k-1})$. We have $L_f \subset \mathcal{X}_{k-1} \cap f^{-1}(\mathcal{X}_{k-1})$. This implies that when $(g, n) \neq (2, 0)$ we have $h'_f = h$ since $L_f$ has trivial 
pointwise stabilizer, and when $(g, n) = (2, 0)$ we have $h'_f = h$ or $h'_f = h \circ i$ since the pointwise stabilizer of $L_f$ is the center of $Mod_R^*$. Then we have  $h'_f([x]) = h([x])$ for each vertex $[x] \in \mathcal{N}(R)$. Hence, we have  
$h'_f([x]) = h([x]) = \lambda([x])$ for each $x \in \mathcal{X}_{k-1} \cup f^{-1}(\mathcal{X}_{k-1})$. So, we get $h([x]) = \lambda([x])$ for each $x \in \mathcal{X}_{k-1} \cup f(\mathcal{X}_{k-1}) \cup f^{-1}(\mathcal{X}_{k-1})$ for any $f \in G$. This gives us $h([x]) = \lambda([x])$ for each $x \in \mathcal{X}_k$. Hence, by induction 
$h([x]) = \lambda([x])$ for each $x \in \mathcal{X}_k$ for all $k \geq 1$. Since 
$\bigcup _{k=1} ^{\infty} \mathcal{X}_k$ is the set of all vertices in  $\mathcal{N}(R)$, we have 
$h([x]) = \lambda([x])$ for every vertex $[x] \in \mathcal{N}(R)$. We see that this homeomorphism is unique up to isotopy when $(g, n) \neq (2, 0)$. When $(g, n) = (2, 0)$, we also have $(h \circ i) ([x]) =\lambda([x])$ for every vertex $[x]$ in 
$\mathcal{N}(R)$.\end{proof}
 
\section{Edge-preserving Maps of the Curve Graphs}

In this section we prove that an edge-preserving map $\theta : \mathcal{C}(R) \rightarrow \mathcal{C}(R)$ restricts to an 
edge-preserving map $\theta|_{\mathcal{N}(R)} : \mathcal{N}(R) \rightarrow \mathcal{N}(R)$. By Theorem \ref{A}, we know that 
there is a homeomorphism $h$ that induces $\theta|_{\mathcal{N}(R)}$. We will prove that $\theta$ is induced by $h$ as well.  
 
\begin{figure}[htb]
\begin{center}
\hspace{-0.4cm} \epsfxsize=1.6in \epsfbox{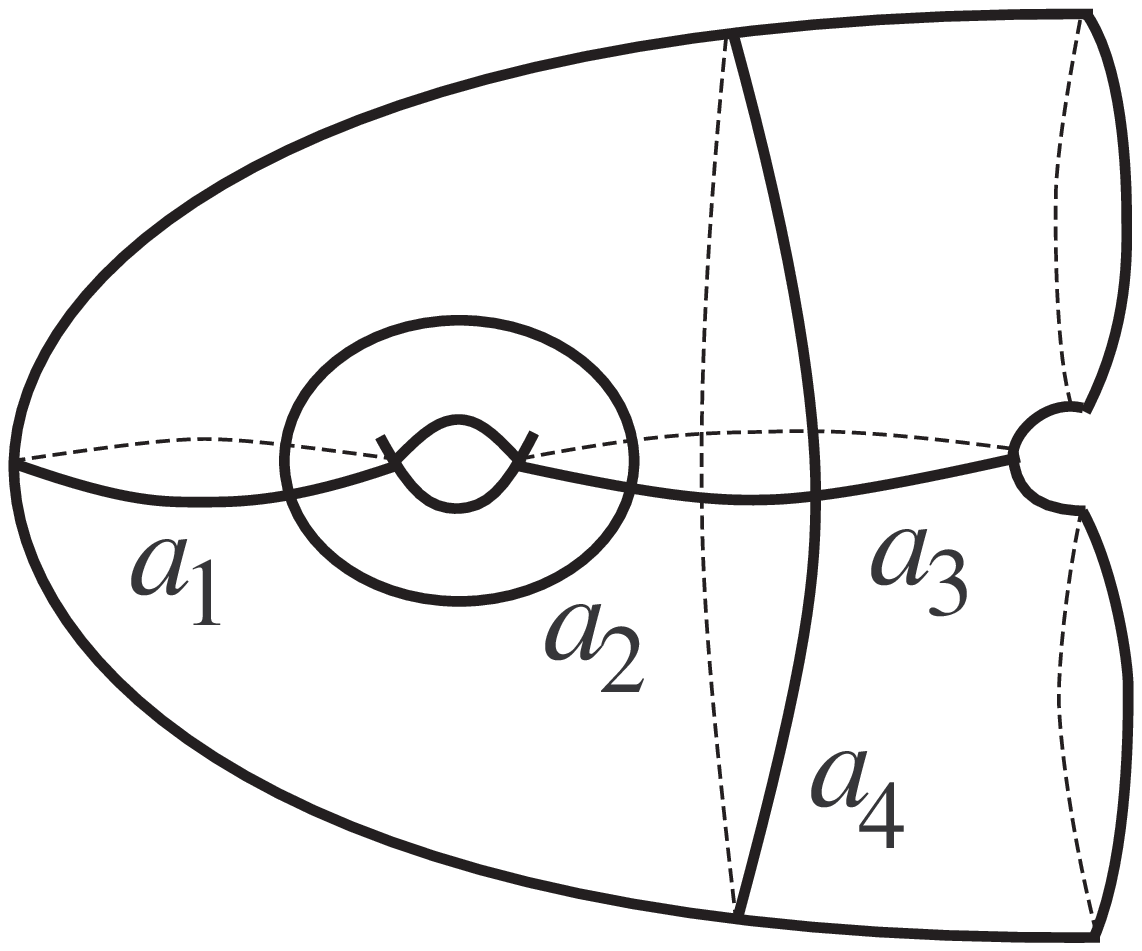}  
   
\caption{Curves $a_1, a_2$ intersecting once} \label{figure-90}
\end{center}
\end{figure}

The proof of the following lemma was given in \cite{H2} when $g \geq 3$. The same proof works for $g=2$ case as well. 

\begin{lemma}
\label{22} Suppose $g \geq 2, n \geq 0$. Let $\alpha_1, \alpha_2$ be two vertices of $\mathcal{C}(R)$ such that $i(\alpha_1, \alpha_2) = 1$. Then $i(\theta(\alpha_1), \theta(\alpha_2)) \neq 0$.\end{lemma}
 
\begin{proof} Let $a_1, a_2$ be representatives in minimal position of $\alpha_1, \alpha_2$ respectively. We complete $a_1, a_2$ to a curve configuration $\{a_1, a_2, a_3, a_4\}$ as shown in Figure \ref{figure-90}. Then we complete $\{a_1, a_4\}$ to a pants decomposition $P$ on $R$ such that all the elements of $P \setminus \{a_4\}$ are disjoint from $a_3$. Let $P'$ be a set of pairwise disjoint representatives of $\lambda([P])$. The set $P'$ is a pants decomposition on $R$. We see that $i([a_2], [x]) = 0$ for all $x \in P \setminus \{a_1\}$ and there is an edge between $[a_2]$ and $[x]$ for all $x \in P \setminus \{a_1\}$. Since 
$\theta$ is edge-preserving we have $i(\theta([a_2]), \theta([x])) = 0$ for all $x \in P \setminus \{a_1\}$ and there is an edge between 
$\theta([a_2])$ and $\theta([x])$ for all $x \in P \setminus \{a_1\}$. This implies that either $i(\theta([a_1]), \theta([a_2])) \neq 0$ or 
$\theta([a_1]) = \theta([a_2])$. With a similar argument we can see that either $i(\theta([a_3]), \theta([a_2])) \neq 0$ or 
$\theta([a_3]) = \theta([a_2])$. If $\theta([a_1]) = \theta([a_2])$ then we couldn't have $i(\theta([a_3]), \theta([a_2])) \neq 0$ and
$\theta([a_3]) = \theta([a_2])$ since $\theta$ is edge-preserving. Hence, $i(\theta(\alpha_1), \theta(\alpha_2)) \neq 0$.\end{proof}\\

By using Lemma \ref{22}, we can obtain the proofs of the following lemmas by following the proofs of Lemma \ref{adj}, Lemma \ref{embedded}, Lemma \ref{int-one}.  

\begin{lemma}
\label{adj-2} Suppose $g \geq 2$ and $n \geq 0$. Let $P = \{a_1, a_3, \cdots, a_{2g+1}, b_4, b_6, \cdots, b_{2g-2}, c_4,$ 
$c_6, \cdots, c_{2g-2}\}$ when $R$ is closed, and  
$P = \{a_1, a_3, \cdots, a_{2g-1}, b_4, b_6, \cdots, b_{2g}, c_4, c_6,$ 
$\cdots, c_{2g},$ $d_1, d_2,$ $ \cdots, d_{n-1}\}$ when $R$ has boundary
where the curves are as shown in Figure \ref{fig1} (i)-(ii). Let $P'$ be a pair of pants
decomposition of $R$ such that $\theta([P]) = [P']$. Let $a'_1, a'_3, b'_4$ be the representatives of $\theta([a_1]), \theta([a_3]),
\theta([b_4])$ in $P'$ respectively. Then any two of $a'_1, a'_3, b'_4$ are adjacent to each other with respect to $P'$.\end{lemma}

\begin{lemma}
\label{embedded-2} Suppose $g \geq 2$ and $n \geq 0$. If $\alpha, \beta, \gamma$ are distinct vertices in $\mathcal{C}(R)$ having pairwise disjoint nonseparating representatives which bound a pair of pants on $R$, then $\theta(\alpha), \theta(\beta), \theta(\gamma)$ 
are distinct vertices in $\mathcal{C}(R)$ having pairwise disjoint representatives which bound a pair of pants on $R$.\end{lemma}

\begin{lemma} \label{int-one-2} Suppose $g \geq 2$ and $n \geq 0$.  
Let $\alpha_1, \alpha_2$ be two vertices of $\mathcal{C}(R)$. If $i(\alpha_1, \alpha_2) = 1$, 
then $i(\theta(\alpha_1), \theta(\alpha_2) )= 1$.\end{lemma}

\begin{figure}[htb]
	\begin{center}
		\hspace{-0.4cm} \epsfxsize=2.67in \epsfbox{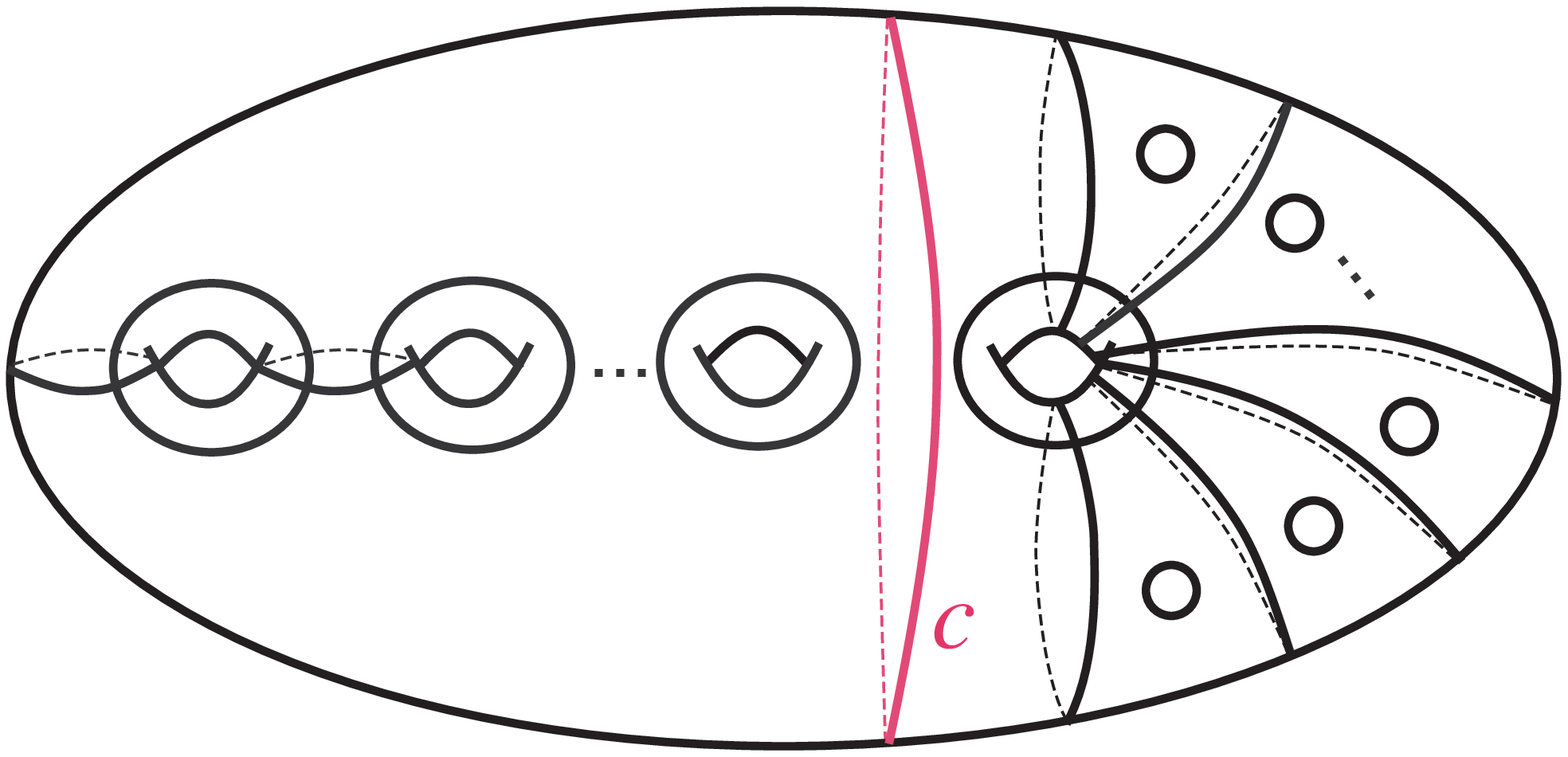} \hspace{0.2cm} \epsfxsize=2.67in \epsfbox{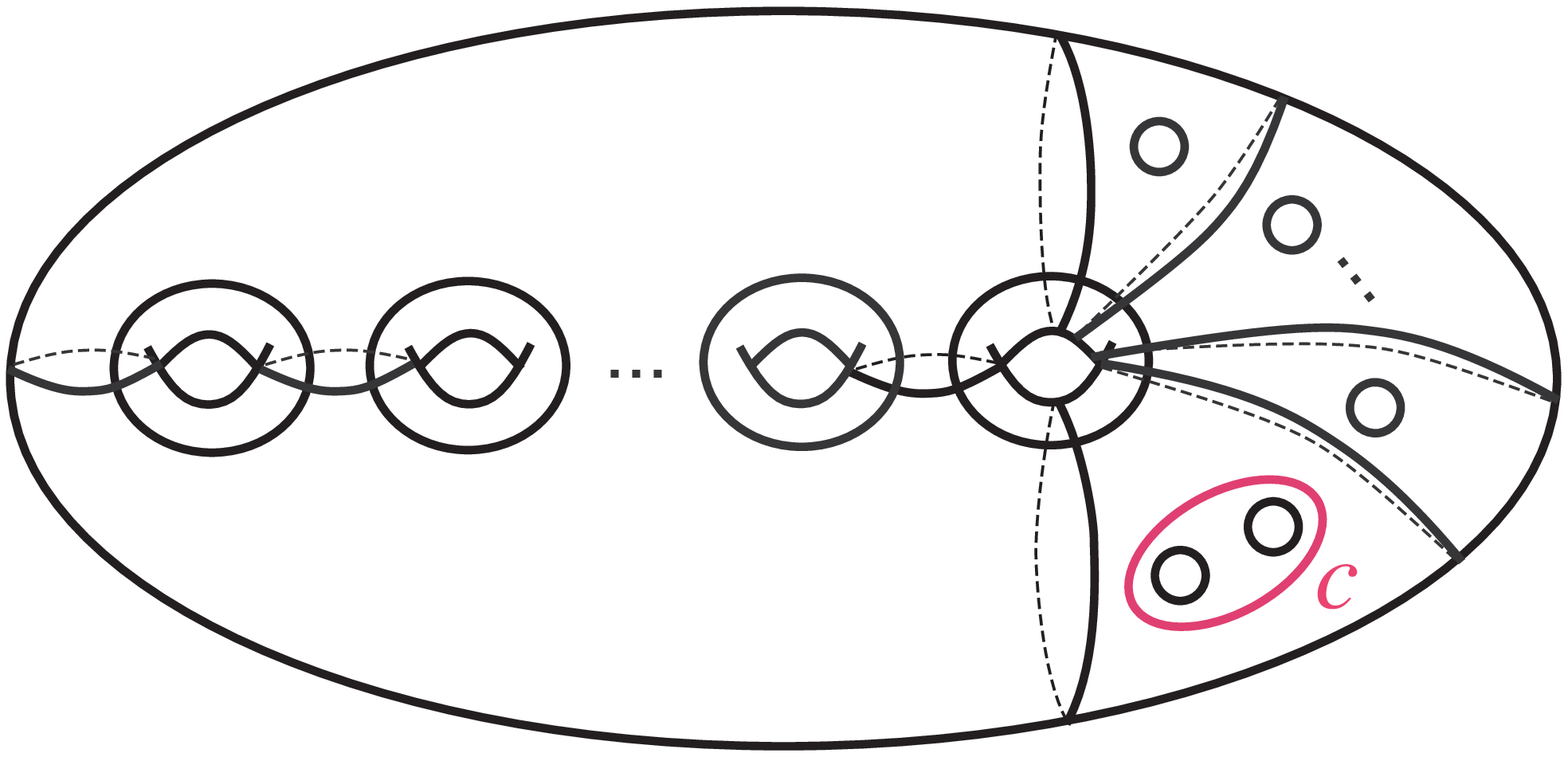} 
		
		\hspace{-1cm} (i) \hspace{6.7cm} (ii)
		
		\hspace{-0.4cm} \epsfxsize=2.67in \epsfbox{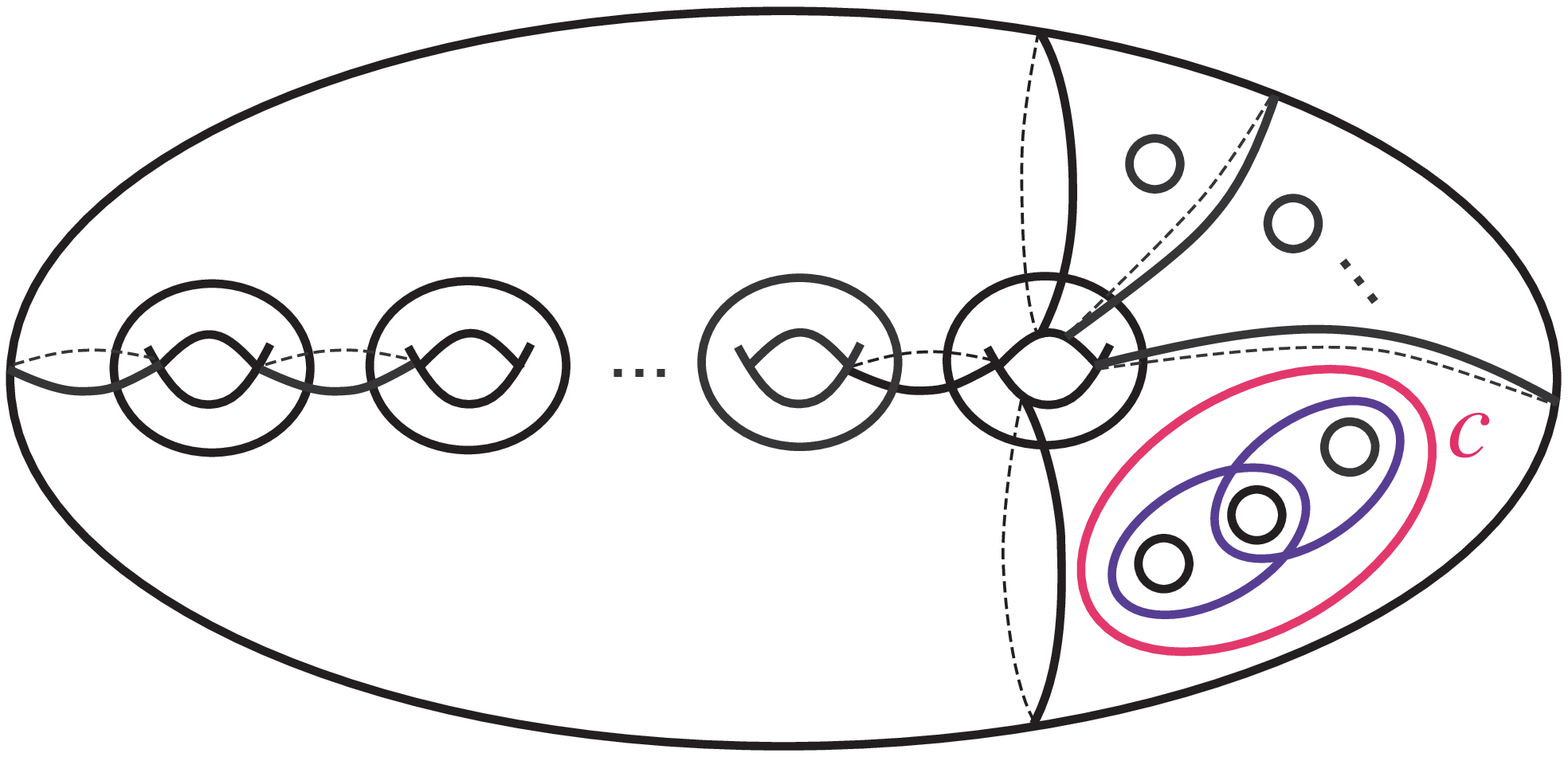} \hspace{0.2cm} \epsfxsize=2.67in \epsfbox{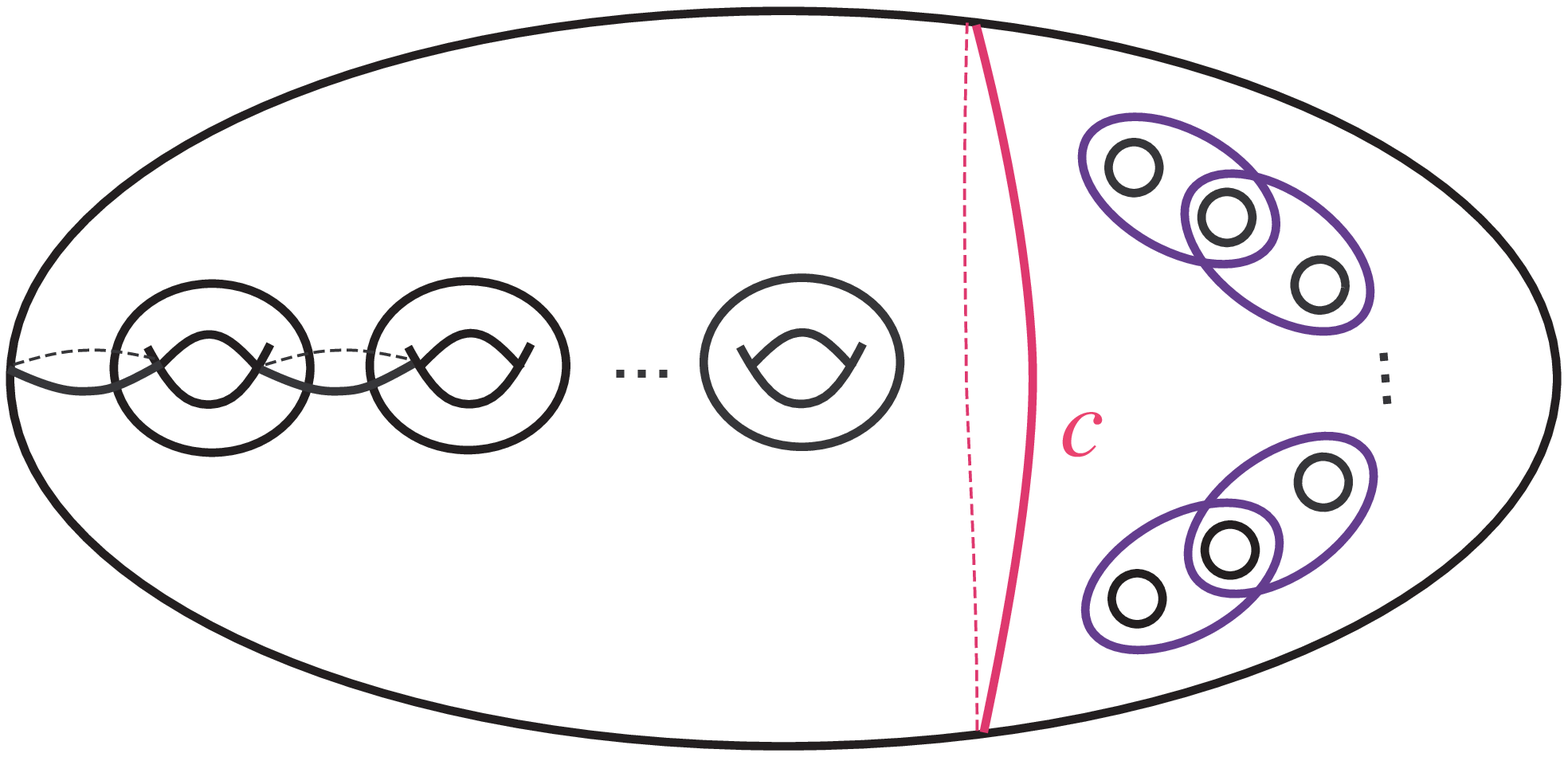} 
		
		\hspace{-0.8cm} (iii) \hspace{6.3cm} (iv)
		
		\caption{Curve configurations for separating curves} \label{fig100}
	\end{center}
\end{figure}

\begin{theorem} Let $g \geq 2$ and $n \geq 0$. If $\theta :\mathcal{C}(R) \rightarrow \mathcal{C}(R)$ is an edge-preserving map, then there exists a homeomorphism $h : R \rightarrow R$ such that $[h]_*(\alpha) = \theta(\alpha)$ for every 
vertex $\alpha$ in $\mathcal{C}(R)$ and this homeomorphism is unique 
up to isotopy when $(g, n) \neq (2, 0)$.\end{theorem}

\begin{proof} Let $a$ be a nonseparating simple closed curve on $R$. We choose another nonseparating simple closed curve $b$ such that $i([a], [b]) =1$. 
By Lemma \ref{int-one-2} we have $i(\theta([a]), \theta([b])) =1$. This implies that $\theta([a])$ is the class of a nonseparating curve. So, $\theta$ restricts to an edge-preserving map $\theta|_{\mathcal{N}(R)} : \mathcal{N}(R) \rightarrow \mathcal{N}(R)$.  By Theorem \ref{A}, there is a homeomorphism $h : R \rightarrow R$ that induces $\theta|_{\mathcal{N}(R)}$, and this homeomorphism is unique 
up to isotopy when $(g, n) \neq (2, 0)$. We will prove that $\theta$ is induced by $h$ as well. Let $c$ be a separating simple closed curve on $R$. If both of the connected components of the complement of $c$ have genus at least one, then we can find a curve configuration $\mathcal{B}$ consisting of only nonseparating simple closed curves on $R$ such that $c$ is the unique nontrivial simple closed curve up to isotopy disjoint from all the curves in $\mathcal{B}$, see Figure \ref{fig100} (i). Since we know that $h([x]) = \theta([x])$ for all the curves in $\mathcal{B}$ and $\theta$ is edge-preserving we have $h([c]) = \theta([c])$. If $c$ and two boundary components of $R$ bound a pair of pants as shown in Figure \ref{fig100} (ii), then we can find a curve configuration $\mathcal{B}$ consisting of only nonseparating simple closed curves on $R$ such that $c$ is the unique nontrivial simple closed curve up to isotopy disjoint from all the curves in $\mathcal{B}$ as shown in Figure \ref{fig100} (ii). Since we know that $h([x]) = \theta([x])$ for all the curves in $\mathcal{B}$ and $\theta$ is edge-preserving we have $h([c]) = \theta([c])$. If $c$ cuts a sphere with at least four boundary components, then we can find a curve configuration $\mathcal{B}$ consisting of only nonseparating simple closed curves and separating simple closed curves that 
cut a pair of pants on $R$ such that $c$ is the unique nontrivial simple closed curve up to isotopy disjoint from all the curves in $\mathcal{B}$, see Figure \ref{fig100} (iii) and (iv). Since we know that $h([x]) = \theta([x])$ for all the curves in $\mathcal{B}$ and $\theta$ is edge-preserving we have $h([c]) = \theta([c])$. So, we see that $h([x]) = \theta([x])$ for every  
vertex $[x]$ in $\mathcal{C}(R)$. Hence, $\theta$ is induced by $h$.\end{proof}
 
\section{Rectangle Preserving Maps of the Hatcher-Thurston Graphs}

We remind that if $C= \{c_1, c_2, \ldots, c_g\}$ is a set of pairwise disjoint nonseparating simple
closed curves on $R$ such that $R_C$ is a sphere with
$2g+n$ boundary components, then the set $\{ [c_1], [c_2], \cdots,
[c_g] \}$ is called a cut system and denoted by $\langle [c_1], [c_2],
\cdots, [c_g] \rangle$. Let $v$ and $w$ be two cut systems. Suppose that there are $[c] \in
v$ and $[d] \in w$ such that $i([c], [d])=1$ and $v \setminus\{ [c] \}= w \setminus \{ [d] \}$.
We will say that $w$ is obtained from $v$ by an elementary move and  
write $v \leftrightarrow w$.
If $\langle [c_1], [c_2], \cdots, [c_i], \cdots, [c_g] \rangle
\leftrightarrow \langle [c_1], [c_2], \cdots, [c'_i], \cdots, [c_g] \rangle$
is an elementary move, then we will drop the unchanged curves from the
notation and write $\langle [c_i] \rangle \leftrightarrow \langle
[c'_i] \rangle$.

Let $\mathcal{HT}(R)$ be the Hatcher-Thurston graph which has the cut systems
as vertices, and two vertices $v$ and $w$ span an edge if $w$ is obtained from $v$ by an elementary move. A sequence of cut systems $(v_1, \ldots, v_n)$ forms a path in
$\mathcal{HT}(R)$ if every consecutive pair in the sequence is connected
by an edge in $\mathcal{HT}(R)$. There
are three types of distinguished closed paths: triangles, rectangles and pentagons.
If three vertices have $g-1$ common elements and if the remaining classes
$[c], [c'], [c'']$ satisfy $i([c], [c'])=i([c], [c''])=i([c'], [c''])=1$, then they form a {\it triangle} 
$\langle c \rangle \leftrightarrow \langle c' \rangle
\leftrightarrow \langle c'' \rangle \leftrightarrow \langle c
\rangle$ (see Figure \ref{fig8} (i)). If four vertices have $g-2$ common elements and if the remaining classes
$[c_1], [c_2], [d_1], [d_2]$ have representatives $c_1, c_2, d_1, d_2$ given as in
Figure \ref{fig8} (ii), then they form a {\it rectangle} $\langle [c_1], [d_1]
\rangle \leftrightarrow \langle [c_1], [d_2] \rangle \leftrightarrow
\langle [c_2], [d_2] \rangle \leftrightarrow \langle [c_2], [d_1] \rangle
\leftrightarrow \langle [c_1], [d_1] \rangle $. If five vertices have $g-2$ common elements and if the remaining classes
$[c_1], [c_2], [c_3], [c_4], [c_5]$ have representatives $c_1, c_2, c_3, c_4, c_5$
intersecting each other as in Figure \ref{fig8} (iii), then they form a {\it pentagon} $\langle c_1,c_4 \rangle \leftrightarrow \langle
c_2,c_4 \rangle \leftrightarrow\langle c_2,c_5 \rangle
\leftrightarrow \langle c_3,c_5 \rangle \leftrightarrow\langle
c_1,c_3 \rangle  \leftrightarrow \langle c_1,c_4 \rangle$. The graph $\mathcal{HT}(R)$
is the $1$-skeleton of the Hatcher-Thurston complex which is a
two-dimensional CW-complex obtained from $\mathcal{HT}(R)$ by attaching
a $2$-cell along each triangle, rectangle and pentagon.

We say that a map $\tau: \mathcal{HT}(R) \rightarrow \mathcal{HT}(R)$ is rectangle preserving if it sends rectangles to rectangles. 
A map $\tau: \mathcal{HT}(R) \rightarrow \mathcal{HT}(R)$ is called alternating if the restriction to the star of any vertex maps cut systems that differ in exactly two curves to cut systems that differ in exactly two curves. In \cite{H3}, Hern\'andez proved that if $\tau : \mathcal{HT}(R) \rightarrow \mathcal{HT}(R)$ is an alternating, edge-preserving map then $\tau$ is induced by a homeomorphism when $g \geq 3$, $n =0$. Our Theorem \ref{B} proves this statement when $g \geq 2$ and $n = 0$ since being edge-preserving and alternating is equivalent to sending rectangles to rectangles in $\mathcal{HT}(R)$. 

\begin{figure}[htb]
\begin{center}
\hspace{0.3cm} \epsfxsize=1.45in \epsfbox{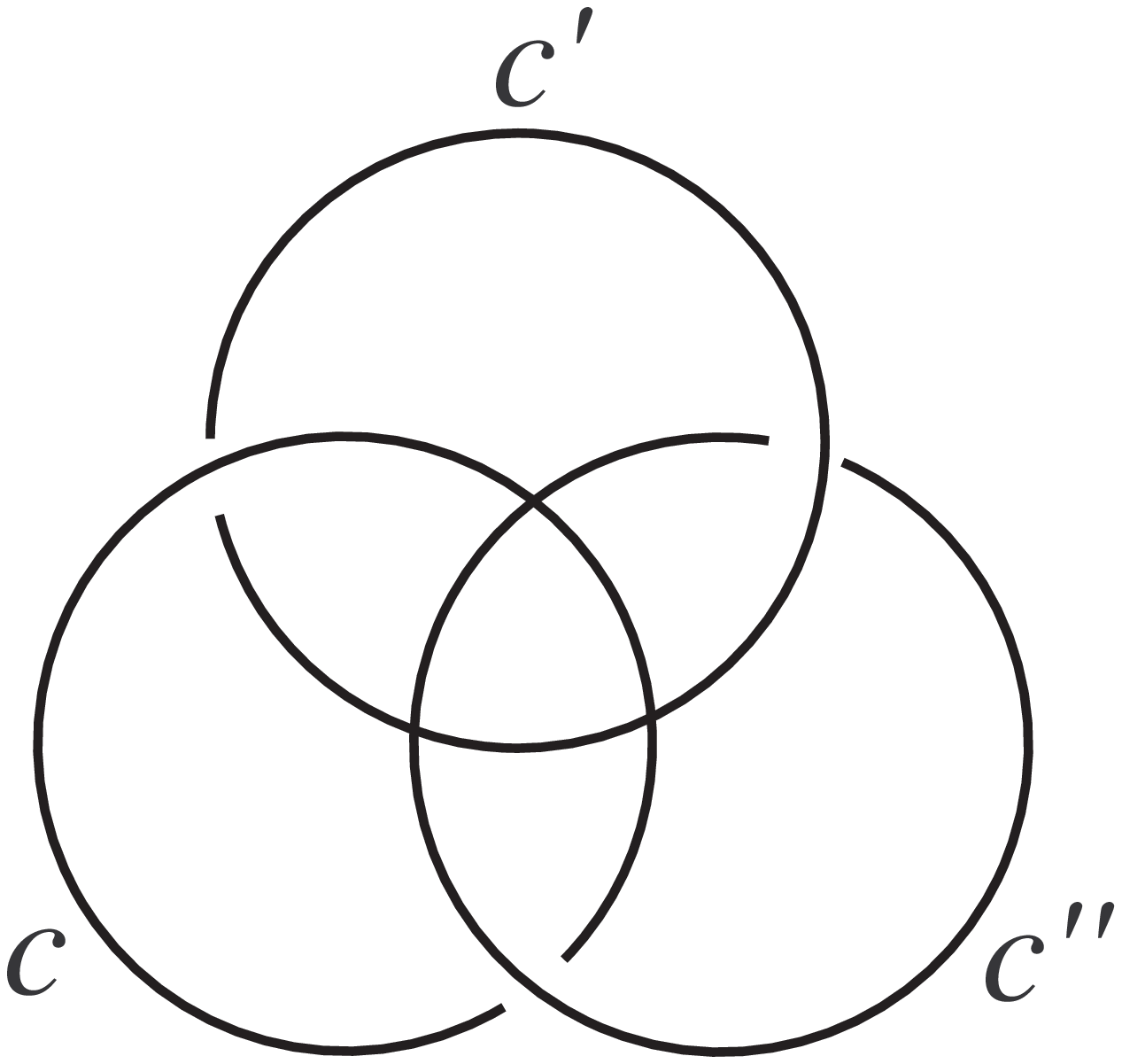}  \hspace{0.3cm} \epsfxsize=1.85in \epsfbox{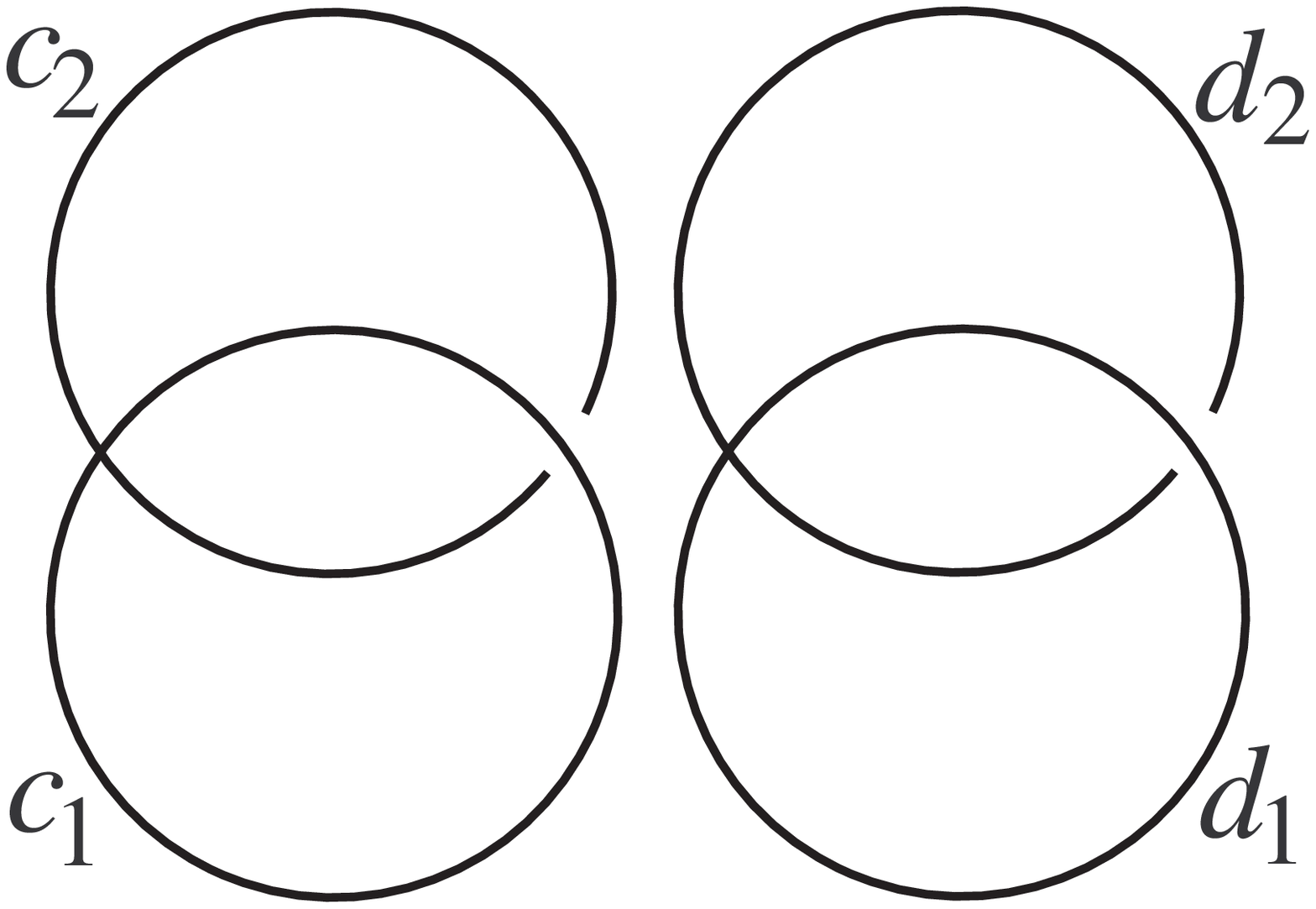} \hspace{0.1cm}  \epsfxsize=1.95in \epsfbox{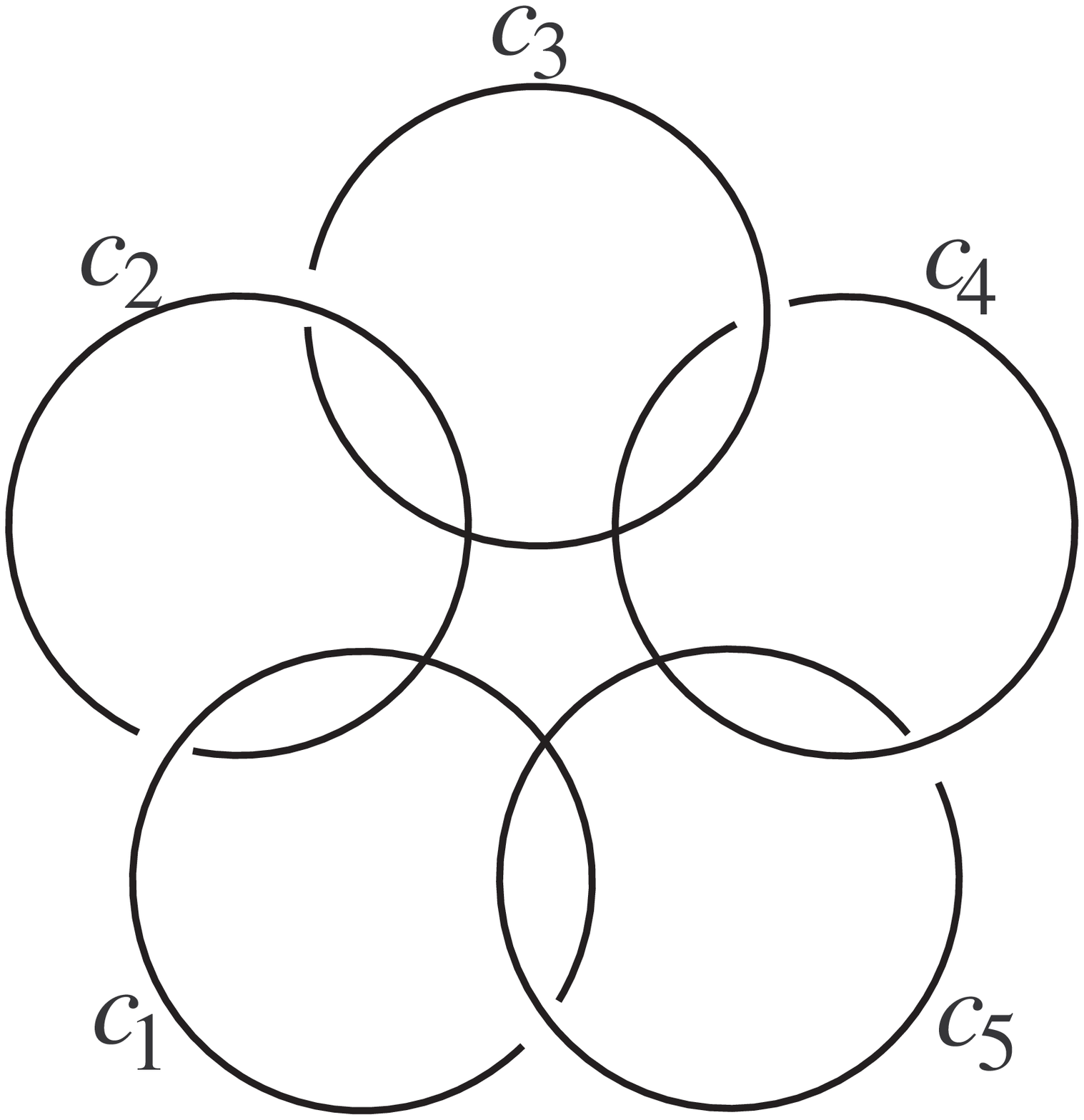}

\hspace{-0.4cm} (i) \hspace{4cm} (ii) \hspace{4.2cm} (iii)
    
\caption{A triangle, a rectangle and a pentagon}\label{fig8}
\end{center}
\end{figure}

Let $\tau: \mathcal{HT}(R) \rightarrow\mathcal{HT}(R)$ be a rectangle preserving map. Since every edge in $\mathcal{HT}(R)$ can be completed to a rectangle, $\tau$ is edge-preserving. It is easy to see that $\tau$ sends triangles to triangles. We will define $\tilde{\tau}: \mathcal{N}(R) \rightarrow\mathcal{N}(R)$ as follows:  Let $a$ be a nonseparating simple closed curve on $R$. We choose pairwise disjoint nonseparating simple closed curves $a_2, a_3, \cdots a_g$ on $R$ such that 
$v = \langle [a], [a_2], \cdots, [a_g] \rangle$ is a cut system. We choose another curve $b$ on $R$ such that $a$ and $b$ are dual (i.e. they have tranverse intersection one) and $b$ does not intersect any of $a_i$. Then $w = \langle [b], [a_2], \cdots, [a_g] \rangle$ is also a cut system and the vertices $v$ and $w$ are connected by an edge in $\mathcal{HT}(R)$. Since $\tau$ is edge-preserving, the vertices $\tau(v)$ and $\tau(w)$ are connected by an edge as well. So $\tau(v) \setminus \tau(w)$ contains only one element. We define $\tilde{\tau}([a])$ to be this unique class. Well-definedness of $\tilde{\tau}([a])$, the independence from all the choices made were shown in \cite{IrK} when $\tau$ is an automorphism  of the Hatcher-Thurston complex. Since $\tau$ is edge-preserving, triangle preserving and rectangle preserving on the Hatcher-Thurston graph, the same construction works for $\tau$ as well. So, $\tilde{\tau} : \mathcal{N}(R) \rightarrow \mathcal{N}(R)$ is a well-defined map. 

\begin{lemma} \label{lemma-1}: $\tilde{\tau}$ preserves
geometric intersection one property.
\end{lemma}

\begin{proof} If $a$ and $b$ are two nonseparating simple closed curves on $R$ such that
$i([a], [b]) =1$, then $\{a\}$ can be completed to a cut system $v= \langle [a], [a_2], [a_3], \cdots, [a_g] \rangle$ such that 
$w= \langle [b], [a_2], [a_3], \cdots, [a_g] \rangle$ is also a cut system. We see that $v$ and $w$ are connected by an edge in 
$\mathcal{HT}(R)$. Since $\tau$ is edge-preserving, $\tau(v)$ and $\tau(w)$ are connected by an edge in $\mathcal{HT}(R)$. By the definition of $\tilde{\tau}$
we have $\tilde{\tau}([a]) = \tau(v) \setminus \tau(w)$ and $\tilde{\tau}([b]) = \tau(w) \setminus \tau(v)$. Since $\tau(v)$ and $\tau(w)$ are connected by an edge, we get  $i(\tilde{\tau}([a]), \tilde{\tau}([b]))=1$. 
\end{proof} 

\begin{figure}[t] \begin{center}
		\epsfxsize=2.99in \epsfbox{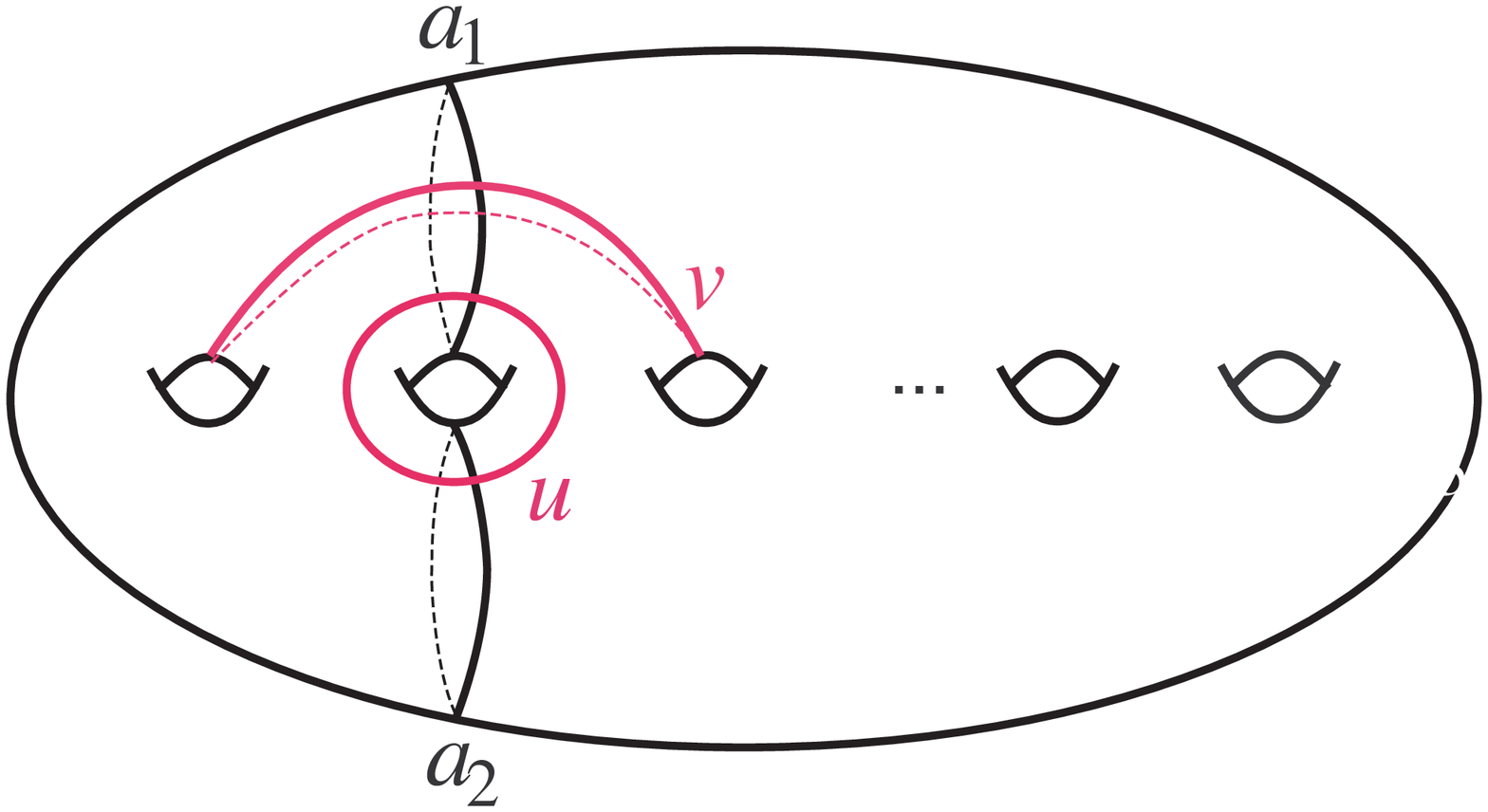} 
		\hspace{0cm} \epsfxsize=2.99in \epsfbox{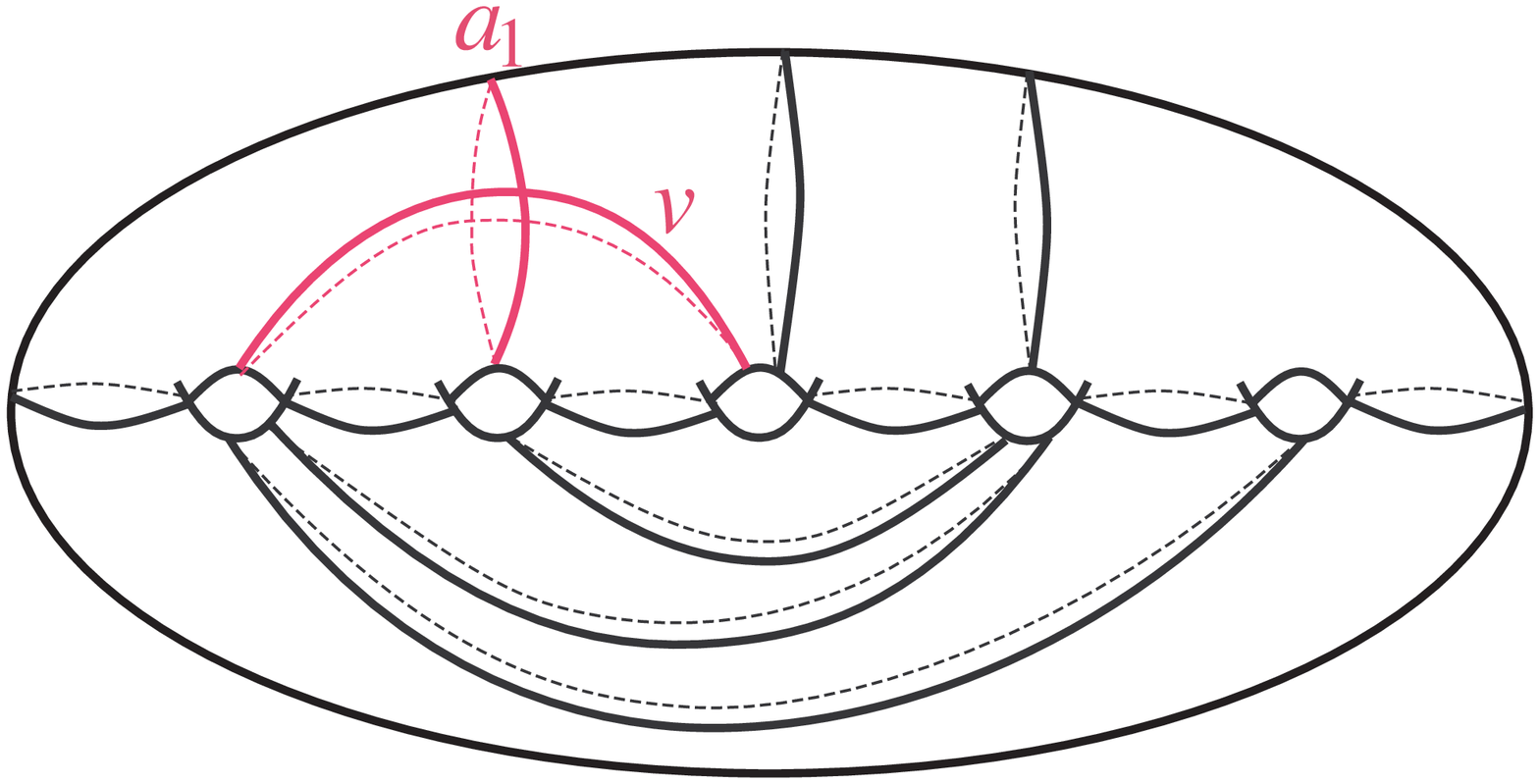} 
		
		(i)  \hspace{7cm} (ii)
		
			\hspace{0cm} \epsfxsize=2.99in \epsfbox{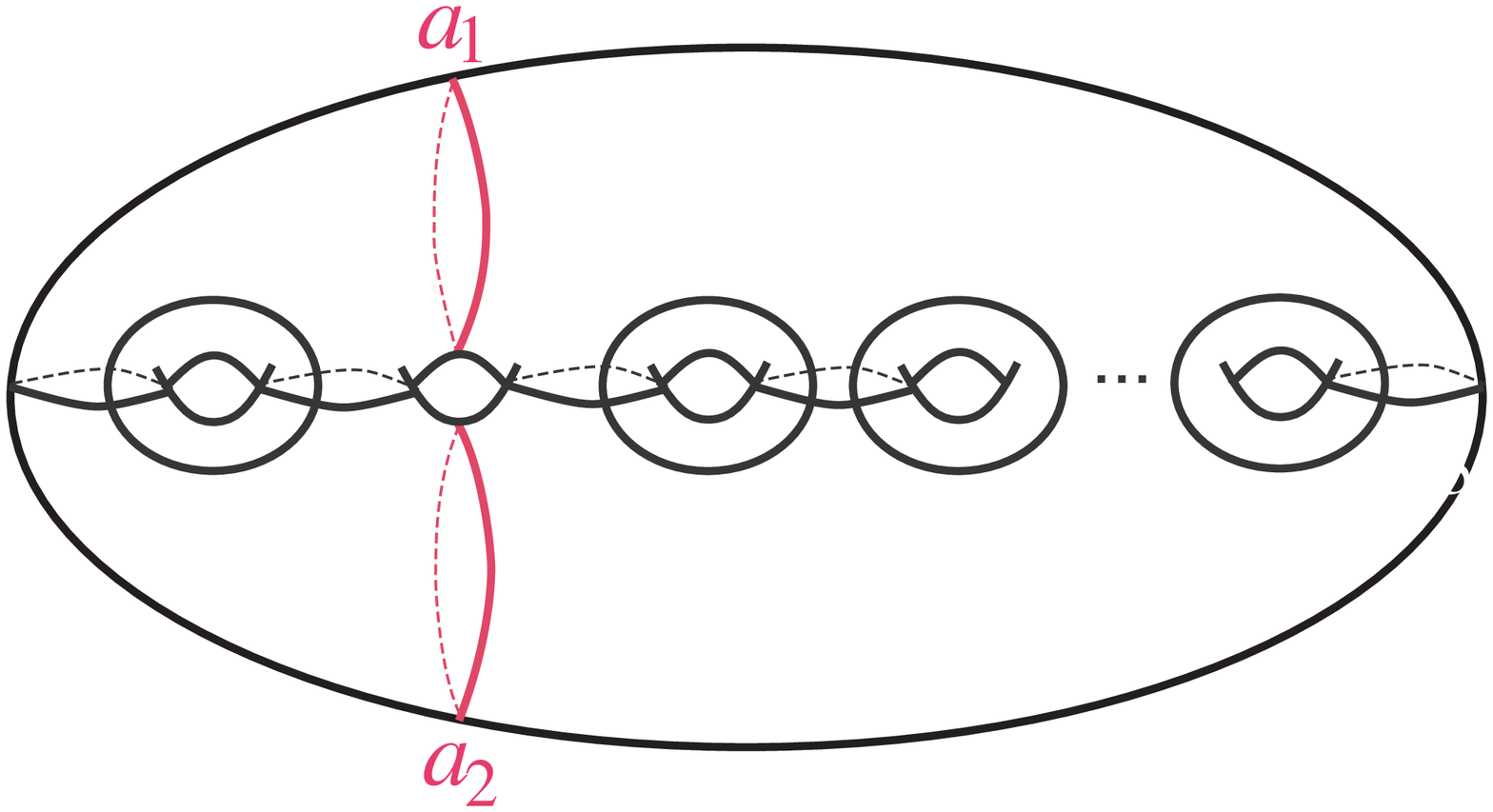} 
	 
	 (iii)
		
		\caption{Curves $a_1, a_2$ separating together} \label{fig-last}
	\end{center}
\end{figure}
 
\begin{lemma} \label{lemma-2}: $\tilde{\tau}$ is edge-preserving. 
\end{lemma} 
 
\begin{proof} Let $a_1$ and $a_2$ be two disjoint, nonisotopic, nonseparating simple closed curves on $R$. 
	
Case (i): Suppose $R_{a_1 \cup a_2}$ is connected. In this case  $\{a_1, a_2\}$ can be completed to a cut system $w= \langle [a_1], [a_2], \cdots, [a_g] \rangle$. Then $\tilde{\tau}([a_1])$ and $\tilde{\tau}([a_2])$ both belong to $\tau(w)$ by definition. This implies that $\tilde{\tau}([a_1])$ and $\tilde{\tau}([a_2])$ are connected by an edge in $\mathcal{N}(R)$. 
	
Case (ii): Suppose $R_{a_1 \cup a_2}$ is not connected. This case happens only when $g\geq 3$. Assume $a_1, a_2$ separate the surface together as shown in Figure \ref{fig-last} (i). Consider the curves $u, v$ as shown in Figure \ref{fig-last} (i). Since $a_1$ and $u$ intersect once, by Lemma \ref{lemma-1} we know that $i(\tilde{\tau}([a_1]), \tilde{\tau}([u]))=1$. Since $R_{u \cup v}$ is connected and $u$ and $v$ are disjoint curves, by case (i) we know that $i(\tilde{\tau}([u]), \tilde{\tau}([v]))=0$. These imply that $\tilde{\tau}([a_1]) \neq \tilde{\tau}([v])$. 

We can find a set $P$ of pairwise disjoint nonseparating simple closed curves such that $P_1= \{a_1\} \cup P$ and $P_2= \{v\} \cup P$ are two pair of pants decompositions such that in each of them no two curves separate the surface together (see Figure \ref{fig-last} (ii) for genus $g=5$ case). By using case (i) we can see that 
$\tilde{\tau}(P_1)$ and $\tilde{\tau}(P_2)$ correspond to two maximal complete subgraphs in $\mathcal{N}(R)$. Then, since $\tilde{\tau}([a_1]) \neq \tilde{\tau}([v])$, we see that  $i(\tilde{\tau}([a_1]), \tilde{\tau}([v])) \neq 0$. Since $R_{a_2 \cup v}$ is connected and $a_2$ and $v$ are disjoint curves, by case (i) we know that $i(\tilde{\tau}([a_2]), \tilde{\tau}([v]))=0$.     Since  $i(\tilde{\tau}([a_2]), \tilde{\tau}([v]))=0$ and $i(\tilde{\tau}([a_1]), \tilde{\tau}([v])) \neq 0$, we conclude that 
$\tilde{\tau}([a_1]) \neq \tilde{\tau}([a_2])$. Similarly, if $a_1, a_2$ separate the surface together in a different way, we can see that $\tilde{\tau}([a_1]) \neq \tilde{\tau}([a_2])$. 

Now we will see that $\tilde{\tau}([a_1])$ and $\tilde{\tau}([a_2])$ are connected by an edge in $\mathcal{N}(R)$ as follows:
On the complement of $a_1, a_2$ we take two chains $C_1, C_2$ as shown in Figure \ref{fig-last} (iii). We can see that $R_{x \cup y}$ is connected for each pair $x, y \in C_1 \cup C_2$, $R_{x \cup a_1}$ is connected for each $x \in C_1 \cup C_2$ and $R_{x \cup a_2}$ is connected for each $x \in C_1 \cup C_2$. By using case (i) and Lemma \ref{lemma-1}, we can see that chains will go to chains. This means that we can choose represantatives in minimal position of the sets $\{\tilde{\tau}([x]): x \in C_1 \}$ and $\{\tilde{\tau}([x]): x \in C_2 \}$, say $C'_1$ and $C'_2$ respectively such that $C'_1$ and $C'_2$ are two chains. Let $a'_1, a'_2$ be representatives of $\tilde{\tau}([a_1])$ and $\tilde{\tau}([a_2])$ such that they intersect minimally with the elements of $C'_1$ and $C'_2$. The complement of disjoint regular neighborhoods of $C'_1$ and $C'_2$ are two disjoint annuli. It is easy to see that $a'_1, a'_2$ have to be in these two annuli. Since $\tilde{\tau}([a_1]) \neq \tilde{\tau}([a_2])$, we know that $a'_1$ and $a'_2$ are not isotopic to each other. Hence, $a'_1$ and $a'_2$ will be in different annuli, so they will be disjoint (if $a_1$ and $a_2$ together separate the surface in a different way, we can take similar chains to see that $a'_1$ and $a'_2$ will be disjoint). Hence, $\tilde{\tau}([a_1])$ and $\tilde{\tau}([a_2])$ are connected by an edge in $\mathcal{N}(R)$.\end{proof}

\begin{theorem} \label{B} Let $g \geq 2$ and $n = 0$. If $\tau: \mathcal{HT}(R) \rightarrow\mathcal{HT}(R)$ is a rectangle preserving map, then $\tau$ is induced by a homeomorphism of $R$, i.e. there exists a homeomorphism $h$ of $R$ such that $\tau( \langle [a_1], [a_2], \cdots, [a_g] \rangle) = \langle [h(a_1)], [h(a_2)], \cdots, [h(a_g)] \rangle $ for every vertex $ \langle [a_1], [a_2], \cdots, [a_g] \rangle \in \mathcal{HT}(R)$, and this homeomorphism is unique 
up to isotopy when $(g, n) \neq (2, 0)$.\end{theorem}

\begin{proof} By Lemma \ref{lemma-2}, we know that $\tau$ induces an edge-preserving map $\tilde{\tau} : \mathcal{N}(R) \rightarrow \mathcal{N}(R)$. By Theorem \ref{A}, $\tilde{\tau}$ is induced by a homeomorphism $h : R \rightarrow R$, and this homeomorphism is unique up to isotopy when $(g, n) \neq (2, 0)$. We have $h([x]) =\tilde{\tau}([x])$ for every vertex 
$[x] \in \mathcal{N}(R)$. Let $v = \langle [a_1], [a_2], \cdots, [a_g] \rangle$ be a vertex in $\mathcal{HT}(R)$
where $a_i$'s are pairwise disjoint. Then we see that 
$\tau(v) = \langle \tilde{\tau}([a_1]), \tilde{\tau}([a_2]), \cdots, \tilde{\tau}([a_g]) \rangle $. So,  
$\tau(v) = \langle [h(a_1)], [h(a_2)], \cdots, [h(a_g)] \rangle $. Hence, $\tau$ is induced by $h$.\end{proof}\\

{\bf Acknowledgements}\\

The author thanks Peter Scott for some discussions and his comments about this paper. The author also thanks referees for their comments.

{\bf Elmas Irmak} \vspace{0.2cm}
 
University of Michigan, Department of Mathematics, Ann Arbor, 48105, MI, USA 

e-mail: eirmak@umich.edu\\
\end{document}